\documentclass[11pt,a4paper]{article}
\usepackage[american]{babel}
\usepackage{hyperref}
\usepackage{a4}
\usepackage{dsfont}
\usepackage{amsmath, amssymb, amsthm}
\usepackage{graphicx}
\usepackage{subfig}
\usepackage{pgf,tikz}
\usepackage{mathrsfs}
\usepackage{overpic}
\usepackage{contour}\contourlength{0.2ex}
\usetikzlibrary{positioning,arrows,patterns}

\renewcommand{\vec}[1]{{\bf{#1}}}

\newtheorem{theorem}{Theorem}
\newtheorem{proposition}[theorem]{Proposition}
\newtheorem{lemma}[theorem]{Lemma}
\newtheorem{corollary}[theorem]{Corollary}

\theoremstyle{definition}

\newtheorem*{definition}{Definition}

\newtheorem*{remark}{Remark}

\newcommand{\cput}[3]{\put(#1,#2){\hbox to 0pt{\hss#3\hss}}}

\begin{document}

\title{Smooth polyhedral surfaces}

\author{Felix G\"unther\footnote{Erwin Schr\"odinger International Institute for Mathematical Physics, Boltzmanngasse 9, 1090 Vienna, Austria and Max Planck Institute for Mathematics, Vivatsgasse 7, 53111 Bonn, Germany. E-mail: fguenth@math.tu-berlin.de} \and Caigui Jiang\footnote{Max Planck Institute for Informatics, Campus E1 4, 66123 Saarbr\"ucken, Germany and Visual Computing Center, King Abdullah University of Science and Technology, KSA. E-mail: cjiang@mpi-inf.mpg.de} \and Helmut Pottmann\footnote{Center for Geometry and Computational Design, Technische Universit\"at Wien, Wiedner Hauptstra{\ss}e 8/104, 1040 Vienna, Austria. E-mail: pottmann@geometrie.tuwien.ac.at}}

\date{}
\maketitle

\begin{abstract}
\noindent
Polyhedral surfaces are fundamental objects in architectural geometry and industrial design. Whereas closeness of a given mesh to a smooth reference surface and its suitability for numerical simulations were already studied extensively, the aim of our work is to find and to discuss suitable assessments of smoothness of polyhedral surfaces that only take the geometry of the polyhedral surface itself into account. Motivated by analogies to classical differential geometry, we propose a theory of smoothness of polyhedral surfaces including suitable notions of normal vectors, tangent planes, asymptotic directions, and parabolic curves that are invariant under projective transformations. It is remarkable that seemingly mild conditions significantly limit the shapes of faces of a smooth polyhedral surface. Besides being of theoretical interest, we believe that smoothness of polyhedral surfaces is of interest in the architectural context, where vertices and edges of polyhedral surfaces are highly visible.\\ \vspace{0.5ex}

\noindent
\textbf{2010 Mathematics Subject Classification:} 52B70; 53A05.\\ \vspace{0.5ex}

\noindent
\textbf{Keywords:} Discrete differential geometry, polyhedral surface, smoothness, discrete Gaussian curvature, asymptotic direction, projective transformation.
\end{abstract}

\raggedbottom
\setlength{\parindent}{0pt}
\setlength{\parskip}{1ex}


\section{Introduction}\label{sec:intro}

In modern architecture, facades and glass roofs often model smooth shapes but are realized as polyhedral surfaces. Bad approximations may be observed as wiggly meshes, even though the polyhedral mesh is close to a smooth reference surface. High demands on the aesthetics of the realized surface necessitate adequate assessments of smoothness. But what does it mean for a polyhedral surface to be smooth?

\begin{figure}[htbp]
		\centerline{
			\begin{overpic}[height=0.3\textwidth]{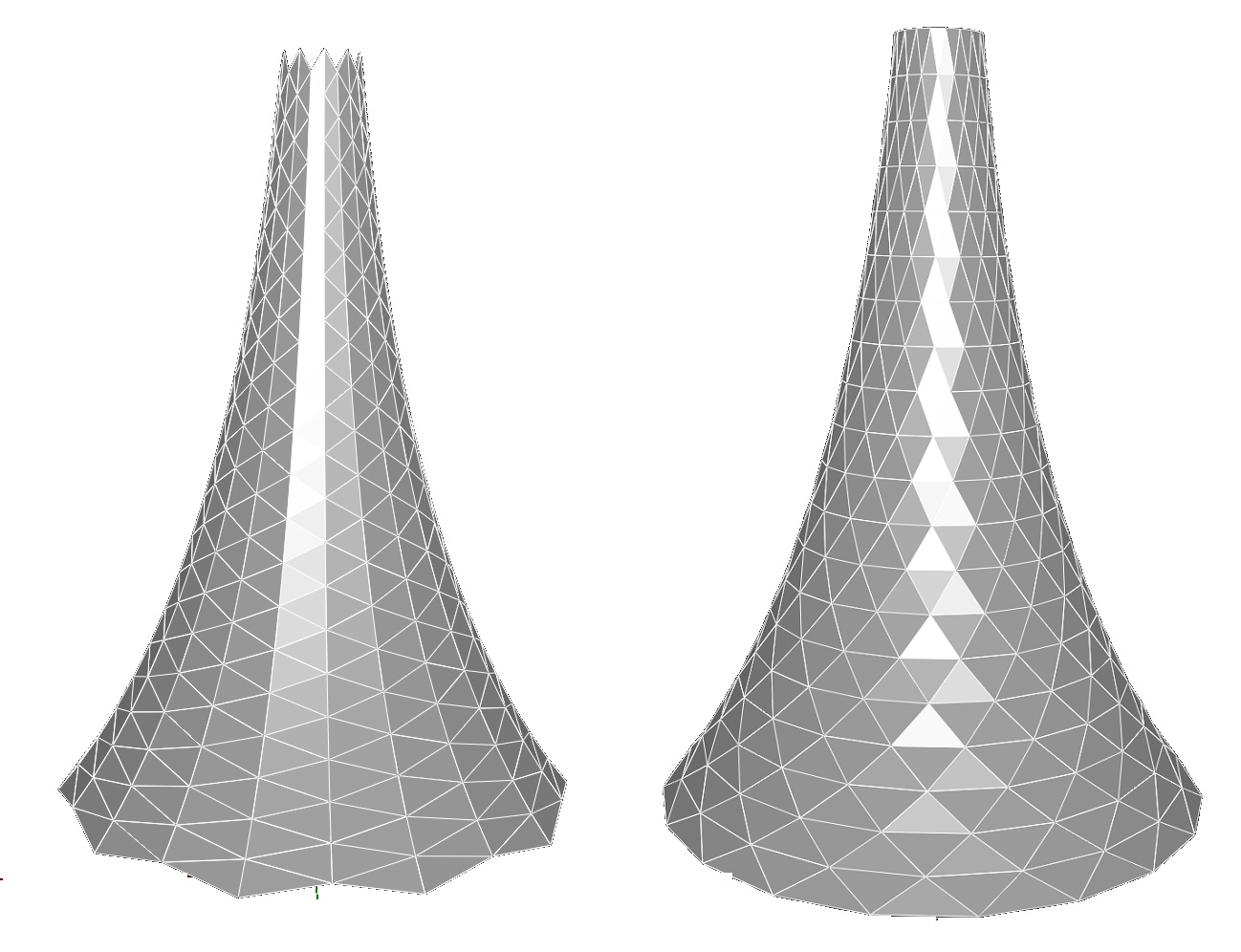}
				\put(2,68){\contour{white}{mesh 1}}
				\put(52,68){\contour{white}{mesh 2}}
			\end{overpic}
			\relax\\
			\begin{overpic}[height=0.3\textwidth]{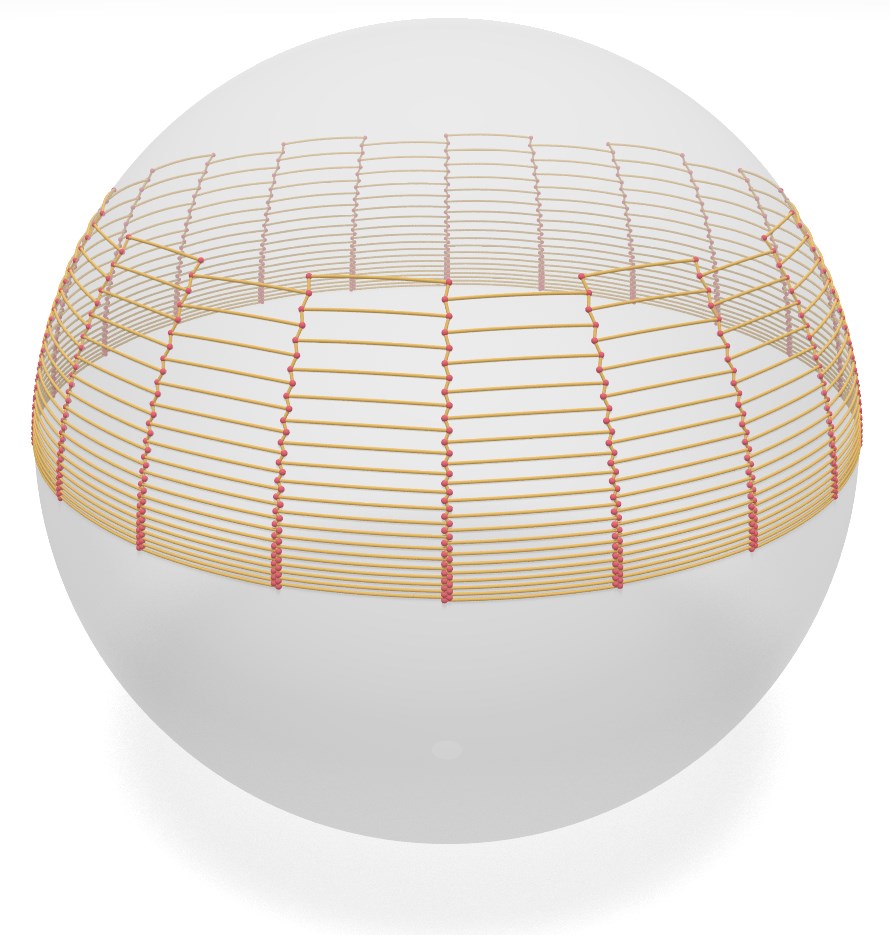}
				\cput{46}{2}{Gauss image 1}
			\end{overpic}
			\relax\\
			\begin{overpic}[height=0.3\textwidth]{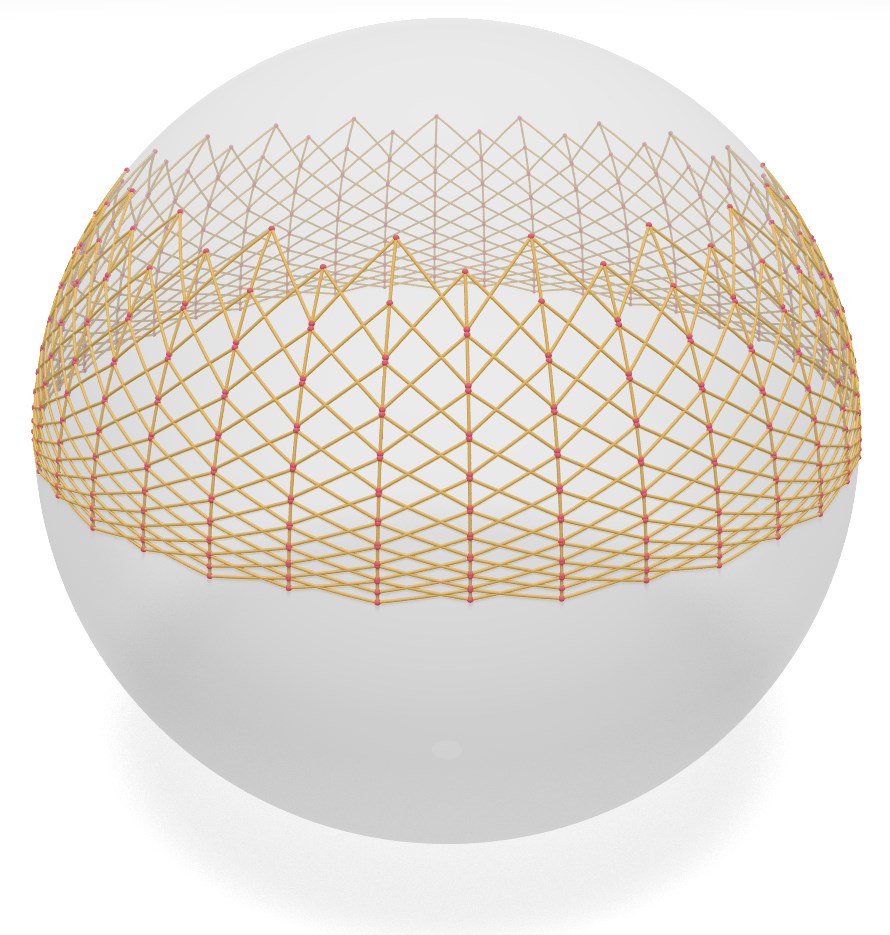}
				\cput{46}{2}{Gauss image 2}
			\end{overpic}
		}
	\caption{Evaluation of reflections as a subjective assessment of smoothness: The observation that the reflection pattern of mesh 1 behaves better than the reflections of mesh 2 is encoded in the regularity of its Gauss image}
	\label{fig:architecture}
\end{figure}


Reflective surfaces visibly expose the kink angles between adjacent faces. A reflection pattern that does not closely follow the actual shape of the polyhedral surface is not appealing to the human eye and may serve as a subjective assessment of smoothness, see Fig.~\ref{fig:architecture}. The reflections of parallel light rays are determined by the face normals of the polyhedral surface. So if we are looking for a reflection pattern that is reasonably aligned with the shape of the surface, the Gauss image should contain as few degeneracies such as overlaps and self-intersections as possible. Indeed, the different reflections of the two meshes depicted in Fig.~\ref{fig:architecture} exactly correspond to the different behavior of their Gauss images. Our first requirement for smoothness of a polyhedral surface will be therefore the absence of self-intersections in its Gauss image.


\subsection{Prior work}\label{sec:prior}

In \cite{JTVWP15} the design and optimization of polyhedral patterns, which are patterns of planar polygonal faces on freeform surfaces, were studied. It turned out that some patterns adapt their shape according to the sign of the Gaussian curvature and others do not. Polyhedral patterns of the first type look somewhat smoother than patterns of the other type and are better suitable to approximate smooth surfaces.

Orden et al. \cite{ORSSSW04} investigated planar frameworks of forces in static equilibrium. For a planar graph with forces in equilibrium, one can rearrange the forces at each vertex to obtain polygons of forces that themselves can be arranged to form a realization of the dual graph. This realization is called the reciprocal framework. The authors derived conditions on the shapes of polygons and the sign of forces (tension or compression) such that both the original and the reciprocal framework are free of self-intersections.
The framework on a planar graph can be viewed as a projection of a (possibly self-intersecting) polyhedron in Euclidean space onto a plane. The dihedral angles correspond to forces, such that inflection faces are encoded by sign changes of the forces. The reciprocal is then the projection of the polar of the original polyhedron, where polarity with respect to the paraboloid $z=x^2+y^2$ is considered. The result of \cite{ORSSSW04} is that the only planar faces which produce a reciprocal without self-intersections are convex polygons, pseudo-quadrilaterals and pseudo-triangles, where zero, four, or two or four sign changes appear, respectively.


\subsection{Contributions and overview}\label{sec:contributions}

The aim of our paper is to find suitable assessments of smoothness of polyhedral surfaces without a smooth reference surface. The property we start with in Section~\ref{sec:Gauss} is that the Gauss image of the star of a vertex of either positive or negative discrete Gaussian curvature shall have no self-intersections. Already this property limits the possible shapes of Gauss images to convex spherical polygons in the case of positive discrete Gaussian curvature and to spherical pseudo-digons, -triangles, and -quadrilaterals in the case of negative discrete Gaussian curvature. If we restrict to vertex stars that exhibit a so-called transverse plane onto which they orthogonally project in a one-to-one way, then spherical pseudo-digons do not appear as Gauss images and we recover the result of \cite{ORSSSW04}. In contrast to the authors of \cite{ORSSSW04}, we use arguments of spherical geometry only and do not use a projective dualization argument. That is also why the existence of a transverse plane is not necessary for our argument.

Whereas the normal of such a transverse plane is a suitable normal at the vertex of the polyhedral surface, we propose a possibly different tangent plane that is motivated by a discretization of the Dupin indicatrix. In the smooth theory, planes parallel and close to the tangent plane at a point of positive or negative curvature intersect the surface in curves that resemble ellipses and hyperbolas, respectively. A similar behavior in the discrete setup is achieved when the Gauss image of a negatively curved vertex star is star-shaped. A normal vector with respect to which the Gauss image is star-shaped then defines a suitable tangent plane. As a result, the asymptotic directions lie within the faces of inflection at the vertex star.

Assessing smoothness around a face in Section~\ref{sec:global} leads to similar observations. For example, in regions of negative curvature the polyhedral faces have to be pseudo-triangles or -quadrilaterals. The fact that they are dual to the shapes of Gauss images of vertex stars lies in the projective invariance of our notion that is discussed later. This also explains the analogy to \cite{ORSSSW04}. Where we are relating angles of the face to spherical angles, Orden et al. considered the planar angles in the reciprocal vertex star, corresponding to the projective dual described in Theorem~\ref{th:dual}. For their construction to work they use the fact that the face is an interior face of the framework, translating to the face being a transverse plane for its neighborhood. In that situation, we can apply the dualization of Theorem~\ref{th:dual}. In contrast, our arguments are valid even if this theorem cannot be applied.

In addition, we define asymptotic directions in star-shaped faces of negative discrete Gaussian curvature. Even the shapes of faces with vertices of positive and of negative curvature are limited. Suitable constraints allow us to define a discrete parabolic curve as a line segment connecting the two edges where the discrete Gaussian curvature changes it sign.

The notion we finally propose for a smooth polyhedral surface can be found in Section~\ref{sec:smooth}. We demand that the discrete Gaussian curvature is non-zero for all vertices and that its sign changes at either zero or two edges of a face. Furthermore, the Gauss images of all interior vertex stars shall be star-shaped and contained in open hemispheres. For each face, the interior angles of Gauss images of neighboring vertex stars shall add up to $2\pi$ or $0$ depending on whether the discrete Gaussian curvature has the same sign for all vertices or not. In the first case, we demand in addition that $f$ is star-shaped; in the latter case, we require that the convex hull of vertices of positive curvature does not contain any vertex of negative curvature.

In Section~\ref{sec:projective}, we show that our notion of smooth polyhedral surfaces is projectively invariant. Moreover, projective transformations map discrete tangent planes, discrete asymptotic directions and discrete parabolic curves to the corresponding objects of the image surface. Certain duality properties such as the similar shapes of faces of a negatively curved smooth polyhedral surface and of its Gauss image are reflected by the fact that the projective dual of a negatively curved smooth polyhedral surface is again a negatively curved smooth polyhedral surface.

Finally, we would like to remark that the optimization of a given polyhedral surface approximating a given smooth reference surface toward greater smoothness in the sense of star-shaped Gauss images is the objective of the recent paper \cite{JGWP16} for which the current paper lays the theoretical foundation. 


\section{Gauss image of vertex stars}\label{sec:Gauss}

In this section, we investigate the Gauss image of a vertex star of a polyhedral surface and gradually propose conditions on the polyhedral surface that correspond to smoothness. We motivate these criteria by relating them to corresponding properties of smooth surfaces and by giving examples.

Of crucial importance for our investigation is the fact that the algebraic area enclosed by the Gauss image $g({\vec{v}})$ of the star of a vertex $\vec{v}$ equals the discrete Gaussian curvature $K({\vec{v}})$. Even though this statement is widely believed in discrete differential geometry, the first complete proof was given in the recent paper \cite{BG16}. It is based on the index $i({\vec{v}},{\vec{n}})$ introduced by Banchoff \cite{B70}. Due to its importance for our present discussion, we will shortly discuss it here.

We will start with some basic notations and fundamental lemmas in Sections~\ref{sec:notation} and~\ref{sec:lemmas} before we will explore the shape of Gauss images that are free of self-intersections in Section~\ref{sec:selfintersection}. The discrete Dupin indicatrices and asymptotic directions will be investigated in Section~\ref{sec:Dupin}.


\subsection{Notation}\label{sec:notation}

Let $P$ be an orientable polyhedral surface \textit{immersed} into three-dimensional Euclidean space. That means, we consider a connected 2-manifold that is a union of polygonal regions such that each vertex of $P$ possesses a neighborhood that is embedded in $\mathds{R}^3$. We do not require any specific condition on the faces $f$ of $P$ besides being simple and planar, but we assume that faces that share an edge are not coplanar and that no interior angle equals $\pi$. However, faces can be non-convex. Clearly, non-convex faces do not occur in simplicial, i.e., triangulated surfaces, but the faces of the projective dual of a triangular surface have in general more than three vertices and non-convexity usually arises in areas of negative discrete Gaussian curvature.

A vertex of $P$ shall be denoted by $\vec{v}$. In what follows, $\vec{v}$ should always be an \textit{interior vertex} of $P$, boundary vertices will not be considered. Two vertices $\vec{v}_1$ and $\vec{v}_2$ are \textit{adjacent} if they are connected by an edge, two faces $f_1$ and $f_2$ are \textit{adjacent} if they share an edge. A face $f$ or an edge $e$ is \textit{incident} to a vertex $\vec{v}$, if the vertex is a corner of the face or the edge, respectively. For adjacency and incidence we will use the notations $\vec{v}_1 \sim {\vec{v}}_2$, $f_1 \sim f_2$ and $f,e \sim {\vec{v}}$, $\vec{v} \sim f,e$, respectively. See Fig.~\ref{fig:vertex_star} for the notation of the \textit{star} of the interior vertex $\vec{v}$, i.e., the set of all faces of $P$ incident to $\vec{v}$.

$\vec{v}$ is said to be \textit{convex}, if the star of $\vec{v}$ determines the boundary of a convex polyhedral cone.

\begin{figure}[!ht]
	\centerline{
		\begin{overpic}[height=0.25\textwidth]{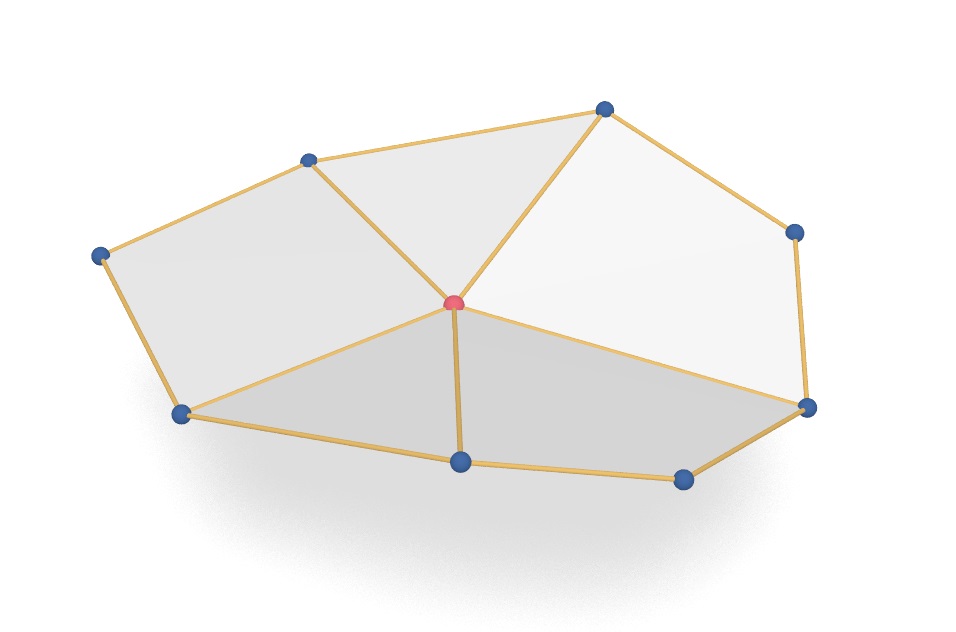}
			\put(66,38){\contour{white}{$f_s$}}
			\cput{47}{36}{$\vec{v}$}
			\cput{65}{57}{$\vec{v}_s$}
			\cput{92}{23}{$\vec{v}_{s-1}$}
		\end{overpic}
	}
	\caption{Notation for the star of the vertex $\vec{v}$}\label{fig:vertex_star}
\end{figure}

We fix one orientation of $P$ and denote by ${\vec{n}}_f \in S^2$ the outer unit normal vector of a face $f$. For an interior vertex $\vec{v}$, we connect the normals of adjacent faces ${\vec{n}}_{s-1},{\vec{n}}_{s} \sim {\vec{v}}$ by the shorter of the two great circle arcs connecting ${\vec{n}}_{s-1}$ and ${\vec{n}}_{s}$. Since $P$ is immersed, ${\vec{n}}_{s-1}\neq \pm{\vec{n}}_{s}$, so the shorter arc is non-trivial and well defined. Hence, the normals ${\vec{n}}_f$ for $f \sim {\vec{v}}$ define a spherical polygon $g({\vec{v}})$ that we call the \textit{Gauss image of the vertex star}.

\begin{figure}[!ht]
	\centerline{
		\begin{overpic}[height=0.3\textwidth]{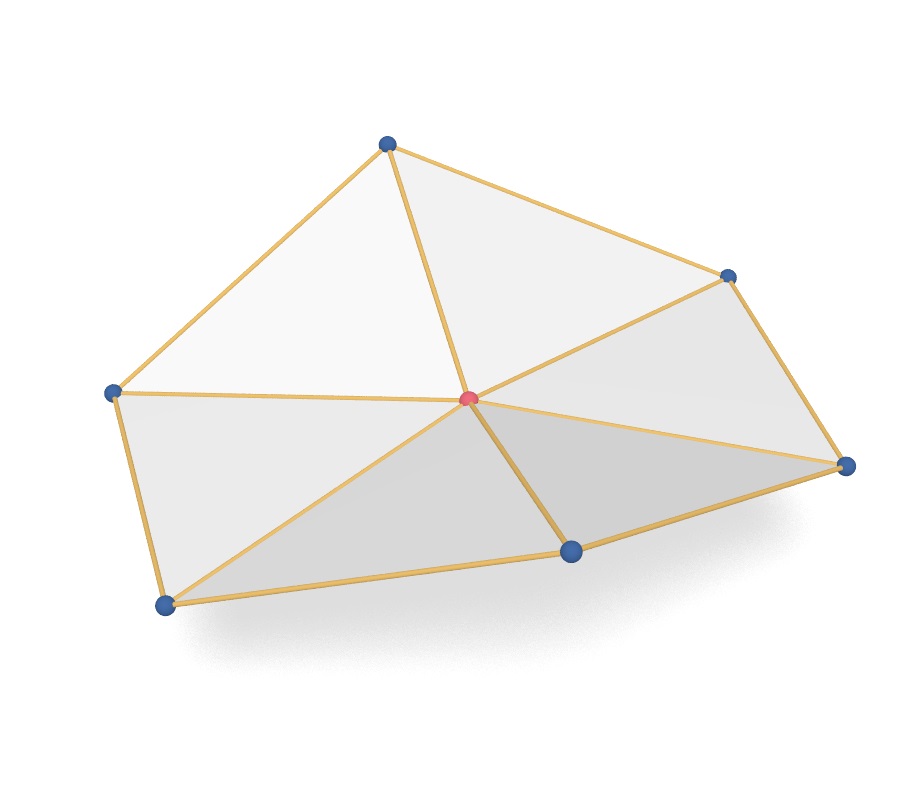}
			\put(30,37){\contour{white}{$f_1$}}
			\put(45,32){\contour{white}{$f_2$}}
			\put(63,34){\contour{white}{$f_3$}}
			\put(70,45){\contour{white}{$f_4$}}
			\put(55,55){\contour{white}{$f_5$}}
			\put(33,50){\contour{white}{$f_6$}}
			\cput{52}{46}{$\vec{v}$}
		\end{overpic}
		\relax\\
		\begin{overpic}[height=0.3\textwidth]{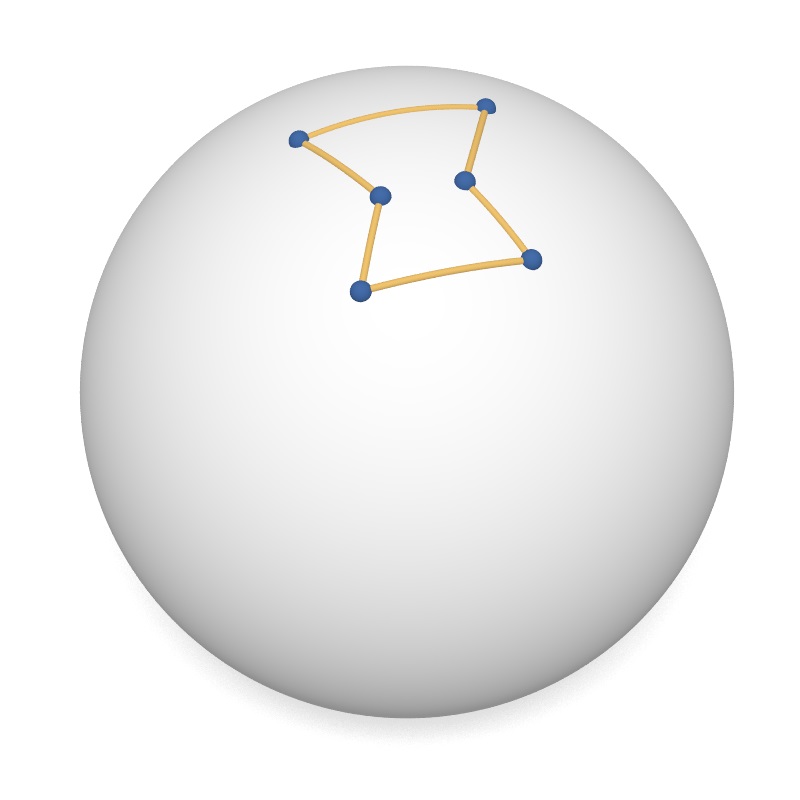}
			\cput{41}{75}{$\vec{n}_{1}$}
			\cput{33}{85}{$\vec{n}_{2}$}
			\cput{63}{89}{$\vec{n}_{3}$}
			\cput{64}{77}{$\vec{n}_{4}$}
			\cput{68}{62}{$\vec{n}_{5}$}
			\cput{43}{59}{$\vec{n}_{6}$}	
		\end{overpic}
	}
	\caption{A vertex star and its Gauss image for negative discrete Gaussian curvature}\label{fig:Gaussian}
\end{figure}
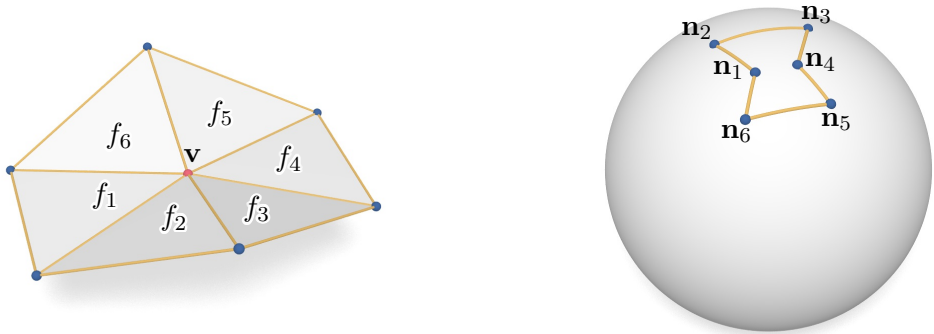


Even though we have an intuition how the Gauss image should be extended across a vertex in a way that reflects the geometry of the vertex star, it is not that trivial to give a precise definition. Still, the algebraic area of $g({\vec{v}})$ can be defined in a way that agrees with our intuition, and equals the discrete Gaussian curvature $K({\vec{v}})$. We will discuss this relation in the following subsection.

\begin{definition}
Let $\vec{v}$ be fixed and let $\alpha_f$ be the angle of a face $f\sim{\vec{v}}$. Then, the \textit{discrete Gaussian curvature} at $\vec{v}$ is defined as \[K({\vec{v}}):=2\pi-\sum\limits_{f \sim {{\vec{v}}}}\alpha_f.\]
\end{definition}

To compare the angles of the faces of the vertex star at $\vec{v}$ with the angles of the spherical polygon $g({\vec{v}})$ in Lemma~\ref{lem:angle}, we introduce the notion of an \textit{inflection face}.

\begin{definition}
Let $f$ be a face of the star of the vertex $\vec{v}$. If the two other faces of the vertex star that are adjacent to $f$ lie in different half-spaces that the plane through $f$ determines, then $f$ is said to be an inflection face.
\end{definition}


\subsection{Gauss image and discrete Gaussian curvature}\label{sec:lemmas}

That the discrete Gaussian curvature of a vertex equals the algebraic area that its Gauss image encloses, is a popular exercise in some lecture notes on discrete differential geometry. However, the idea of proof going back to Alexandrov \cite{A05} by comparing the spherical angle $\angle {\vec{n}}_{s+1}{\vec{n}}_s{\vec{n}}_{s-1}$ with the angle $\alpha_{s}$ only works in the case that $\alpha_{s}<\pi$ and no inflection faces are present, i.e., in the case of a convex corner. The general situation is slightly more complicated.

\begin{lemma}\label{lem:angle}
Let $f_1, f_2, f_3$ be three consecutive faces of the star of the vertex $\vec{v}$ that are ordered in counterclockwise direction around $\vec{v}$ with respect to the given orientation of $P$. Let $\alpha=\alpha_{2}$ be the angle of $f_2$ at $\vec{v}$ and $\alpha'$ the spherical angle $\angle {\vec{n}}_{3}{\vec{n}}_{2}{\vec{n}}_{1}$.
\begin{enumerate}
\item If $\alpha < \pi$ and $f_2$ is not an inflection face, then $\alpha'=\pi-\alpha$.
\item If $\alpha < \pi$ and $f_2$ is an inflection face, then $\alpha'=2\pi-\alpha$.
\item If $\alpha > \pi$ and $f_2$ is not an inflection face, then $\alpha'=3\pi-\alpha$.
\item If $\alpha > \pi$ and $f_2$ is an inflection face, then $\alpha'=2\pi-\alpha$.
\end{enumerate}
\end{lemma}
\begin{proof}

(i) Consider the two planes each of which is orthogonal to one of the two edges of $f_2$ incident to $\vec{v}$ and intersects this edge. Since both ${\vec{n}}_{1}$ and ${\vec{n}}_{2}$ are orthogonal to the edge that is shared by $f_1$ and $f_2$, and ${\vec{n}}_{2}$ and ${\vec{n}}_{3}$ are orthogonal to the edge that is shared by $f_2$ and $f_3$, the spherical angle $\angle {\vec{n}}_{3}{\vec{n}}_{2}{\vec{n}}_{1}$ equals one of the two intersection angles of these planes.

Using that $f_2$ is not an inflection face, $\alpha'$ corresponds to the angle opposite to $\alpha$ in the quadrilateral determined by $\vec{v}$, the two edges of $f_2$ incident to $\vec{v}$, and the two planes constructed above. The two angles between these edges and the two planes are given by $\pi/2$, such that $\alpha'=\pi-\alpha$ follows.

\begin{figure}[!ht]
	\centerline{
		\begin{overpic}[height=0.3\textwidth]{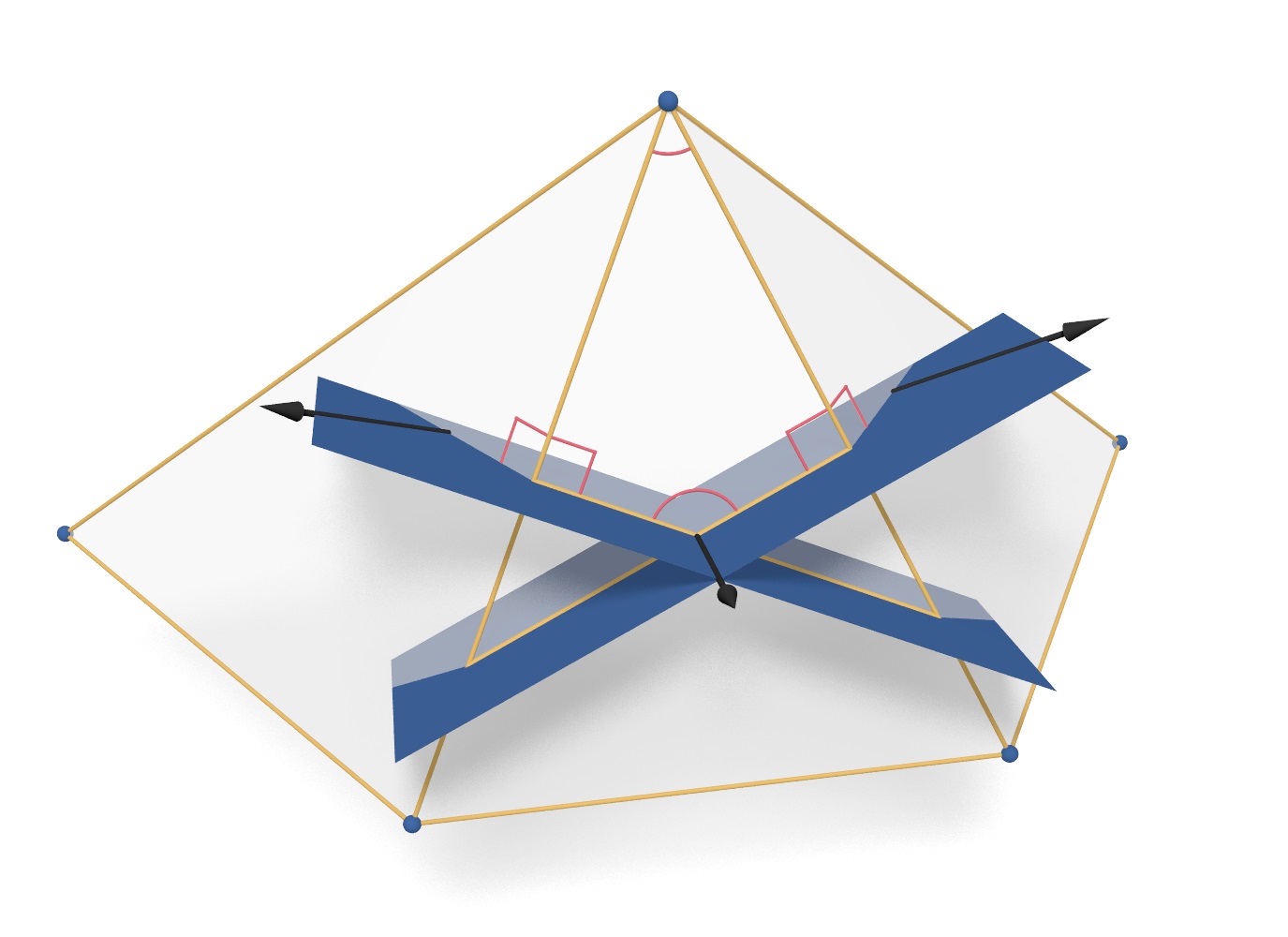}
			\put(25,28){\contour{white}{$f_1$}}
			\put(50,18){\contour{white}{$f_2$}}
			\put(78,32){\contour{white}{$f_3$}}
			\cput{52}{67}{$\vec{v}$}
			\cput{20.5}{37}{$\vec{n}_{1}$}
			\cput{58}{22}{$\vec{n}_{2}$}
			\cput{88}{49}{$\vec{n}_{3}$}
			\color{red}
			\cput{54}{36}{$\alpha'$}
			\cput{52}{57}{$\alpha$}
		\end{overpic}
		\relax\\
		\begin{overpic}[height=0.3\textwidth]{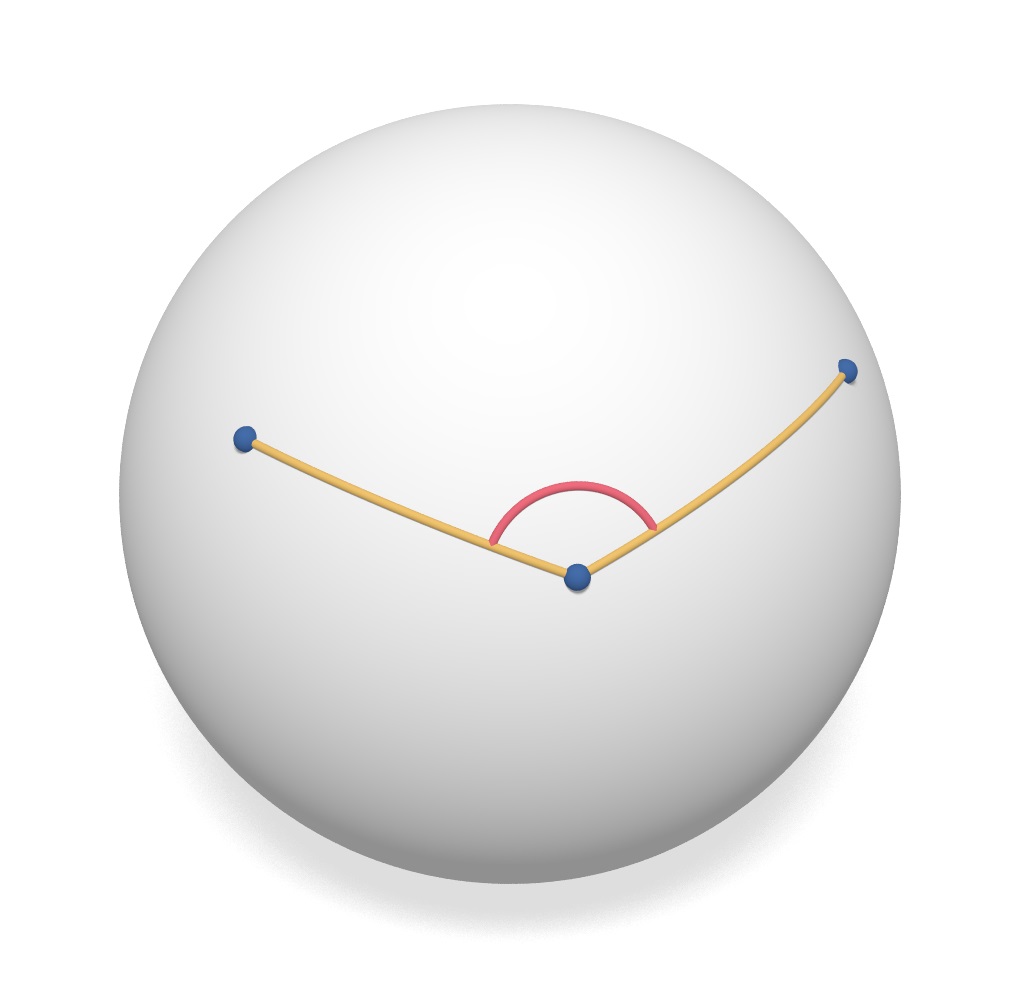}
			\cput{21}{50}{$\vec{n}_{1}$}
			\cput{58}{35}{$\vec{n}_{2}$}
			\cput{88}{61}{$\vec{n}_{3}$}
			\color{red}
			\cput{57}{44}{$\alpha'$}
		\end{overpic}
	}
	\caption{Lemma~\ref{lem:angle} (i): $\alpha < \pi$ and $f_2$ is not an inflection face}\label{fig:angle1}
\end{figure}

(ii) Consider the edge neighboring to $f_1$ and $f_2$ as a hinge, and move $f_1$ to the other side of $f_2$ such that we obtain the situation of (i). Now, hinge $f_1$ down. By (i), the spherical angle $\angle {\vec{n}}_{3}{\vec{n}}_{2}{\vec{n}}_{1}$ equals $\pi-\alpha$ as long as $f_1$ is on the same side of $f_2$ as $f_3$. Thus, ${\vec{n}}_{1}$ moves on the great circle through ${\vec{n}}_{2}$ that intersects the arc ${\vec{n}}_{2}{\vec{n}}_{3}$ in an angle of size $\pi-\alpha$ as in situation (i). When $f_1$ comes to the other side of $f_2$, ${\vec{n}}_{1}$ passes through ${\vec{n}}_{2}$ and the spherical angle $\angle {\vec{n}}_{3}{\vec{n}}_{2}{\vec{n}}_{1}$ becomes larger by $\pi$. Hence, $\alpha'=\pi+\pi-\alpha=2\pi-\alpha$.

\begin{figure}[!ht]
	\centerline{
		\begin{overpic}[height=0.28\textwidth]{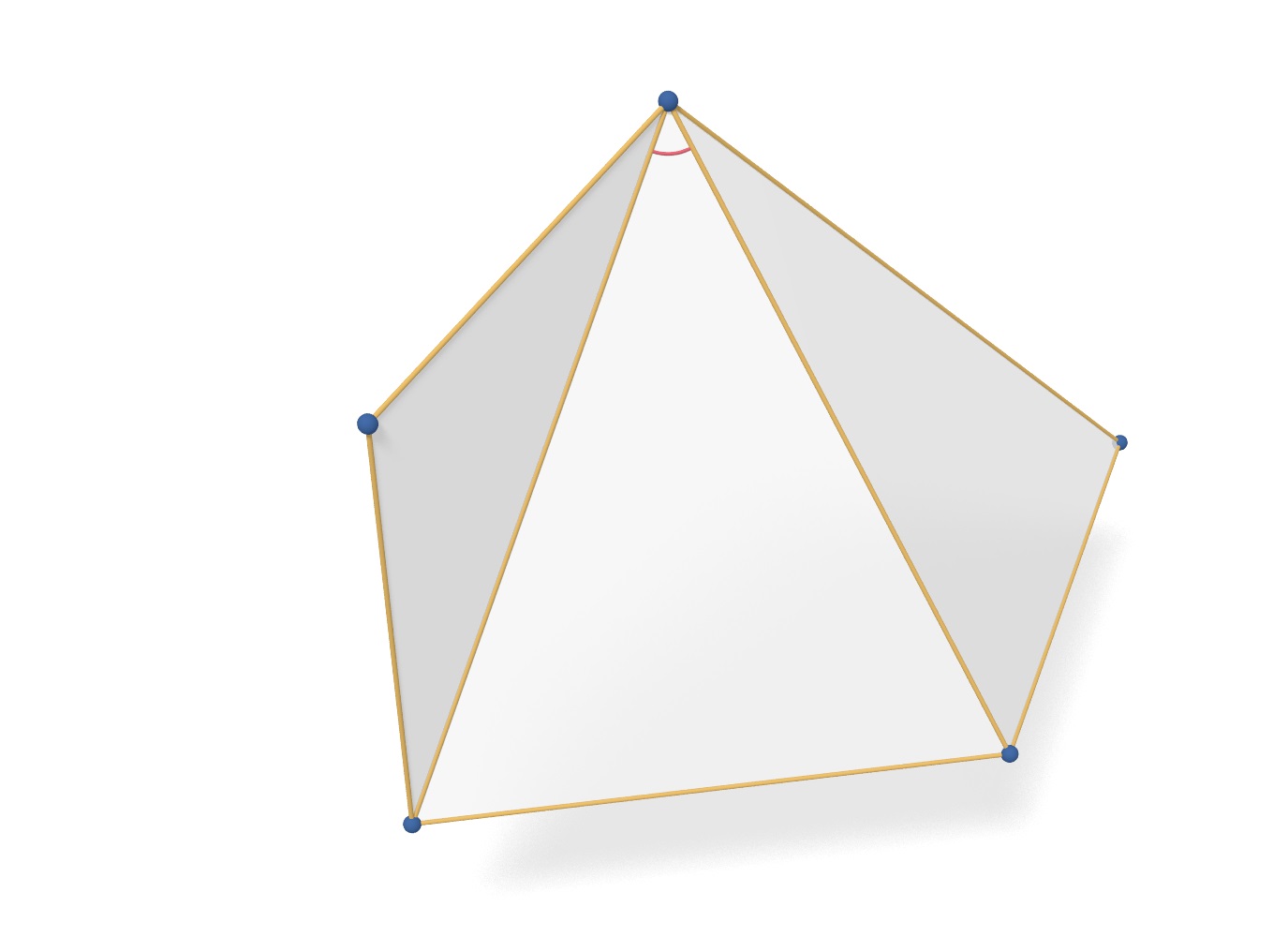}
		\put(34,35){\contour{white}{$f_1$}}
		\put(50,20){\contour{white}{$f_2$}}
		\put(76,32){\contour{white}{$f_3$}}
		\cput{52}{67}{$\vec{v}$}
		\color{red}
		\cput{52}{57}{$\alpha$}
		\end{overpic}
		\relax\\
		\begin{overpic}[height=0.3\textwidth]{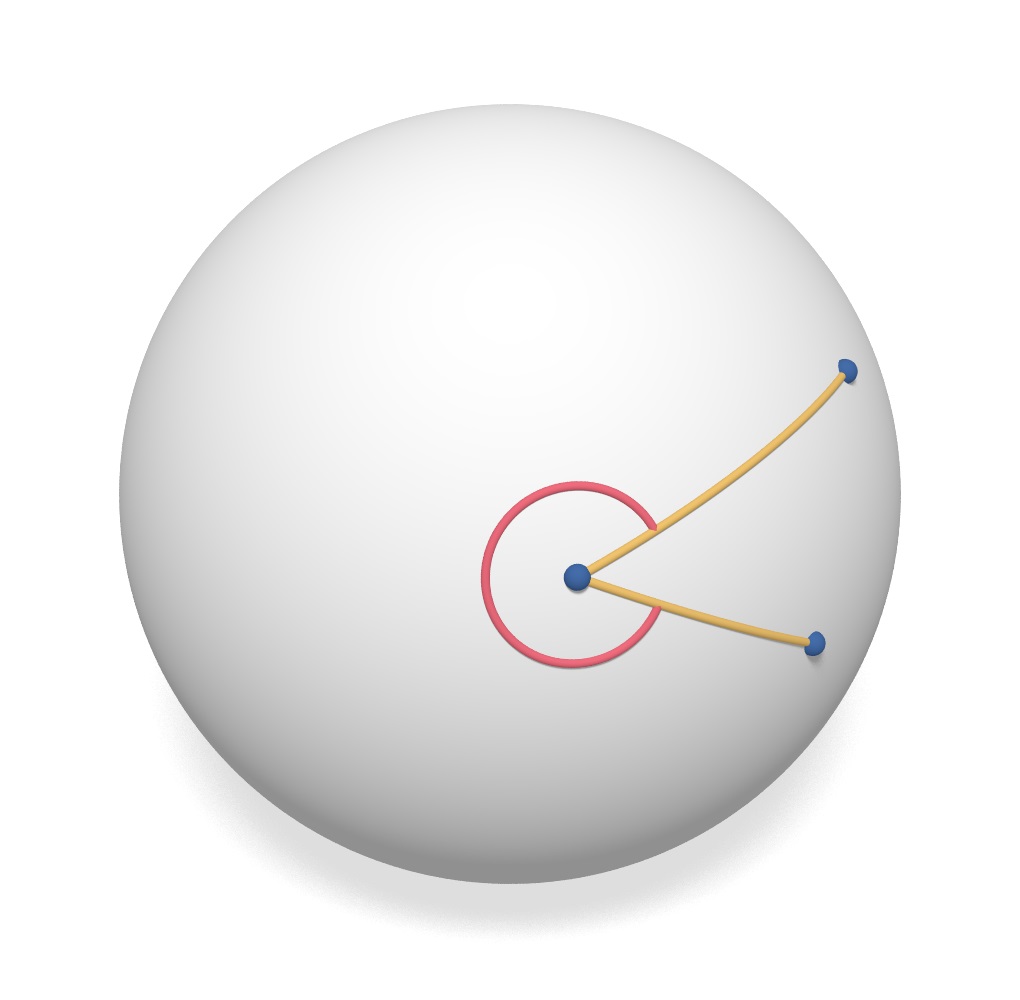}
			\cput{83}{38}{$\vec{n}_{1}$}
			\cput{67}{40.5}{$\vec{n}_{2}$}
			\cput{88}{61}{$\vec{n}_{3}$}
			\color{red}
			\cput{53}{40}{$\alpha'$}
		\end{overpic}
	}
	\caption{Lemma~\ref{lem:angle} (ii): $\alpha < \pi$ and $f_2$ is an inflection face}\label{fig:angle2}
\end{figure}

(iii) Consider the triangle $f'_2$ spanned by the two edges of $f_2$ incident to $\vec{v}$ and replace $f_2$ by $f'_2$. Then, we are in case (i) and ${\vec{n}}_{f'_2}=-{\vec{n}}_{2}$. It follows that the spherical angle $\angle {\vec{n}}_{3}(-{\vec{n}}_{2}){\vec{n}}_{1}$ equals $\pi-(2\pi-\alpha)=\alpha-\pi$. Thus, \[\alpha'=\angle {\vec{n}}_{3}{\vec{n}}_{2}{\vec{n}}_{1}=2\pi-\angle {\vec{n}}_{1}{\vec{n}}_{2}{\vec{n}}_{3}=2\pi-\angle {\vec{n}}_{3}(-{\vec{n}}_{2}){\vec{n}}_{1}=3\pi-\alpha.\]

\begin{figure}[!ht]
	\centerline{
		\begin{overpic}[height=0.3\textwidth]{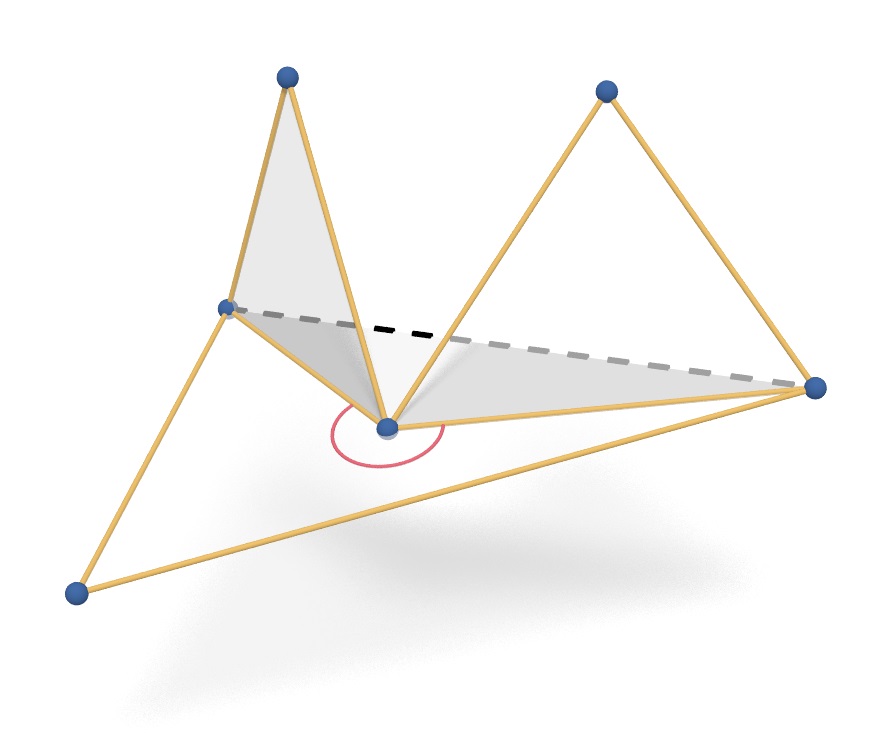}
			\put(31,49){\contour{white}{$f_1$}}
			\put(26,29){\contour{white}{$f_2$}}
			\put(65,48){\contour{white}{$f_3$}}
			\put(42,45){\contour{white}{$f_2'$}}
			\cput{47}{34}{$\vec{v}$}
			\color{red}
			\cput{41}{32}{$\alpha$}
		\end{overpic}
		\relax\\
		\begin{overpic}[height=0.3\textwidth]{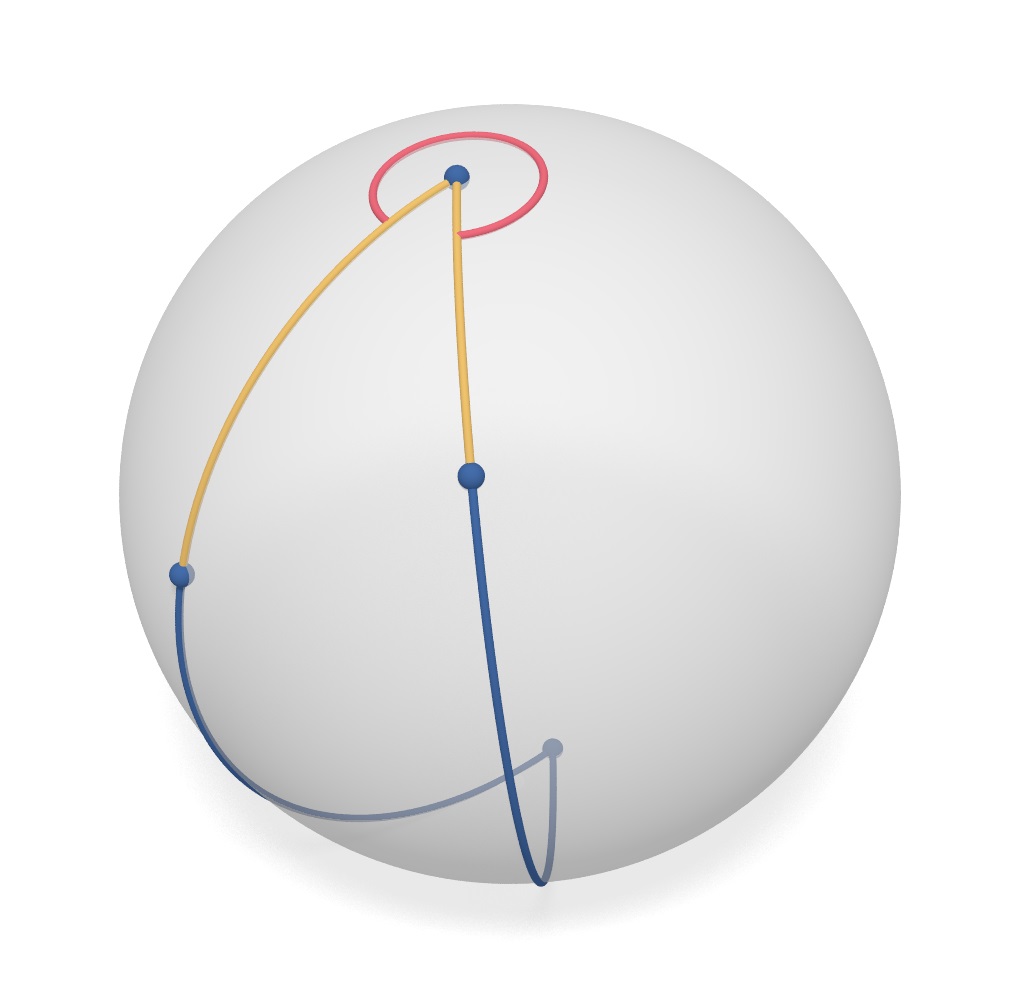}
			\cput{13}{39}{$\vec{n}_{1}$}
			\cput{45}{87}{$\vec{n}_{2}$}
			\cput{52}{50}{$\vec{n}_{3}$}
			\cput{57}{27}{$\mbox{-}\vec{n}_{2}$}
			\color{red}
			\cput{49}{77}{$\alpha'$}
		\end{overpic}
	}
	\caption{Lemma~\ref{lem:angle} (iii): $\alpha > \pi$ and $f_2$ is not an inflection face}\label{fig:angle3}
\end{figure}

(iv) With a similar argument as in (iii), we can use the result in (ii) to obtain that the spherical angle $\angle {\vec{n}}_{3}(-{\vec{n}}_{2}){\vec{n}}_{1}$ equals $2\pi-(2\pi-\alpha)=\alpha$, such that \[\alpha'=\angle {\vec{n}}_{3}{\vec{n}}_{2}{\vec{n}}_{1}=2\pi-\angle {\vec{n}}_{3}(-{\vec{n}}_{2}){\vec{n}}_{1}=2\pi-\alpha.\qed\]

\begin{figure}[!ht]
	\centerline{
		\begin{overpic}[height=0.3\textwidth]{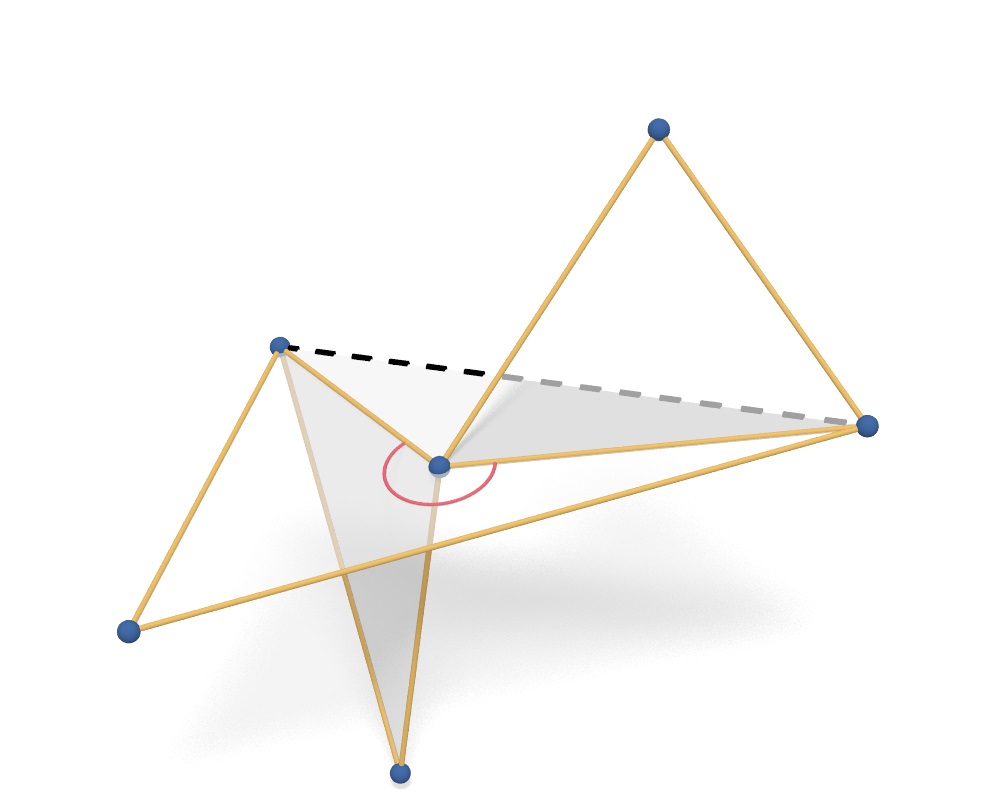}
			\put(37,18){\contour{white}{$f_1$}}
			\put(28,28){\contour{white}{$f_2$}}
			\put(65,48){\contour{white}{$f_3$}}
			\put(40,40){\contour{white}{$f_2'$}}
			\cput{48}{34}{$\vec{v}$}
			\color{red}
			\cput{42.5}{30.5}{$\alpha$}
		\end{overpic}
		\relax\\
		\begin{overpic}[height=0.3\textwidth]{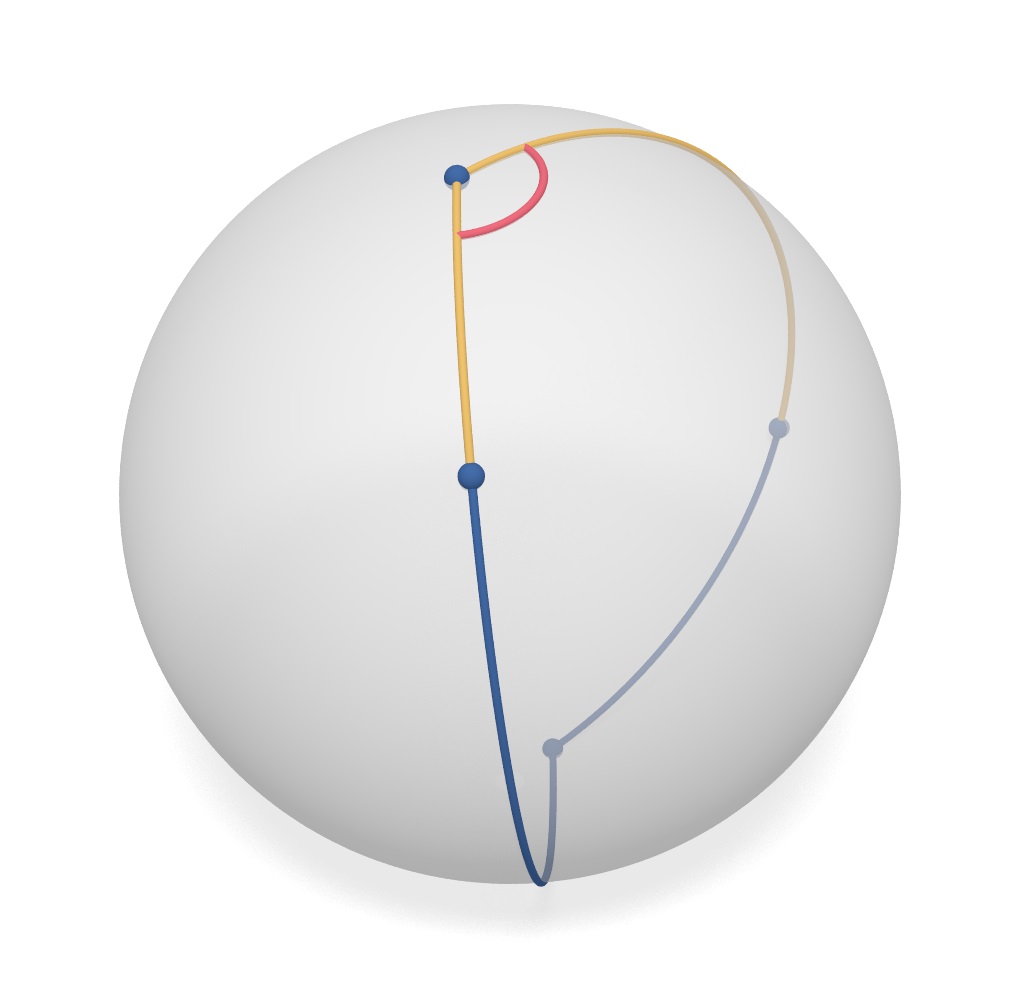}
			\cput{82}{55}{$\vec{n}_{1}$}
			\cput{43}{85}{$\vec{n}_{2}$}
			\cput{52}{50}{$\vec{n}_{3}$}
			\cput{62}{22}{$\mbox{-}\vec{n}_{2}$}
			\color{red}
			\cput{49}{77}{$\alpha'$}
		\end{overpic}
	}
	\caption{Lemma~\ref{lem:angle} (iv): $\alpha > \pi$ and $f_2$ is an inflection face}\label{fig:angle4}
\end{figure}

\end{proof}

\begin{definition}\label{def:index}
Let $\vec{v}$ be an interior vertex of the polyhedral surface $P$ and let $\xi \in S^2$. We define the \textit{index} \[i({\vec{v}},\xi):=1-\frac{1}{2}\left(\textnormal{number of incident triangles for which }{\vec{v}}\textnormal{ is middle for }\xi\right).\]

Here, we triangulate the faces of $P$ and call $\vec{v}$ \textit{middle} for $\xi$ in an incident triangle $\triangle$ if $\langle{\vec{v}},\xi\rangle$ lies in between $\langle{\vec{v}}_1,\xi\rangle$ and $\langle{\vec{v}}_2,\xi\rangle$ for the other two vertices $\vec{v}_1,{\vec{v}}_2$ of $\triangle$.
\end{definition}

Clearly, the above definition that goes back to Banchoff \cite{B70} does not depend on the choice of triangulations of the faces of $P$. Banchoff showed that integration of the index with respect to all $\xi \in S^2$ yields twice the discrete Gaussian curvature $K({\vec{v}})$. In \cite{BG16}, it is shown that $i({\vec{v}},\xi)$ equals the sum of appropriate algebraic indices of ${\vec{n}}$ and $-{\vec{n}}$ with respect to $g({\vec{v}})$. These algebraic indices encode the corresponding algebraic multiplicities of these points in the area enclosed by $g({\vec{v}})$, such that integration of the index $i({\vec{v}},{\vec{n}})$ over the unit sphere gives twice the algebraic area enclosed by $g({\vec{v}})$. This sketches the proof of the following key lemma \cite{BG16}:

\begin{lemma}\label{lem:area}
The algebraic area of the Gauss image $g({\vec{v}})$ equals the discrete Gaussian curvature $K({\vec{v}})$.
\end{lemma}

\begin{remark}
We would like to remark that Lemma~\ref{lem:area} is a popular exercise in some lecture notes and text books. However, the proof the authors usually have in mind relies on Lemma~\ref{lem:angle}~(i) that is just valid for a convex corner $\vec{v}$. In the general case, it is a priori not that clear how the algebraic image of the Gauss image shall be defined since the area a curve on the sphere encloses is at first defined only up to a multiple of $4 \pi$. But in the case we are mainly interested in, namely that the Gauss image is contained in an open hemisphere $H$, Banchoff's index $i({\vec{v}},\xi)$ coincides with the geometrically correct algebraic multiplicity of the point $\xi \in H$ in $g({\vec{v}})$. Note that $i({\vec{v}},\xi)=i({\vec{v}},-\xi)$, so we cannot claim this statement for all $\xi \in S^2$.
\end{remark}

In \cite{BG16}, it was discussed how Lemma~\ref{lem:area} significantly limits the possible shapes for $g({\vec{v}})$. The case of polyhedral surfaces with non-convex faces was also considered in the end. Since non-convex faces play an important role in the investigation of polyhedral patterns \cite{JTVWP15} and occur in the projective dual of negatively curved simplicial surfaces, we will repeat the classification of possible Gauss images without self-intersections in Theorem~\ref{th:shape}.


\subsection{Gauss image without self-intersections}\label{sec:selfintersection}

In the smooth theory, the Gauss map is locally injective in areas of non-zero curvature. In the case of positive Gaussian curvature, the Gauss map is orientation-preserving, in the case of negative curvature it is orientation-reversing.

So a natural requirement for \textit{smoothness} of a polyhedral surface is that if $K({\vec{v}})$ is non-zero, then $g({\vec{v}})$ shall have no self-intersections. In this case, we can also talk about orientation: going counterclockwise around the star of the vertex $\vec{v}$,  we traverse its Gauss image $g({\vec{v}})$ in the same orientation if $K({\vec{v}})>0$ and in the reverse orientation if $K({\vec{v}})<0$, see Fig.~\ref{fig:Gaussian}.

In the following, we use Lemma~\ref{lem:angle} and Lemma~\ref{lem:area} to deduce which shapes of $g({\vec{v}})$ actually occur if it has no self-intersections. For negative discrete Gaussian curvature, only spherical pseudo-$n$-gons occur, where $n=2,3,4$.

\begin{definition}
A spherical polygon without self-intersections is called a \textit{pseudo-$n$-gon}, if exactly $n$ of its interior angles are less than $\pi$. The corresponding $n$ vertices are called \textit{corners}.
\end{definition}

\begin{theorem}\label{th:shape}
Assume that $g({\vec{v}})$ has no self-intersections.
\begin{enumerate}
\item If $K({\vec{v}})> 0$, then $g({\vec{v}})$ is a convex spherical polygon. It follows that $\vec{v}$ is a convex corner.
\begin{figure}[!ht]
	\centerline{
		\begin{overpic}[height=0.25\textwidth]{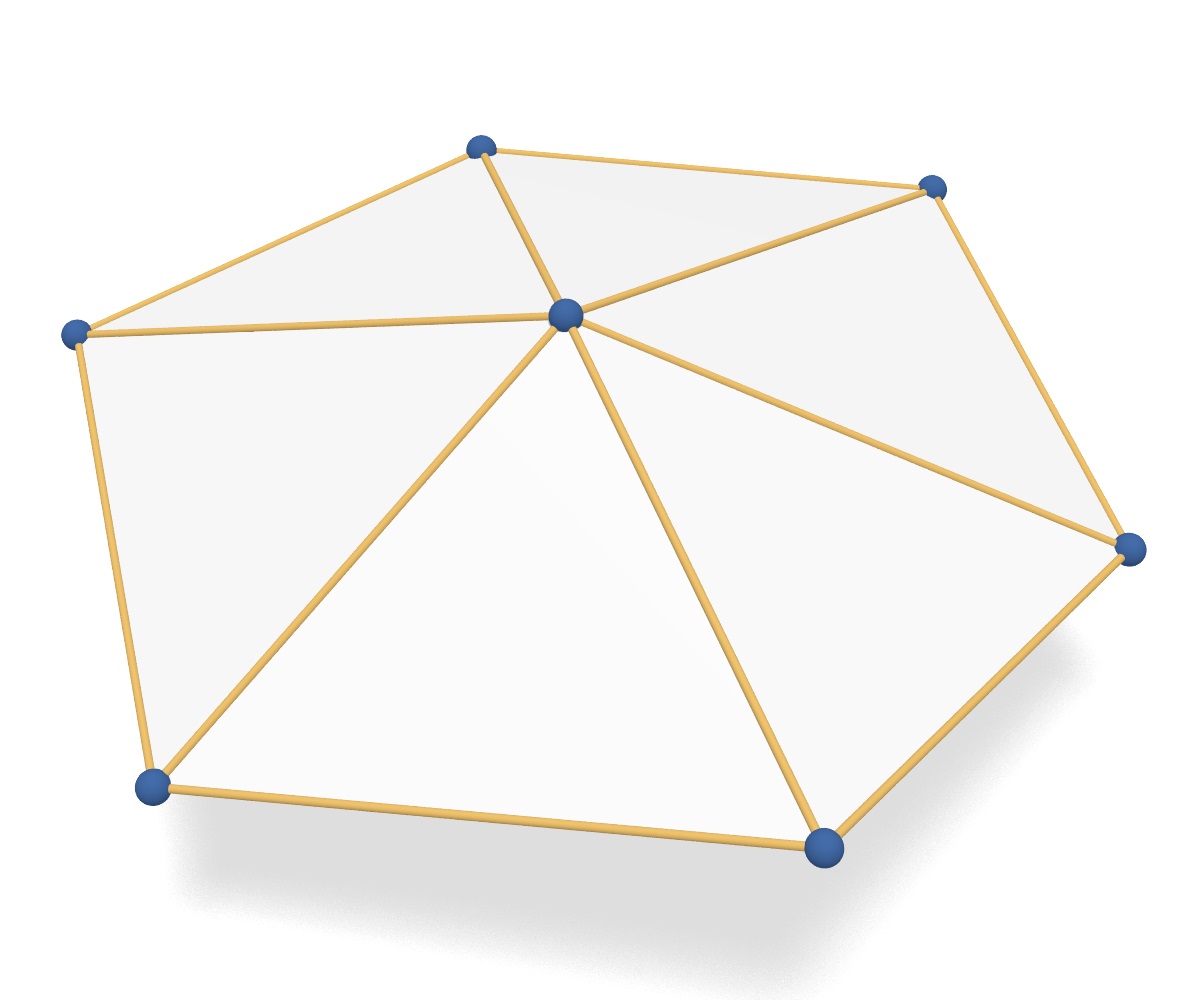}
						\put(21,42){\contour{white}{$f_1$}}
						\put(41,32){\contour{white}{$f_2$}}
						\put(63,36){\contour{white}{$f_3$}}
						\put(70,52){\contour{white}{$f_4$}}
						\put(50,63){\contour{white}{$f_5$}}
						\put(30,60){\contour{white}{$f_6$}}
		\end{overpic}
		\relax\\
		\begin{overpic}[height=0.3\textwidth]{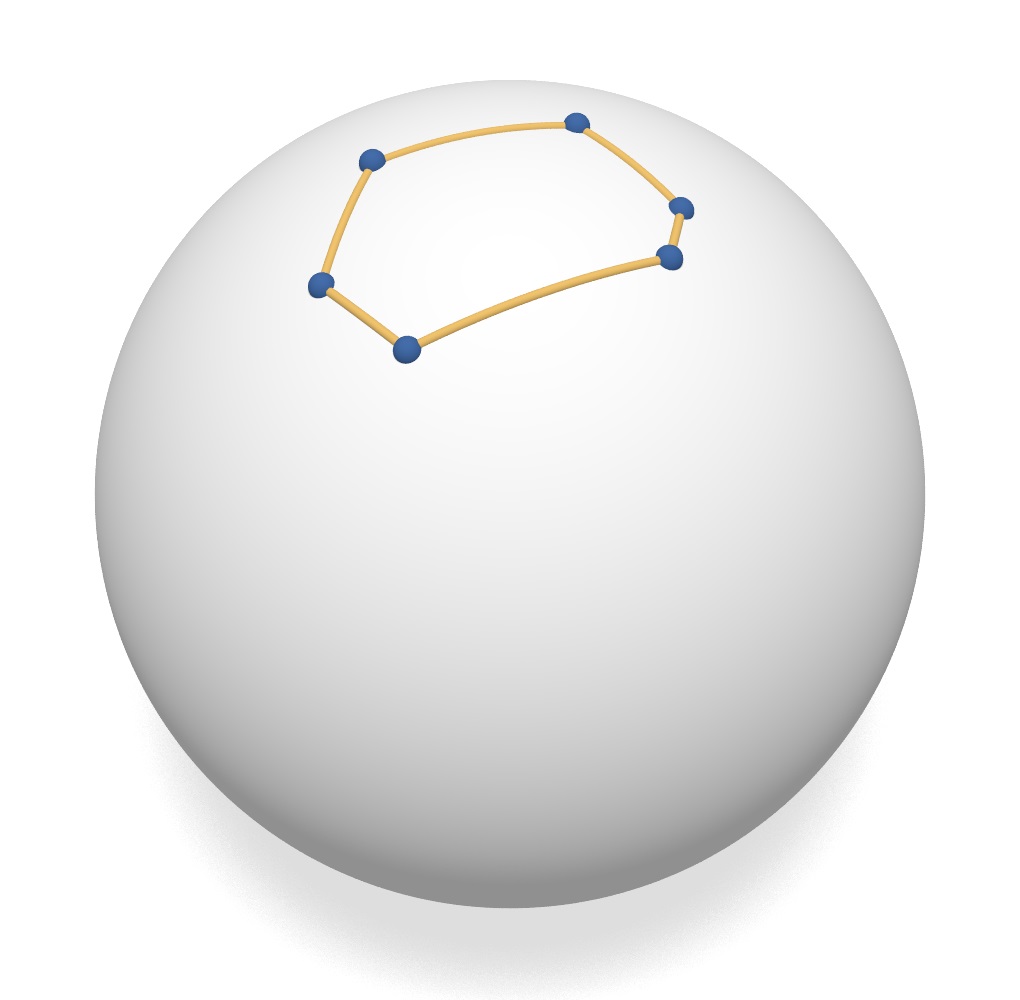}
				\cput{30}{65}{$\vec{n}_{1}$}
				\cput{43}{58}{$\vec{n}_{2}$}
				\cput{70}{67}{$\vec{n}_{3}$}
				\cput{73}{77}{$\vec{n}_{4}$}
				\cput{61}{87}{$\vec{n}_{5}$}
				\cput{33}{85}{$\vec{n}_{6}$}			
		\end{overpic}
	}
	\caption{Theorem~\ref{th:shape} (i): $g({\vec{v}})$ is a convex spherical polygon}\label{fig:shape1}
\end{figure}
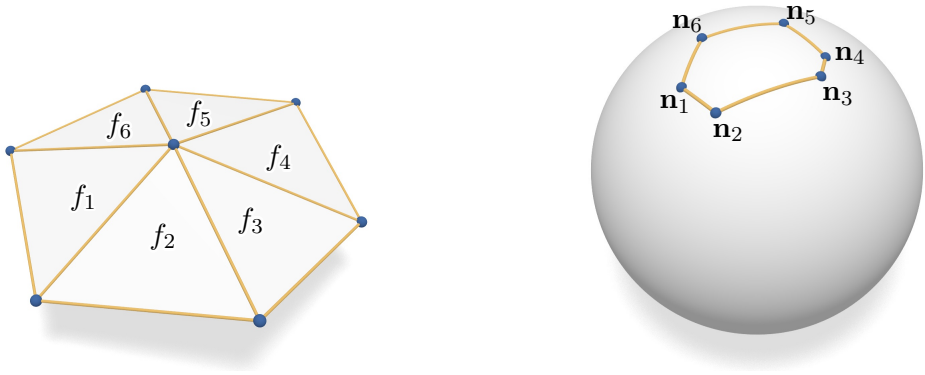
\item If $K({\vec{v}})< 0$ and $\alpha_f< \pi$ for all $f \sim {\vec{v}}$, then $g({\vec{v}})$ is a spherical pseudo-quadrilateral. Its four corners are the normals to the exactly four inflection faces in the star of $\vec{v}$.

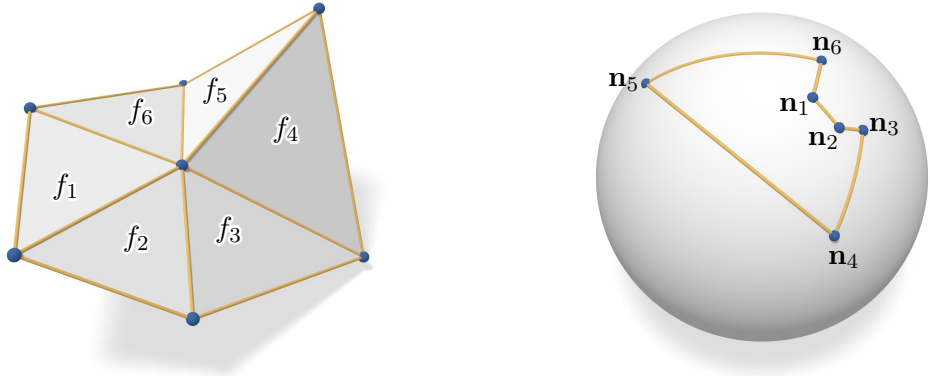
\begin{figure}[!ht]
	\centerline{
				\begin{overpic}[height=0.3\textwidth]{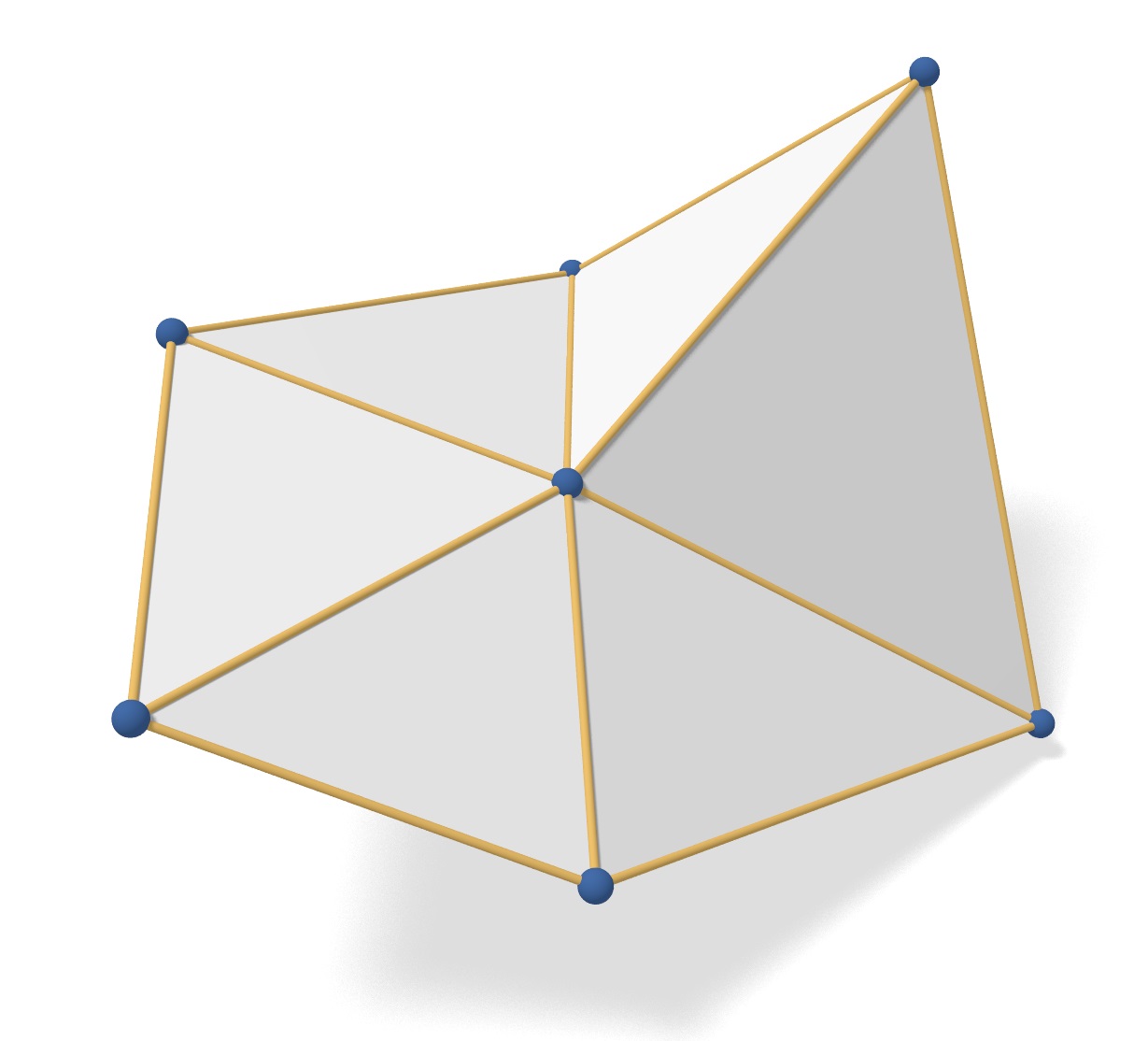}
					\put(20,42){\contour{white}{$f_1$}}
					\put(36,30){\contour{white}{$f_2$}}
					\put(57,32){\contour{white}{$f_3$}}
					\put(70,55){\contour{white}{$f_4$}}
					\put(54,64){\contour{white}{$f_5$}}
					\put(37,59){\contour{white}{$f_6$}}
				\end{overpic}
				\relax\\
				\begin{overpic}[height=0.3\textwidth]{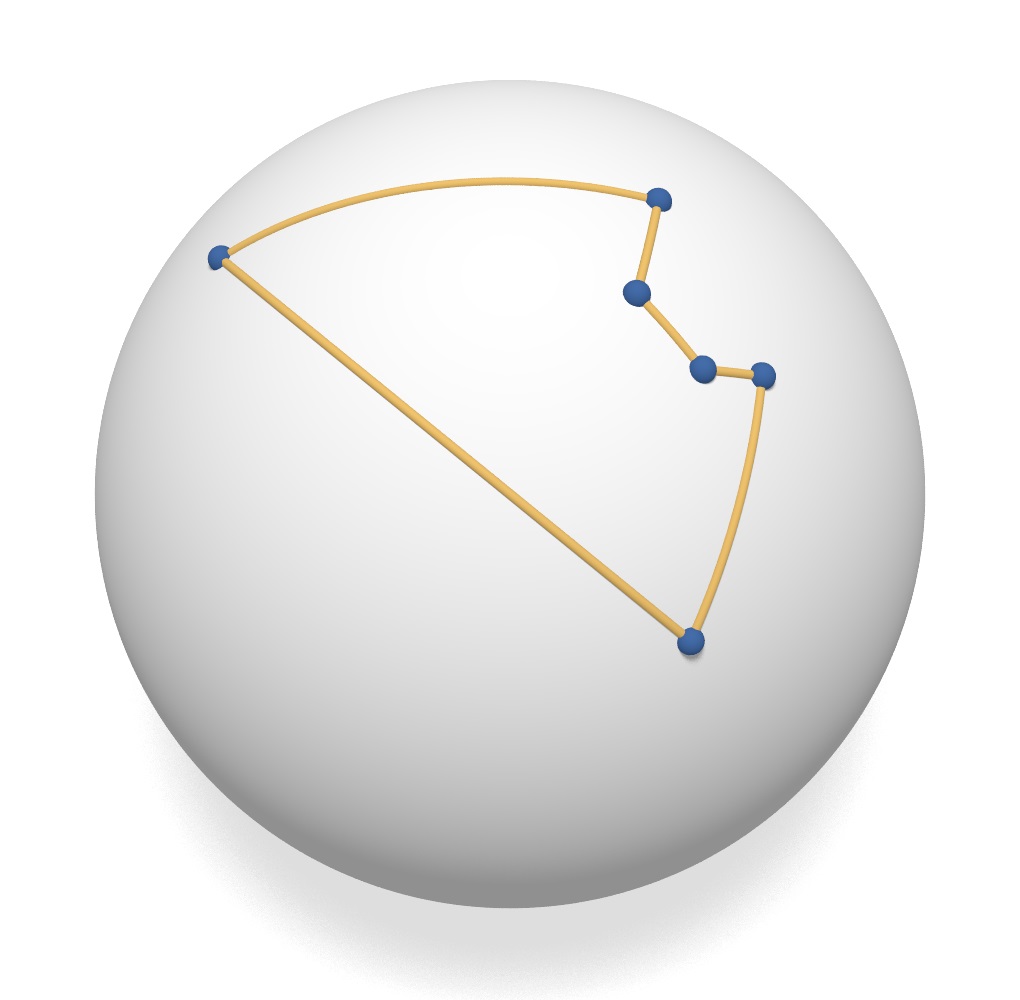}
					\cput{58}{66}{$\vec{n}_{1}$}
					\cput{65}{57.5}{$\vec{n}_{2}$}
					\cput{80}{61}{$\vec{n}_{3}$}
					\cput{70}{28}{$\vec{n}_{4}$}
					\cput{16}{72}{$\vec{n}_{5}$}
					\cput{66}{81}{$\vec{n}_{6}$}	
				\end{overpic}
	}
	\caption{Theorem~\ref{th:shape} (ii): $g({\vec{v}})$ is a spherical pseudo-quadrilateral, $f_3,f_4,f_5,f_6$ are inflection faces}\label{fig:shape2}
\end{figure}

\item If $K({\vec{v}})< 0$ and $\alpha_f > \pi$ for exactly one $f \sim {\vec{v}}$, then $g({\vec{v}})$ is a spherical pseudo-triangle. In the case that the face $f \sim {\vec{v}}$ with $\alpha_f > \pi$ is an inflection face, the three corners of $g({\vec{v}})$ are the normals of the three other inflection faces. If this $f$ is not an inflection face, ${\vec{n}}_f$ is a corner of $g({\vec{v}})$ and there are just two inflection faces in the star of $\vec{v}$, whose normals are the other two corners of $g({\vec{v}})$.

\begin{figure}[!ht]
	\centerline{
		\begin{overpic}[height=0.28\textwidth]{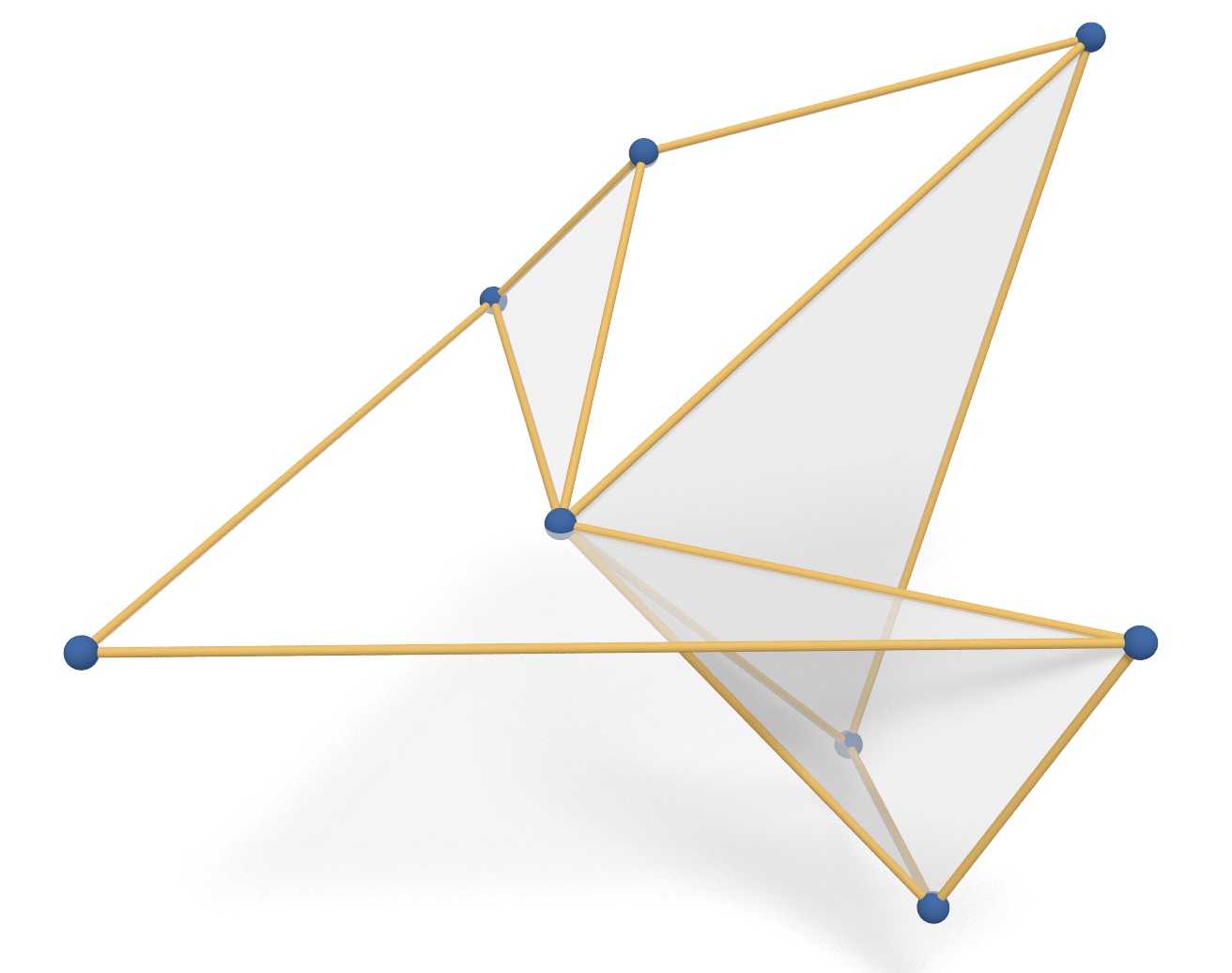}
			\put(42,52){\contour{white}{$f_1$}}
			\put(25,32){\contour{white}{$f_2$}}
			\put(75,18){\contour{white}{$f_3$}}
			\put(63,15){\contour{white}{$f_4$}}
			\put(65,42){\contour{white}{$f_5$}}
			\put(55,55){\contour{white}{$f_6$}}	
		\end{overpic}
		\relax\\
		\begin{overpic}[height=0.28\textwidth]{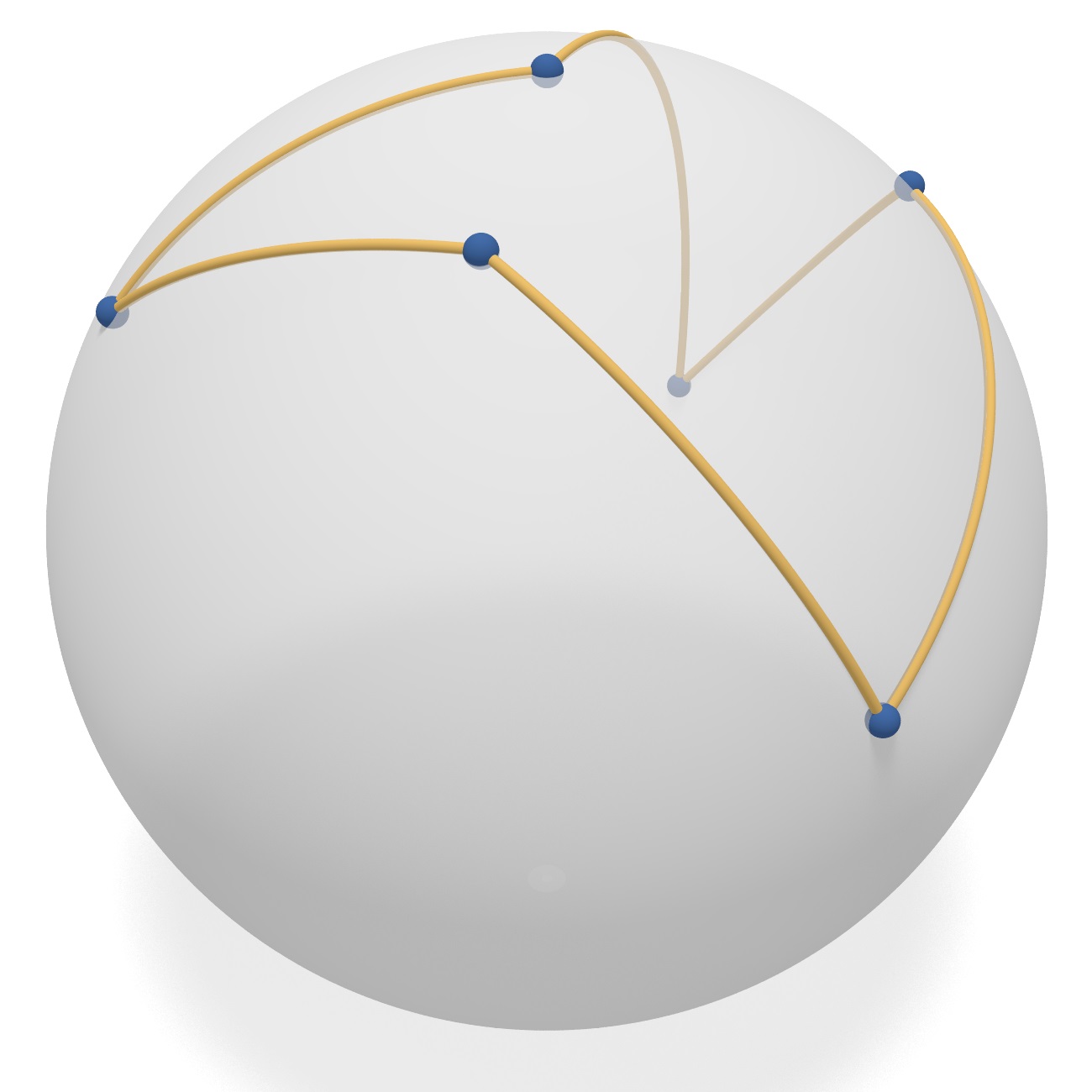}
				\cput{11}{65}{$\vec{n}_{1}$}
				\cput{47}{98}{$\vec{n}_{2}$}
				\cput{70}{62}{$\vec{n}_{3}$}
				\cput{84}{86}{$\vec{n}_{4}$}
				\cput{85}{27}{$\vec{n}_{5}$}
				\cput{41}{71}{$\vec{n}_{6}$}
		\end{overpic}
	}
	\caption{Theorem~\ref{th:shape} (iii): a) $g({\vec{v}})$ is a spherical pseudo-triangle and there are exactly four inflection faces $f_1,f_2,f_3,f_5$}\label{fig:shape31}
\end{figure}


\begin{figure}[!ht]
	\centerline{
				\begin{overpic}[height=0.28\textwidth]{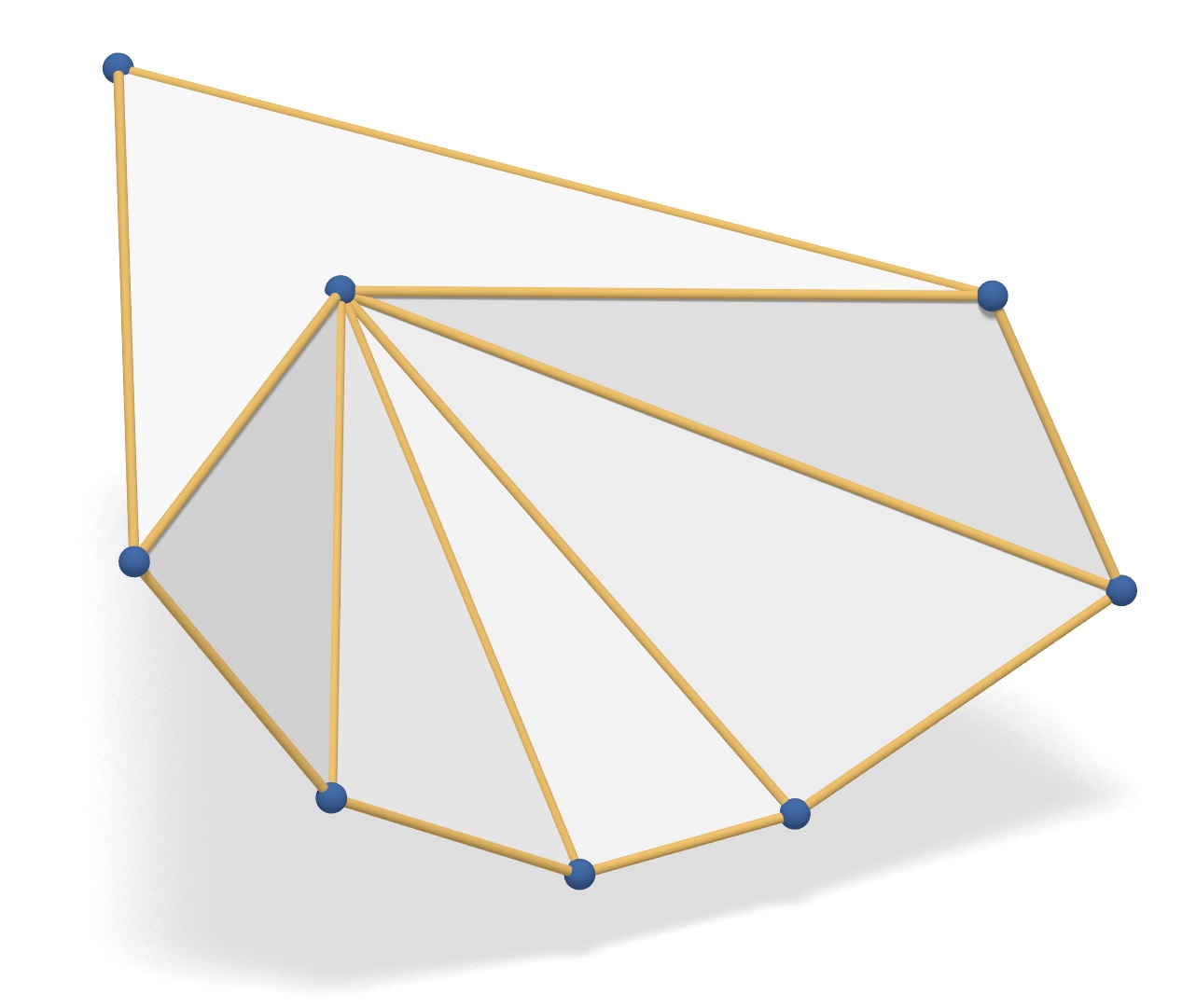}
					\put(78,47){\contour{white}{$f_1$}}
					\put(18,65){\contour{white}{$f_2$}}
					\put(17,36){\contour{white}{$f_3$}}
					\put(31,21){\contour{white}{$f_4$}}
					\put(50,18){\contour{white}{$f_5$}}
					\put(66,28){\contour{white}{$f_6$}}
				\end{overpic}
				\relax\\
				\begin{overpic}[height=0.28\textwidth]{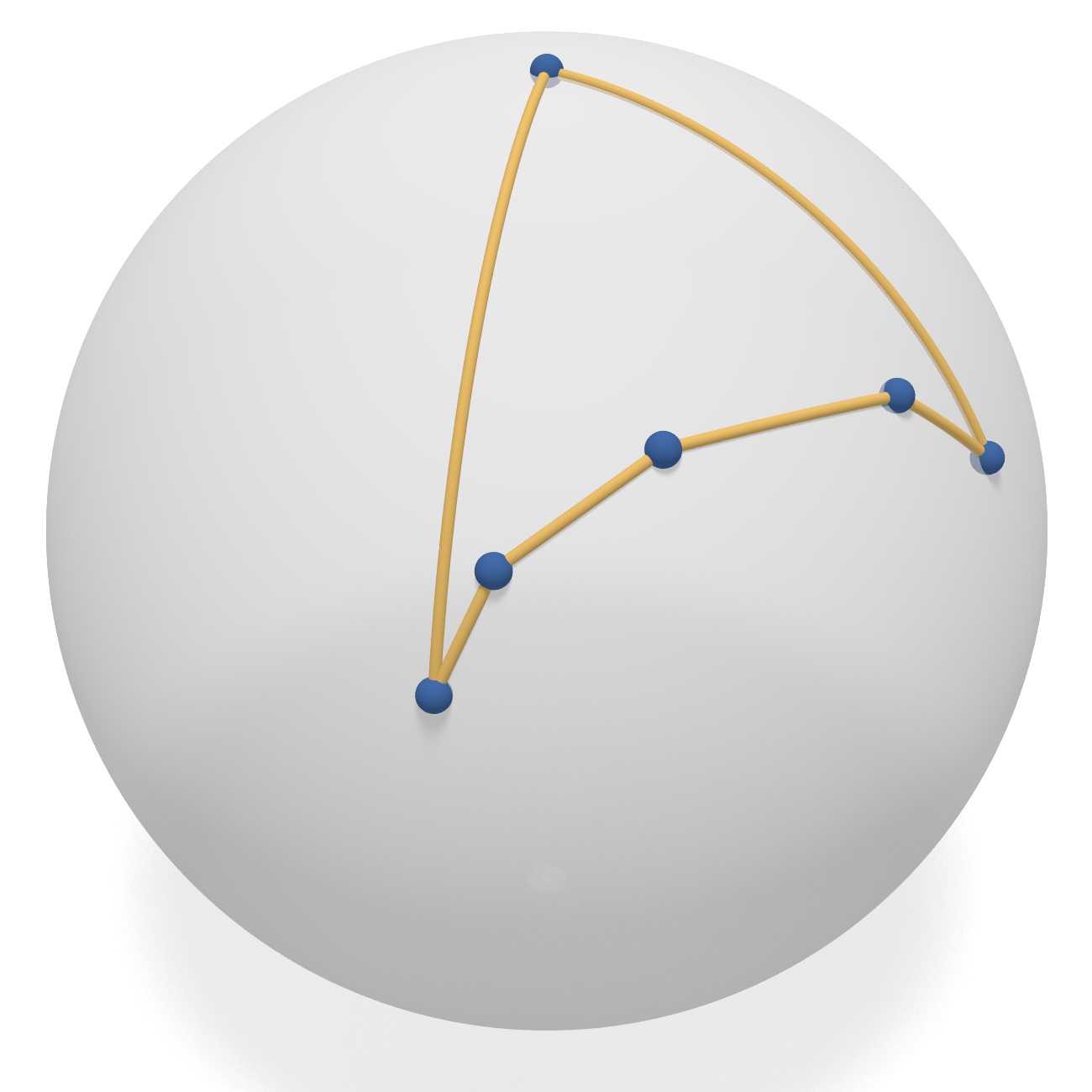}
					\cput{36}{29}{$\vec{n}_{1}$}
					\cput{43}{95}{$\vec{n}_{2}$}
					\cput{94}{52}{$\vec{n}_{3}$}
					\cput{81}{56}{$\vec{n}_{4}$}
					\cput{65}{52.5}{$\vec{n}_{5}$}
					\cput{52}{43}{$\vec{n}_{6}$}	
				\end{overpic}
	}
	\caption{Theorem~\ref{th:shape} (iii): b) $g({\vec{v}})$ is a spherical pseudo-triangle and there are exactly two inflection faces $f_1,f_4$}\label{fig:shape32}
\end{figure}


\item If $K({\vec{v}})< 0$ and $\alpha_f > \pi$ for more than one $f \sim {\vec{v}}$, then $g({\vec{v}})$ is a spherical pseudo-digon. There are exactly two faces $f$ such that $\alpha_f > \pi$ and these two faces are inflection faces. The two corners of $g({\vec{v}})$ are the normals of the other two inflection faces in the vertex star.

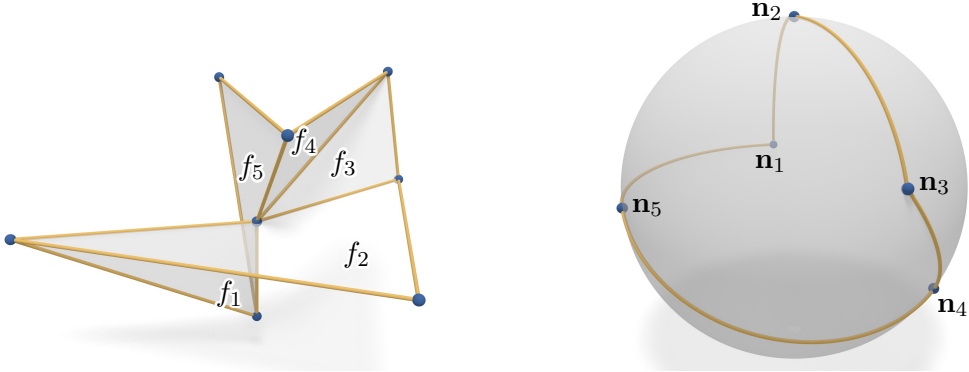
\begin{figure}[!ht]
	\centerline{
		\begin{overpic}[height=0.25\textwidth]{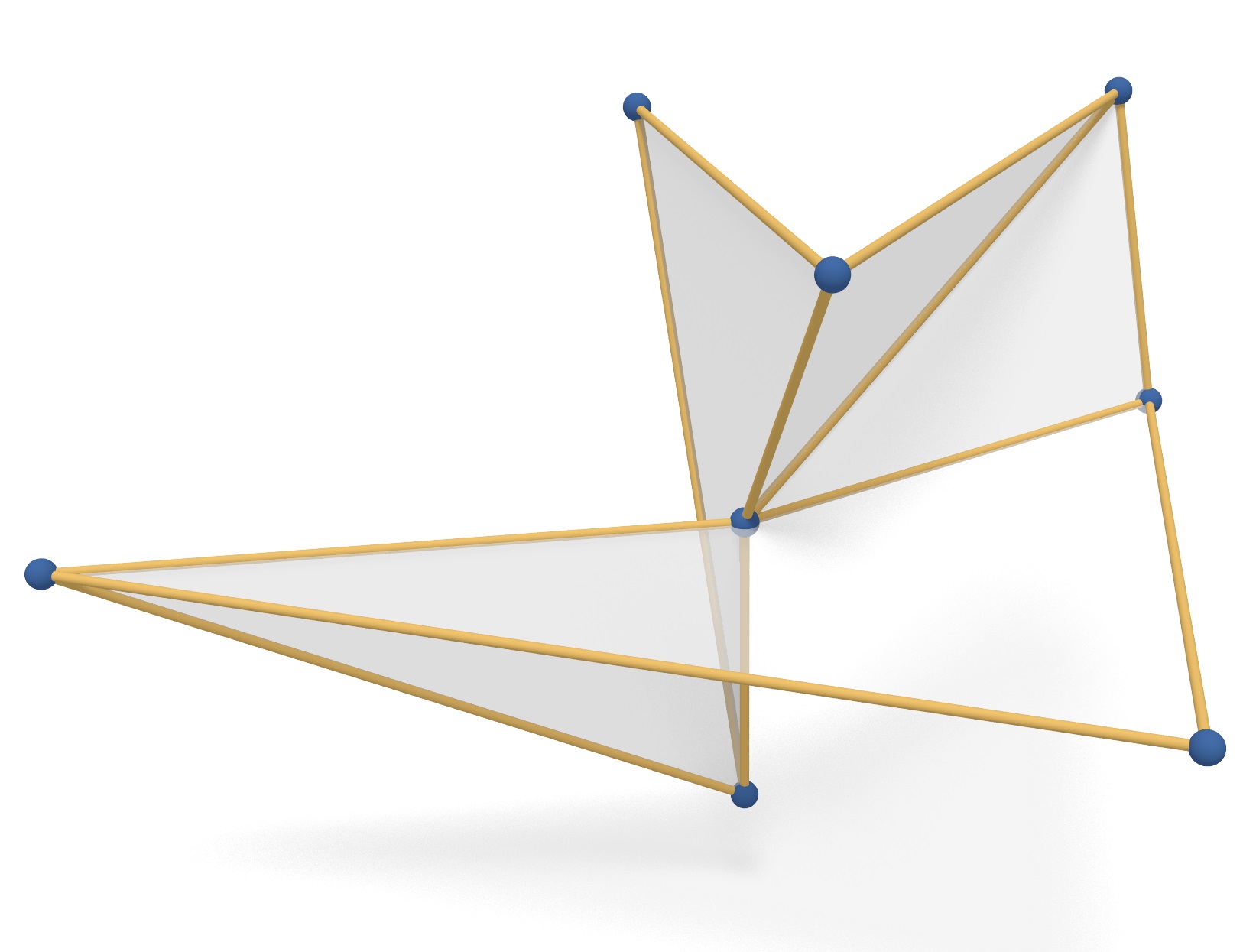}
			\put(50,16){\contour{white}{$f_1$}}
			\put(79,25){\contour{white}{$f_2$}}
			\put(76,46){\contour{white}{$f_3$}}
			\put(67,51){\contour{white}{$f_4$}}
			\put(55,45){\contour{white}{$f_5$}}
		\end{overpic}
		\relax\\
		\begin{overpic}[height=0.28\textwidth]{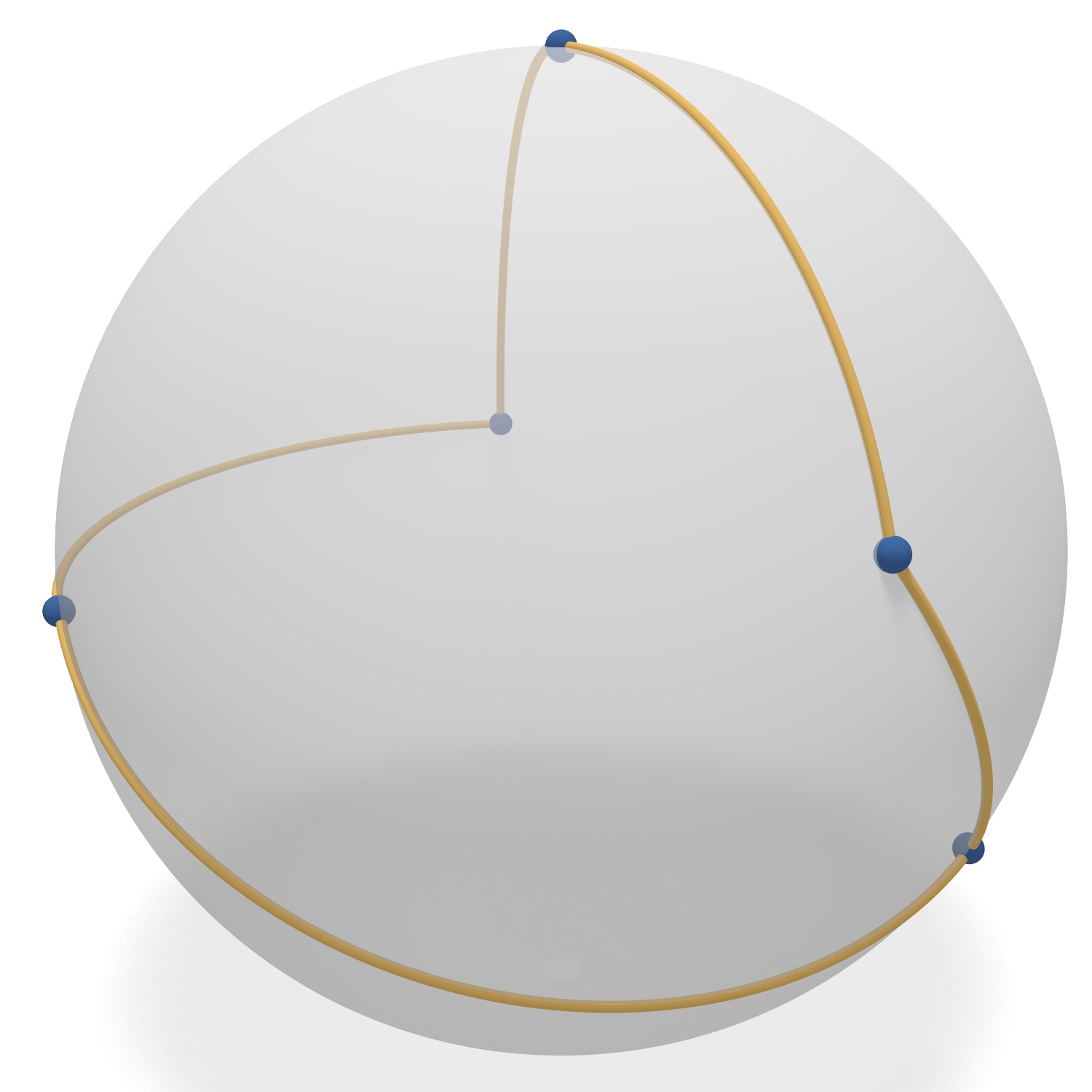}
			\cput{45}{55}{$\vec{n}_{1}$}
			\cput{44}{96}{$\vec{n}_{2}$}
			\cput{89}{49}{$\vec{n}_{3}$}
			\cput{94}{16}{$\vec{n}_{4}$}
			\cput{12}{43}{$\vec{n}_{5}$}
		\end{overpic}
	}
	\caption{Theorem~\ref{th:shape} (iv): $g({\vec{v}})$ is a spherical pseudo-digon, $f_1,f_2,f_3,f_4$ are inflection faces}\label{fig:shape4}
\end{figure}

\end{enumerate}
\end{theorem}
\begin{proof}
For a face $f\sim{\vec{v}}$, let $\hat{\alpha}_f>0$ be the interior angle of $g({\vec{v}})$ at ${\vec{n}}_{f}$. Let $n$ be the number of faces in the star of $\vec{v}$. Then, $g({\vec{v}})$ is a spherical $n$-gon without self-intersections, so it can be triangulated into $n-2$ spherical triangles. Thus, its algebraic area is $\pm\left(\sum_{f\sim{\vec{v}}}\hat{\alpha}_f-(n-2)\pi\right)$ depending on the orientation of $g({\vec{v}})$.

The orientation of $g({\vec{v}})$ is the same as the one of the vertex star if $K({\vec{v}})> 0$, and it is opposite to it if $K({\vec{v}})< 0$. It follows by Lemma~\ref{lem:area} that if $K({\vec{v}})> 0$, then \begin{align}&\sum\limits_{f\sim{\vec{v}}}\hat{\alpha}_f-(n-2)\pi=K({\vec{v}})=2\pi-\sum\limits_{f\sim{\vec{v}}} \alpha_f\notag \\ \Longrightarrow &\sum\limits_{f\sim{\vec{v}}}\hat{\alpha}_f=\sum\limits_{f\sim{\vec{v}}} (\pi-\alpha_f)\label{eq:area1},\end{align}

and that if $K({\vec{v}})< 0$, then \begin{align}&-\sum\limits_{f\sim{\vec{v}}}\hat{\alpha}_f+(n-2)\pi=K({\vec{v}})=2\pi-\sum\limits_{f\sim{\vec{v}}} \alpha_f\notag \\ \Longrightarrow &\sum\limits_{f\sim{\vec{v}}}\hat{\alpha}_f=\sum\limits_{f\sim{\vec{v}}} (\pi+\alpha_f)-4\pi.\label{eq:area2}\end{align}

(i) By Lemma~\ref{lem:angle}~(i) and (ii), $\hat{\alpha}_f=\pi-\alpha_f$ if $f$ is not an inflection face and $\hat{\alpha}_f=2\pi-\alpha_f$ otherwise. Comparison with Equation~(\ref{eq:area1}) yields $\hat{\alpha}_f=\pi-\alpha_f$ for all $f$, so $g({\vec{v}})$ is convex and no face is an inflection face. In particular, the star of $\vec{v}$ forms the boundary of a convex cone.

(ii) Having the opposite orientation of $g({\vec{v}})$ in mind, we deduce from Lemma~\ref{lem:angle}~(i) and (ii) that $\hat{\alpha}_f$ equals $2\pi-(2\pi-\alpha_f)=\alpha_f$ if $f$ is an inflection face and $\hat{\alpha}_f$ equals $2\pi-(\pi-\alpha_f)=\pi+\alpha_f$ else. Comparison with Equation~(\ref{eq:area2}) gives $\hat{\alpha}_f=\pi+\alpha_f$ for all but four $f$. It follows that exactly four faces of the vertex star are inflection faces. Only for such a face $f$, $\hat{\alpha}_f=\alpha_f$ is less than $\pi$. Therefore, $g({\vec{v}})$ is a spherical pseudo-quadrilateral and its four corners are the normals of the inflection faces.

(iii) In addition to the argument given in the second part, by Lemma~\ref{lem:angle}~(iii) and (iv) we have that $\hat{\alpha}_f=2\pi-(2\pi-\alpha_f)=\alpha_f$ if the face $f$ with $\alpha_f> \pi$ is an inflection face and $\hat{\alpha}_f=2\pi-(3\pi-\alpha_f)=\alpha_f-\pi$ else. Note that if $\alpha_f=\pi$, the latter case cannot occur since $g({\vec{v}})$ has no self-intersections.

In the case that the face $f$ with $\alpha_f> \pi$ is an inflection face, we deduce from Equation~(\ref{eq:area2}) that $\hat{\alpha}_{f'}=\alpha_{f'}$ for exactly three other faces $f'$ of the vertex star. Then, $g({\vec{v}})$ is a spherical pseudo-triangle and its three corners are the normals of the inflection faces that have an angle of less than $\pi$ at ${\vec{v}}$.

If the face $f$ with $\alpha_f> \pi$ is not an inflection face, then $\alpha_f \neq \pi$. Hence, comparison with Equation~(\ref{eq:area2}) yields that $\hat{\alpha}_{f'}=\alpha_{f'}$ for exactly two faces $f'$ of the vertex star. Then, $g({\vec{v}})$ is a spherical pseudo-triangle and its three corners are the normals of the two inflection faces and of the face $f$. 

(iv) Let $f$ and $f'$ be two faces of the vertex star for which $\alpha_f,\alpha_{f'}>\pi$. Then, $f$ intersects the planar convex cone spanned by the two edges of $f'$ incident to $\vec{v}$ and vice versa (if say $\alpha_{f'}=\pi$, consider the corresponding opposite half-plane), see Fig.~\ref{fig:concave}. In particular, these two faces cannot share an edge since a disk neighborhood of the vertex star is embedded.

\begin{figure}[!ht]
	\centerline{
		\begin{overpic}[height=0.3\textwidth]{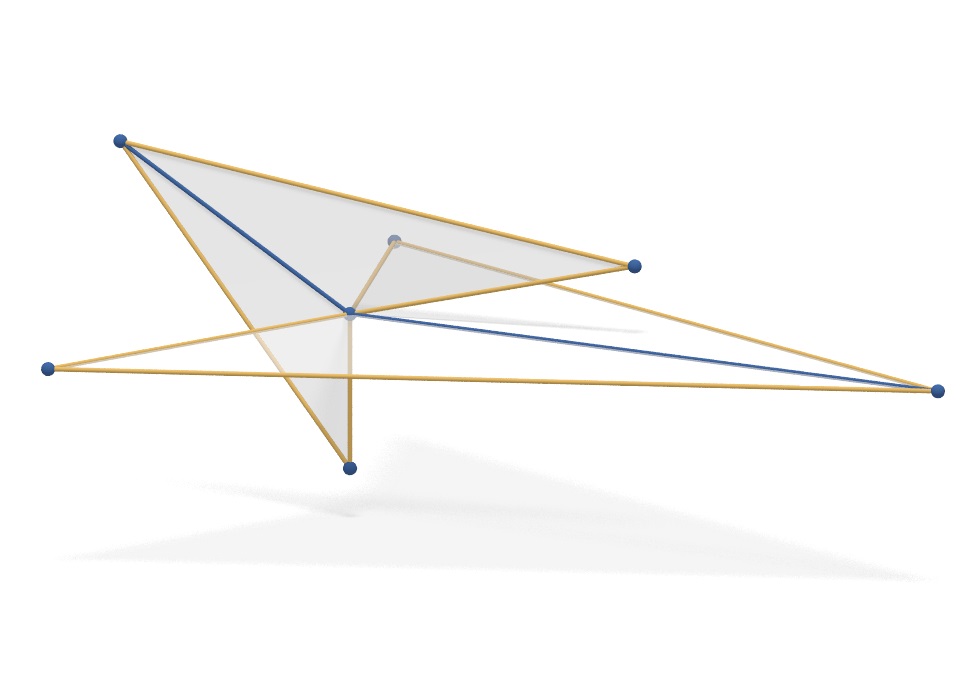}
			\put(19,37){\contour{white}{$f$}}
			\put(50,27){\contour{white}{$f_{'}$}}
			\put(25,39){\contour{white}{$f_1$}}
			\put(32,45){\contour{white}{$f_2$}}
			\put(55,37){\contour{white}{$f_{2'}$}}
			\put(37,33){\contour{white}{$f_{1'}$}}
			\cput{38}{39}{$\vec{v}$}
		\end{overpic}
		\relax\\
		\begin{overpic}[height=0.25\textwidth]{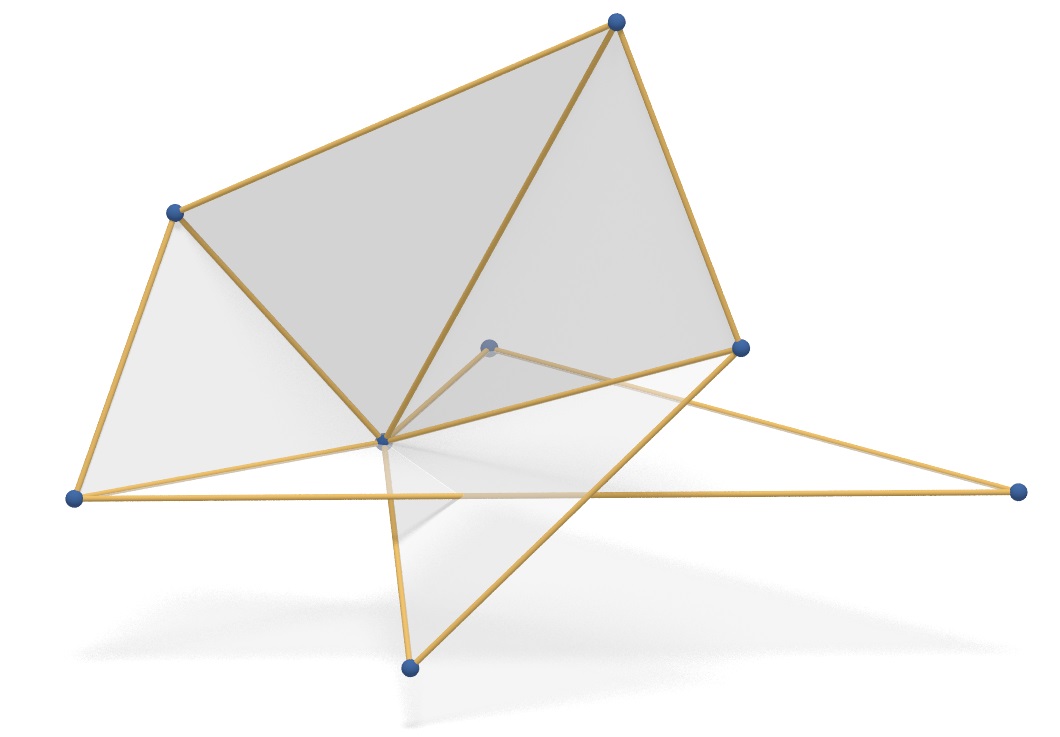}
		\end{overpic}
	}
	\caption{Between any two faces with a reflex angle at $\vec{v}$ there has to be an inflection face}\label{fig:concave}
\end{figure}
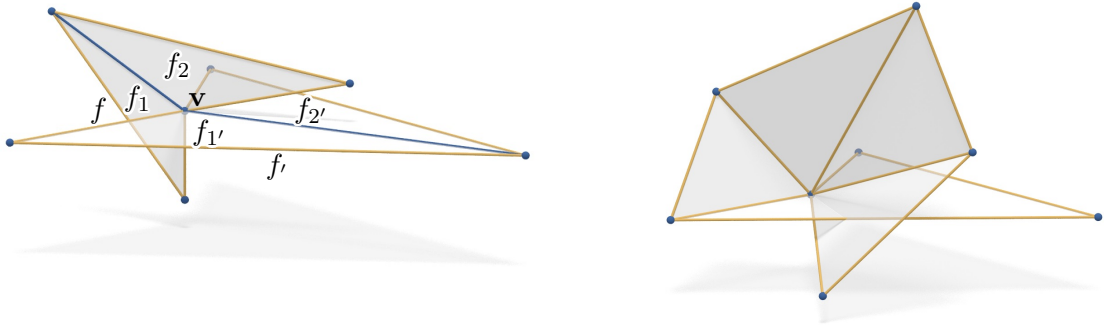

We claim that if we go along the star of $\vec{v}$ from $f$ to $f'$, there has to be at least one inflection face in between. Let us split $f$ and $f'$ each into two faces $f_1$, $f_2$ that have an angle of less than $\pi$ at $\vec{v}$ and suppose that $f_1,f_2,f'_1,f'_2$ appear in counterclockwise order around the vertex star. If we consider the faces in the star of $\vec{v}$ following $f_2$, then they will form part of the boundary of a convex cone till the first inflection face or $f'_2$ is reached.

Suppose there is no inflection face in between $f_2$ and $f'_1$. Then, the faces $f_2$ and the ones following it form part of the boundary of a convex cone. If not all of these faces lie in the same of the two half-spaces determined by $f$, one face would intersect $f$, contradicting embeddedness of a disk neighborhood of the vertex star as in Fig.~\ref{fig:concave}. But $f'$ intersects the plane through $f$, so $f'_1$ and $f'_2$ cannot lie in a common half-space. So assuming the absence of an inflection face between $f$ and $f'$ results in an intersection of these two faces. 

It follows that there at least as many inflection faces as faces $f$ with $\alpha_f>\pi$. Using the arguments given in (ii) and (iii) and comparing with Equation~(\ref{eq:area2}) gives that there are just two faces $f$ with $\alpha_f>\pi$, these two faces are inflection faces, and exactly two other faces are inflection faces as well. The normals of the latter are the corners of $g({\vec{v}})$, which is a spherical pseudo-digon.
\end{proof}

Note that the first three cases of Theorem~\ref{th:shape} occurred in the paper \cite{ORSSSW04} of Orden et al. There, the authors investigated planar embeddings of planar frameworks and their reciprocals. The forces of a self-stress on the planar graph correspond to the dihedral angles of a possibly self-intersecting spherical polyhedron that projects to the graph. Sign changes of forces hence correspond to inflection faces. Considering polarity with respect to the paraboloid $z=x^2+y^2$ then leads to the reciprocal. Orden et al. were interested in the case when both the original framework and its reciprocal are crossing-free, translating to non-self-intersecting Gauss images in our setting. The requirement that the spherical polyhedron projects to the graph translates to the existence of a plane onto which the Gauss image projects in a one-to-one way. In Corollary~\ref{cor:digon} we will show that this condition excludes the digon case. In Section~\ref{sec:global}, we will encounter the dual shapes of faces whose vertices all have non-self-intersecting Gauss images of the same orientation (meaning that the sign of discrete Gaussian curvature is the same for all vertices). If also the neighborhood of a face bijectively projects onto the face plane as was required in \cite{ORSSSW04}, then the cases discussed in Propositions~\ref{prop:face_positive} and~\ref{prop:face_negative} are indeed the projective duals of the vertex stars discussed in Theorem~\ref{th:shape}. Actually, we can take the projective dual under slightly more general circumstances, see Theorem~\ref{th:correlation}.

\begin{remark}
If we think of a smooth saddle of negative Gaussian curvature, we have two hills and two valleys, so four inflections in between. We encounter exactly the same behavior in the generic case of a vertex $\vec{v}$ of discrete negative Gaussian curvature and a Gauss image $g({\vec{v}})$ without self-intersections, see Theorem~\ref{th:shape}~(ii). Note that requiring exactly four inflection faces does not prevent $\vec{v}$ of having self-intersections as is shown in Fig.~\ref{fig:saddlef}.

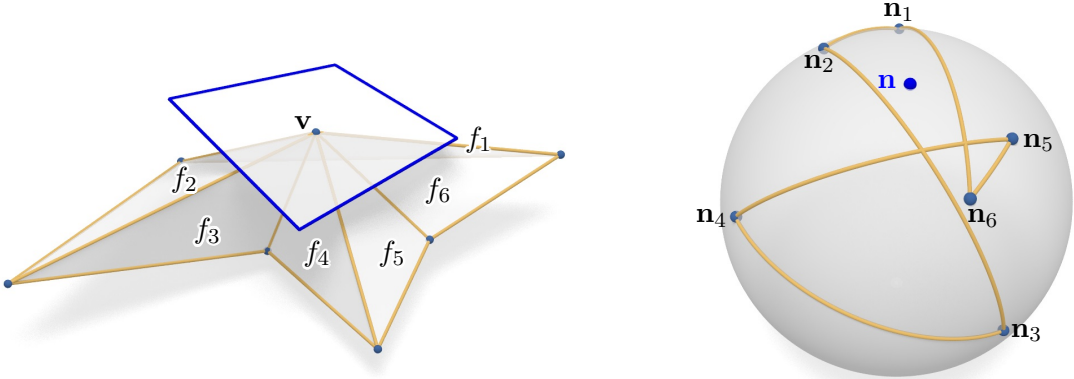
\begin{figure}[htbp]
		\centerline{
			\begin{overpic}[height=0.3\textwidth]{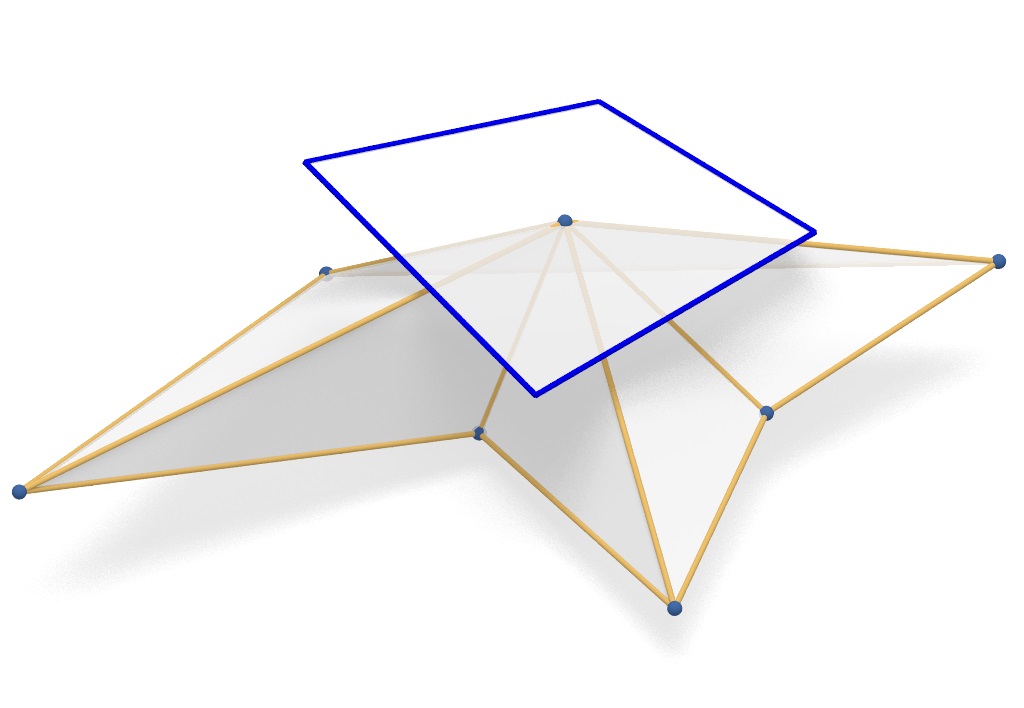}
				\put(81,45){\contour{white}{$f_1$}}
				\put(30,38){\contour{white}{$f_2$}}
				\put(34,29){\contour{white}{$f_3$}}
				\put(53,25){\contour{white}{$f_4$}}
				\put(66,25){\contour{white}{$f_5$}}
				\put(74,36){\contour{white}{$f_6$}}
				\cput{53}{48}{$\vec{v}$}	
			\end{overpic}
			\relax\\
			\begin{overpic}[height=0.3\textwidth]{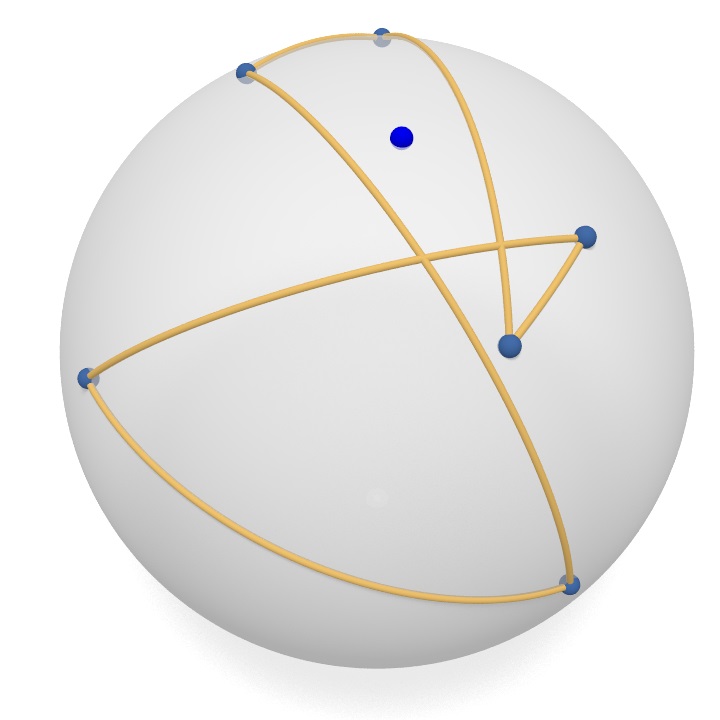}
				\cput{50}{80}{\color{blue}{$\vec{n}$}}
					\cput{53}{98}{$\vec{n}_1$}
					\cput{33}{85}{$\vec{n}_2$}
					\cput{85}{17}{$\vec{n}_3$}
					\cput{6}{46}{$\vec{n}_4$}
					\cput{88}{65}{$\vec{n}_5$}
					\cput{74}{46}{$\vec{n}_6$}
			\end{overpic}
		}
	\caption{Four inflection faces in the vertex star, but $g({\vec{v}})$ has self-intersections}\label{fig:saddlef}
\end{figure}

In Fig.~\ref{fig:saddlef}, we have exactly four inflection faces, but the vertex star does not look like a saddle even though $K({\vec{v}})<0$. Indeed, the whole vertex star lies on the same side of a plane through $\vec{v}$ whose normal ${\vec{n}}$ is in the component of $g({\vec{v}})$ with positive algebraic area. Such a plane does not exist in the classical theory, so it is reasonable to exclude this case by demanding that $g({\vec{v}})$ has no self-intersections rather than just requiring that exactly four faces of the vertex star are inflection faces.

Note that in the second case of Theorem~\ref{th:shape}~(iii), there are just two inflection faces in the vertex star. Having a classical saddle in mind, two inflections are missing; however, they are both incorporated in the face $f$ with $\alpha_f>\pi$. In this sense, $f$ counts as two inflection faces although it is in fact not an inflection face. The two cases of Theorem~\ref{th:shape}~(iii) will lead to slightly different shapes of their discrete Dupin indicatrices, as we will discuss in Section~\ref{sec:Dupin}.
\end{remark}

In Section~\ref{sec:projective}, we will investigate the projective dual of $P$. In the smooth theory, the projective dual of a negatively curved surface is again negatively curved since asymptotic directions are mapped to asymptotic directions. So ideally, the projective dual $P^*$ of a smooth polyhedral surface $P$ of negative curvature should be again a smooth polyhedral surface of negative curvature. Now, the faces of $P^*$ correspond to the faces of the Gauss image of $P$ by projection. But there is no planar polygon that projects to a spherical pseudo-digon as in Theorem~\ref{th:shape}~(iv). Hence, we would like to exclude this case.

Locally, the Gauss image of a smooth surface patch covers just a small piece of the sphere. A reasonable additional condition for smoothness of $P$ is therefore that for any vertex $\vec{v}$, $g({\vec{v}})$ should be contained in some open hemisphere. Not only that this excludes the case of Theorem~\ref{th:shape}~(iv), it also provides us with the existence of a \textit{transverse plane}.

\begin{definition}\label{def:transverse}
A plane $E$ passing through an interior vertex $\vec{v}$ of $P$ is said to be a \textit{transverse plane} if a disk neighborhood of the star of $\vec{v}$ projects orthogonally to $E$ in a one-to-one way.
\end{definition}

In the case of simplicial surfaces, this definition goes back to Banchoff and mimics the corresponding property of the tangent plane of a smooth surface. Existence of a transverse plane is equivalent to the condition that $g({\vec{v}})$ is contained in an open hemisphere.

\begin{proposition}\label{prop:transverse}
There exists a transverse plane through an interior vertex $\vec{v}$ of $P$ if and only if $g({\vec{v}})$ is contained in an open hemisphere.
\end{proposition}
\begin{proof}
Let $E$ be a transverse plane through $\vec{v}$, and let ${\vec{n}}$ be the normal of $E$ consistent with the orientation of $P$, i.e., the counterclockwise direction around the star of $\vec{v}$ agrees with the counterclockwise direction around ${\vec{n}}$.

\begin{figure}[!ht]
	\centerline{
		\begin{overpic}[height=0.3\textwidth]{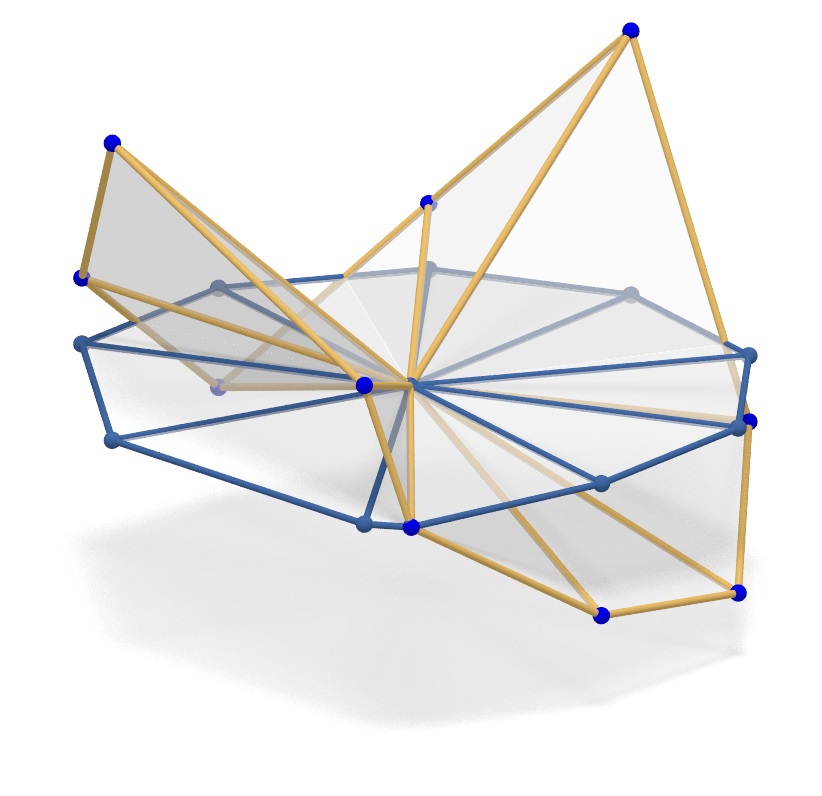}
	\cput{47}{53}{$\vec{v}$}	
		\end{overpic}
		\relax\\
		\begin{overpic}[height=0.3\textwidth]{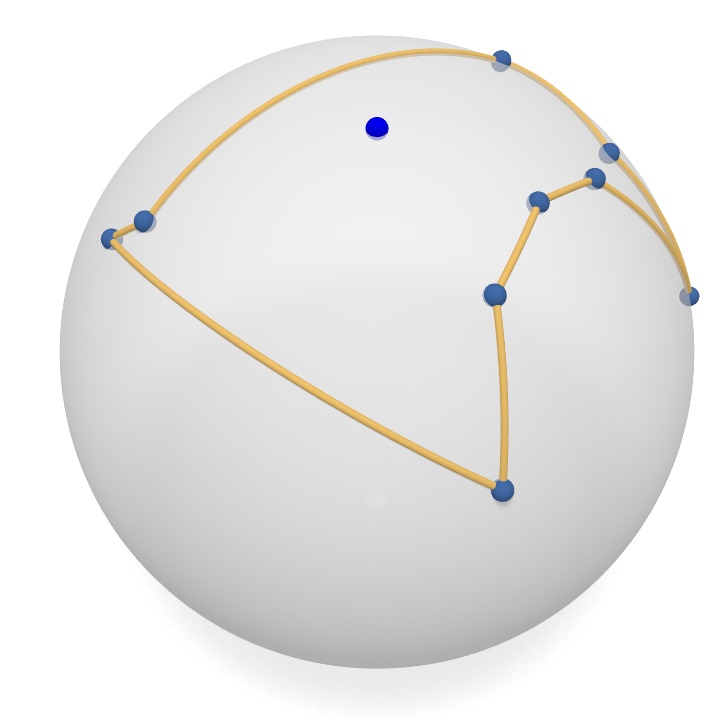}
			\cput{47}{80}{\color{blue}{$\vec{n}$}}	
		\end{overpic}
	}
	\caption{Relation between existence of a transverse plane and the Gauss image}\label{fig:transversal}
\end{figure}

For a face $f \sim {\vec{v}}$, its image under orthogonal projection to $E$ when seen from the side of ${\vec{n}}$ has the same orientation as $f$ if ${\vec{n}}_f$ is contained in the upper hemisphere with pole ${\vec{n}}$, it degenerates to a line segment if ${\vec{n}}_f$ is contained in the equator, and the orientation is reversed if ${\vec{n}}_f$ is contained in the lower hemisphere. Since a disk neighborhood of the star of $\vec{v}$ projects to $E$ in a one-to-one way, it follows that for all faces $f \sim {\vec{v}}$, ${\vec{n}}_f$ is contained in the upper hemisphere with pole ${\vec{n}}$. In particular, $g({\vec{v}})$ is contained in that open hemisphere.

Conversely, if $g({\vec{v}})$ is contained in the open hemisphere defined by a vector ${\vec{n}} \in S^2$, we conclude by a similar reasoning that a disk neighborhood of the star of $\vec{v}$ orthogonally projects to the plane $E$ through $\vec{v}$ orthogonal to ${\vec{n}}$ in a bijective and orientation-preserving way.
\end{proof}

\begin{remark}
Note that the Gauss image only depends on a disk neighborhood of the vertex star. In particular, it may happen that even though there exists a plane onto which a disk neighborhood of the vertex star projects orthogonally and bijectively, it is not true for the whole vertex star. However, such a situation cannot occur for simplicial surfaces.
\end{remark}

Due to Theorem~\ref{th:shape}~(i), $g({\vec{v}})$ is contained in an open hemisphere if it has no self-intersections and $K({\vec{v}})>0$. Hence, the assumption that $g({\vec{v}})$ is contained in an open hemisphere is only a restriction in the case of negative discrete Gaussian curvature.

\begin{definition}
Let $\vec{v}$ be an interior vertex of $P$ and assume that $g({\vec{v}})$ is contained in an open hemisphere with pole ${\vec{n}}\in S^2$ lying inside $g({\vec{v}})$. Then, ${\vec{n}}$ is said to be a \textit{discrete normal}.
\end{definition}

\begin{remark}
If ${\vec{n}}$ is a discrete normal, then the plane orthogonal to it is a transverse plane. The reason why ${\vec{n}}$ is a suitable normal of the polyhedral surface, but its orthogonal complement is not necessarily a suitable tangent plane, lies in the fact that we are aiming at projectively invariant notions. Projective transformations do not preserve angles. Since given a discrete normal ${\vec{n}}$, any projection of a disk neighborhood of the vertex star onto a plane along lines parallel to ${\vec{n}}$ is one-to-one, the direction of projection and not the plane itself leads to a definition that is preserved under projective transformations. We will see in the following section how suitable tangent planes are motivated.
\end{remark}

We already mentioned before that if $g({\vec{v}})$ is contained in an open hemisphere $H$, $i({\vec{v}},\xi)$ agrees with the correct algebraic multiplicity of $\xi \in H$ in $g({\vec{v}})$. In the setting that $g({\vec{v}})$ is free of self-intersections, this means that $i({\vec{v}},\xi)=\textnormal{sign } K({\vec{v}})$ for all $\xi$ inside $g({\vec{v}})$.

\begin{corollary}\label{cor:digon}
Let $\vec{v}$ be an interior vertex of $P$ of negative discrete Gaussian curvature and assume that $g({\vec{v}})$ is contained in an open hemisphere. Then, $g({\vec{v}})$ is not a spherical pseudo-digon.
\end{corollary}
\begin{proof}
By Proposition~\ref{prop:transverse}, there is a plane $E$ such that a disk neighborhood of the star of $\vec{v}$ projects orthogonally to $E$ in a one-to-one way. Under orthogonal projection, any reflex angle will be mapped to a reflex angle (or a straight angle in the degenerate case). It follows that the star of $\vec{v}$ contains at most one face $f$ with $\alpha_f>\pi$. In particular, $g({\vec{v}})$ cannot be a spherical pseudo-digon due to the classification result in Theorem~\ref{th:shape}.
\end{proof}

To sum up, our assertion for smoothness of $P$ so far is that the Gauss image of any vertex of non-zero discrete Gaussian curvature has no self-intersections and is contained in an open hemisphere. 


\subsection{Discrete Dupin indicatrices and asymptotic directions}\label{sec:Dupin}

\begin{figure}[htbp]
   \begin{center}
    \includegraphics[scale=1.3]{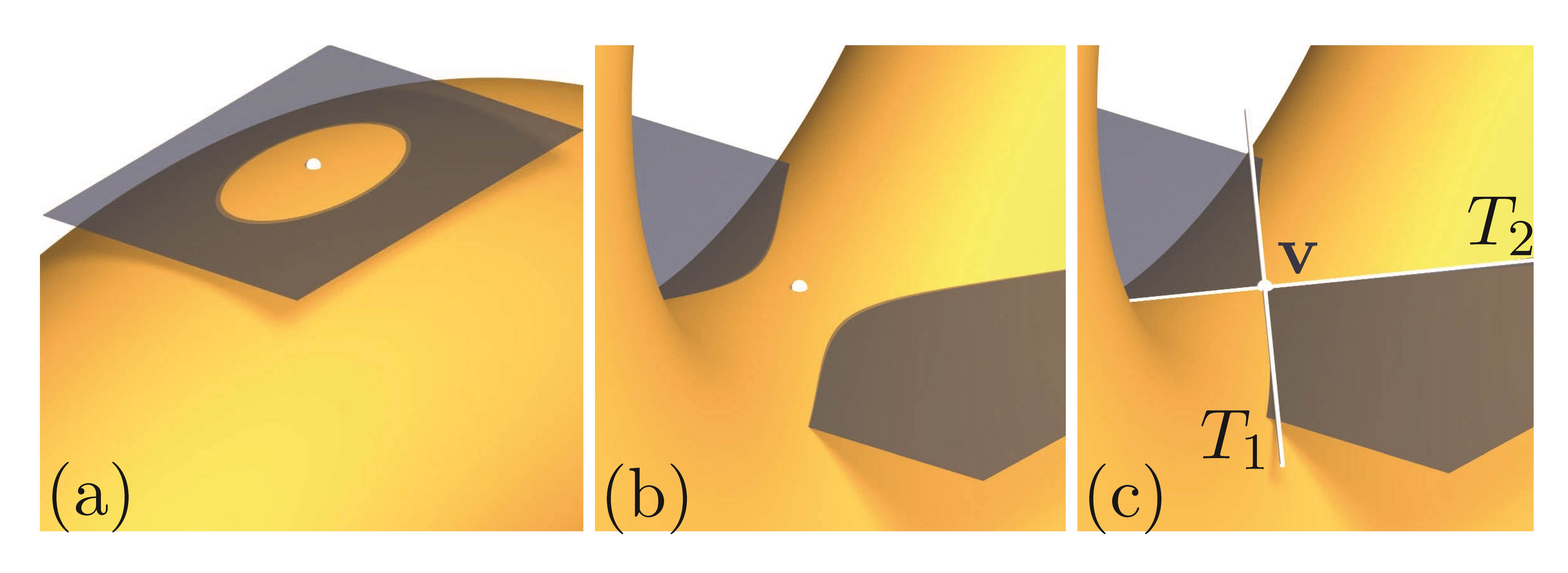}
   \caption{In the smooth theory, the Dupin indicatrix indicates how the intersection of the surface with a plane parallel to the tangent plane at a point that is close to it looks like. (a) For points of positive Gaussian curvature the indicatrix is an ellipse. (b) For points of negative Gaussian curvature it is a hyperbola whose asymptotes indicate the asymptotic directions. (c) In the latter case, the asymptotic directions $T_1,T_2$ at a point are defined as the tangents of the two smooth curves that arise in the intersection of the tangent plane with the surface}
   \label{fig:dupin}
	\end{center}
\end{figure}


\subsubsection{Discrete Dupin indicatrices}\label{sec:discrete_Dupin}

In the following, we will discuss a discrete analog of the Dupin indicatrix for polyhedral surfaces. For this purpose, we propose as suitable tangent planes to a vertex $\vec{v}$ planes orthogonal to some ${\vec{n}}$ such that $g({\vec{v}})$ is star-shaped with respect to ${\vec{n}}$. Of course, this is an additional requirement to $g({\vec{v}})$ if $K({\vec{v}})<0$, but we will explain why it is well-motivated.

\begin{definition}
A convex polygon is said to be a \textit{discrete ellipse}, see for example Fig.~\ref{fig:Dupin_positive}. A \textit{discrete hyperbola} is the union of two infinite simple polylines, none of which contains an inflection edge, such that the convex hulls of the polylines are disjoint as in Fig.~\ref{fig:Dupin_quad}.
\end{definition}

The case of positive discrete Gaussian curvature is fairly simple.

\begin{proposition}\label{prop:Dupin_positive}
Assume that $K({\vec{v}})>0$ and that $g({\vec{v}})$ has no self-intersec\-tions. Then, $g({\vec{v}})$ is star-shaped with respect to any point ${\vec{n}}\in S^2$ in its interior. A plane $E$ orthogonal to ${\vec{n}}$ and close to, but not passing through $\vec{v}$, either does not intersect the star of $\vec{v}$ or intersects it in a discrete ellipse.

\begin{figure}[!ht]
	\centerline{
		\begin{overpic}[height=0.3\textwidth]{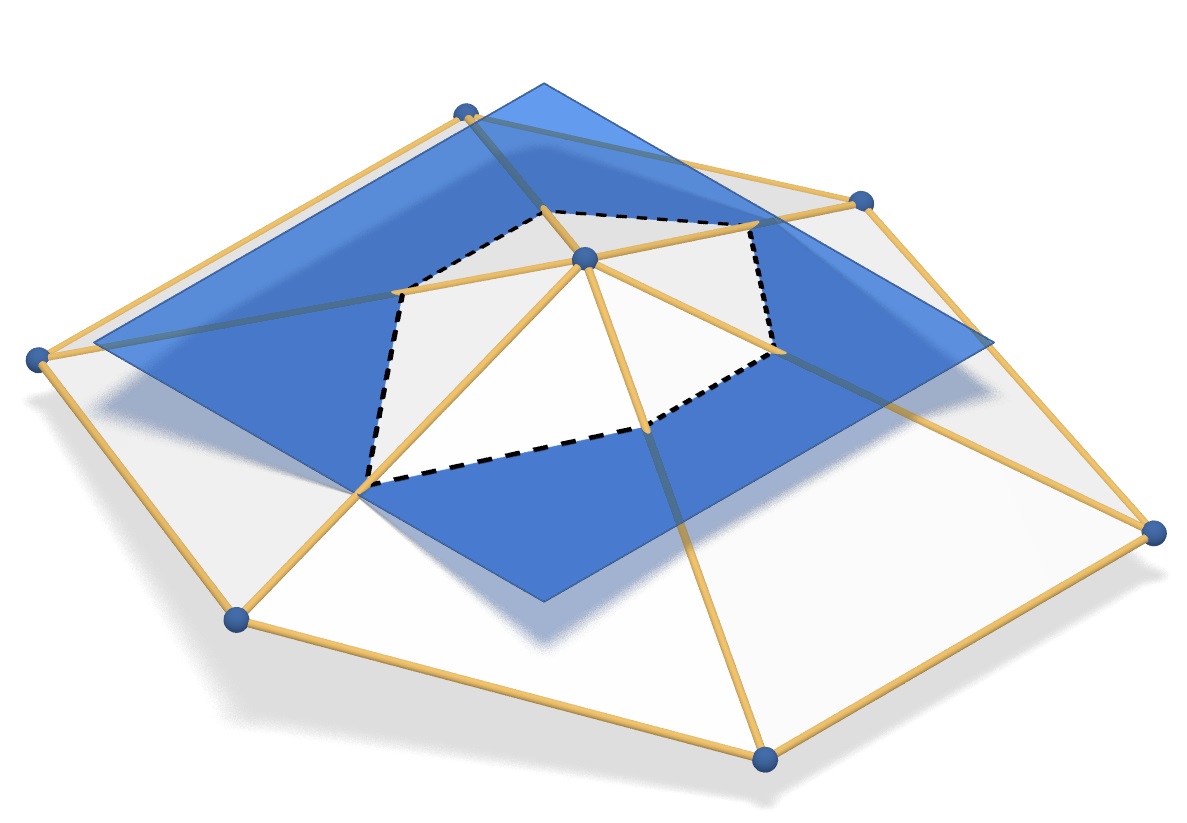}
			\put(45,26){\contour{white}{$f_1$}}
			\put(63,32){\contour{white}{$f_2$}}
			\put(68,44){\contour{white}{$f_3$}}
			\put(51,52){\contour{white}{$f_4$}}
			\put(33,50){\contour{white}{$f_5$}}
			\put(25,34){\contour{white}{$f_6$}}
		\end{overpic}
		\relax\\
		\begin{overpic}[height=0.3\textwidth]{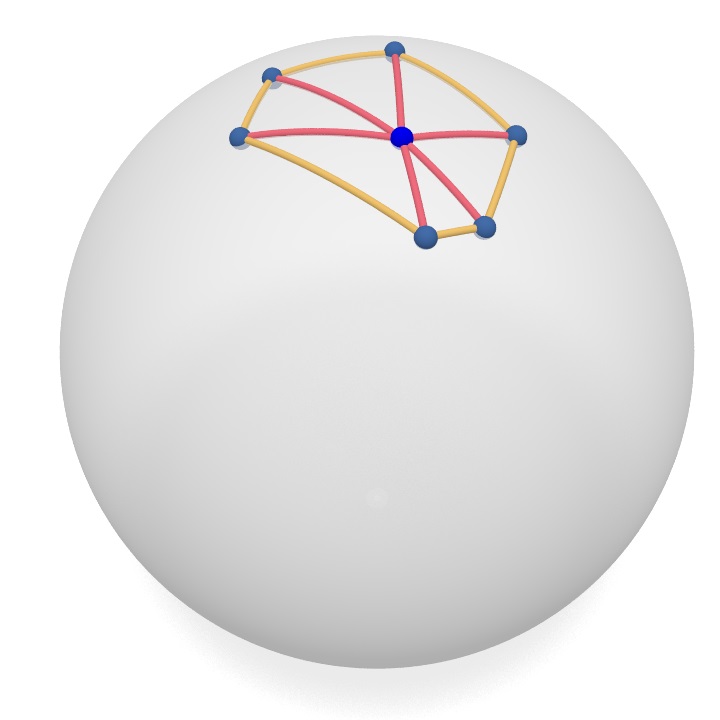}
			\cput{61}{61}{$\vec{n}_{1}$}
			\cput{72}{62}{$\vec{n}_{2}$}
			\cput{78}{79}{$\vec{n}_{3}$}
			\cput{61}{93}{$\vec{n}_{4}$}
			\cput{33}{92}{$\vec{n}_{5}$}
			\cput{31}{75}{$\vec{n}_{6}$}	
				\cput{60}{82}{\color{blue}{$\vec{n}$}}
		\end{overpic}
	}
	\caption{Proposition~\ref{prop:Dupin_positive}: Discrete ellipse as discrete Dupin indicatrix}\label{fig:Dupin_positive}
\end{figure}
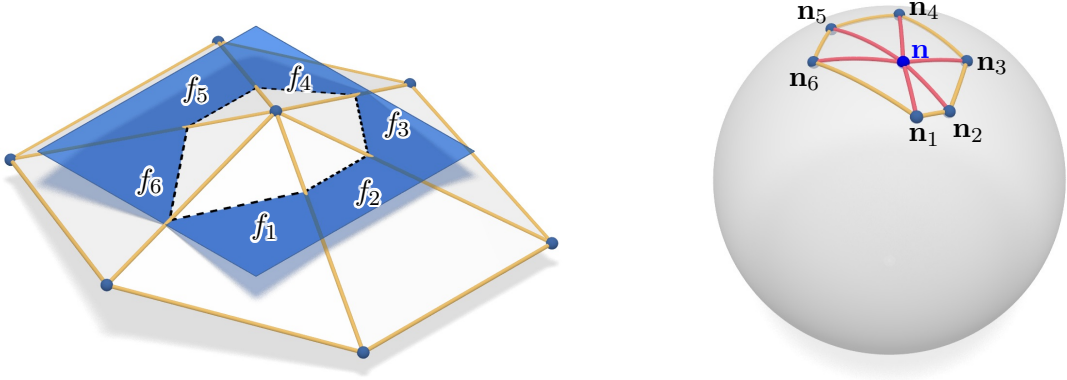

\end{proposition}
\begin{proof}
By Theorem~\ref{th:shape}~(i), $g({\vec{v}})$ is a convex spherical polygon, in particular star-shaped with respect to any of its interior points. Furthermore, $\vec{v}$ is convex. It follows that for any ${\vec{n}}$ in the interior of $g({\vec{v}})$, the star of $\vec{v}$ lies on one side of the plane orthogonal to ${\vec{n}}$ and passing through $\vec{v}$. If the plane is moved toward the vertex star, then the intersection is a convex polygon, if it is moved away, then the intersection is empty.
\end{proof}

\begin{theorem}\label{th:Dupin_negative}
To prevent disconnected intersections in the intersection of a plane with a face in the star of $\vec{v}$, we consider in the following an embedded disk neighborhood of $\vec{v}$ in the star of $\vec{v}$ that is composed of one triangle per original face $f \sim {\vec{v}}$ with convex angle at $\vec{v}$ and two triangles per face with a reflex angle. For simplicity, we consider these triangles to be infinitely large and only bounded by their edges incident to $\vec{v}$. The resulting infinite vertex star is denoted by $V$.

Assume that $K({\vec{v}})<0$ and that $g({\vec{v}})$ has no self-intersections. Let ${\vec{n}}\in S^2$ be a point in its interior and $E$ be a plane orthogonal to ${\vec{n}}$ and not passing through $\vec{v}$. 

\begin{enumerate}
\item If $-{\vec{n}}$ is contained in $g({\vec{v}})$, then the intersection of $E$ with $V$ has three connected components.

\begin{figure}[!ht]
	\centerline{
		\begin{overpic}[height=0.3\textwidth]{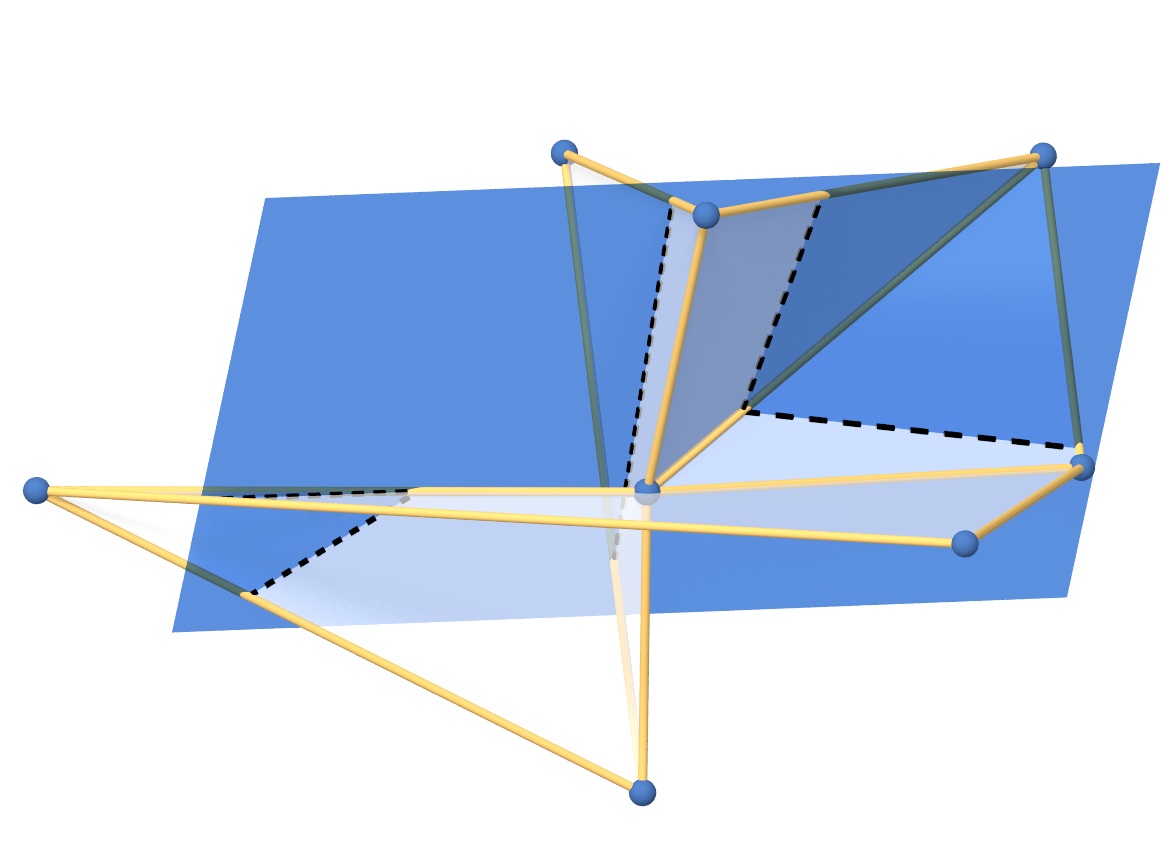}
			\put(40,22){\contour{white}{$f_1$}}
			\put(76,30){\contour{white}{$f_2$}}
			\put(80,43){\contour{white}{$f_3$}}
			\put(65,49){\contour{white}{$f_4$}}
			\put(53,45){\contour{white}{$f_5$}}
		\end{overpic}
		\relax\\
		\begin{overpic}[height=0.3\textwidth]{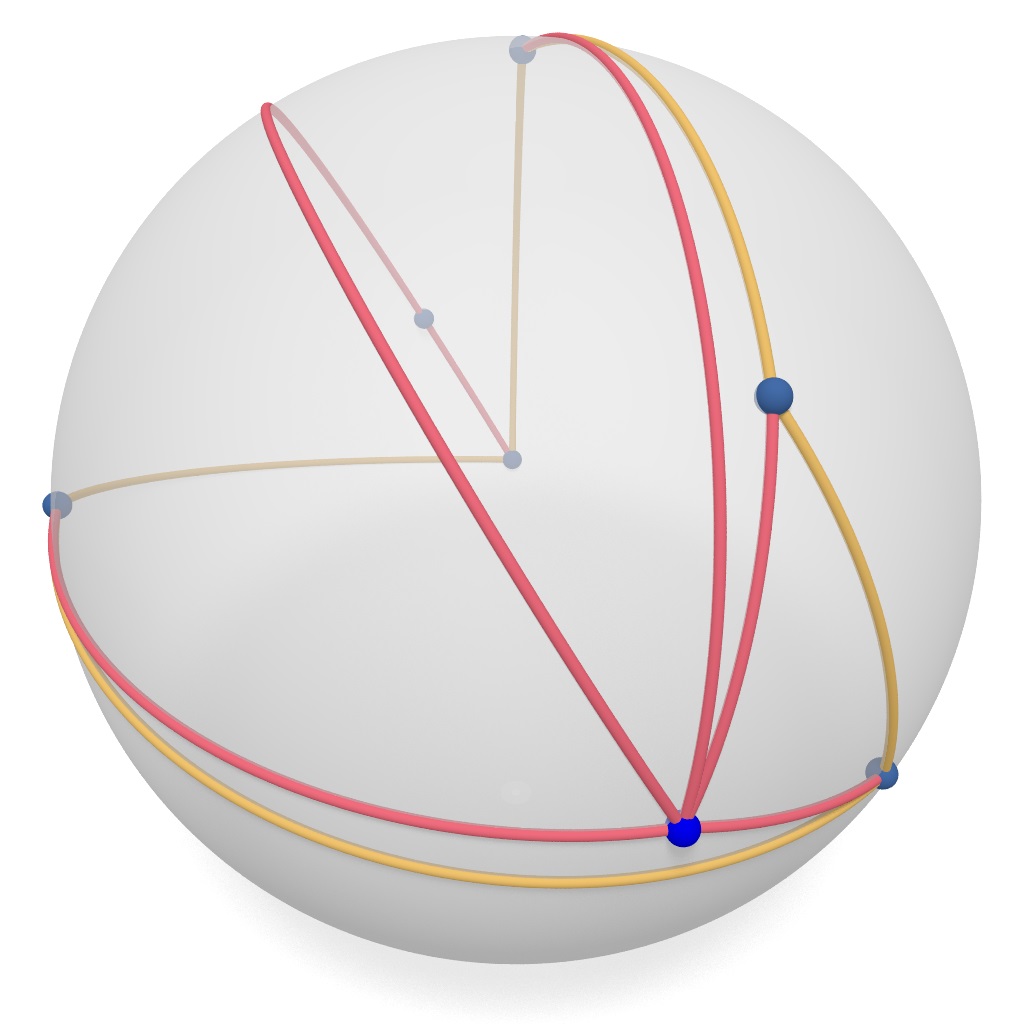}
			\cput{56}{53}{$\vec{n}_{1}$}
			\cput{53}{98}{$\vec{n}_{2}$}
			\cput{83}{59}{$\vec{n}_{3}$}
			\cput{93}{21}{$\vec{n}_{4}$}
			\cput{5}{53}{$\vec{n}_{5}$}
			\cput{43}{70}{$\mbox{-}\vec{n}$}
			\cput{62.5}{19}{\color{blue}{$\vec{n}$}}
		\end{overpic}
	}
	\caption{Theorem~\ref{th:Dupin_negative}~(i): Three components when $\mbox{-}{\vec{n}}$ is contained in $g({\vec{v}})$}\label{fig:Dupin_antipodal}
\end{figure}

\item Assume now that $-{\vec{n}}$ is not contained in $g({\vec{v}})$. Then, the intersection of $E$ with $V$ consists of two distinct polylines. If and only if $g({\vec{v}})$ is star-shaped with respect to ${\vec{n}}$, then the intersection of any such $E$ with $V$ is a \textit{discrete hyperbola}. There exist planes $E$ such that one of the polylines is a single straight line segment if and only if the vertex star contains a face $f$ with $\alpha_f > \pi$ which is not an inflection face, see Fig.~\ref{fig:Dupin_triangle_degenerate}.

\begin{figure}[!ht]
	\centerline{
		\begin{overpic}[height=0.3\textwidth]{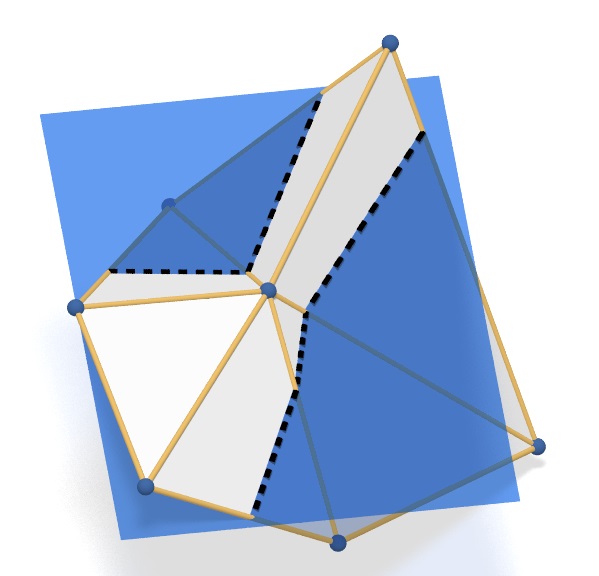}
			\put(25,34){\contour{white}{$f_1$}}
			\put(35,22){\contour{white}{$f_2$}}
			\put(61,26){\contour{white}{$f_3$}}
			\put(60,50){\contour{white}{$f_4$}}
			\put(36,60){\contour{white}{$f_5$}}
			\put(25,50){\contour{white}{$f_6$}}
			\cput{48}{48}{$\vec{v}$}	
		\end{overpic}
		\relax\\
		\begin{overpic}[height=0.3\textwidth]{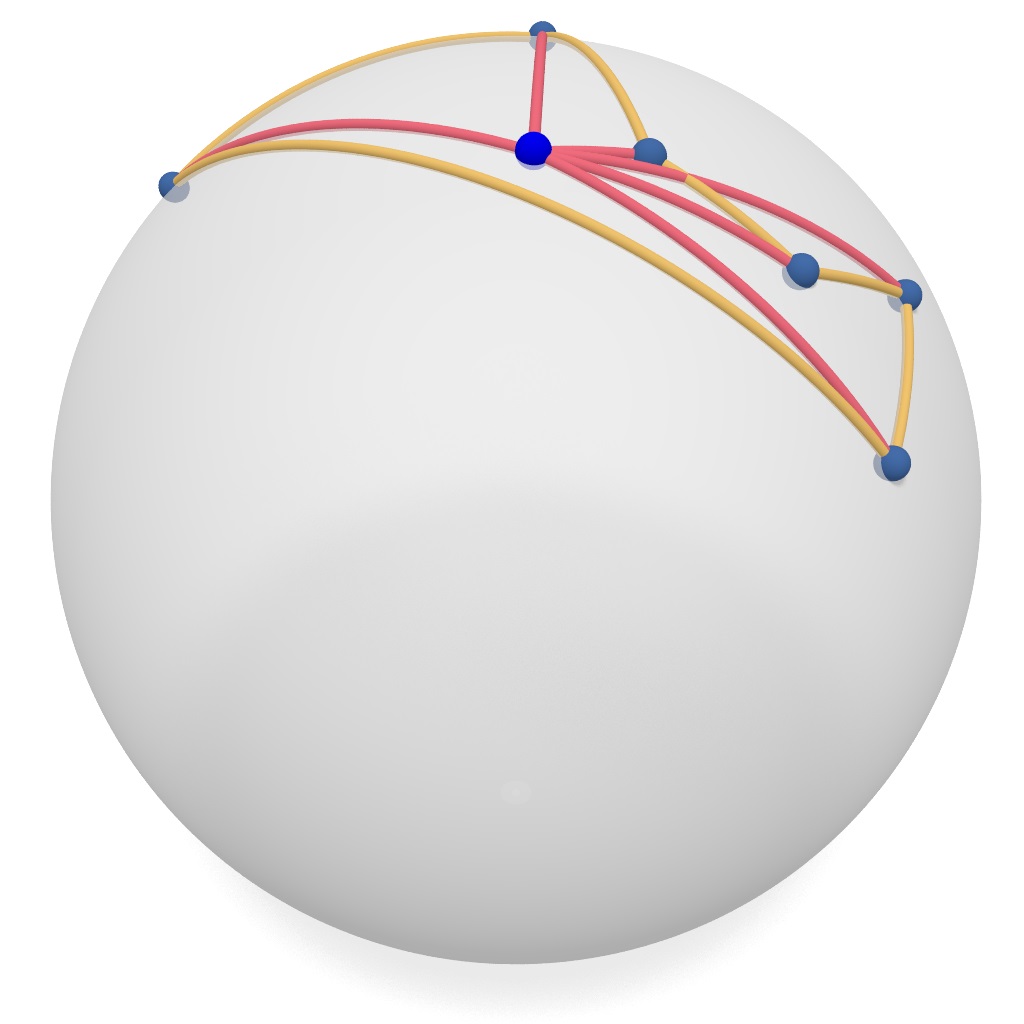}
			\cput{70}{85}{$\vec{n}_{1}$}
			\cput{82}{67}{$\vec{n}_{2}$}
			\cput{95}{69}{$\vec{n}_{3}$}
			\cput{90}{49}{$\vec{n}_{4}$}
			\cput{11}{81}{$\vec{n}_{5}$}
			\cput{52}{99}{$\vec{n}_{6}$}	
				\cput{48}{88}{\color{blue}{$\vec{n}$}}
		\end{overpic}
	}
	\caption{Theorem~\ref{th:Dupin_negative}~(ii), case (ii) of Theorem~\ref{th:shape}: An inflection edge appears when the Gauss image is not star-shaped}\label{fig:Dupin_quad_not}
\end{figure}

\begin{figure}[!ht]
	\centerline{
		\begin{overpic}[height=0.3\textwidth]{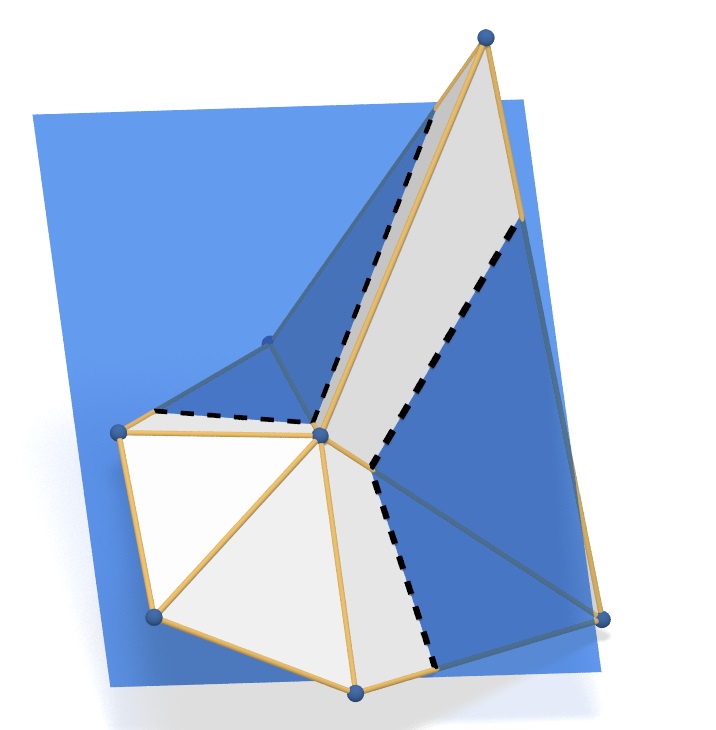}
	\put(25,32){\contour{white}{$f_1$}}
	\put(35,22){\contour{white}{$f_2$}}
	\put(53,20){\contour{white}{$f_3$}}
	\put(60,45){\contour{white}{$f_4$}}
	\put(42,55){\contour{white}{$f_5$}}
	\put(30,43){\contour{white}{$f_6$}}
	\cput{48}{40}{$\vec{v}$}	
		\end{overpic}
		\relax\\
		\begin{overpic}[height=0.3\textwidth]{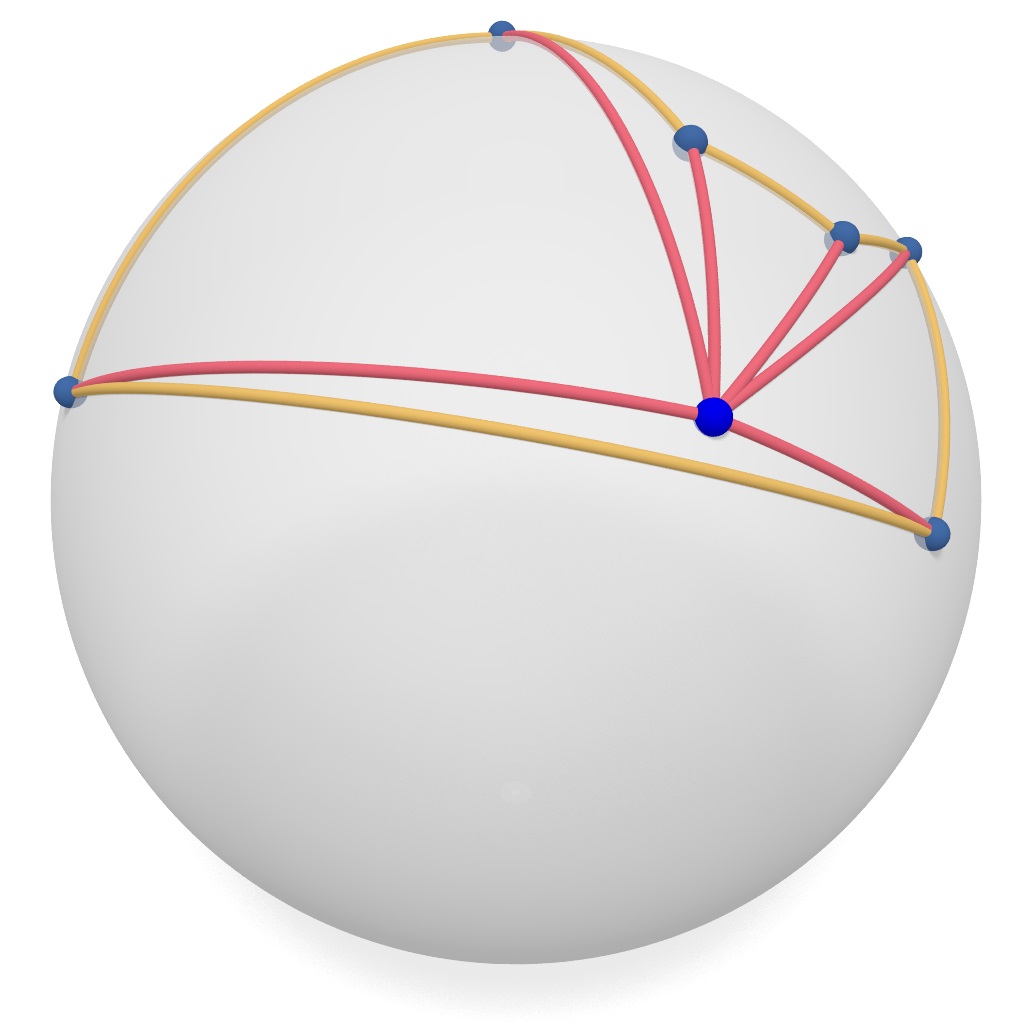}
	\cput{74}{86}{$\vec{n}_{1}$}
	\cput{86}{78.5}{$\vec{n}_{2}$}
	\cput{95}{73}{$\vec{n}_{3}$}
	\cput{90}{42}{$\vec{n}_{4}$}
	\cput{1}{61}{$\vec{n}_{5}$}
	\cput{49}{99}{$\vec{n}_{6}$}	
	\cput{64}{62}{\color{blue}{$\vec{n}$}}
		\end{overpic}
	}
	\caption{Theorem~\ref{th:Dupin_negative}~(ii), case (ii) of Theorem~\ref{th:shape}: Discrete hyperbola as discrete Dupin indicatrix}\label{fig:Dupin_quad}
\end{figure}

\begin{figure}[!ht]
	\centerline{
		\begin{overpic}[height=0.3\textwidth]{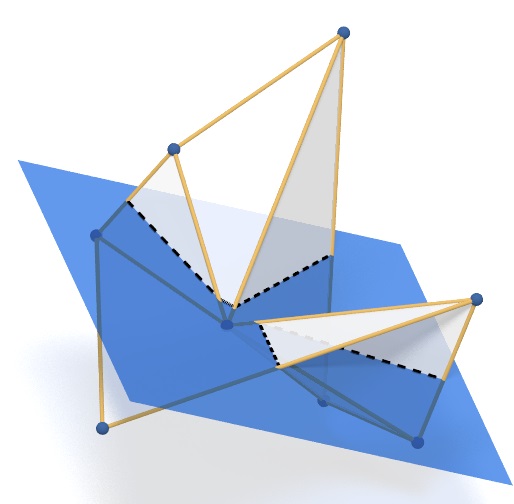}
			\put(30,48){\contour{white}{$f_1$}}
			\put(28,32){\contour{white}{$f_2$}}
			\put(73,26){\contour{white}{$f_3$}}
			\put(60,20){\contour{white}{$f_4$}}
			\put(55,55){\contour{white}{$f_5$}}
			\put(43,60){\contour{white}{$f_6$}}
		\end{overpic}
		\relax\\
		\begin{overpic}[height=0.3\textwidth]{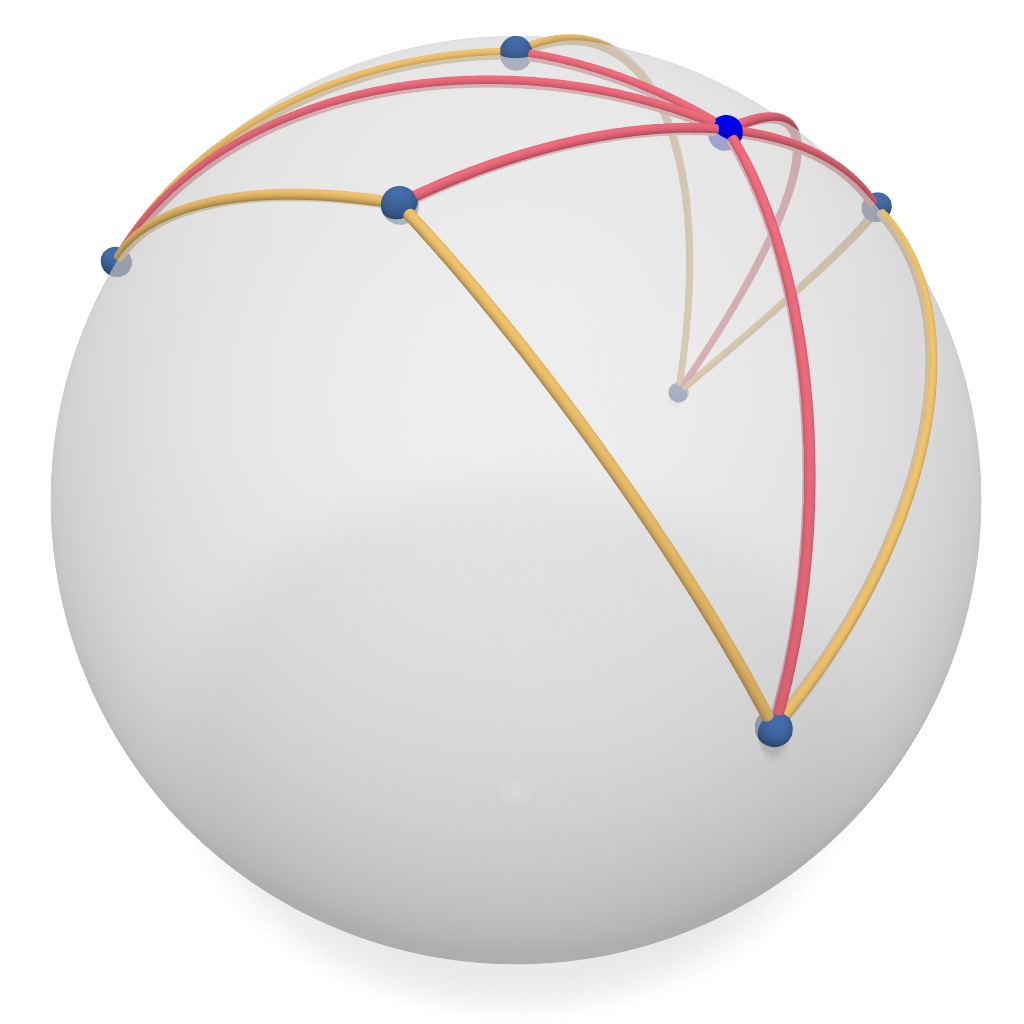}
			\cput{11}{69}{$\vec{n}_{1}$}
			\cput{46}{98}{$\vec{n}_{2}$}
			\cput{66}{56}{$\vec{n}_{3}$}
			\cput{93}{78}{$\vec{n}_{4}$}
			\cput{76}{23}{$\vec{n}_{5}$}
			\cput{38}{74}{$\vec{n}_{6}$}
			\cput{72}{90}{\color{blue}{$\vec{n}$}}
		\end{overpic}
	}
	\caption{Theorem~\ref{th:Dupin_negative}~(ii), case (iii) of Theorem~\ref{th:shape} with four inflection faces: Discrete hyperbola as discrete Dupin indicatrix}\label{fig:Dupin_triangle}
\end{figure}

\begin{figure}[!ht]
	\centerline{
		\begin{overpic}[height=0.3\textwidth]{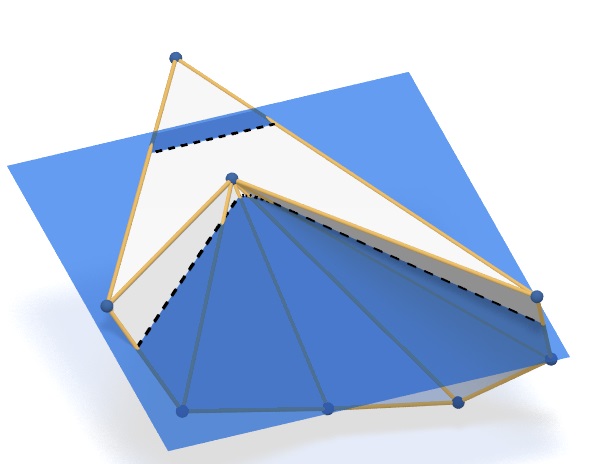}
			\put(78,27){\contour{white}{$f_1$}}
			\put(29,48){\contour{white}{$f_2$}}
			\put(23,23){\contour{white}{$f_3$}}
			\put(38,17){\contour{white}{$f_4$}}
			\put(54,18){\contour{white}{$f_5$}}
			\put(68,22){\contour{white}{$f_6$}}
			\cput{38}{50}{$\vec{v}$}	
		\end{overpic}
		\relax\\
		\begin{overpic}[height=0.3\textwidth]{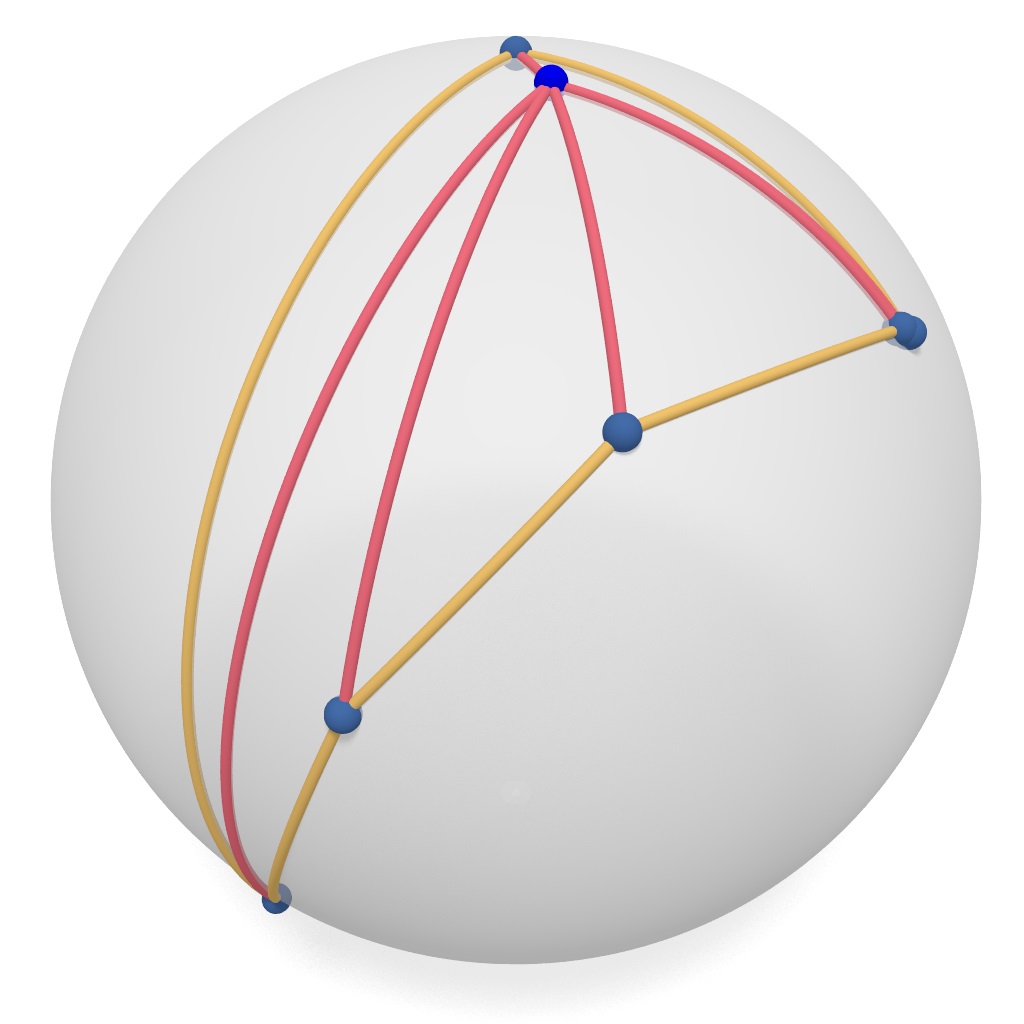}
			\cput{25}{7}{$\vec{n}_{1}$}
			\cput{45}{98}{$\vec{n}_{2}$}
			\cput{92}{70}{$\vec{n}_{3}$}
			\cput{92}{61}{$\vec{n}_{4}$}
			\cput{67}{53}{$\vec{n}_{5}$}
			\cput{40}{25}{$\vec{n}_{6}$}	
			\cput{59}{92}{\color{blue}{$\vec{n}$}}
		\end{overpic}
	}
	\caption{Theorem~\ref{th:Dupin_negative}~(ii), case (iii) of Theorem~\ref{th:shape} with two inflection faces: Discrete hyperbola as discrete Dupin indicatrix}\label{fig:Dupin_triangle_degenerate}
\end{figure}


\begin{figure}[!ht]
	\centerline{
		\begin{overpic}[height=0.3\textwidth]{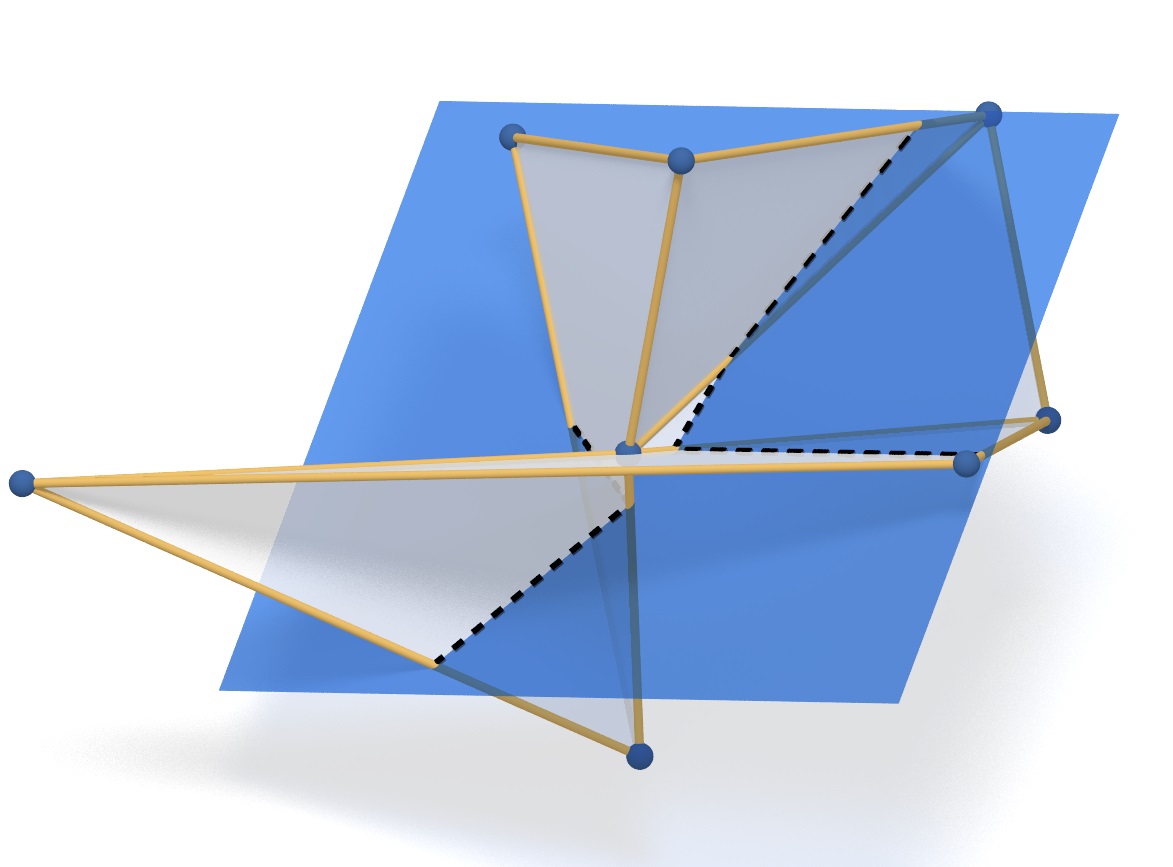}
			\put(36,22){\contour{white}{$f_1$}}
			\put(73,34){\contour{white}{$f_2$}}
			\put(77,45){\contour{white}{$f_3$}}
			\put(60,50){\contour{white}{$f_4$}}
			\put(49,53){\contour{white}{$f_5$}}
		\end{overpic}
		\relax\\
		\begin{overpic}[height=0.3\textwidth]{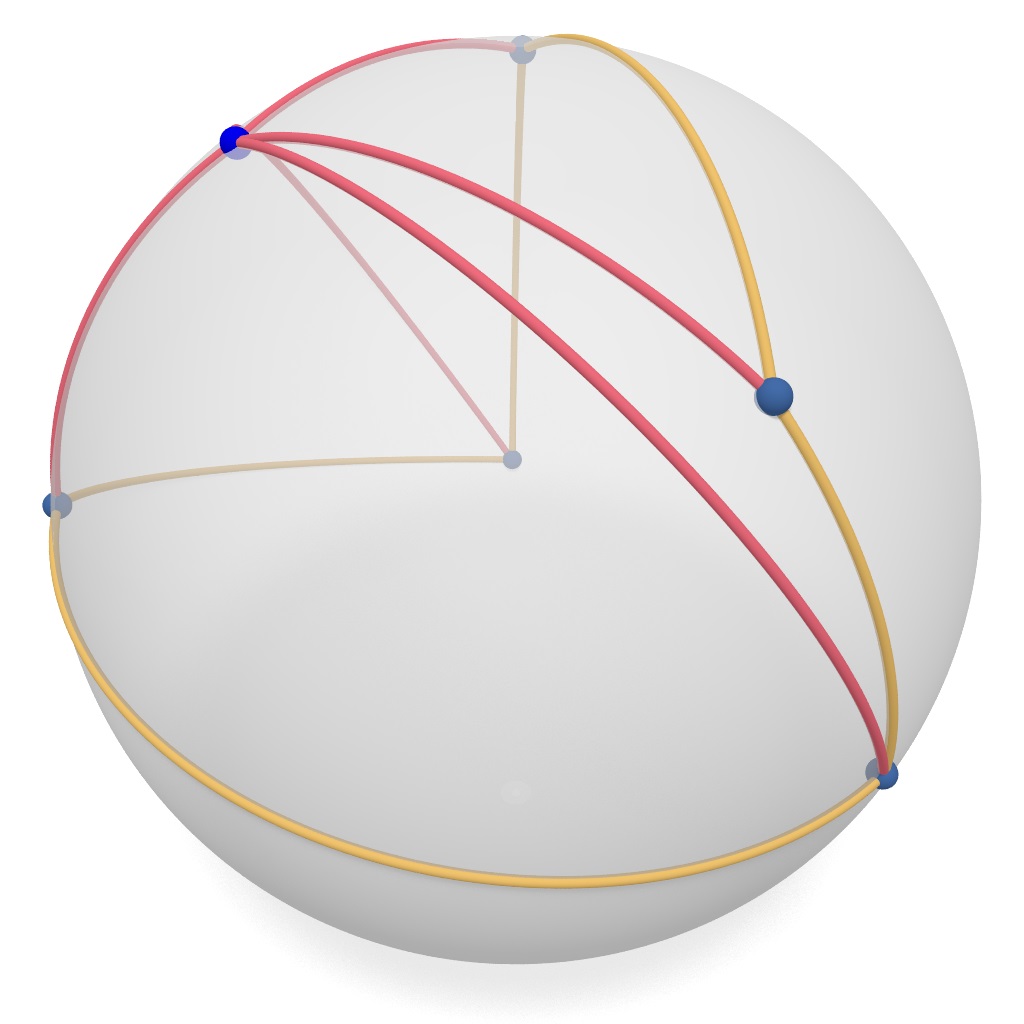}
		\cput{56}{54}{$\vec{n}_{1}$}
		\cput{53}{98}{$\vec{n}_{2}$}
		\cput{83}{61}{$\vec{n}_{3}$}
		\cput{93}{20}{$\vec{n}_{4}$}
		\cput{0}{50}{$\vec{n}_{5}$}
		\cput{23}{89}{\color{blue}{$\vec{n}$}}
		\end{overpic}
	}
	\caption{Theorem~\ref{th:Dupin_negative}~(ii), case (iv) of Theorem~\ref{th:shape}: Discrete hyperbola as discrete Dupin indicatrix}\label{fig:Dupin_Digon}
\end{figure}

\end{enumerate}
\end{theorem}
\begin{proof}
In \cite{BG16} it was discussed that $i({\vec{v}},\xi)$ equals the sum of winding numbers of $g({\vec{v}})$ around $\xi$ and $-\xi$. In the case that the Gauss image is negatively oriented and free of self-intersections, $i({\vec{v}},\xi)=0$ if $\xi$ is outside $g({\vec{v}})$, $i({\vec{v}},\xi)=-1$ if $\xi$ is inside $g({\vec{v}})$ and $-\xi$ lies outside $g({\vec{v}})$, and $i({\vec{v}},\xi)=-2$ if $g({\vec{v}})$ contains both $\xi$ and $-\xi$.

(i) \cite{BG16} considered the case of triangulated surfaces. In the definition of the index, we can replace the triangles in the counting by the actual faces of the vertex star, but we have to bear in mind that a face $f$ with $\alpha_f>\pi$ counts twice if $\vec{v}$ is an interior point of the intersection of the vertex star with the plane $E_0$ orthogonal to $\xi$ and passing through $\vec{v}$. In any case, $i({\vec{v}},{\vec{n}})=-2$ translates to six line segments emanating from $\vec{v}$ in the intersection of $E_0$ with the star of $\vec{v}$. Thus, a plane $E$ parallel and close to $E_0$ will intersect the vertex star in three connected components as in Fig.~\ref{fig:Dupin_antipodal}.

(ii) Similar to our discussion in the first part, $i({\vec{v}},{\vec{n}})=-1$ implies that the intersection of $E$ with the star of $\vec{v}$ consists of two connected components, namely two polylines.

These two polylines each bound a convex set if and only if no polyline contains an inflection edge. An inflection edge is a line segment whose neighboring two line segments are on different sides of the line through the inflection edge. Since these line segments are intersections of $E$ with faces of $V$, any inflection edge is the intersection of $E$ with an inflection face in the star of $\vec{v}$. Thus, an inflection edge occurs if and only if $E$ intersects an inflection face of $V$ and both its neighboring faces.

Consider the plane $E_0$ orthogonal to ${\vec{n}}$ that passes through $\vec{v}$. Then, there exists $E$ intersecting an inflection face of the vertex star and both its neighboring faces if and only if $E_0$ intersects this inflection face in $\vec{v}$ only.

Banchoff proved in \cite{B70} that $\vec{v}$ is a middle vertex of a face $f$ with angle $\alpha_f < \pi$ with respect to $\xi \in S^2$ if and only if $\xi$ is contained in one of the two digons in $S^2$ with vertices $\pm{\vec{n}}_f$, opening angle $\alpha_f$, and bounded by the great circles that pass through the arcs of $g({\vec{v}})$ incident to $\vec{v}$. Clearly, $\vec{v}$ is a middle vertex of $f$ with respect to ${\vec{n}}$ if and only if $E_0$ intersects $f$ non-trivially.

By Theorem~\ref{th:shape}, $g({\vec{v}})$ is a spherical pseudo-$n$-gon and its corners correspond to inflection faces $f$ with angle $\alpha_f < \pi$ or a non-inflection face $f$ with angle $\alpha_f > \pi$ (second case of Theorem~\ref{th:shape}~(iii)).

Let us first consider the case that any face $f$ with angle $\alpha_f > \pi$ is an inflection face, see Figs.~\ref{fig:Dupin_quad} and~\ref{fig:Dupin_quad_not} to see the difference between star-shaped and non-star-shaped $g({\vec{v}})$. Then, there are exactly four inflection faces. A spherical pseudo-$n$-gon is star-shaped with respect to ${\vec{n}}$ if and only if its corners can be seen from ${\vec{n}}$, meaning that the great circle arcs connecting the corners with ${\vec{n}}$ lie completely inside $g({\vec{v}})$. Since the faces $f$ with angle $\alpha_f > \pi$ are intersected by $E_0$ anyway, it follows that $g({\vec{v}})$ is star-shaped with respect to ${\vec{n}}$ if and only if $E_0$ intersects exactly all four inflection faces, well fitting to the fact that $i({\vec{v}},{\vec{n}})=-1$ means that the intersection of $E_0$ with $V$ consists of four line segments emanating from $\vec{v}$.

Furthermore, this intersection does not contain a line segment passing through $\vec{v}$, so both polylines in the intersection of $E$ with $V$ consist of at least two line segments. It remains to show that the two polylines always bound disjoint convex sets. We can look equivalently at the intersection of $E_0$ with $V$ and show that the angle between any two consecutive line segments is less than $\pi$. Instead of looking at the intersection angles between $E_0$ and faces $f\sim{\vec{v}}$, we may look at the intersection angles between planes that are orthogonal to both $E_0$ and $f$. Such planes are given by the planes spanned by ${\vec{n}},{\vec{n}}_f$. It easily follows that if the intersections of $E_0$ with $f$ and $f'$ are consecutive in counterclockwise order, then the angle $\alpha$ between the line segments corresponding to $f$ and $f'$ equals $\pi-\angle {\vec{n}}_{f'} {\vec{n}} {\vec{n}}_f$. Since $g({\vec{v}})$ is star-shaped with respect to the interior point ${\vec{n}}$, $0<\angle {\vec{n}}_{f'} {\vec{n}} {\vec{n}}_f<\pi$ and so $0<\alpha<\pi$.

We now come to the second case of Theorem~\ref{th:shape}~(iii), i.e., that there is a face $f\sim{\vec{v}}$ with angle $\alpha_f > \pi$ that is not an inflection face, see Fig.~\ref{fig:Dupin_triangle}. Then, $g({\vec{v}})$ is a spherical pseudo-triangle, two of its corners correspond to inflection faces and the other corner is ${\vec{n}}_f$. The two polylines in the intersection of $E$ with $V$ bound convex sets if and only if the two inflection faces are intersected or equivalently if the normals of the two inflection faces can be seen from ${\vec{n}}$.

Assume that also ${\vec{n}}_f$ can be seen from ${\vec{n}}$. Then, using similar arguments as above, $E_0$ intersects $f$ in a line segment with $\vec{v}$ in its interior (remember that $f$ counts as two inflection faces). It follows that the convex sets which the two polylines in the intersection of $E$ with $V$ bound are disjoint and that one of the polylines is a straight line segment if $\vec{v}$ is on the appropriate side of $E$.

We are left with the case that ${\vec{n}}_f$ cannot be seen from ${\vec{n}}$, see Fig.~\ref{fig:Dupin_triangle_degenerate}. Then, $E_0$ does not intersect $f$ in a line segment containing $\vec{v}$ in its interior, so there has to be a non-inflection face $f'$ that is intersected non-trivially by $E_0$. By the argument sketched above, ${\vec{n}}$ has to lie in one of the digons with vertices $\pm{\vec{n}}_{f'}$ and opening angle $\alpha_{f'}$. It follows that if we compare the orientation of the faces around the vertex star with the reversed orientation of its normals as seen from ${\vec{n}}$, then the order of ${\vec{n}}_{f}$ and ${\vec{n}}_{f'}$ is interchanged.

Clearly, the line segments corresponding to $f$ and $f'$ that appear in the intersection of $E_0$ with $V$ are consecutive if one goes along the vertex star. Let us assume without loss of generality that the intersection with $f'$ comes after the one with $f$ in counterclockwise direction. The angle $\alpha$ between them is given by $\pi-\angle {\vec{n}}_{f'} {\vec{n}} {\vec{n}}_f$. But now, $0>\angle {\vec{n}}_{f'} {\vec{n}} {\vec{n}}_f$, so $\alpha>\pi$. Thus, for suitable $E$ the two polylines in the intersection of $E$ with $V$ bound convex sets such that one is containing the other, see Fig.~\ref{fig:Dupin_triangle_not}.

\begin{figure}[!ht]
	\centerline{
		\begin{overpic}[height=0.3\textwidth]{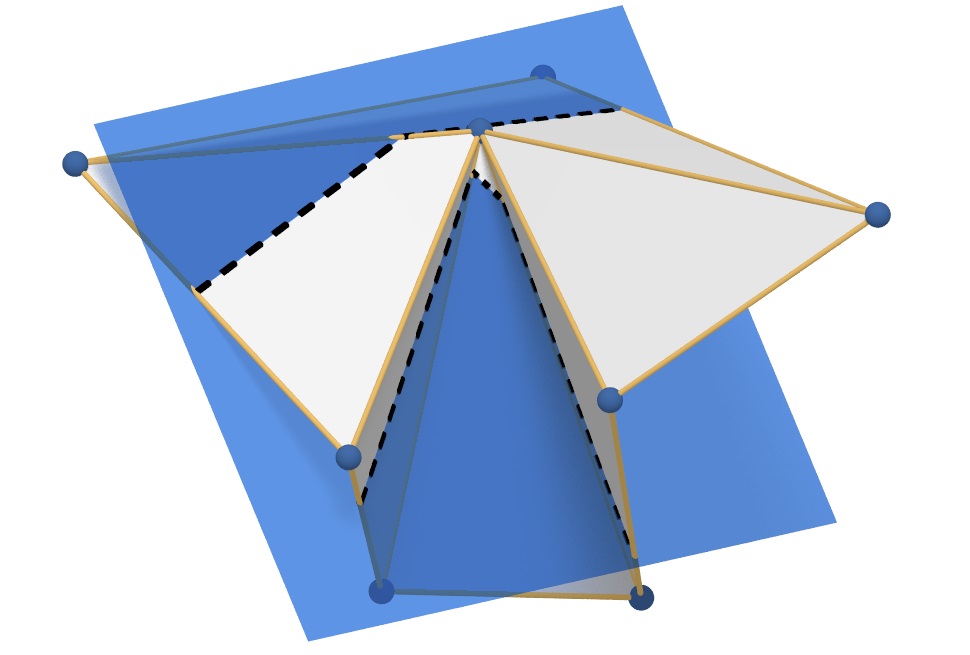}
			\put(63,43){\contour{white}{$f_1$}}
			\put(62,52){\contour{white}{$f_2$}}
			\put(29,40){\contour{white}{$f_3$}}
			\put(37,19){\contour{white}{$f_4$}}
			\put(48,25){\contour{white}{$f_5$}}
			\put(60,20){\contour{white}{$f_6$}}
		\end{overpic}
		\relax\\
		\begin{overpic}[height=0.3\textwidth]{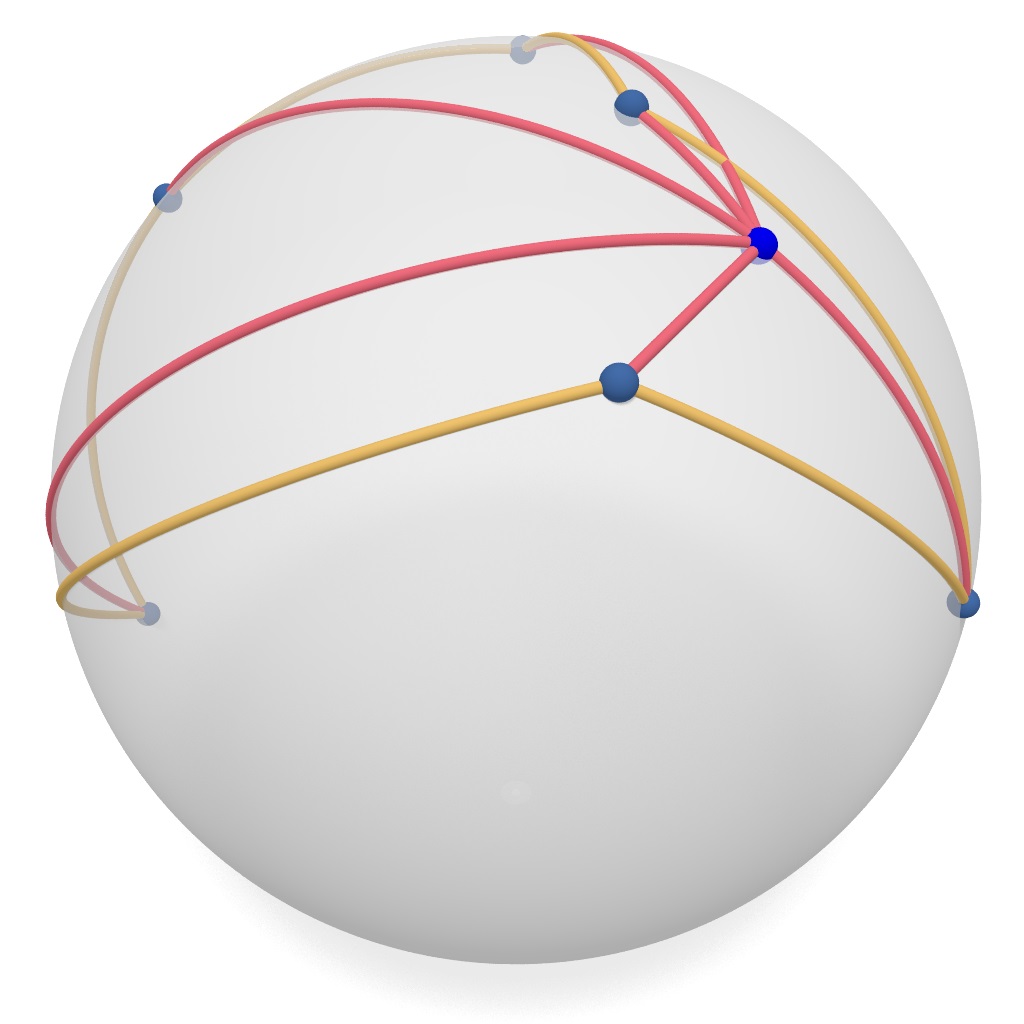}
		\cput{11.5}{80}{$\vec{n}_{1}$}
		\cput{51}{99}{$\vec{n}_{2}$}
		\cput{56}{88}{$\vec{n}_{3}$}
		\cput{96}{35}{$\vec{n}_{4}$}
		\cput{59}{55}{$\vec{n}_{5}$}
		\cput{74}{69}{\color{blue}{$\vec{n}$}}
			\cput{21}{38}{$\vec{n}_{6}$}	
		\end{overpic}
	}
	\caption{Polylines bound convex sets such that one is containing the other}\label{fig:Dupin_triangle_not}
\end{figure}
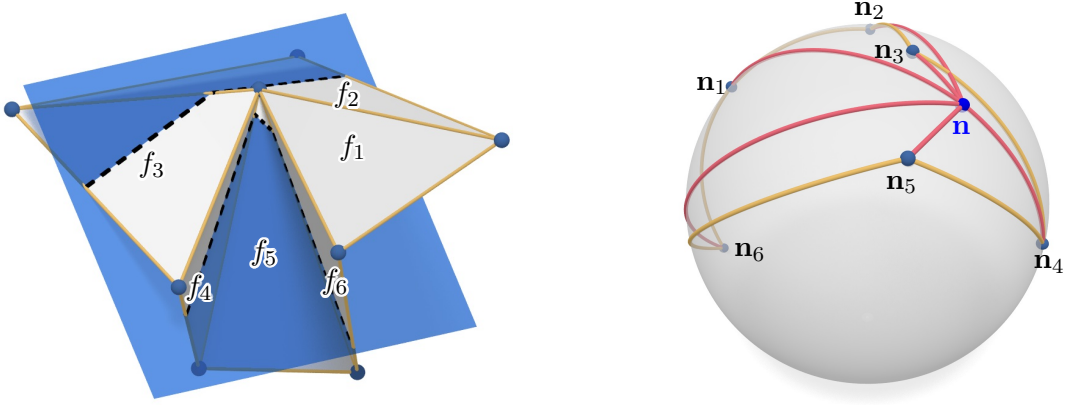

\end{proof}

In Theorem~\ref{th:Dupin_negative}~(i), we have seen that it is a reasonable condition for smoothness of $P$ to demand that $g({\vec{v}})$ contains no antipodal points. This is guaranteed if $g({\vec{v}})$ is contained in an open hemisphere. In addition, it is more appropriate to require that $g({\vec{v}})$ is star-shaped (with respect to an interior point, see Fig.~\ref{fig:starshaped} for an example that is not star-shaped with respect to the interior, but can be seen from an exterior point) rather than just free of self-intersections. Algorithmically, star-shapedness can be checked more easily \cite{JGWP16}.

\begin{remark}
Motivated by the last paragraph in Section~\ref{sec:selfintersection} and Propositions~\ref{prop:Dupin_positive} and Theorem~\ref{th:Dupin_negative}~(ii), both ${\vec{n}}\in S^2$ such that $g({\vec{v}})$ is star-shaped with respect to ${\vec{n}}$ and  ${\vec{n}}'\in S^2$ such that $g({\vec{v}})$ is contained in the open hemisphere with pole ${\vec{n}}'$ may be reasonable normal vectors at $\vec{v}$. Actually, we have shown that a plane orthogonal to ${\vec{n}}$ and passing through $\vec{v}$ is a reasonable tangent plane: Such tangent planes $E_0$ have the property that planes $E$ parallel and close to $E_0$ intersect a neighborhood of the vertex star in a discrete ellipse or a discrete hyperbola.

\begin{figure}[htbp]
		\centerline{
			\begin{overpic}[height=0.3\textwidth]{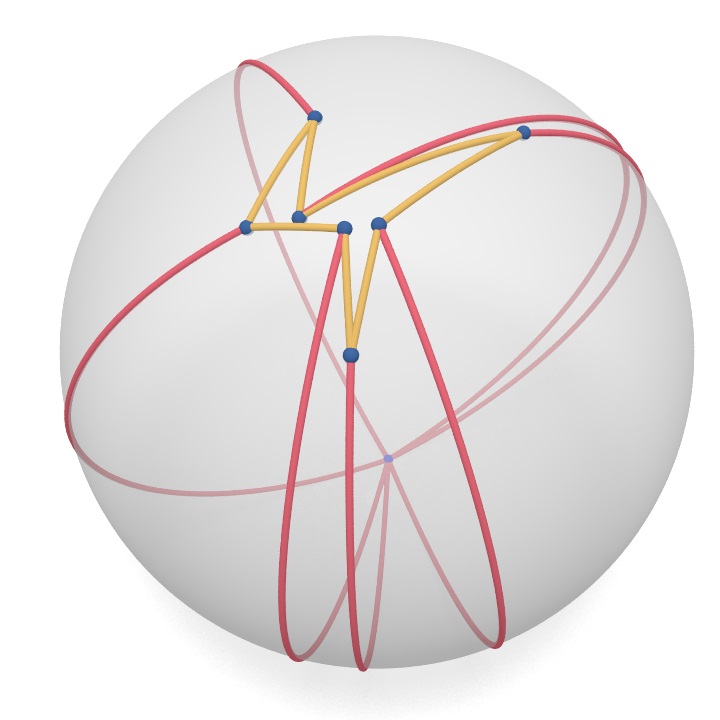}
			\end{overpic}
			\relax\\
			\begin{overpic}[height=0.3\textwidth]{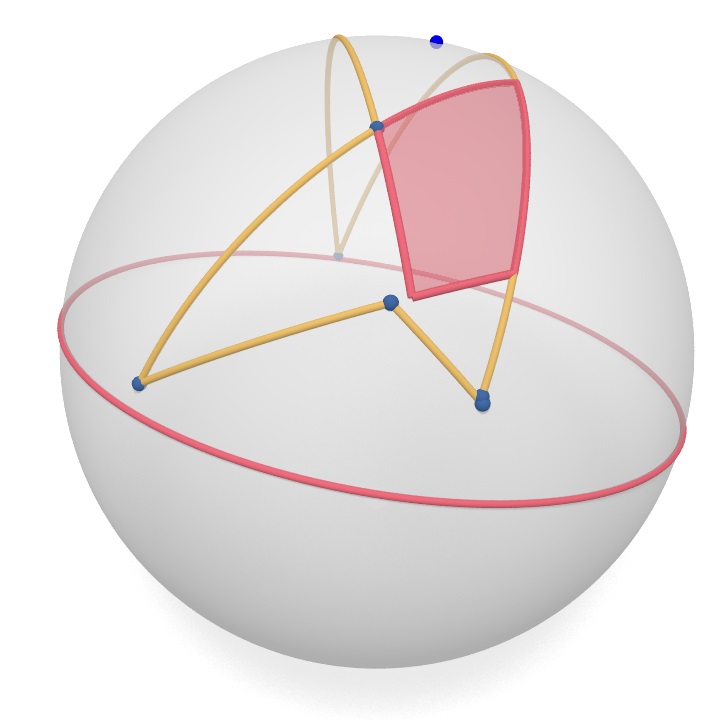}
				\cput{62}{96}{\color{blue}{$\vec{n}$}}	
			\end{overpic}
		}
	\caption{Gauss images that are star-shaped only with respect to the exterior (left) or where no pole of an open hemisphere containing it is in the kernel (right)}\label{fig:starshaped}
\end{figure}

However, Fig.~\ref{fig:starshaped} shows that it might happen that even though $g({\vec{v}})$ is star-shaped and contained in an open hemisphere, it is never contained in an open hemisphere whose pole is in the \textit{kernel} of $g({\vec{v}})$, i.e., the set of points on $S^2$ inside $g({\vec{v}})$ with respect to that $g({\vec{v}})$ is star-shaped. This will not happen if the Gauss image is small enough (for example, if it is contained in an open spherical triangle with three angles $\pi/2$).
\end{remark}

To give a summary so far, our assertion for smoothness of a polyhedral vertex star is the following:

\begin{definition}\label{def:smooth}
The star of an interior point $\vec{v}$ of $P$ is said to be \textit{smooth}, if there exist ${\vec{n}},{\vec{n}}'\in S^2$ inside the Gauss image $g({\vec{v}})$ such that it is star-shaped with respect to ${\vec{n}}$ and contained in the open hemisphere with pole ${\vec{n}}'$. A plane orthogonal to ${\vec{n}}$ and passing through $\vec{v}$ defines a \textit{discrete tangent plane}.
\end{definition}

\begin{remark}
We have already seen in the previous section that ${\vec{n}}'$ is a discrete normal. Now, not the plane orthogonal to ${\vec{n}}'$, but the one orthogonal to ${\vec{n}}$ is the better discretization of a tangent plane. When we discuss projective transformations in Section~\ref{sec:projective}, this difference will become apparent. Whereas ${\vec{n}}'$ defines a direction (i.e. a line) in which the vertex star projects bijectively, ${\vec{n}}$ defines a band of planes that intersect the vertex star in a discrete ellipse or discrete hyperbola. Since projective transformations map lines to lines and planes to planes, but do not preserve angles, these two concepts should be treated differently.
\end{remark}


\subsubsection{Discrete asymptotic directions}\label{sec:asymptotes}

In the smooth case, asymptotic directions at a point of negative Gaussian curvature are the asymptotes of its Dupin indicatrix, which is a hyperbola. Let us now consider a smooth polyhedral vertex star around $\vec{v}$ with $K({\vec{v}})<0$. Motivated by Theorem~\ref{th:Dupin_negative} and its proof, we define discrete asymptotic directions at $\vec{v}$ as follows.

\begin{definition}
Let $\vec{v}$ be an interior point of $P$ of negative discrete Gaussian curvature whose vertex star is smooth. Given a discrete tangent plane at $\vec{v}$, the four directions defined by the four line segments in the discrete tangent plane with a disk neighborhood of $\vec{v}$ are said to be \textit{discrete asymptotic directions}.
\end{definition}

Note that it may happen that two asymptotic directions are collinear as in Fig.~\ref{fig:Dupin_triangle_degenerate}. We will now shortly discuss how all asymptotic directions that can appear if the set of all possible tangent planes is considered are characterized. 

The Gauss image $g({\vec{v}})$ is star-shaped with respect to an interior point ${\vec{n}}$ if and only if all corners can be seen from ${\vec{n}}$. Hence, the kernel of $g({\vec{v}})$ is given by the interior of the intersection of all digons spanned by two spherical arcs incident to a corner of $g({\vec{v}})$. In particular, the kernel $C \subset g({\vec{v}})$ is open and convex.

We know that a tangent plane only intersects inflection faces. Let us first consider the case that the face with a reflex angle (if it exists) is an inflection face. Then, the interior angle of $g({\vec{v}})$ at ${\vec{n}}_f$ equals $\alpha_f$. Consider for each inflection face $f$ the spherical digon $D_f$ spanned by the two tangent arcs from ${\vec{n}}_f$ to the boundary of $C$. Since any point in $C$ can see all ${\vec{n}}_f$, such tangents exist.

If the edges of $f$ incident to $\vec{v}$ are denoted by $e_1$ and $e_2$, let $\alpha_{i}$ be the angle between $D_f$ and the arc in $g({\vec{v}})$ corresponding to $e_i$ (if this arc is contained in $D_f$, we set $\alpha_{i}:=0$). Let $\alpha_{0}$ be the angle of $D_f$. Then, $\alpha_f=\alpha_{0}+\alpha_{1}+\alpha_{2}$. In the same way as we computed the angles between the line segments in the intersection of $E_0$ with the vertex star in the proof of Theorem~\ref{th:Dupin_negative}~(ii), it follows that all possible asymptotic directions in $f$ lie in the open convex cone in $f$ with tip $\vec{v}$ and angle $\alpha_{0}$ whose boundary forms angles $\alpha_{1}$ and $\alpha_{2}$ with $e_1$ and $e_2$, respectively.

It remains the case that there is a face $f$ in the vertex star with angle $\alpha_f > \pi$ that is not an inflection face. Then, ${\vec{n}}_f$ is a corner of $g({\vec{v}})$ and the interior angle equals $\alpha_f-\pi$. Furthermore, the intersection of any tangent plane $E_0$ with $f$ will be (locally around $\vec{v}$) a line segment with $\vec{v}$ in its interior. We define $D_f$ and angles $\alpha_{0},\alpha_{1},\alpha_{2}$ as before and conclude with a similar reasoning as before that $\alpha_f-\pi=\alpha_{0}+\alpha_{1}+\alpha_{2}$ and that any pair of possible asymptotic directions in $f$ lie on a line segment in the open double cone in $f$ with tip $\vec{v}$ and angle $\alpha_{0}$ whose boundary forms angles $\alpha_{1}$ and $\alpha_{2}$ with the edges $e_1$ and $e_2$ of $f$, respectively.

In particular, asymptotic directions lie in the planar cones spanned by the edges incident to $\vec{v}$ of an inflection face $f\sim{\vec{v}}$ with angle $\alpha_f <\pi$, in the cones that are opposite to the planar cones spanned by the edges incident to $\vec{v}$ of an inflection face $f\sim{\vec{v}}$ with angle $\alpha_f >\pi$ (and the whole half-plane containing $f$ if $\alpha_f =\pi$), and in the double cone that is spanned by the edges incident to $\vec{v}$ of a non-inflection face $f\sim{\vec{v}}$ with angle $\alpha_f >\pi$.

\begin{remark}

In \cite{JGWP16}, the relationship between the Gauss image and asymptotic lines was discussed using a regular triangular mesh on a hyperboloid, see Fig.~\ref{fig:asymptotes}. The observation made there was that the Gauss image has no self-intersections if the quadrants bounded by the asymptotic directions do not contain faces except for a pair of faces contained in opposite quadrants. Our discussion above gives now the mathematical explanation for that observation if we identify smoothness with star-shapedness: In order to get a discrete hyperbola as the discrete Dupin indicatrix, all four inflection faces have to be intersected by the tangent plane. A close look on Fig.~\ref{fig:asymptotes} shows that the two non-inflection faces around the origin are lying in the upper and the lower quadrant. Hence, the mesh (ii) is smooth since exactly the four inflection faces each contain an asymptotic direction, in contrast to mesh (i), where the two non-inflection faces contain two asymptotic directions each. The mesh in the left part is a border case where the asymptotic directions lie on edges.

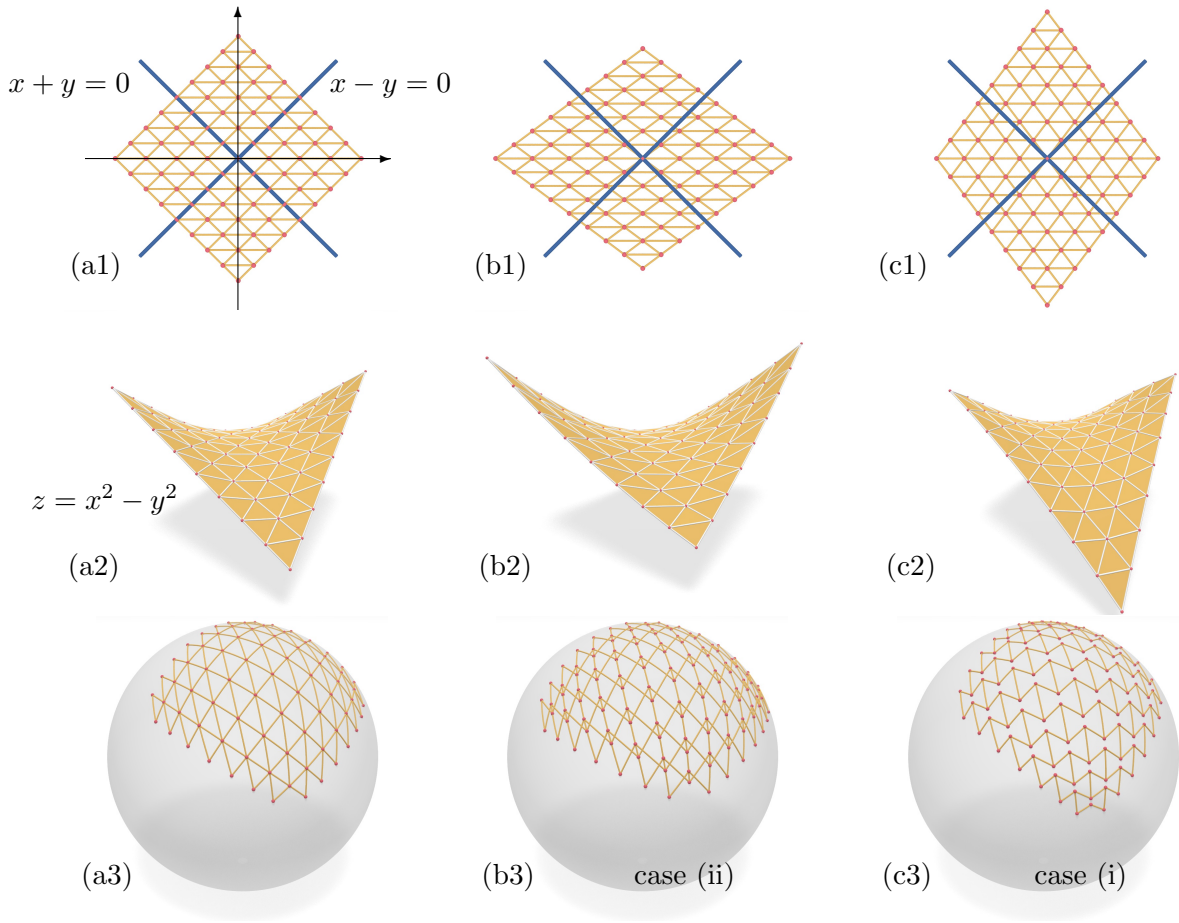
\begin{figure}[htbp]
   \centerline
   {
   	\begin{overpic}[height=0.23\textwidth]{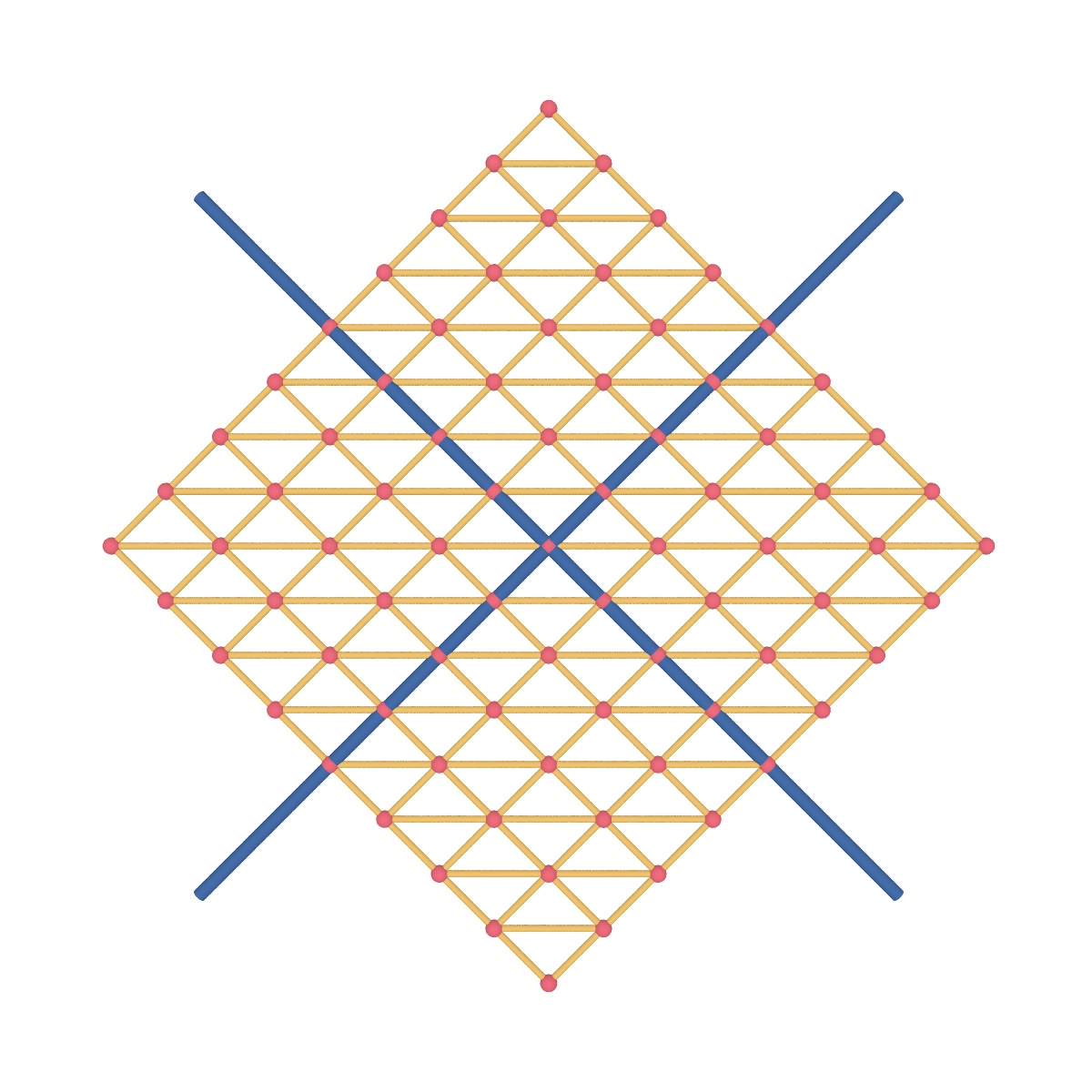}
   		\put(-25,72){\contour{white}{$x+y=0$}}
   		\put(80,72){\contour{white}{$x-y=0$}}
   		\put(-5,12){\contour{white}{(a1)}}
   		 		\put(0,50){\vector(1,0){100}}
   		 	   	\put(50,0){\vector(0,1){100}}
   	\end{overpic}
   	\relax\\
   	\begin{overpic}[height=0.23\textwidth]{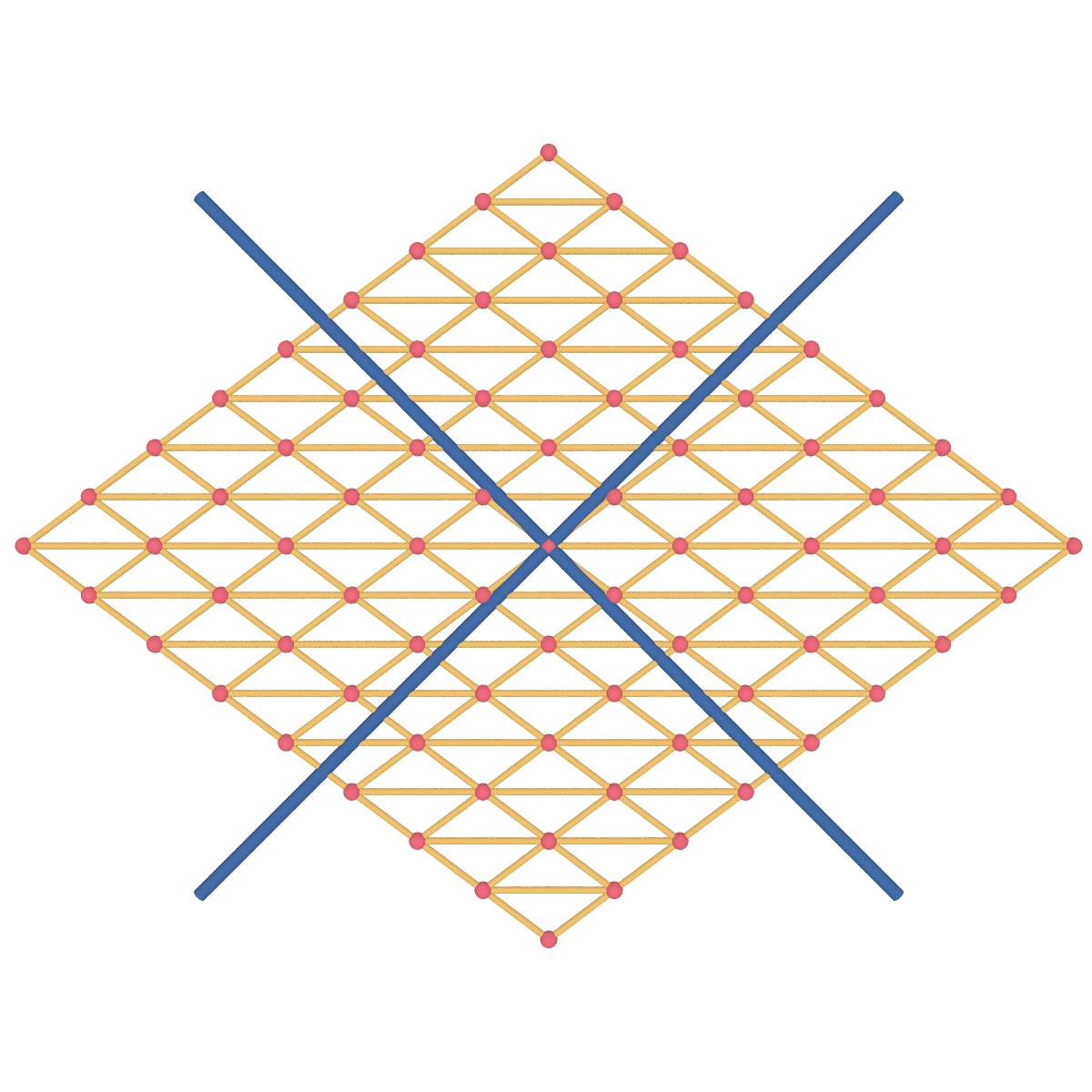}
   	\put(-5,12){\contour{white}{(b1)}}
   	\end{overpic}
  	\relax\\ 	
   	\begin{overpic}[height=0.23\textwidth]{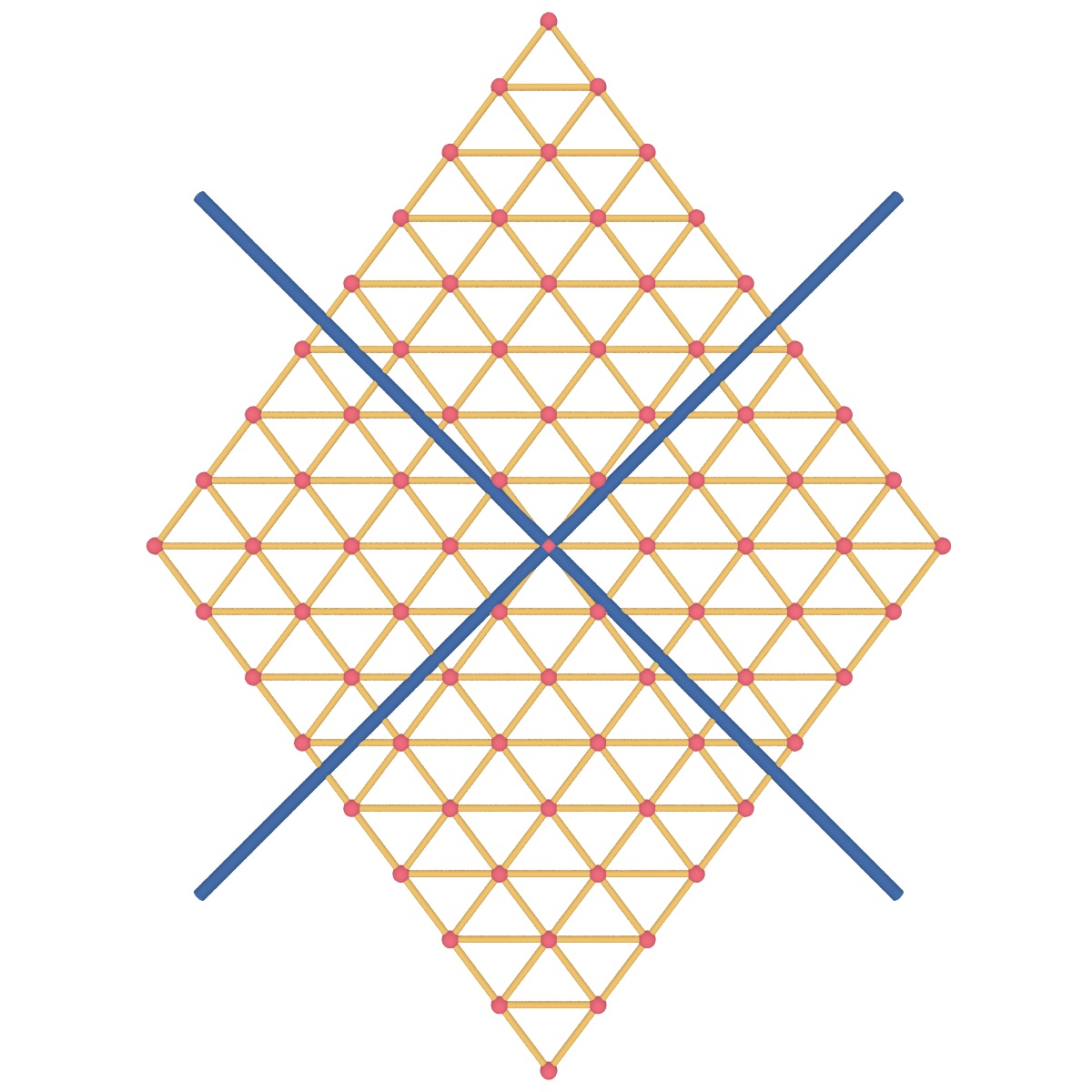}
   	\put(-5,12){\contour{white}{(c1)}}
   	\end{overpic} 
   }
   \centerline
	{
   	\begin{overpic}[height=0.23\textwidth]{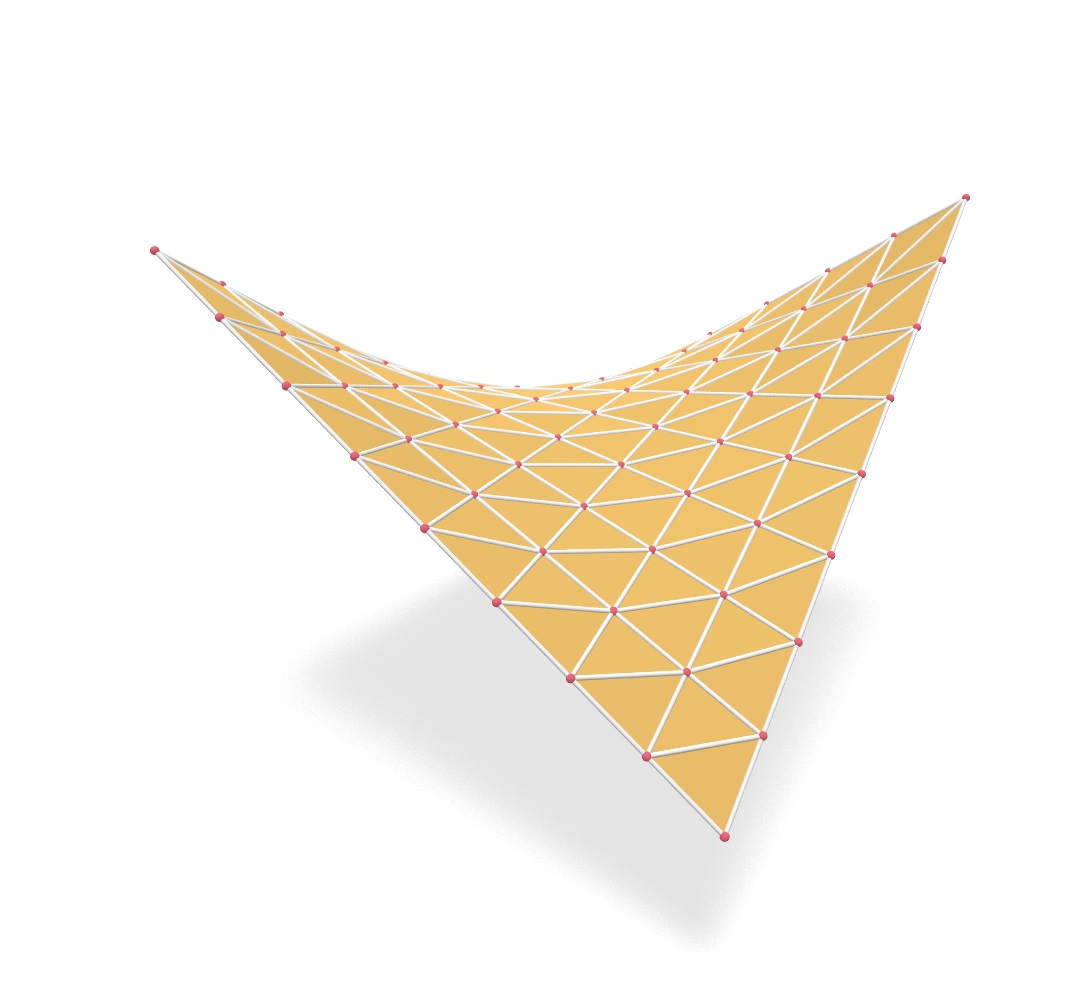}
   	\put(1,12){\contour{white}{(a2)}}
   	\put(-10,32){\contour{white}{$z=x^2-y^2$}}
   	\end{overpic}
   	\relax\\
   	\begin{overpic}[height=0.23\textwidth]{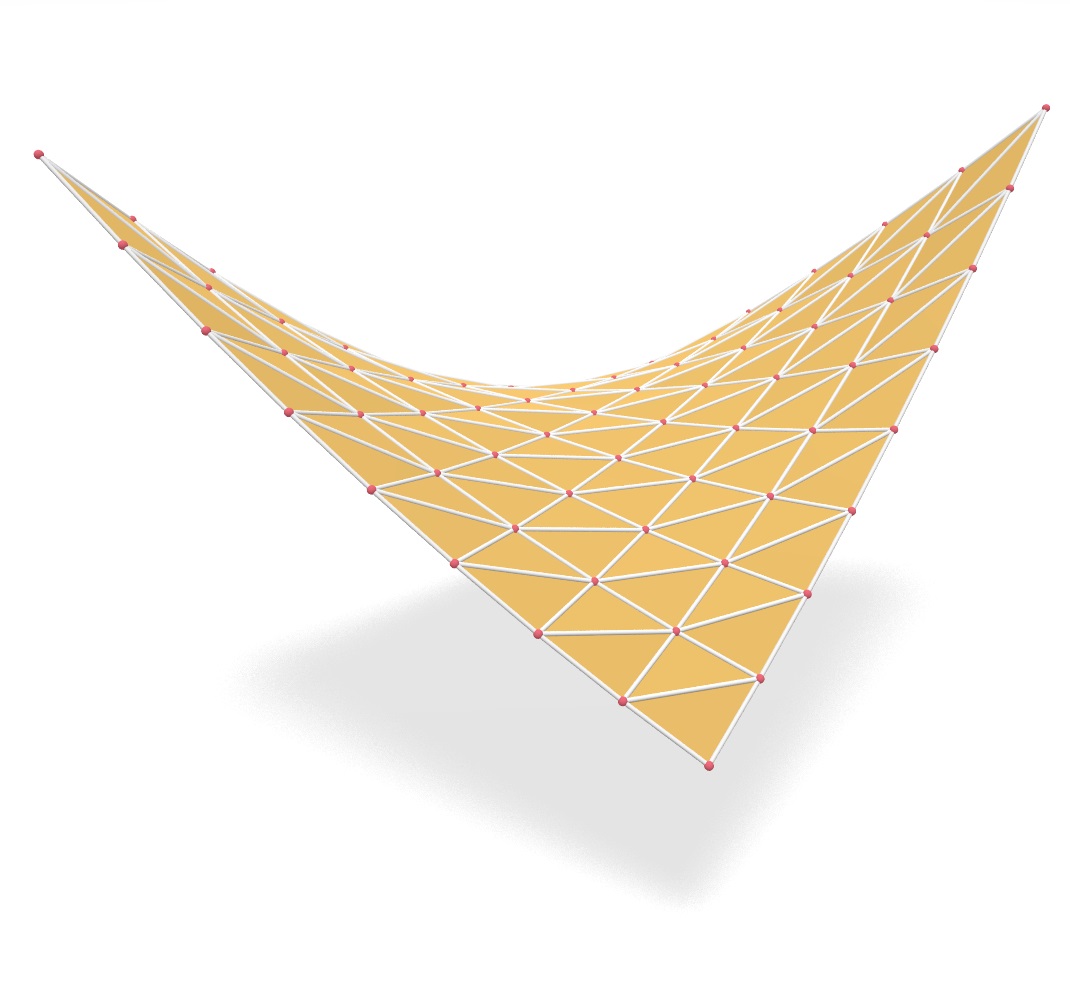}
   		\put(1,12){\contour{white}{(b2)}}
   	\end{overpic}
   	\relax\\
   	\begin{overpic}[height=0.23\textwidth]{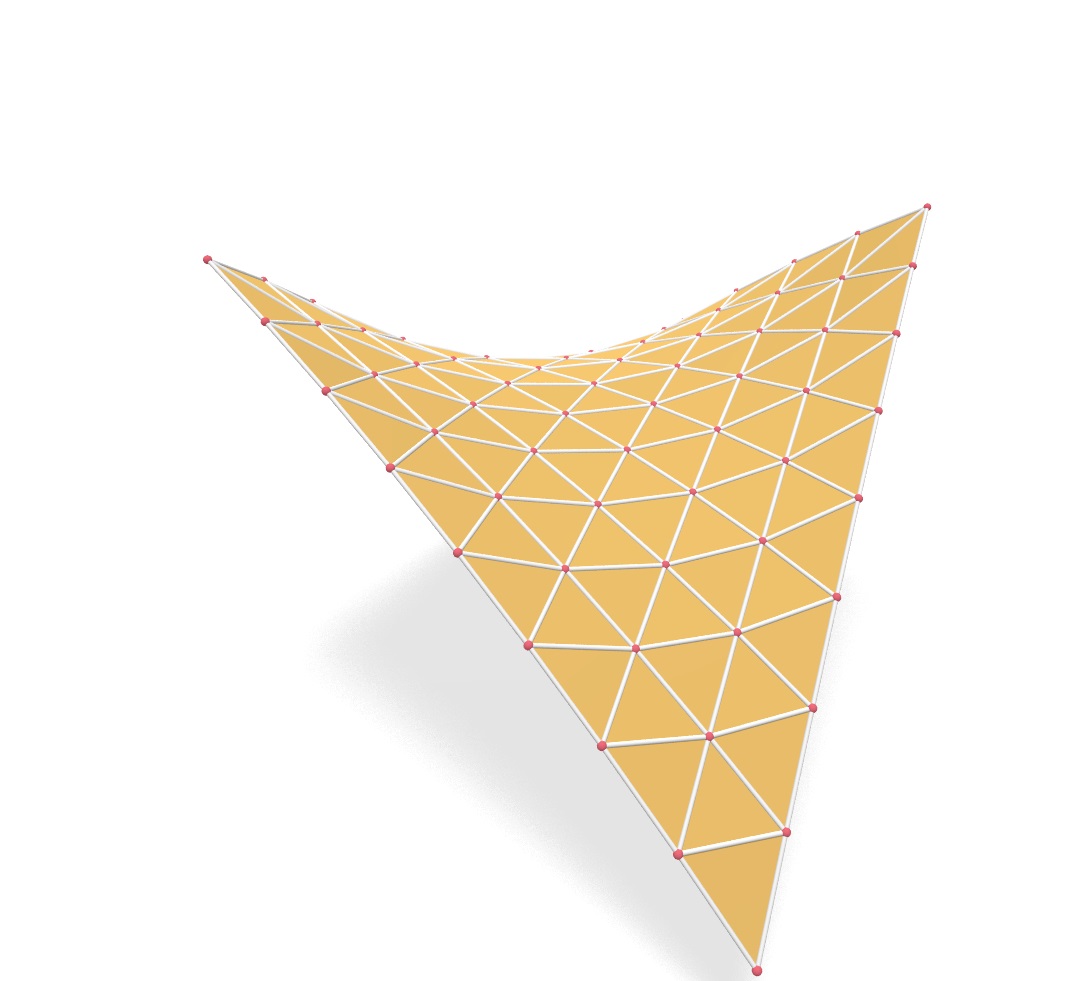}
   		\put(0,12){\contour{white}{(c2)}}
   	\end{overpic}
   }
   \centerline
   {
   	   	\begin{overpic}[height=0.23\textwidth]{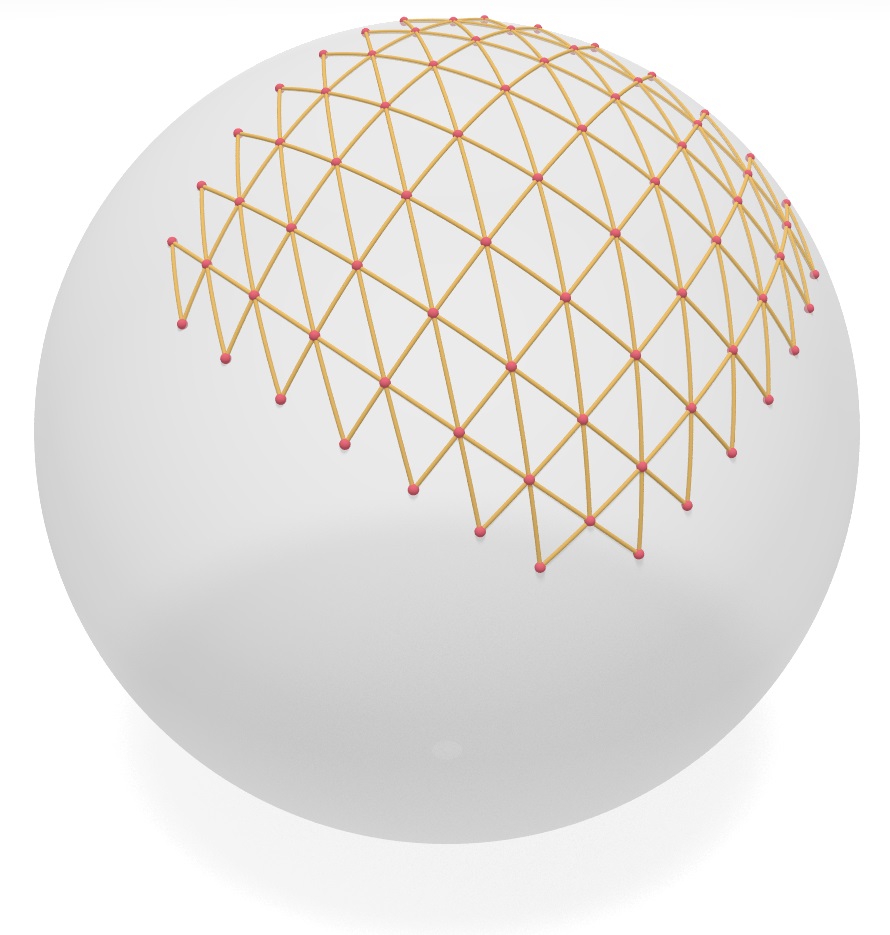}
   	   		\put(-5,12){\contour{white}{(a3)}}
   	   	\end{overpic}
   	   	\relax\\
   	   	\begin{overpic}[height=0.23\textwidth]{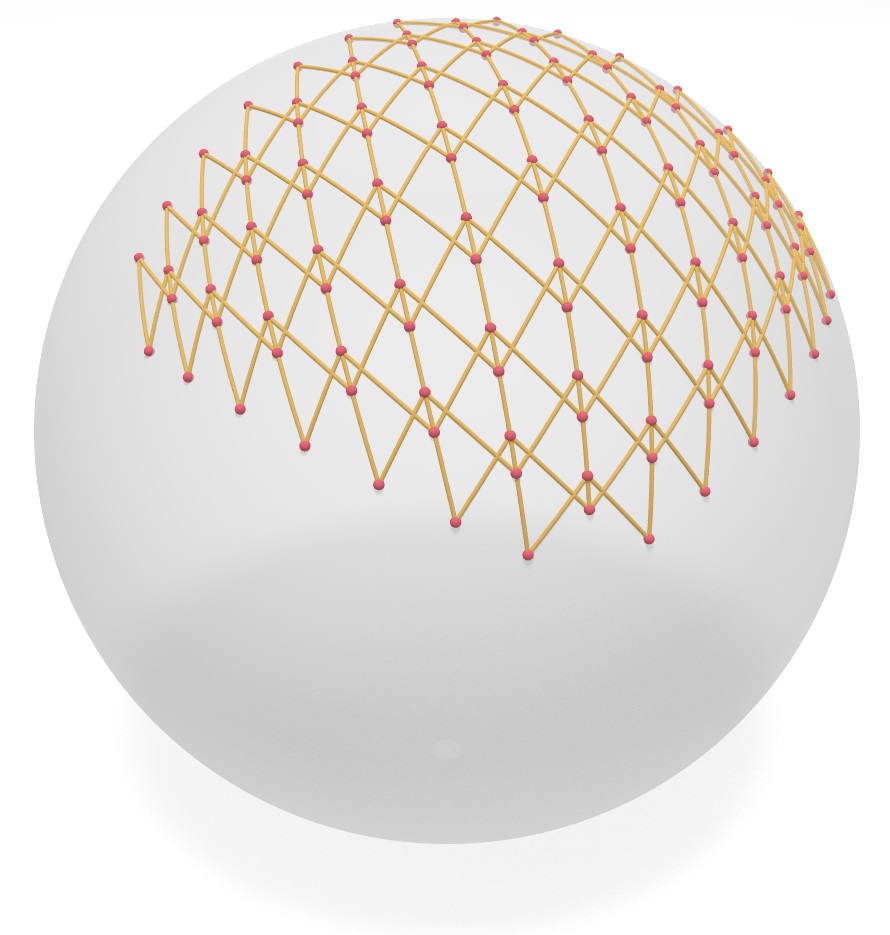}
   	   		\put(-5,12){\contour{white}{(b3)}}
   	   		\put(45,12){{case (ii)}}
   	   	\end{overpic}
   	   	\relax\\
   	   	\begin{overpic}[height=0.23\textwidth]{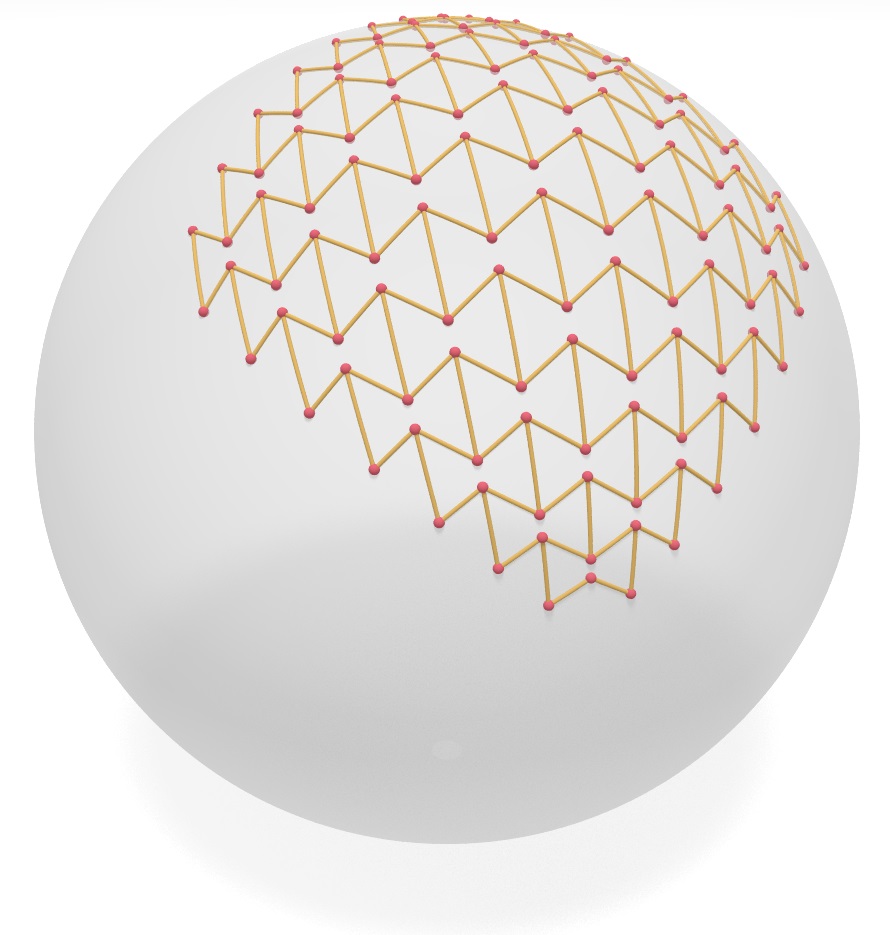}
   	   			\put(-5,12){\contour{white}{(c3)}}
   	   			\put(45,12){{case (i)}}
   	   	\end{overpic}
   	}
   \caption{Relevance of the alignment of the mesh along asymptotic directions for smoothness: The graph of the function $z = x^2+y^2$ carries two families of straight lines which correspond to $x \pm y =\textnormal{const.}$, and which are also the asymptotic directions. Images (a1)--(c1) show different tilings of the $xy$-plane by triangles, which in (a2)--(c2) are lifted to the graph surface and generate a mesh. Their respective Gauss images have dual faces without self-intersections in (c3), and with self-intersections in (b3). Image (a3) illustrates a border case where self-intersections begin to occur, and where the edges coincide with the surface's asymptotic directions}
   \label{fig:asymptotes}
\end{figure}

\end{remark}


\section{Gauss images around faces and shapes of faces}\label{sec:global}

The aim of this section is to assess smoothness of $P$ in a slightly more global way, which we will do by investigating how the Gauss images of the stars of all vertices of a face fit together. If we require that there are locally as few overlaps as possible, then this will impose conditions on the shape of the face. The case of a region of positive curvature in Section~\ref{sec:positive} is fairly easy. In Section~\ref{sec:negative} we will see that faces in negatively curved regions are pseudo-triangles and -quadrilaterals if we are looking for a smooth behavior. Requiring in addition that the face is star-shaped gives us the opportunity to define asymptotic directions within the face. Finally, we will consider areas of both positive and negative discrete Gaussian curvature in Section~\ref{sec:parabolic} and we will discuss parabolic curves.

Whereas we expect few overlaps in either positively or negatively curved areas, it is natural to expect overlaps of the Gauss images near \textit{parabolic} points, i.e. points of zero Gaussian curvature.

A smooth surface of vanishing Gaussian curvature is developable. To study discrete developable surfaces is not the aim of our exposition and for the modeling of discrete developable surfaces using polyhedral surfaces with quadrilateral faces we refer for example to \cite{LPWLW06}.

For this reason, we assume that $K({\vec{v}})\neq 0$ for all $\vec{v}$. Since the faces of $P$ are flat, it is also quite natural to locate the parabolic points on the faces and not on the vertices.

Our main interest lies in generic parabolic points. Isolated parabolic points, that occur for example on a monkey saddle, are not generic since a small perturbation will deform the isolated parabolic point into a closed parabolic curve. Even though we can model polyhedral monkey saddles around faces in Section~\ref{sec:negative}, we will not include such saddles in our notion of smoothness.

Let us consider a face $f_1$ of $P$ and all the stars of vertices adjacent to it. We want to investigate how the corresponding Gauss images fit together if all the Gauss images of vertex stars around $\vec{v}\sim f_1$ have no self-intersections. As a summary, we will give a notion of smooth polyhedral surfaces in Section~\ref{sec:smooth} and compare the discrete definitions with their smooth counterparts by the means of examples in Section~\ref{sec:comparison}.


\subsection{Region of positive discrete Gaussian curvature}\label{sec:positive}

We assume in addition that all $n$ vertices of $f_1$ have positive discrete Gaussian curvature. By Theorem~\ref{th:shape}, the interior angle of $f_1$ at any vertex is less than $\pi$ and $f_1$ is never an inflection face. Furthermore, all vertices are convex corners. It follows that $f_1$ is a convex polygon (in particular, it is star-shaped) and a neighborhood of $f_1$ forms part of the boundary of convex polyhedron as in Fig.~\ref{fig:positive_nice}.  

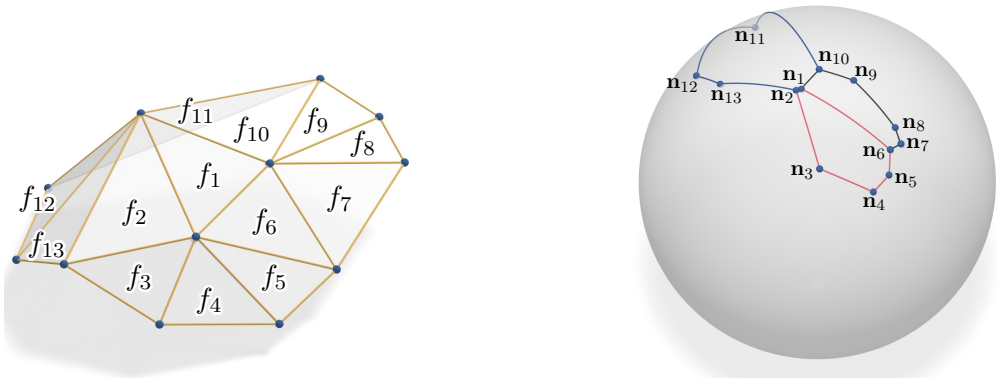
\begin{figure}[!ht]
	\centerline{
		\begin{overpic}[height=0.25\textwidth]{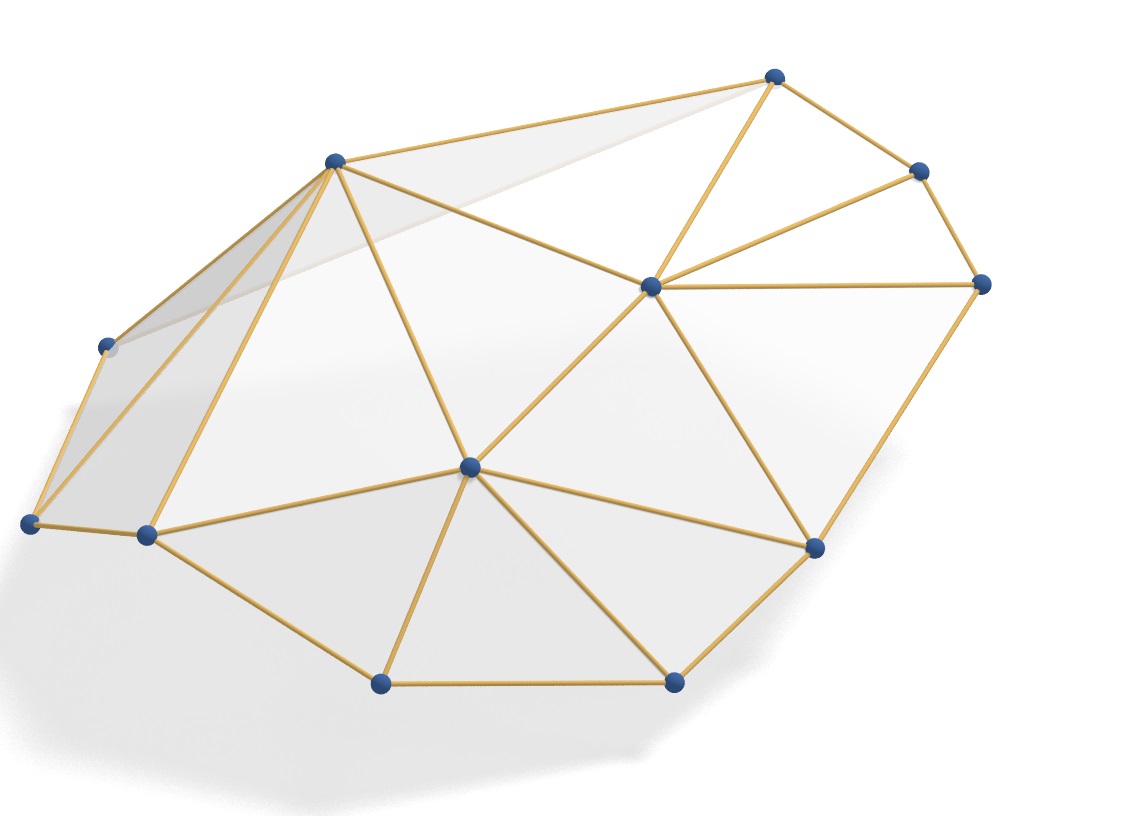}
			\put(41,42){\contour{white}{$f_1$}}
			\put(25,34){\contour{white}{$f_2$}}
			\put(26,20){\contour{white}{$f_3$}}
			\put(41,15){\contour{white}{$f_4$}}
			\put(55,20){\contour{white}{$f_5$}}
			\put(53,32){\contour{white}{$f_6$}}
			\put(69,37){\contour{white}{$f_7$}}
			\put(74,49){\contour{white}{$f_8$}}
			\put(64,54){\contour{white}{$f_9$}}
			\put(49,52){\contour{white}{$f_{10}$}}
			\put(37,56){\contour{white}{$f_{11}$}}
			\put(3,37){\contour{white}{$f_{12}$}}
			\put(5,27){\contour{white}{$f_{13}$}}
		\end{overpic}
		\relax\\
		\begin{overpic}[height=0.3\textwidth]{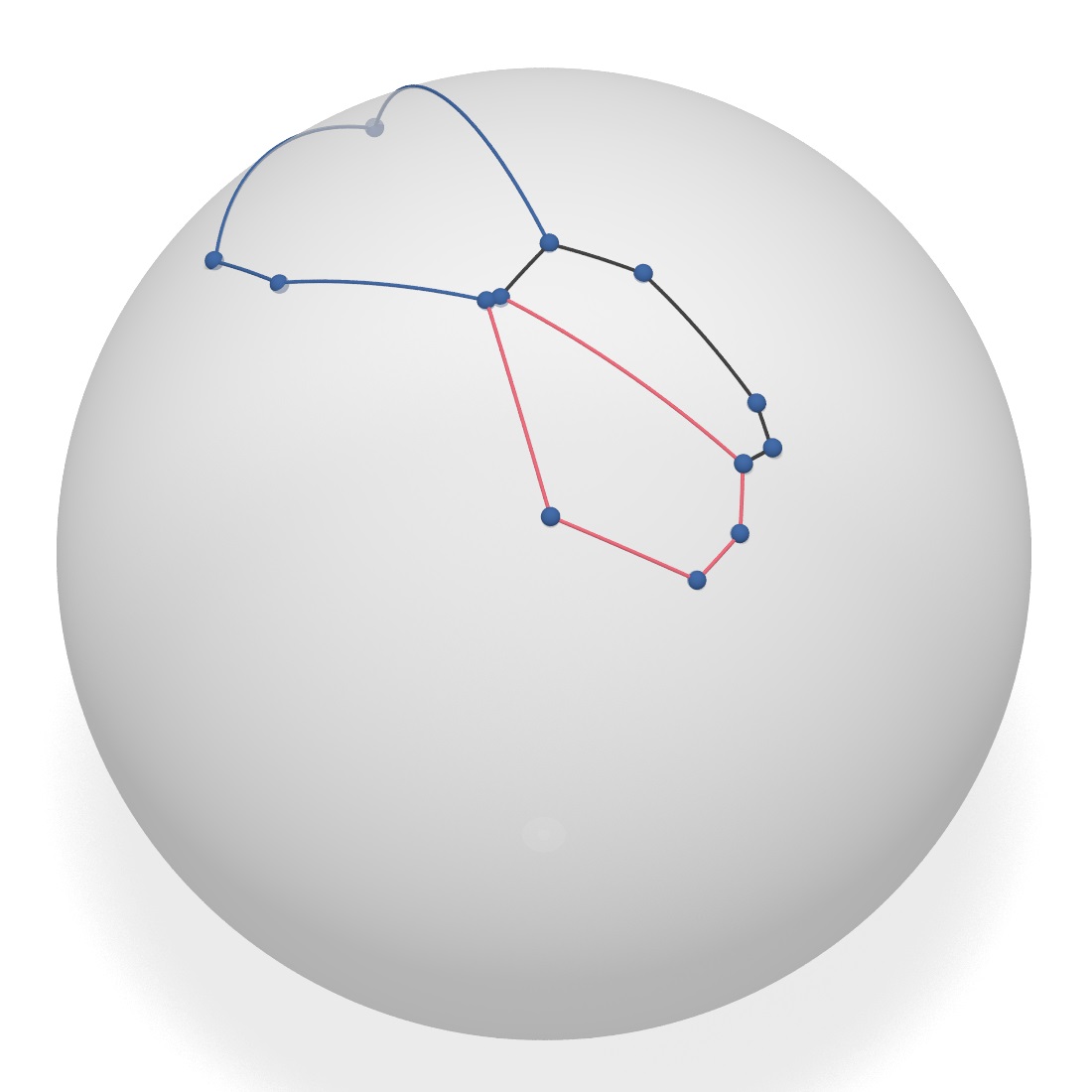}
			\cput{44}{75}{\scalebox{0.75}{$\vec{n}_{1}$}}
			\cput{41}{70}{\scalebox{0.75}{$\vec{n}_{2}$}}
			\cput{46}{51}{\scalebox{0.75}{$\vec{n}_{3}$}}
			\cput{64}{43}{\scalebox{0.75}{$\vec{n}_{4}$}}
			\cput{72}{49}{\scalebox{0.75}{$\vec{n}_{5}$}}
			\cput{64}{56}{\scalebox{0.75}{$\vec{n}_{6}$}}
			\cput{75}{58}{\scalebox{0.75}{$\vec{n}_{7}$}}
			\cput{74}{63}{\scalebox{0.75}{$\vec{n}_{8}$}}
			\cput{62}{76}{\scalebox{0.75}{$\vec{n}_{9}$}}	
			\cput{54}{80}{\scalebox{0.75}{$\vec{n}_{10}$}}		
			\cput{33}{85}{\scalebox{0.75}{$\vec{n}_{11}$}}
			\cput{16}{73}{\scalebox{0.75}{$\vec{n}_{12}$}}	
			\cput{27}{70}{\scalebox{0.75}{$\vec{n}_{13}$}}	
		\end{overpic}
	}
	\caption{Gauss images of the stars of the vertices in a positively curved region}\label{fig:positive_nice}
\end{figure}

From Lemma~\ref{lem:angle} we deduce that the angle of the Gauss image corresponding to the vertex $\vec{v}$ at ${\vec{n}}_{1}$ equals $\pi-\alpha_{\vec{v}}$, where $\alpha_{\vec{v}}$ denotes the interior angle of $f_1$ at $\vec{v}$. Then, \[\sum\limits_{{\vec{v}}\sim f_1} (\pi-\alpha_{\vec{v}})=n\pi-\sum\limits_{{\vec{v}}\sim f_1} \alpha_{\vec{v}}=n\pi-(n-2)\pi=2\pi.\] Using in addition that all Gauss images are convex spherical polygons, it follows that the single Gauss images of the vertex stars fit well together at ${\vec{n}}_{1}$ without intersecting any of the others. In summary, we have proven the following:

\begin{proposition}\label{prop:face_positive}
Let $f$ be a face of $P$. Assume that $K({\vec{v}})>0$ and that the Gauss image of the star of $\vec{v}$ has no self-intersections for all $\vec{v} \sim f$. Then, $f$ is a convex polygon and the Gauss images of the corresponding vertex stars do not intersect.
\end{proposition}

If we now think of the $n$ Gauss images as the central projection of $n$ polygons onto the sphere, then we can choose these polygons in such a way that they form part of the boundary of a convex polyhedron. Actually, we consider the projective dual surface as we will discuss in more detail in Section~\ref{sec:projective}. The fact that the Gauss images fit well together at ${\vec{n}}_{1}$ translates into the plane orthogonal to ${\vec{n}}_{1}$ being a transverse plane for the projective dual surface.

\begin{remark}
On a smooth surface, isolated parabolic points may be surrounded by elliptic points (positive Gaussian curvature). In this case, the Gauss image of a neighborhood of the parabolic point behaves as if the curvature would be throughout positive and the parabolic point would not exist. The situation in the discrete setup is similar: We have seen that a face $f_1$ all of whose vertices have positive discrete Gaussian curvature is necessarily convex and the Gauss images of the corresponding vertex stars will go once around the normal of $f_1$. So there may or may not be an isolated parabolic point on $f_1$, we cannot tell the difference. Therefore, we cannot reasonably model isolated parabolic points surrounded by elliptic points on polyhedral surfaces. 
\end{remark}

For a generic smooth surface, the tangent plane of a point does not contain any neighboring points. Since the plane through a face $f$ is an obvious tangent plane of the polyhedral surface, we now want to specify a point of contact of that plane with the surface. In the case of positive discrete Gaussian curvature, the face is convex and any interior point can serve as the point of contact.

\begin{definition}
Let $f$ be a face of $P$ and assume that all the stars of its vertices are convex. Then, we can choose any interior point of $f$ as the \textit{point of contact} of the \textit{discrete tangent plane} given by the plane through $f$.
\end{definition}


\subsection{Region of negative discrete Gaussian curvature}\label{sec:negative}

We now assume that all vertices of $f_1$ have negative discrete Gaussian curvature and that all the Gauss images of vertex stars have no self-intersections. If we go around the vertices of $f_1$ in counterclockwise order, the corresponding star-shaped Gauss images are attached at ${\vec{n}}_{1}$ in clockwise order. Their interior angles will form a multiple of $2\pi$. In the case that $f_1$ does not contain an isolated parabolic point, the angles should sum up to exactly $2\pi$ as in Fig.~\ref{fig:negative_nice}.

\begin{figure}[!ht]
	\centerline{
		\begin{overpic}[height=0.25\textwidth]{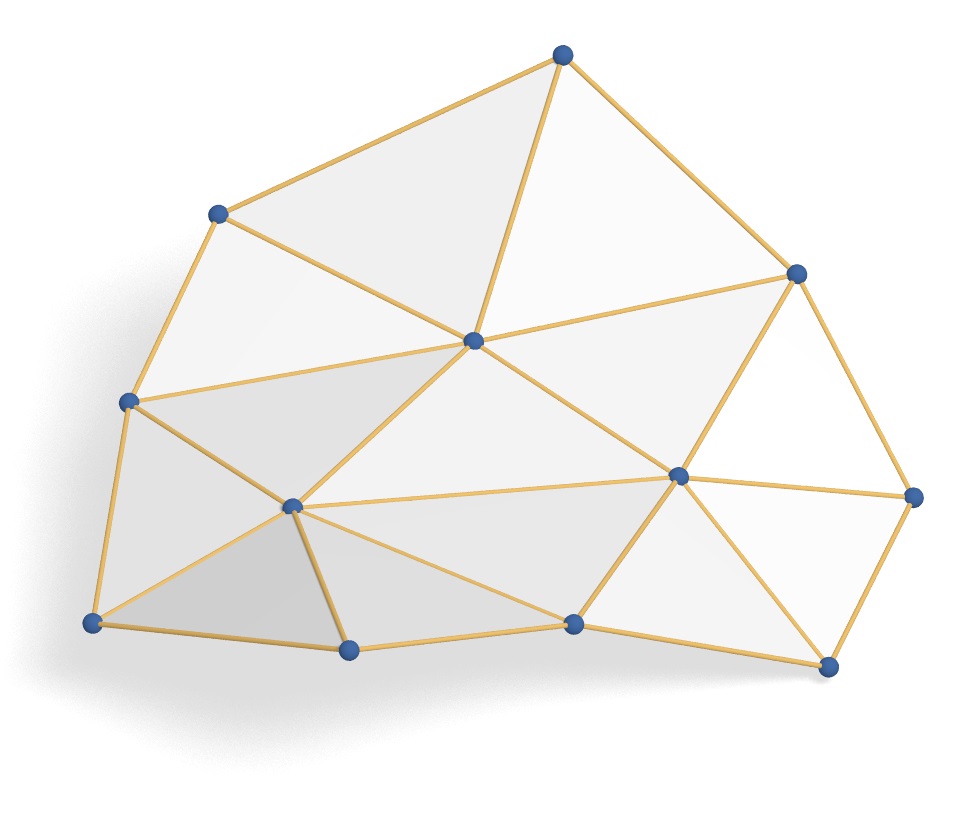}
			\put(47,40){\contour{white}{$f_1$}}
			\put(64,44){\contour{white}{$f_2$}}
			\put(60,60){\contour{white}{$f_3$}}
			\put(39,62){\contour{white}{$f_4$}}
			\put(25,50){\contour{white}{$f_5$}}
			\put(25,39){\contour{white}{$f_6$}}
			\put(13,32){\contour{white}{$f_7$}}
			\put(23,23){\contour{white}{$f_8$}}
			\put(36,22){\contour{white}{$f_9$}}
			\put(48,28){\contour{white}{$f_{10}$}}
			\put(65,23){\contour{white}{$f_{11}$}}
			\put(78,28){\contour{white}{$f_{12}$}}
			\put(77,42){\contour{white}{$f_{13}$}}
		\end{overpic}
		\relax\\
		\begin{overpic}[height=0.3\textwidth]{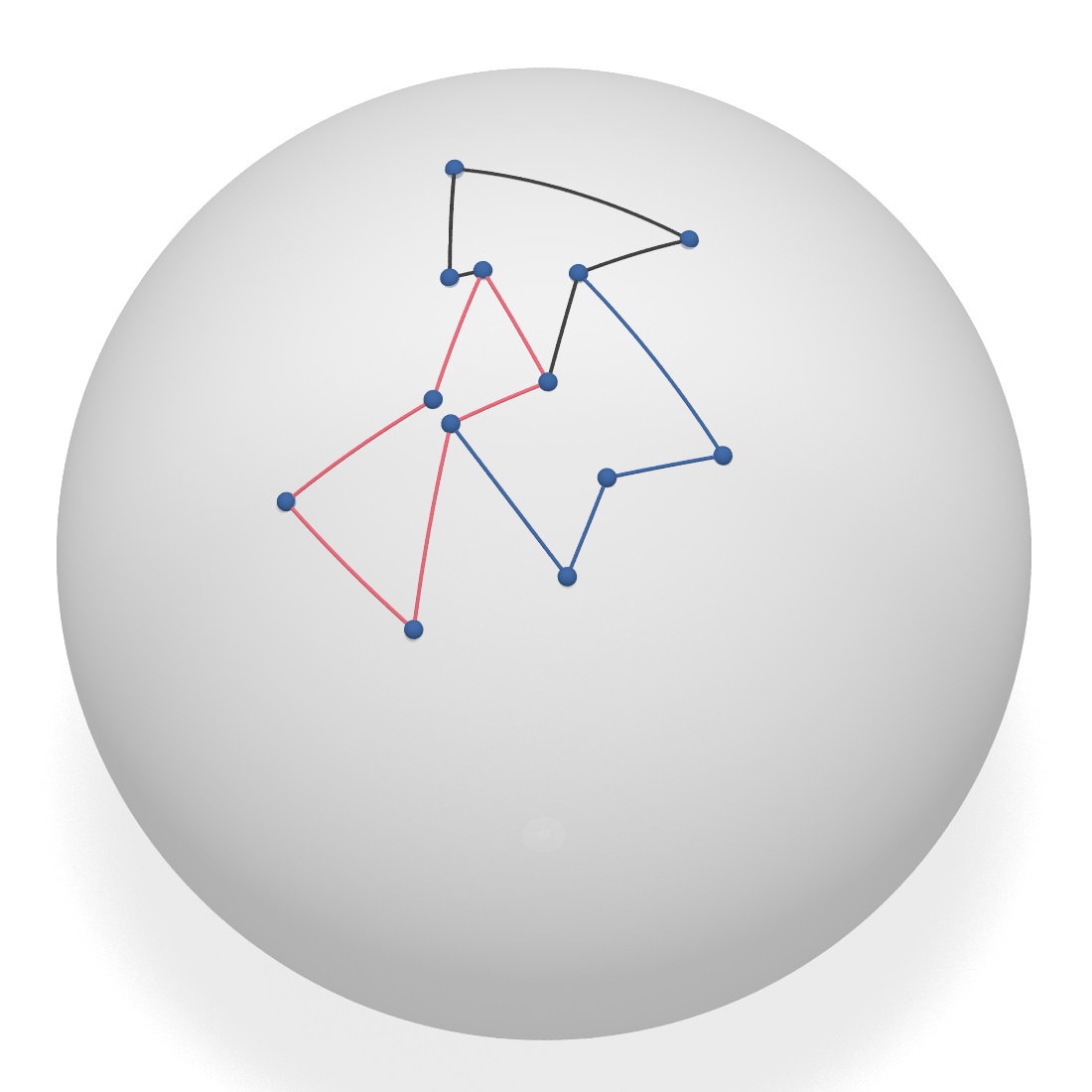}
			\cput{55}{63}{\scalebox{0.75}{$\vec{n}_{1}$}}
			\cput{46}{59.5}{\scalebox{0.75}{$\vec{n}_{2}$}}
			\cput{37}{39}{\scalebox{0.75}{$\vec{n}_{3}$}}
			\cput{22}{53}{\scalebox{0.75}{$\vec{n}_{4}$}}
			\cput{35}{64}{\scalebox{0.75}{$\vec{n}_{5}$}}
			\cput{46}{76.5}{\scalebox{0.75}{$\vec{n}_{6}$}}
			\cput{39}{71}{\scalebox{0.75}{$\vec{n}_{7}$}}
			\cput{43}{86}{\scalebox{0.75}{$\vec{n}_{8}$}}	
			\cput{65}{79.5}{\scalebox{0.75}{$\vec{n}_{9}$}}	
			\cput{60}{73}{\scalebox{0.75}{$\vec{n}_{10}$}}		
			\cput{72}{57}{\scalebox{0.75}{$\vec{n}_{11}$}}	
			\cput{61}{53}{\scalebox{0.75}{$\vec{n}_{12}$}}
			\cput{54}{43}{\scalebox{0.75}{$\vec{n}_{13}$}}
		\end{overpic}
	}
	\caption{Gauss images of the stars of the vertices of $f_1$ do not overlap: $f_6$ and $f_{10}$ are not inflection faces in the left vertex star, $f_1$ and $f_{12}$ are not inflection faces in the right vertex star, and $f_2$ and $f_5$ are not inflection faces in the upper vertex star}\label{fig:negative_nice}
\end{figure}
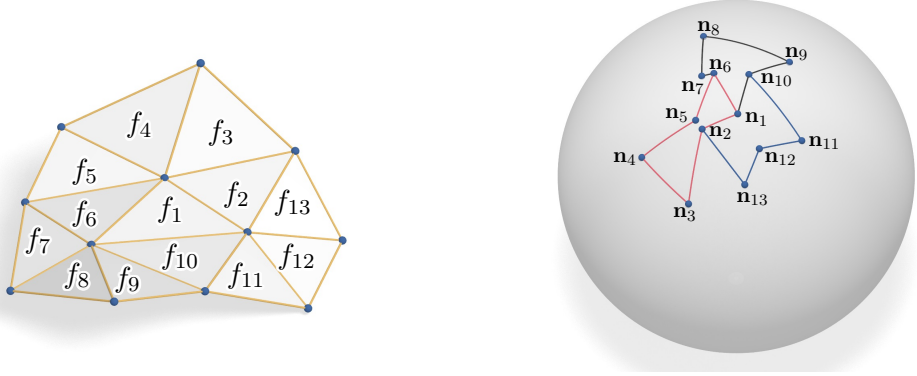

Similar to Section~\ref{sec:selfintersection} where we considered only a neighborhood of the vertex under consideration, we just observe how the individual Gauss images fit together at ${\vec{n}}_{1}$. They may not intersect any of the others as in Fig.~\ref{fig:negative_nice}, but one of them could intersect another Gauss image transversally as in Fig.~\ref{fig:negative_notnice}. 

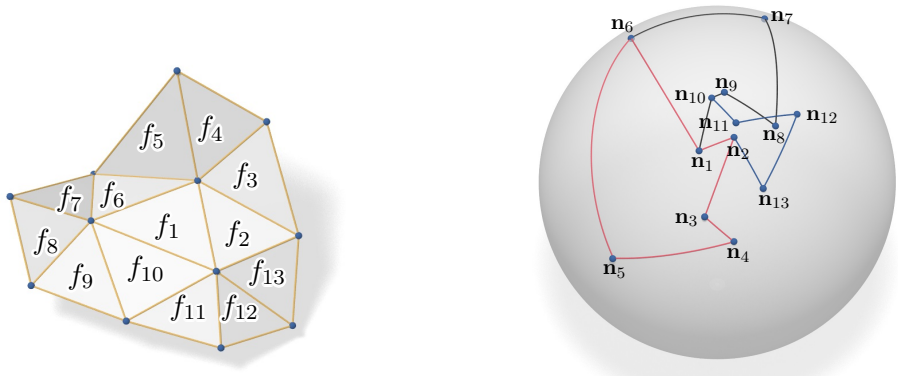
\begin{figure}[!ht]
	\centerline{
		\begin{overpic}[height=0.25\textwidth]{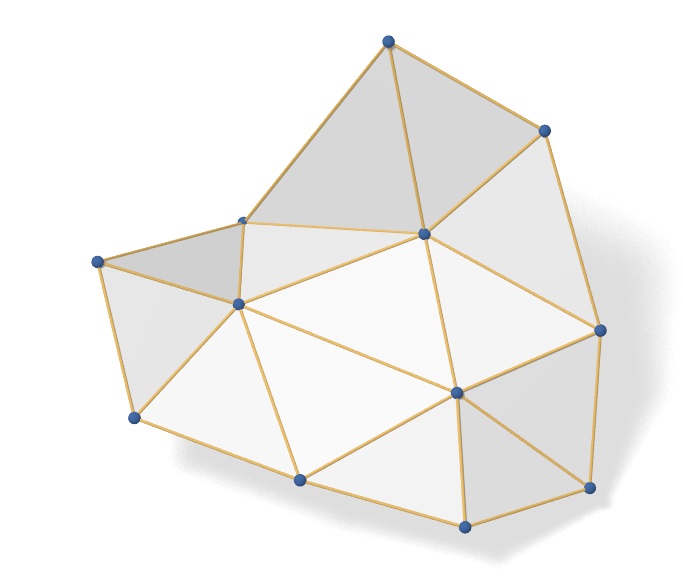}
			\put(51,37){\contour{white}{$f_1$}}
			\put(69,35){\contour{white}{$f_2$}}
			\put(71,50){\contour{white}{$f_3$}}
			\put(62,62){\contour{white}{$f_4$}}
			\put(47,60){\contour{white}{$f_5$}}
			\put(37,45){\contour{white}{$f_6$}}
			\put(26,44){\contour{white}{$f_7$}}
			\put(20,33){\contour{white}{$f_8$}}
			\put(29,24){\contour{white}{$f_9$}}
			\put(44,26){\contour{white}{$f_{10}$}}
			\put(55,16){\contour{white}{$f_{11}$}}
			\put(68,16){\contour{white}{$f_{12}$}}
			\put(75,25){\contour{white}{$f_{13}$}}
		\end{overpic}
		\relax\\
		\begin{overpic}[height=0.3\textwidth]{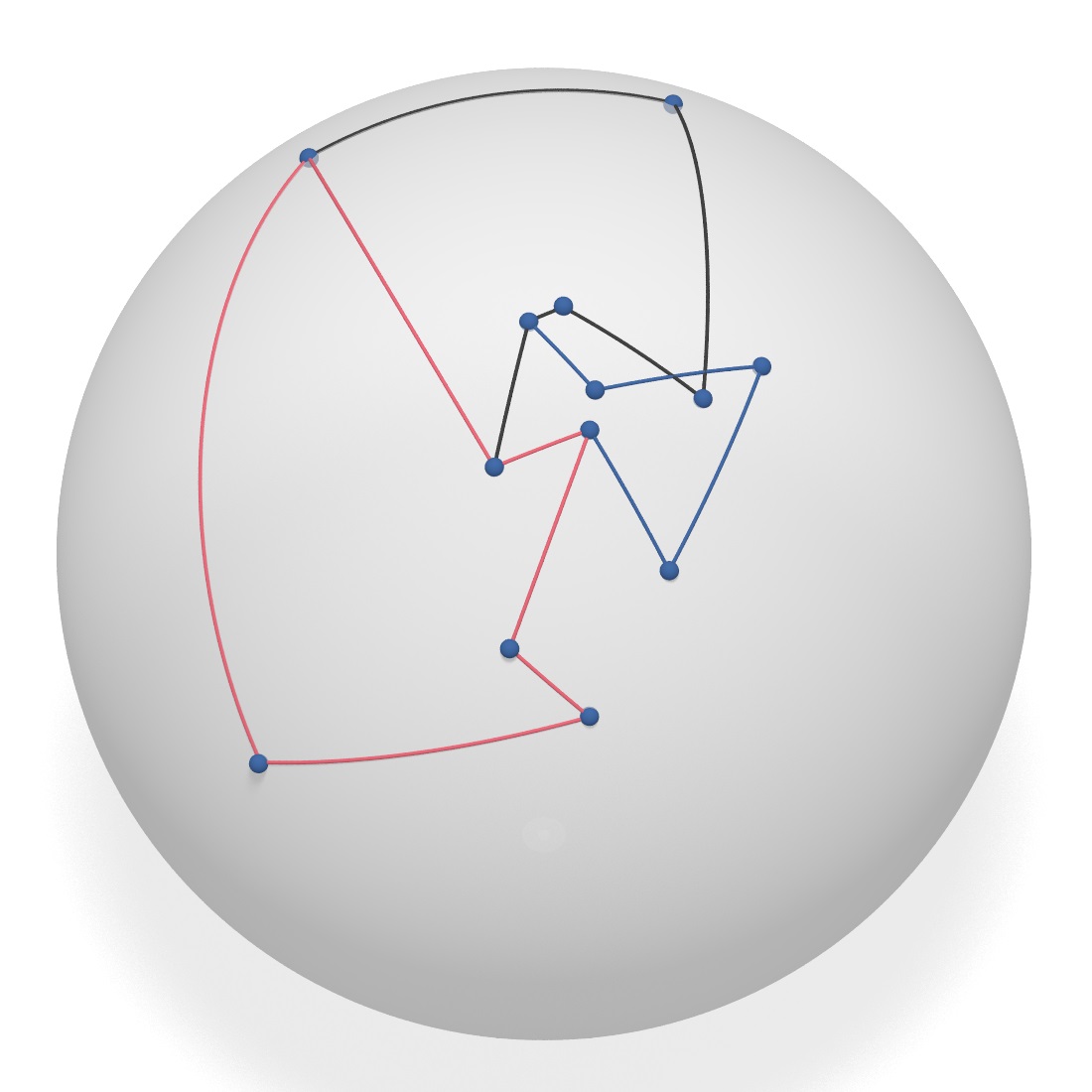}
			\cput{46}{53}{\scalebox{0.75}{$\vec{n}_{1}$}}
			\cput{55}{56}{\scalebox{0.75}{$\vec{n}_{2}$}}
			\cput{42}{39}{\scalebox{0.75}{$\vec{n}_{3}$}}
			\cput{55}{30}{\scalebox{0.75}{$\vec{n}_{4}$}}
			\cput{24}{26}{\scalebox{0.75}{$\vec{n}_{5}$}}
			\cput{26}{88}{\scalebox{0.75}{$\vec{n}_{6}$}}
			\cput{66}{90}{\scalebox{0.75}{$\vec{n}_{7}$}}
			\cput{64}{60}{\scalebox{0.75}{$\vec{n}_{8}$}}	
			\cput{52}{73}{\scalebox{0.75}{$\vec{n}_{9}$}}	
			\cput{43}{70}{\scalebox{0.75}{$\vec{n}_{10}$}}		
			\cput{49}{63}{\scalebox{0.75}{$\vec{n}_{11}$}}	
			\cput{76}{65}{\scalebox{0.75}{$\vec{n}_{12}$}}	
			\cput{64}{43.5}{\scalebox{0.75}{$\vec{n}_{13}$}}	
		\end{overpic}
	}
	\caption{Gauss images of the stars of the vertices of $f_1$ partially overlap: $f_9$ and $f_{10}$ are not inflection faces in the left vertex star, $f_2$ and $f_{11}$ are not inflection faces in the lower vertex star, and $f_1$ and $f_3$ are not inflection faces in the upper vertex star}\label{fig:negative_notnice}
\end{figure}

What we actually can exclude is an overlap as in Fig.~\ref{fig:monkeysaddle2} assuming that also the angles of Gauss images of the negatively curved vertex stars around $f_2$ sum up to $2\pi$ at ${\vec{n}}_{2}$.

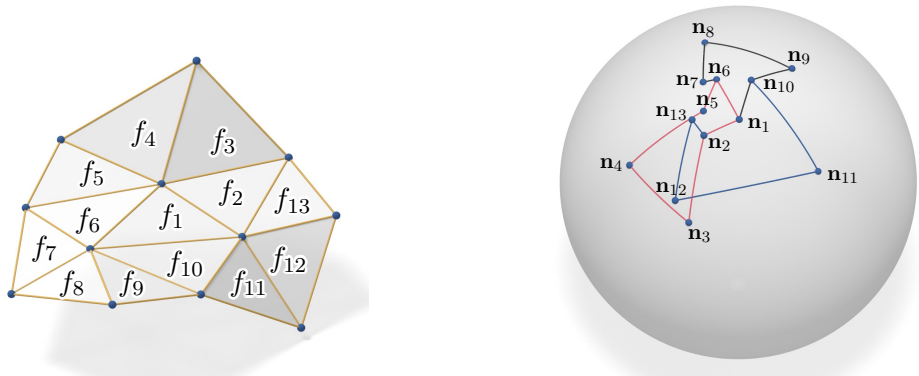
\begin{figure}[!ht]
	\centerline{
		\begin{overpic}[height=0.25\textwidth]{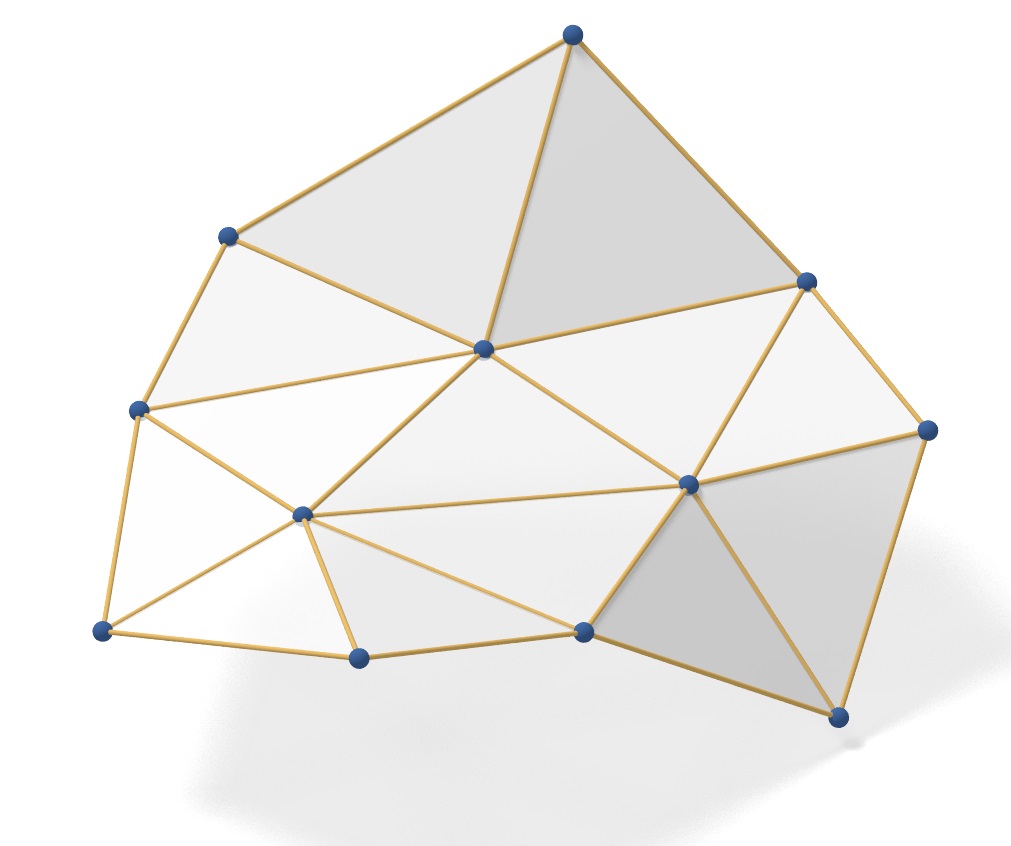}
		\put(47,39){\contour{white}{$f_1$}}
		\put(62,45){\contour{white}{$f_2$}}
		\put(60,58){\contour{white}{$f_3$}}
		\put(40,60){\contour{white}{$f_4$}}
		\put(27,50){\contour{white}{$f_5$}}
		\put(26,38){\contour{white}{$f_6$}}
		\put(15,31){\contour{white}{$f_7$}}
		\put(22,22){\contour{white}{$f_8$}}
		\put(37,22){\contour{white}{$f_9$}}
		\put(49,27){\contour{white}{$f_{10}$}}
		\put(65,22){\contour{white}{$f_{11}$}}
		\put(75,28){\contour{white}{$f_{12}$}}
		\put(76,43){\contour{white}{$f_{13}$}}
		\end{overpic}
		\relax\\
		\begin{overpic}[height=0.3\textwidth]{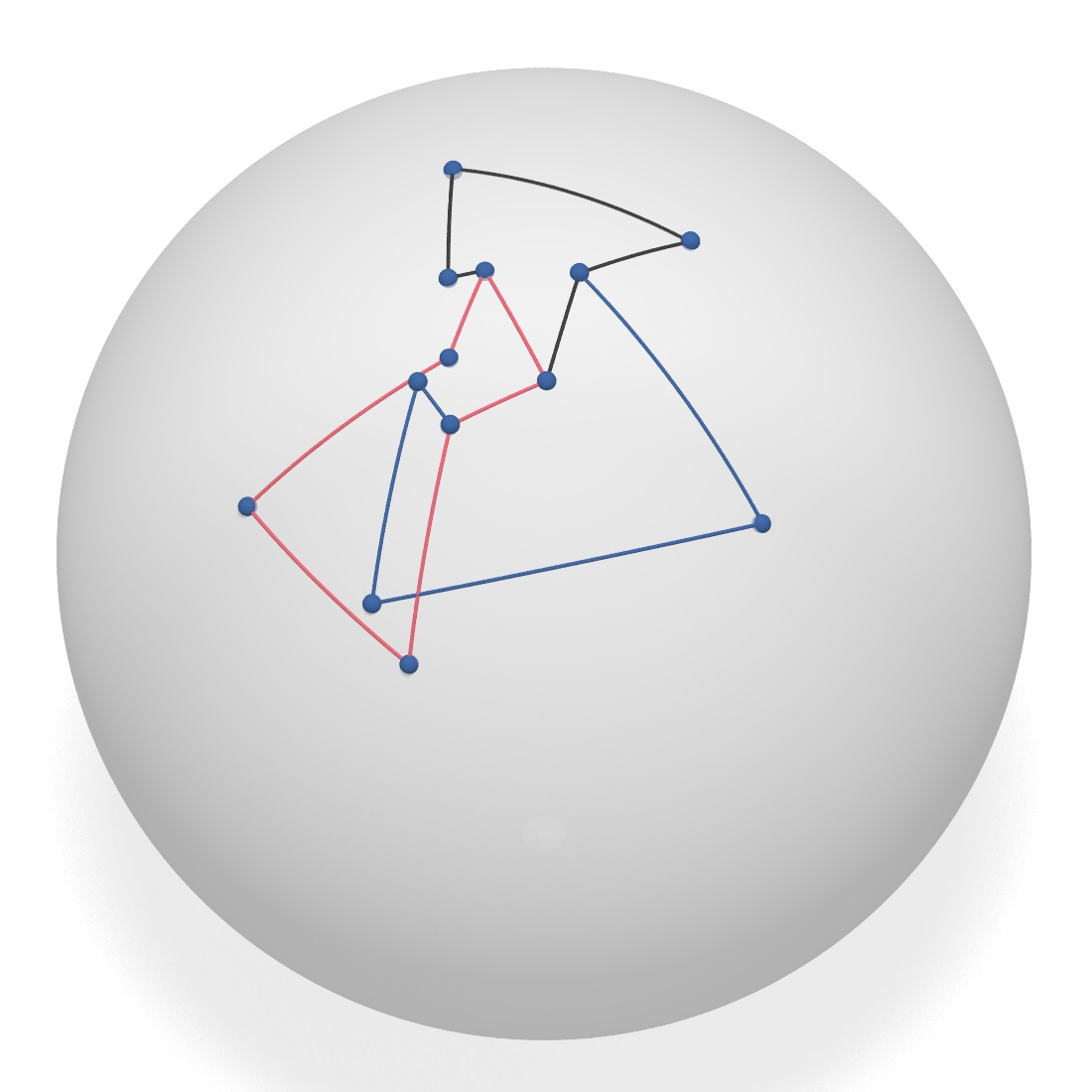}
		\cput{55}{63}{\scalebox{0.75}{$\vec{n}_{1}$}}
		\cput{45}{58}{\scalebox{0.75}{$\vec{n}_{2}$}}
		\cput{40}{35}{\scalebox{0.75}{$\vec{n}_{3}$}}
		\cput{18}{53}{\scalebox{0.75}{$\vec{n}_{4}$}}
		\cput{42}{69}{\scalebox{0.75}{$\vec{n}_{5}$}}
		\cput{45}{77}{\scalebox{0.75}{$\vec{n}_{6}$}}
		\cput{37}{74}{\scalebox{0.75}{$\vec{n}_{7}$}}
		\cput{41}{87}{\scalebox{0.75}{$\vec{n}_{8}$}}	
		\cput{65}{80}{\scalebox{0.75}{$\vec{n}_{9}$}}	
		\cput{60}{73}{\scalebox{0.75}{$\vec{n}_{10}$}}		
		\cput{76}{50}{\scalebox{0.75}{$\vec{n}_{11}$}}	
		\cput{33}{47}{\scalebox{0.75}{$\vec{n}_{12}$}}	
		\cput{33}{66}{\scalebox{0.75}{$\vec{n}_{13}$}}		
		\end{overpic}
	}
	\caption{Polyhedral monkey saddle at $f_2$: $f_6$ and $f_{10}$ are not inflection faces in the left vertex star, $f_1$ and $f_2$ are not inflection faces in the right vertex star, and $f_2$ and $f_5$ are not inflection faces in the upper vertex star}\label{fig:monkeysaddle2}
\end{figure}

\begin{remark}
In fact, such a behavior mimics a polyhedral version of a monkey saddle. If we think about a smooth monkey saddle, then it is a parabolic point $p$ surrounded by hyperbolic points. The Gauss image of a neighborhood of $p$ has a branching of order two at the normal of $p$. Such a behavior can be observed in Fig.~\ref{fig:monkeysaddle} as well: All faces adjacent to $f_1$ lie on the same side of the plane through $f_1$ and all vertices incident to $f_1$ are hyperbolic. In the Gauss image, we see a branching of order two at ${\vec{n}}_{1}$, so we indeed model a polyhedral version of a monkey saddle. Compared to the polyhedral monkey saddle Banchoff described in \cite{B70}, the parabolic point in our setup is located on the face and not on a vertex.
\end{remark}

\begin{figure}[!ht]
	\centerline{
		\begin{overpic}[height=0.25\textwidth]{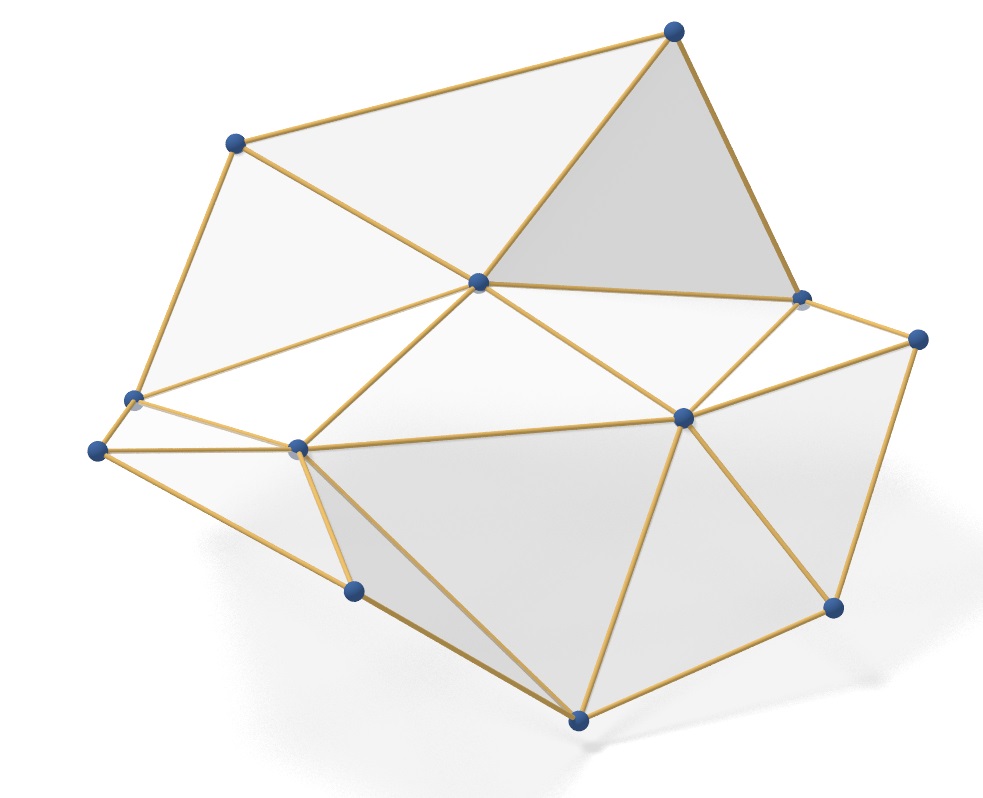}
	\put(47,41){\contour{white}{$f_1$}}
	\put(63,45){\contour{white}{$f_2$}}
	\put(63,58){\contour{white}{$f_3$}}
	\put(42,62){\contour{white}{$f_4$}}
	\put(26,52){\contour{white}{$f_5$}}
	\put(26,39){\contour{white}{$f_6$}}
	\put(15,36){\contour{white}{$f_7$}}
	\put(23,29){\contour{white}{$f_8$}}
	\put(37,22){\contour{white}{$f_9$}}
	\put(49,27){\contour{white}{$f_{10}$}}
	\put(67,22){\contour{white}{$f_{11}$}}
	\put(76,32){\contour{white}{$f_{12}$}}
	\put(76,43){\contour{white}{$f_{13}$}}
		\end{overpic}
		\relax\\
		\begin{overpic}[height=0.3\textwidth]{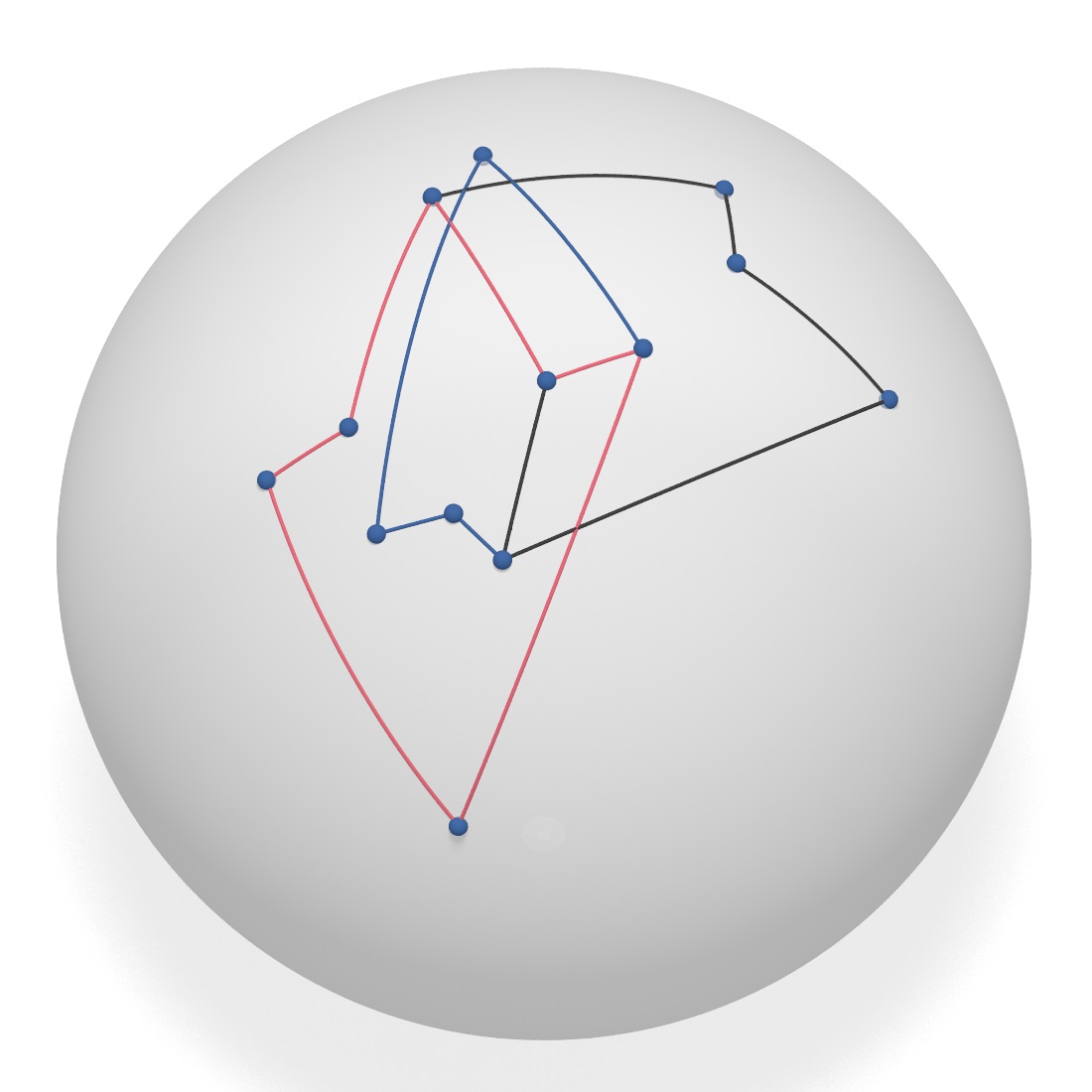}
		\cput{54}{62}{\scalebox{0.75}{$\vec{n}_{1}$}}
		\cput{63}{67}{\scalebox{0.75}{$\vec{n}_{2}$}}
		\cput{41}{21}{\scalebox{0.75}{$\vec{n}_{3}$}}
		\cput{20}{55}{\scalebox{0.75}{$\vec{n}_{4}$}}
		\cput{28}{61}{\scalebox{0.75}{$\vec{n}_{5}$}}
		\cput{38}{84}{\scalebox{0.75}{$\vec{n}_{6}$}}
		\cput{67}{84}{\scalebox{0.75}{$\vec{n}_{7}$}}
		\cput{72}{75}{\scalebox{0.75}{$\vec{n}_{8}$}}	
		\cput{86}{62}{\scalebox{0.75}{$\vec{n}_{9}$}}	
		\cput{45}{45}{\scalebox{0.75}{$\vec{n}_{10}$}}		
		\cput{41}{55}{\scalebox{0.75}{$\vec{n}_{11}$}}	
		\cput{33}{47}{\scalebox{0.75}{$\vec{n}_{12}$}}	
		\cput{43}{88}{\scalebox{0.75}{$\vec{n}_{13}$}}		
		\end{overpic}
	}
	\caption{Polyhedral monkey saddle at $f_1$: $f_1$ and $f_8$ are not inflection faces in the left vertex star, $f_1$ and $f_{11}$ are not inflection faces in the right vertex star, and $f_1$ and $f_5$ are not inflection faces in the upper vertex star}\label{fig:monkeysaddle}
\end{figure}
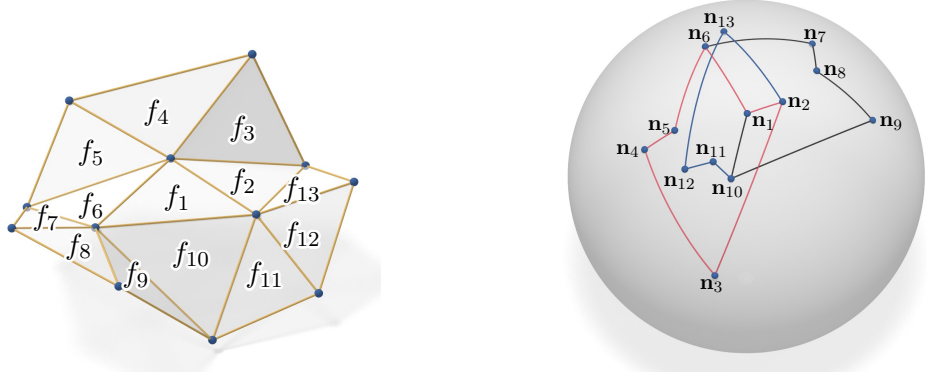

For a polygon $f_1$ with more than just three vertices and branchings in the Gauss image of higher order, corresponding higher order saddles can be modeled as well.

But as we mentioned before, isolated parabolic points are not generic, so we will not include them in our notion of smoothness. The observation that small deformations of the smooth monkey saddle perturb the parabolic point into a small parabolic curve can be transfered to the polyhedral case as follows. We assume that $f_1$ is a triangle and put a point $P$ inside the triangle. That point $P$ is now moved slightly inside the half-space opposite to the one the faces adjacent to $f_1$ are contained in. $f_1$ will be replaced by the three triangles defined by $P$ and the three edges of $f_1$. Then, we will get a convex corner $P$ and three ordinary saddles at the three original vertices of $f_1$, see Fig.~\ref{fig:monkeysaddleP}.

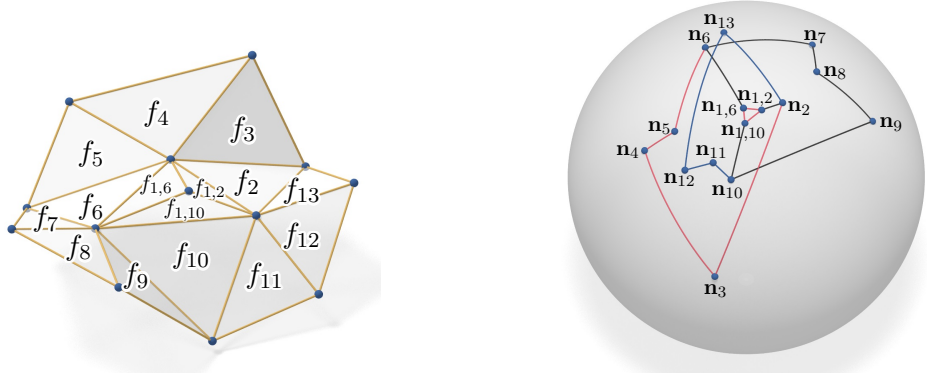
\begin{figure}[htbp]
		\centerline{
			\begin{overpic}[height=0.25\textwidth]{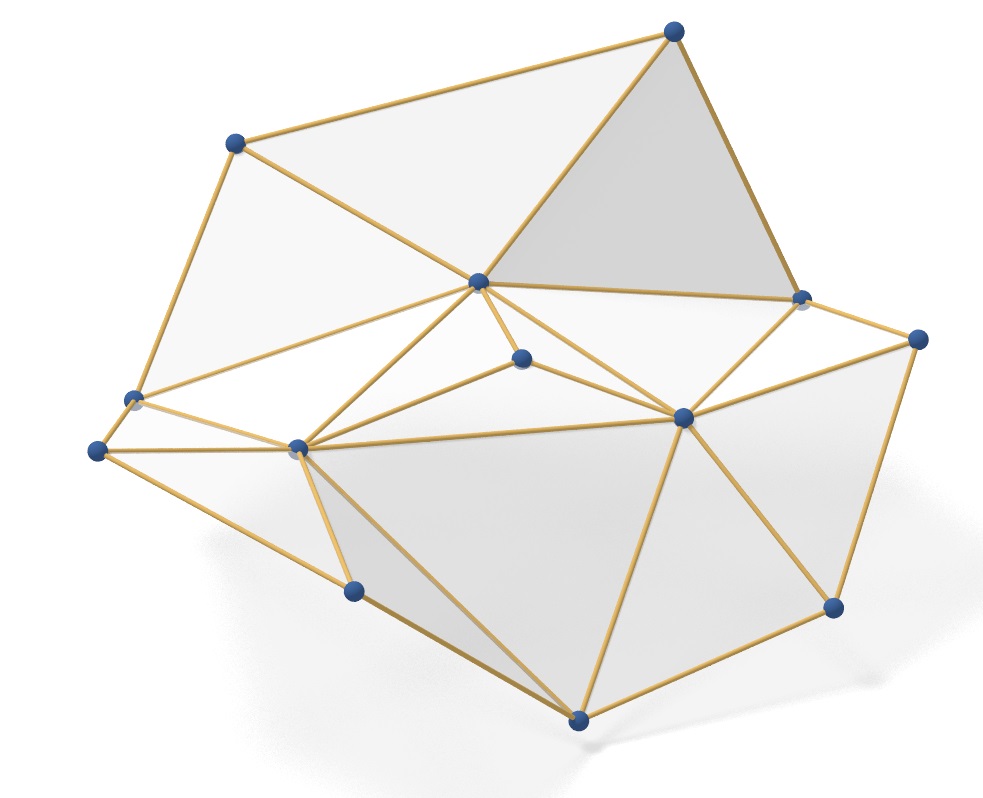}
				\put(47,39){\scalebox{0.75}{\contour{white}{$f_{1,10}$}}}
				\put(54,44){\scalebox{0.75}{\contour{white}{$f_{1,2}$}}}
				\put(41,45){\scalebox{0.75}{\contour{white}{$f_{1,6}$}}}
				\put(64,45){\contour{white}{$f_2$}}
				\put(63,58){\contour{white}{$f_3$}}
				\put(42,62){\contour{white}{$f_4$}}
				\put(26,52){\contour{white}{$f_5$}}
				\put(26,39){\contour{white}{$f_6$}}
				\put(15,36){\contour{white}{$f_7$}}
				\put(23,29){\contour{white}{$f_8$}}
				\put(37,22){\contour{white}{$f_9$}}
				\put(49,27){\contour{white}{$f_{10}$}}
				\put(67,22){\contour{white}{$f_{11}$}}
				\put(76,32){\contour{white}{$f_{12}$}}
				\put(76,43){\contour{white}{$f_{13}$}}
			\end{overpic}
			\relax\\
			\begin{overpic}[height=0.3\textwidth]{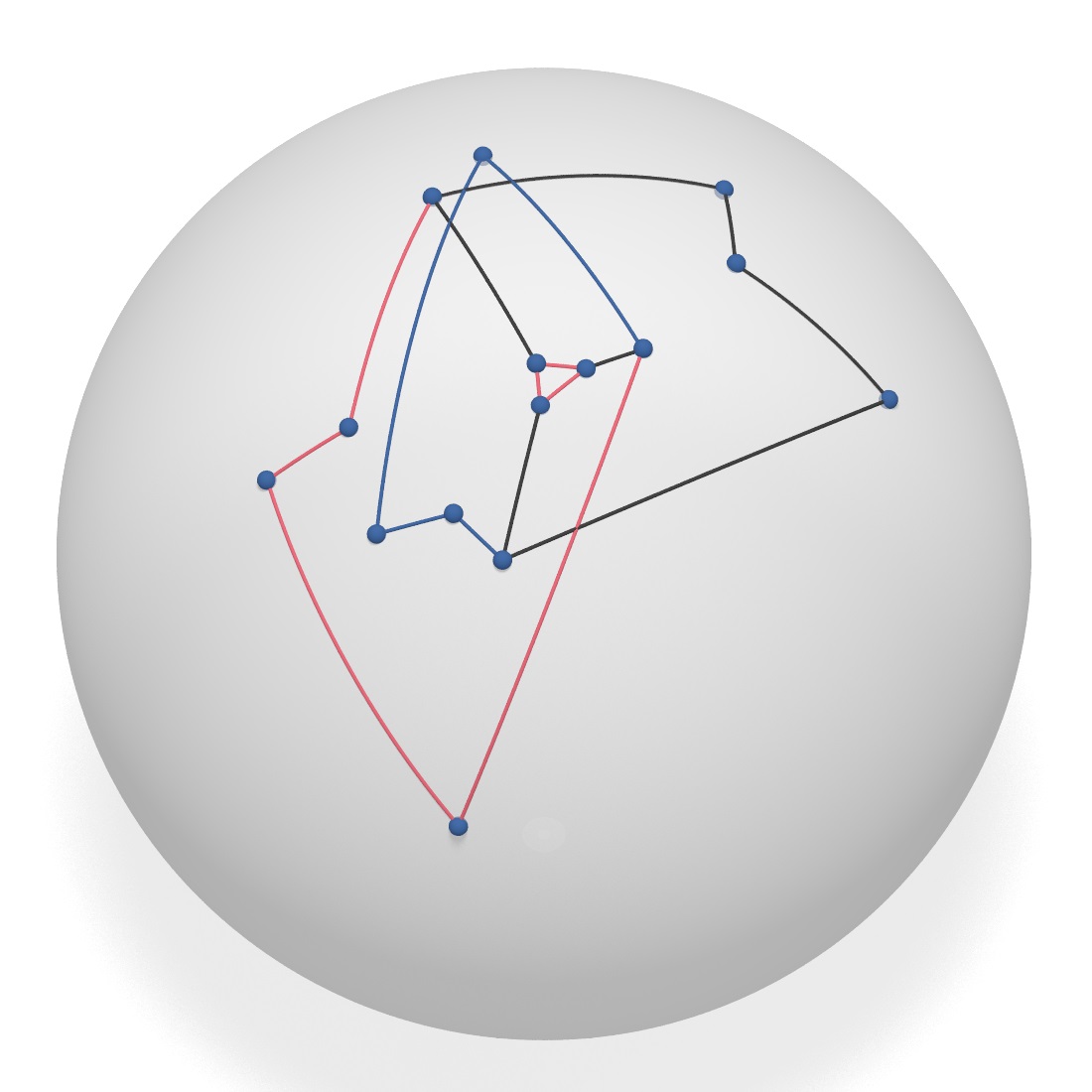}
				\cput{49}{60}{\scalebox{0.75}{$\vec{n}_{1,10}$}}
				\cput{52}{69}{\scalebox{0.75}{$\vec{n}_{1,2}$}}
				\cput{43}{66}{\scalebox{0.75}{$\vec{n}_{1,6}$}}
				\cput{63}{66}{\scalebox{0.75}{$\vec{n}_{2}$}}
				\cput{42}{20}{\scalebox{0.75}{$\vec{n}_{3}$}}
				\cput{20}{55}{\scalebox{0.75}{$\vec{n}_{4}$}}
				\cput{28}{61}{\scalebox{0.75}{$\vec{n}_{5}$}}
				\cput{38}{84}{\scalebox{0.75}{$\vec{n}_{6}$}}
				\cput{67}{84}{\scalebox{0.75}{$\vec{n}_{7}$}}
				\cput{72}{75}{\scalebox{0.75}{$\vec{n}_{8}$}}	
				\cput{86}{62}{\scalebox{0.75}{$\vec{n}_{9}$}}	
				\cput{45}{45}{\scalebox{0.75}{$\vec{n}_{10}$}}		
				\cput{41}{55}{\scalebox{0.75}{$\vec{n}_{11}$}}	
				\cput{33}{48}{\scalebox{0.75}{$\vec{n}_{12}$}}	
				\cput{43}{88}{\scalebox{0.75}{$\vec{n}_{13}$}}		
			\end{overpic}
		}
   \caption{Replacing a polyhedral monkey saddle by three saddles and a convex corner}
   \label{fig:monkeysaddleP}
\end{figure}

 By Lemma~\ref{lem:angle} we know that the interior angle of the Gauss image of the star of $\vec{v}\sim f_1$ at ${\vec{n}}_{1}$ equals
\begin{align*}
2\pi-(\pi-\alpha_{\vec{v}})&=\alpha_{\vec{v}}+\pi \textnormal{ if } \alpha_{\vec{v}}<\pi \textnormal{ and } f_1 \textnormal{ is not an inflection face at } {\vec{v}},\\
2\pi-(2\pi-\alpha_{\vec{v}})&=\alpha_{\vec{v}} \textnormal{ if } \alpha_{\vec{v}}<\pi \textnormal{ and } f_1 \textnormal{ is an inflection face at } {\vec{v}},\\
2\pi-(3\pi-\alpha_{\vec{v}})&=\alpha_{\vec{v}}-\pi \textnormal{ if } \alpha_{\vec{v}}>\pi \textnormal{ and } f_1 \textnormal{ is not an inflection face at } {\vec{v}},\\
2\pi-(2\pi-\alpha_{\vec{v}})&=\alpha_{\vec{v}} \textnormal{ if } \alpha_{\vec{v}}>\pi \textnormal{ and } f_1 \textnormal{ is an inflection face at } {\vec{v}}.
\end{align*}

Here, $\alpha_{\vec{v}}$ denotes the interior angle of $f_1$ at one of its $n$ vertices $\vec{v}$. Let $c_i$, $i=1,2,3,4$, be the number of vertices that correspond to the $i$th line in the list above. Then, \[2\pi=\sum\limits_{{\vec{v}}\sim f_1}\alpha_{\vec{v}}+c_1\pi-c_3\pi=(n-2)\pi+c_1\pi-c_3\pi.\] It follows that $c_1-c_3=4-n$. Now, $c_1+c_2+c_3+c_4=n$, such that \[2c_1+c_2+c_4=4.\] Since $f_1$ is planar, it has to have at least three corners, i.e., $c_1+c_2 \geq 3$. We end up with the following three cases that are illustrated in Fig.~\ref{fig:negative_face}:

\begin{enumerate}
\item $c_1=0$, $c_2=4$, $c_3=n-4$, $c_4=0$: $f_1$ is a pseudo-quadrilateral and $f_1$ is an inflection face exactly in the vertex stars of its corners;
\item $c_1=0$, $c_2=3$, $c_3=n-4$, $c_4=1$: $f_1$ is a pseudo-triangle and $f_1$ is an inflection face exactly in the vertex stars of its three corners and in one further vertex;
\item $c_1=1$, $c_2=2$, $c_3=n-3$, $c_4=0$: $f_1$ is a pseudo-triangle and $f_1$ is an inflection face exactly in the vertex stars of two of its corners.
\end{enumerate}

\begin{figure}[!ht]
\subfloat{		
		\begin{overpic}[height=0.3\textwidth]{f8_12}
			\put(20,43){\contour{white}{$f_1$}}
			\put(36,30){\contour{white}{$f_2$}}
			\put(57,32){\contour{white}{$f_3$}}
			\put(70,54){\contour{white}{$f_4$}}
			\put(54,64){\contour{white}{$f_5$}}
			\put(37,58){\contour{white}{$f_6$}}
		\end{overpic}
		\relax\\
		\begin{overpic}[height=0.3\textwidth]{f8_22}
			\cput{58}{66}{$\vec{n}_{1}$}
			\cput{68}{57}{$\vec{n}_{2}$}
			\cput{81}{60}{$\vec{n}_{3}$}
			\cput{70}{29}{$\vec{n}_{4}$}
			\cput{16}{72}{$\vec{n}_{5}$}
			\cput{66}{81}{$\vec{n}_{6}$}	
		\end{overpic}
		\relax\\		
		\begin{overpic}[height=0.26\textwidth]{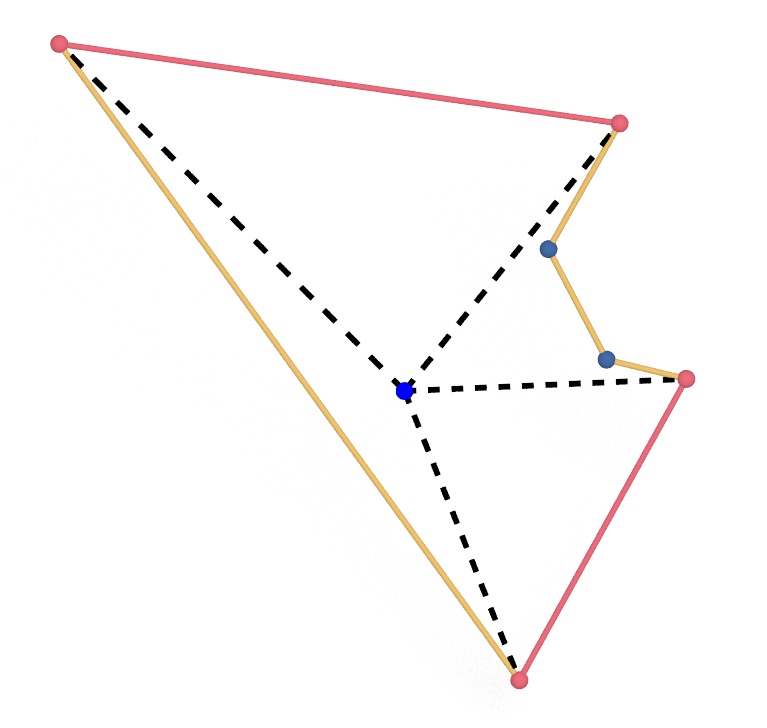}
			\cput{52}{47}{\color{red}{$A$}}
		\end{overpic}		
	}
    \\
	\subfloat{	
		\begin{overpic}[height=0.25\textwidth]{f9_1}
			\put(42,52){\contour{white}{$f_1$}}
			\put(25,32){\contour{white}{$f_2$}}
			\put(75,18){\contour{white}{$f_3$}}
			\put(63,15){\contour{white}{$f_4$}}
			\put(66,42){\contour{white}{$f_5$}}
			\put(55,55){\contour{white}{$f_6$}}	
		\end{overpic}
		\relax\\
		\begin{overpic}[height=0.28\textwidth]{f9_2}
				\cput{11}{65}{$\vec{n}_{1}$}
				\cput{47}{98}{$\vec{n}_{2}$}
				\cput{68}{60}{$\vec{n}_{3}$}
				\cput{84}{86}{$\vec{n}_{4}$}
				\cput{84}{28}{$\vec{n}_{5}$}
				\cput{41}{71}{$\vec{n}_{6}$}
		\end{overpic}
				\relax\\		
				\begin{overpic}[height=0.26\textwidth]{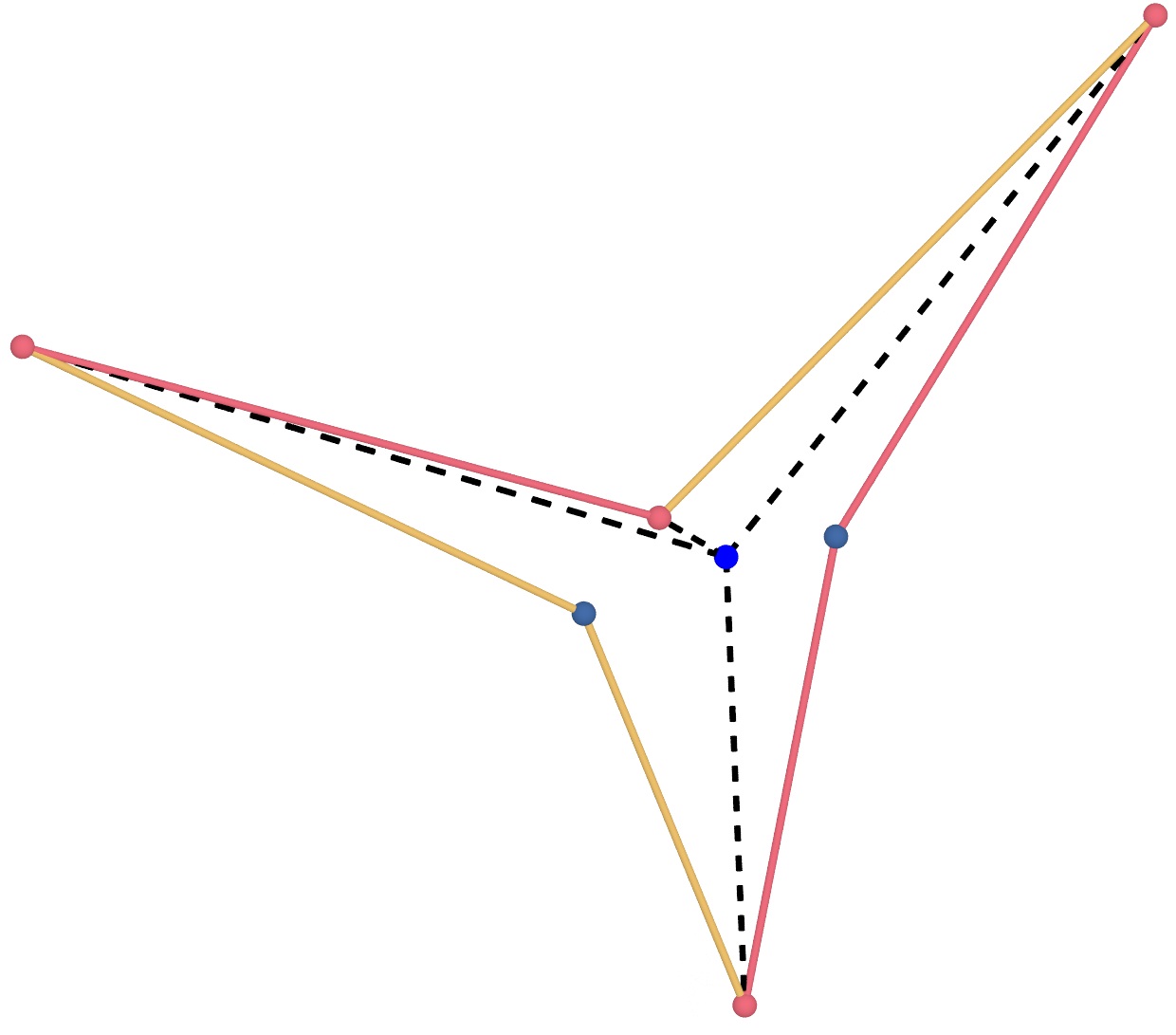}
			\cput{58}{34}{\color{red}{$A$}}	
				\end{overpic}	
		}
\\
\subfloat{
				\begin{overpic}[height=0.28\textwidth]{f10_12}
					\put(76,47){\contour{white}{$f_1$}}
					\put(18,65){\contour{white}{$f_2$}}
					\put(17,36){\contour{white}{$f_3$}}
					\put(31,21){\contour{white}{$f_4$}}
					\put(50,18){\contour{white}{$f_5$}}
					\put(66,28){\contour{white}{$f_6$}}
				\end{overpic}
				\relax\\
				\begin{overpic}[height=0.28\textwidth]{f10_22}
					\cput{38}{29}{$\vec{n}_{1}$}
					\cput{48}{97}{$\vec{n}_{2}$}
					\cput{94}{52}{$\vec{n}_{3}$}
					\cput{82}{57}{$\vec{n}_{4}$}
					\cput{65}{52}{$\vec{n}_{5}$}
					\cput{52}{43}{$\vec{n}_{6}$}	
				\end{overpic}
				\relax\\		
				\begin{overpic}[height=0.28\textwidth]{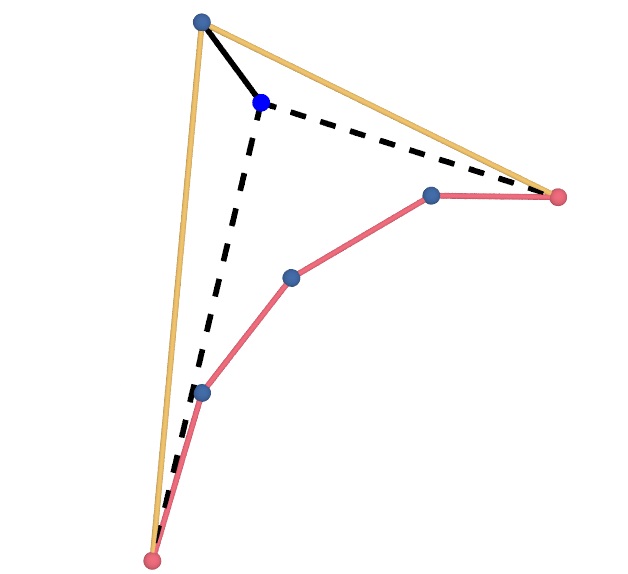}
			\cput{44}{68}{\color{red}{$A$}}	
				\end{overpic}	
	}
		\caption{Three cases for planar faces in negatively curved regions and their projective duals in Theorem~\ref{th:shape}. Vertices where the face is an inflection face are colored red and separate different colored edges. The two colors of edges correspond to adjacent faces that lie below or above the face. Asymptotic directions are marked by dashed lines. In the lower-right picture, the solid line corresponds to two asymptotic directions}\label{fig:negative_face}
\end{figure}
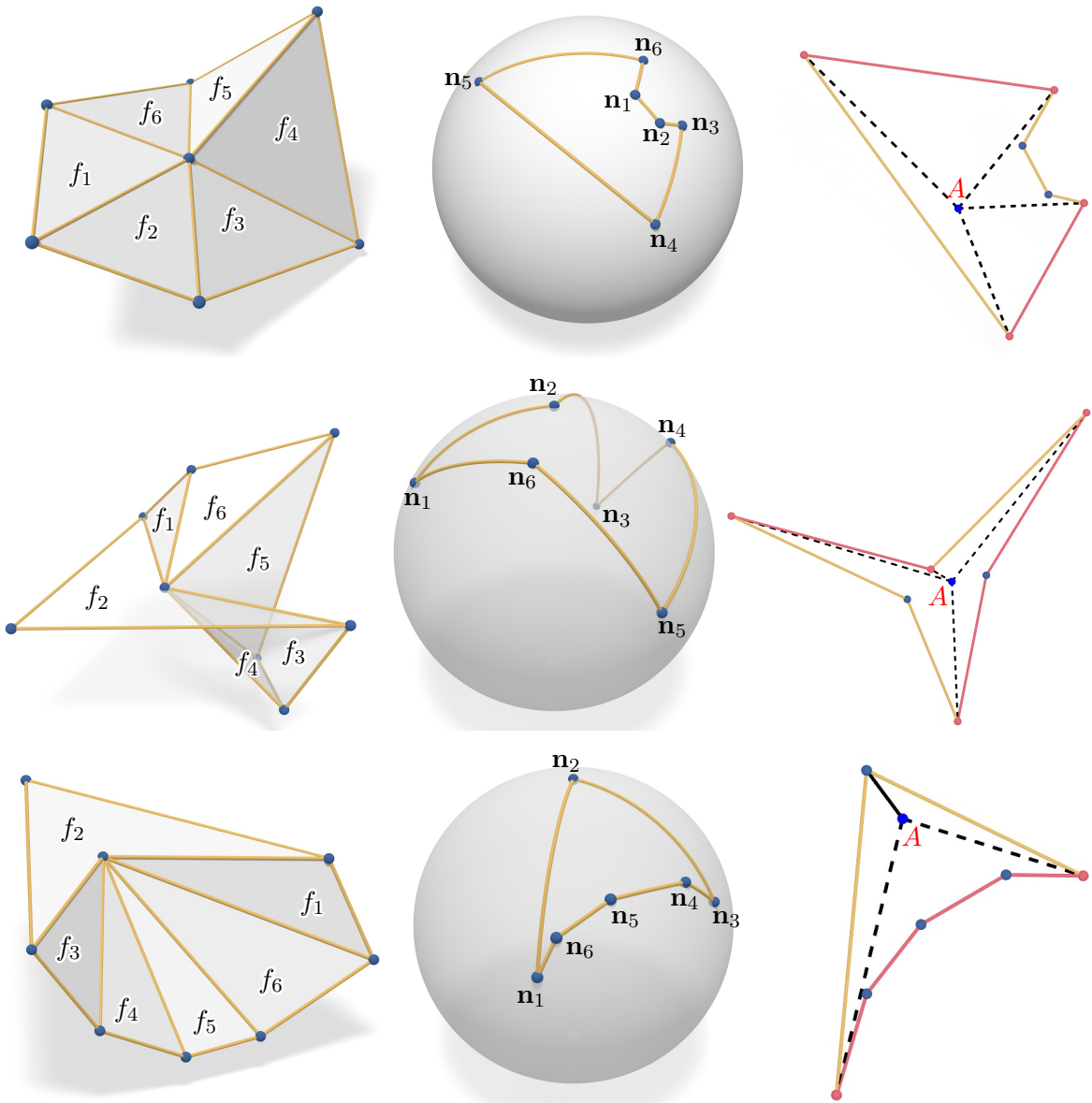

From our discussion of asymptotic directions in Section~\ref{sec:asymptotes} we remember that at each vertex $\vec{v}\sim f_1$ where $f_1$ is an inflection face the asymptotic direction is pointing inward the cone spanned by the two edges of $f_1$ incident to $\vec{v}$ (in the case that $\alpha_{\vec{v}}>\pi$, it is actually the cone opposite to the one spanned by the two incident edges). At vertices $\vec{v}\sim f_1$ with $\alpha_{\vec{v}} > \pi$ where $f_1$ is not an inflection face, the asymptotic directions lie in the double cone that is spanned by the edges incident to $\vec{v}$. 

If we assume in addition that $f_1$ is star-shaped with respect to some point $A \in f_1$, then we can define \textit{discrete asymptotic directions} in $A$.

In the first case that $f_1$ is a pseudo-quadrilateral, line segments connecting the corners of $f_1$ with $A$ are suitable asymptotic directions. Indeed, these line segments lie in the cones where the asymptotic directions at the vertices lie. Furthermore, let us consider an intersection of a plane parallel and either slightly above or below $f_1$ with a neighborhood of $f_1$. Since $f_1$ is an inflection face exactly at the four corners, we obtain in the limit of the plane coming closer to $f_1$ that these intersections converge to one pair of opposite pseudo-edges, i.e., the two opposite paths that each connect two neighboring corners, see the upper-right picture in Fig.~\ref{fig:negative_face}. In particular, the intersections are discrete hyperbolas in the limit. The asymptotic directions that we defined in $A$ interpolate their asymptotes.

In the second case that $f_1$ is a pseudo-triangle and $f_1$ is an inflection face at four of its vertices, then the line segments connecting $A$ with these four vertices should model asymptotic directions. Again, they lie in the cones where the asymptotic directions at these vertices lie. Furthermore, the intersections of planes parallel and close to $f_1$ with a neighborhood of $f_1$ will again be discrete hyperbolas in the limit, even though one branch of one hyperbola and one of the other form together again a branch of a discrete hyperbola, see the middle-right picture in Fig.~\ref{fig:negative_face}. Still, the asymptotic directions in $A$ somehow interpolate their asymptotes.

In contrast to the smooth case, the asymptotic directions we defined either at vertices or in faces do not form a pair of two lines in general. Hence, it is more adequate to consider these asymptotic directions as discrete asymptotic curves passing through a vertex or a point $A$ in a face.

The third case that $f_1$ is a pseudo-triangle, but $f_1$ is an inflection face at just two of its corners, is special. If we have a look at the limit of the intersection of planes parallel and close to $f_1$ with a neighborhood of the face, then it will be either one pseudo-edge or the combination of the other two, depending on the planes being above or below $f_1$. Hence, the limit consists of either a single branch of a discrete hyperbola or two branches of discrete hyperbolas that share a vertex. It is immediate that the two line segments connecting $A$ with the two vertices of $f_1$ where it is an inflection face are reasonable asymptotic directions. If we see these line segments as discrete asymptotic curves, they should both continue to the remaining corner $\vec{v}$ of $f_1$ following the asymptotes of the two branches of discrete hyperbolas that share this vertex. In this sense, the line segment connecting $A$ with $\vec{v}$ counts twice as a discrete asymptotic direction, see the lower-right picture in Fig.~\ref{fig:negative_face}. However, this line segment cannot be an asymptotic direction for $\vec{v}$.

\begin{remark}
The case of $f_1$ being a pseudo-triangle and $f_1$ being an inflection face at just two of its corners should be seen as a degeneracy of the first case. Indeed, if we are in the case of a pseudo-quadrilateral $f_1$ and shrink one pseudo-edge to a vertex, $f_1$ becomes a pseudo-triangle, two vertices where $f_1$ was an inflection face constitute a vertex where $f_1$ is not an inflection face any longer, the pair of opposite pseudo-edges that formed discrete hyperbola degenerate to a single branch of a discrete hyperbola (and a singular point) and two branches of a discrete hyperbola that share a vertex, and the two asymptotic directions that connected $A$ with the two corners that degenerate to just one vertex are overlapping in the end.

Another interpretation will be discussed in more detail in Section~\ref{sec:projective} when we consider the projective dual surface. Figure~\ref{fig:negative_face} depicts the three shapes of faces next to the three dual shapes of vertex stars and their Gauss images. To apply Theorem~\ref{th:correlation}, we need the existence of a point that does not lie on any plane tangent to the neighborhood of the face. This does not always have to be the case, for example it is not for a square with two opposite neighboring faces being slightly above and the other two opposite face being slightly below that square.

We will show in Theorem~\ref{th:correlation} that the discrete asymptotic directions in this degenerate case correspond to the discrete asymptotic directions in the shape discussed in case (iii) of Theorem~\ref{th:shape} with two inflection faces, see Fig.~\ref{fig:Dupin_triangle_degenerate}. There, the face with a reflex angle at the vertex $\vec{v}$ was counting as two inflection faces even though it was not an inflection face, and two asymptotic directions lay on a line passing through the vertex $\vec{v}$. 
\end{remark}

Note that in the case of a simplicial surface of negative curvature, we are always in this degenerate case where the asymptotic directions and the shape of the face near the vertex $\vec{v}$ whose Gauss image has a reflex angle at ${\vec{n}}_f$ resemble the behavior at a parabolic point of a smooth surface. This also explains why we observed in experiments that it seems to be quite hard to construct simplicial surfaces that are smooth in the sense above.

We summarize our results in the following:

\begin{proposition}\label{prop:face_negative}
Let $f$ be a face of $P$. Assume that $K({\vec{v}})<0$ and that the Gauss image of the star of $\vec{v}$ has no self-intersections for all $\vec{v} \sim f$. Assuming that the interior angles of the Gauss images of the corresponding vertex stars add up to $2\pi$ at ${\vec{n}}_f$, $f$ has one of the following shapes that are accompanied with the following restrictions at which vertices $f$ is an inflection face:
\begin{enumerate}
\item $f$ is a pseudo-quadrilateral and $f$ is an inflection face exactly in the vertex stars of its corners;
\item $f$ is a pseudo-triangle and $f$ is an inflection face exactly in the vertex stars of its three corners and in one further vertex;
\item $f$ is a pseudo-triangle and $f$ is an inflection face exactly in the vertex stars of two of its corners.
\end{enumerate}
\end{proposition}

\begin{remark}
In contrast to Proposition~\ref{prop:face_positive}, we cannot claim anymore that the Gauss images of vertex stars do not intersect. But since just transversal intersections are possible, it can only happen that two Gauss images of stars of vertices of $f$ intersect if they enclose a region on the sphere that is not covered by any of the Gauss images of vertex stars around $f$. But this region cannot be covered by Gauss images of vertex stars if we assume that our assumptions on $f$ are valid for all faces under consideration. In this sense, the Gauss images of vertex stars around $f$ generically do not intersect. 
\end{remark}

\begin{definition}
Let $f$ be a face of $P$ and consider the conditions of Proposition~\ref{prop:face_negative}. In addition, we assume that $f$ is star-shaped with respect to an interior point $A$. $A$ is said to be a \textit{point of contact} of the \textit{discrete tangent plane} given by the plane through $f$. \textit{Discrete asymptotic directions} are given by the line segments connecting $A$ with the vertices of $f$ where $f$ is an inflection face. In the case that there are just two such vertices, the line segment connecting $A$ with the other corner of $f$ counts as two discrete asymptotic directions.
\end{definition}

\begin{remark}
In analogy to the smooth case, discrete asymptotic directions are defined in the point of contact. But whereas opposite asymptotic directions in the smooth case are collinear, this is generally not the case for polyhedral surfaces.
\end{remark}


\subsection{Region of mixed discrete Gaussian curvature}\label{sec:parabolic}

We now consider a face $f_1$ all of whose vertex stars have Gauss images without self-intersections, but the discrete Gaussian curvature does not have the same sign for all vertices. Then, if we go along the boundary of $f_1$ in counterclockwise direction, the Gauss images of positively curved vertex stars will arrange in counterclockwise order around $f_1$ and the Gauss images of negatively curved vertex stars in clockwise order. In particular, the Gauss images will overlap as in Fig.~\ref{fig:parabolic_line_smooth}.

\begin{figure}[!ht]
	\centerline{
		\begin{overpic}[height=0.25\textwidth]{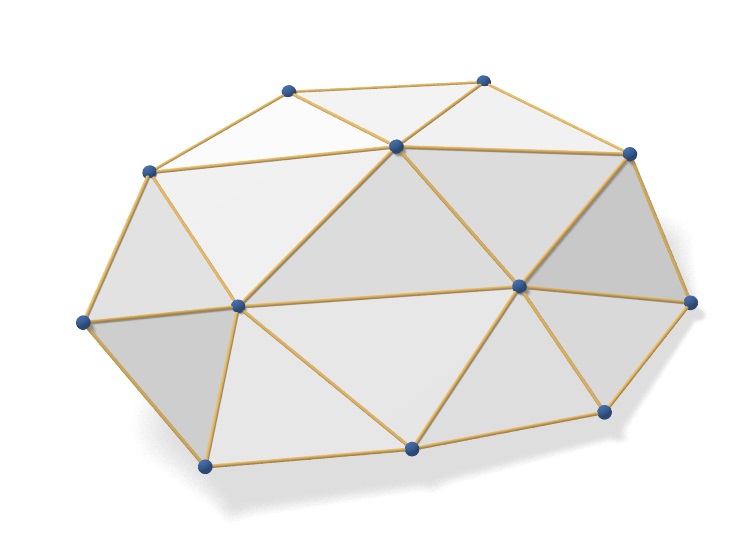}
	\put(49,41){\contour{white}{$f_1$}}
	\put(65,45){\contour{white}{$f_2$}}
	\put(63,58){\contour{white}{$f_3$}}
	\put(48,60){\contour{white}{$f_4$}}
	\put(34,56){\contour{white}{$f_5$}}
	\put(33,45){\contour{white}{$f_6$}}
	\put(18,37){\contour{white}{$f_7$}}
	\put(21,25){\contour{white}{$f_8$}}
	\put(35,18){\contour{white}{$f_9$}}
	\put(49,27){\contour{white}{$f_{10}$}}
	\put(67,22){\contour{white}{$f_{11}$}}
	\put(77,29){\contour{white}{$f_{12}$}}
	\put(79,41){\contour{white}{$f_{13}$}}
		\end{overpic}
		\relax\\
		\begin{overpic}[height=0.3\textwidth]{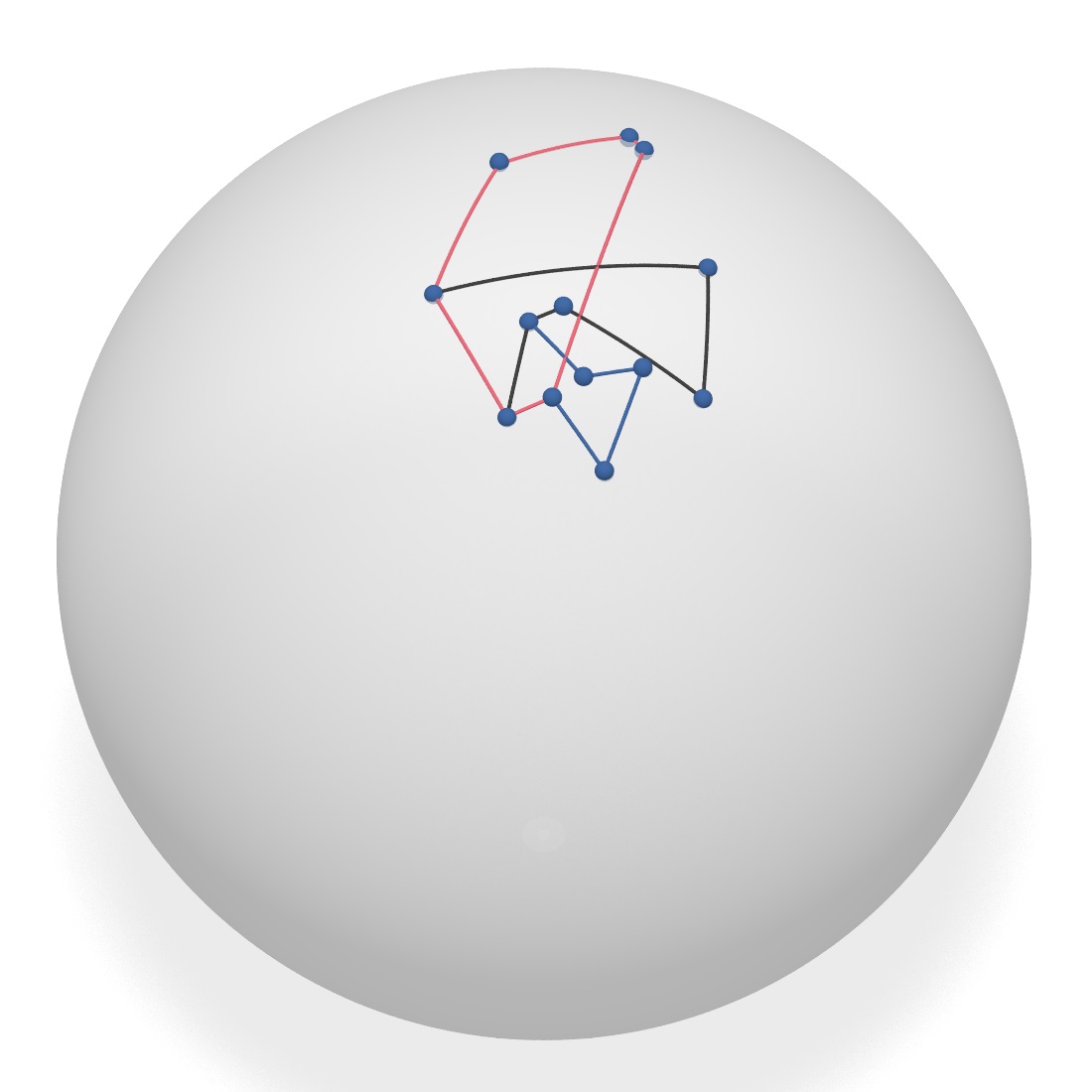}
			\cput{45}{58}{\scalebox{0.75}{$\vec{n}_{1}$}}
			\cput{52}{59}{\scalebox{0.75}{$\vec{n}_{2}$}}
			\cput{63}{85}{\scalebox{0.75}{$\vec{n}_{3}$}}
			\cput{58}{89}{\scalebox{0.75}{$\vec{n}_{4}$}}
			\cput{41}{84}{\scalebox{0.75}{$\vec{n}_{5}$}}
			\cput{35}{72}{\scalebox{0.75}{$\vec{n}_{6}$}}
			\cput{69}{74}{\scalebox{0.75}{$\vec{n}_{7}$}}
			\cput{69}{62}{\scalebox{0.75}{$\vec{n}_{8}$}}	
			\cput{52}{73}{\scalebox{0.75}{$\vec{n}_{9}$}}	
			\cput{43}{69}{\scalebox{0.75}{$\vec{n}_{10}$}}		
			\cput{54}{67}{\scalebox{0.75}{$\vec{n}_{11}$}}	
			\cput{63}{68}{\scalebox{0.75}{$\vec{n}_{12}$}}	
			\cput{57}{53}{\scalebox{0.75}{$\vec{n}_{13}$}}		
		\end{overpic}
	}
	\caption{$f_1$ with change of sign of discrete Gaussian curvature, overlapping in Gauss image: $f_9$ and $f_{10}$ are not inflection faces in the left vertex star, $f_2$ and $f_{11}$ are not inflection faces in the right vertex star, and there are no inflection faces in the upper vertex star}\label{fig:parabolic_line_smooth}
\end{figure}
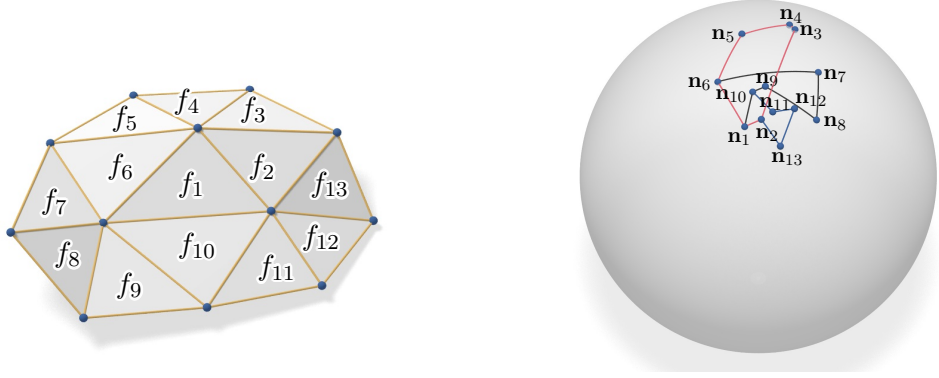

We will now develop criteria for smoothness. First, the number of changes between vertices with positive and negative curvature should be minimal, that is, two. If we think of a parabolic curve separating the regions of positive and negative Gaussian curvature, that means that $f_1$ should contain just one connected segment of a parabolic curve. In the generic smooth setting, parabolic curves do not become arbitrarily close or intersect each other.

Next, the contributions of positive and negative discrete Gaussian curvature should be balanced, i.e., the sum of oriented interior angles of Gauss images at ${\vec{n}}_{1}$ should be zero. Furthermore, the single contributions should not be too large, so we demand that the corresponding angle sums of Gauss images of either positively or negatively curved vertex stars shall be less than $2\pi$. This means that neither the Gauss images of the positively curved vertex stars nor the Gauss images of the negatively curved vertices surround ${\vec{n}}_{1}$ once. The corresponding behavior is observed along a parabolic curve on a generic smooth surface. Actually, we will slightly strengthen this requirement later to obtain a projectively invariant condition.

To easier compare with our discussion in the previous section, we count the interior angle at ${\vec{n}_{1}}$ of a Gauss image of a positively curved vertex star negative and the corresponding angle for a negatively curved vertex star positive. If the interior angle at a vertex $\vec{v}$ of $f_1$ is $\alpha_{\vec{v}}$, then by Theorem~\ref{th:shape} the corresponding interior angle of the Gauss image at ${\vec{n}_{1}}$ is $\alpha_{\vec{v}}-\pi$ if $K({\vec{v}})>0$ (and in particular $\alpha_{\vec{v}}<\pi$) and it is given by one of the angles given in the list in Section~\ref{sec:negative} if $K({\vec{v}})<0$.

Using the same notation of $c_i$ as in the previous Section~\ref{sec:negative} and denoting by $n_+,n_-$ the number of vertices of $f_1$ with positive and negative Gaussian curvature, respectively, our condition for balancedness translates to
\begin{align*}
0&=\sum\limits_{{\vec{v}}\sim f_1} \alpha_{\vec{v}} - n_+\pi + c_1\pi - c_3\pi=(n_++n_--2-n_++c_1-c_3)\pi\\
\Rightarrow n_--2&=c_3-c_1.
\end{align*}
However, $c_1+c_2+c_3+c_4=n_-$. Thus, $c_1$ can be only zero or one.

\subsubsection{One vertex of the first type}\label{sec:c11}

Let us start with the case $c_1=1$. Then, $c_3=n_--1$ and $c_2=c_4=0$. It follows that $f_1$ is never an inflection face at one of its vertices. Furthermore, all vertices of negative discrete Gaussian curvature but one have a reflex angle. For the remaining vertex $\vec{v}_-$, there are two different cases: Either $\vec{v}_-$ is adjacent to a vertex of positive discrete Gaussian curvature or both its neighbors have negative curvature, see Fig.~\ref{fig:mixed1}. Due to our assumption that the sum of interior angles $\pi-\alpha_{\vec{v}}$ of the Gauss images of positively curved vertex stars at ${\vec{n}}_{1}$ is less than $2\pi$, the turning angle of the polyline defined by all edges incident to the corresponding vertices is less than $2\pi$. As a result, $f_1$ is a pseudo-$(n_++1)$-gon with either at most one or exactly two non-trivial pseudo-edges, see the left and the right picture of Fig.~\ref{fig:mixed1}, respectively. Its convex hull contains $\vec{v}_-$.

\begin{figure}[htbp]
   \centering
    \subfloat{
	\begin{overpic}[height=0.3\textwidth]{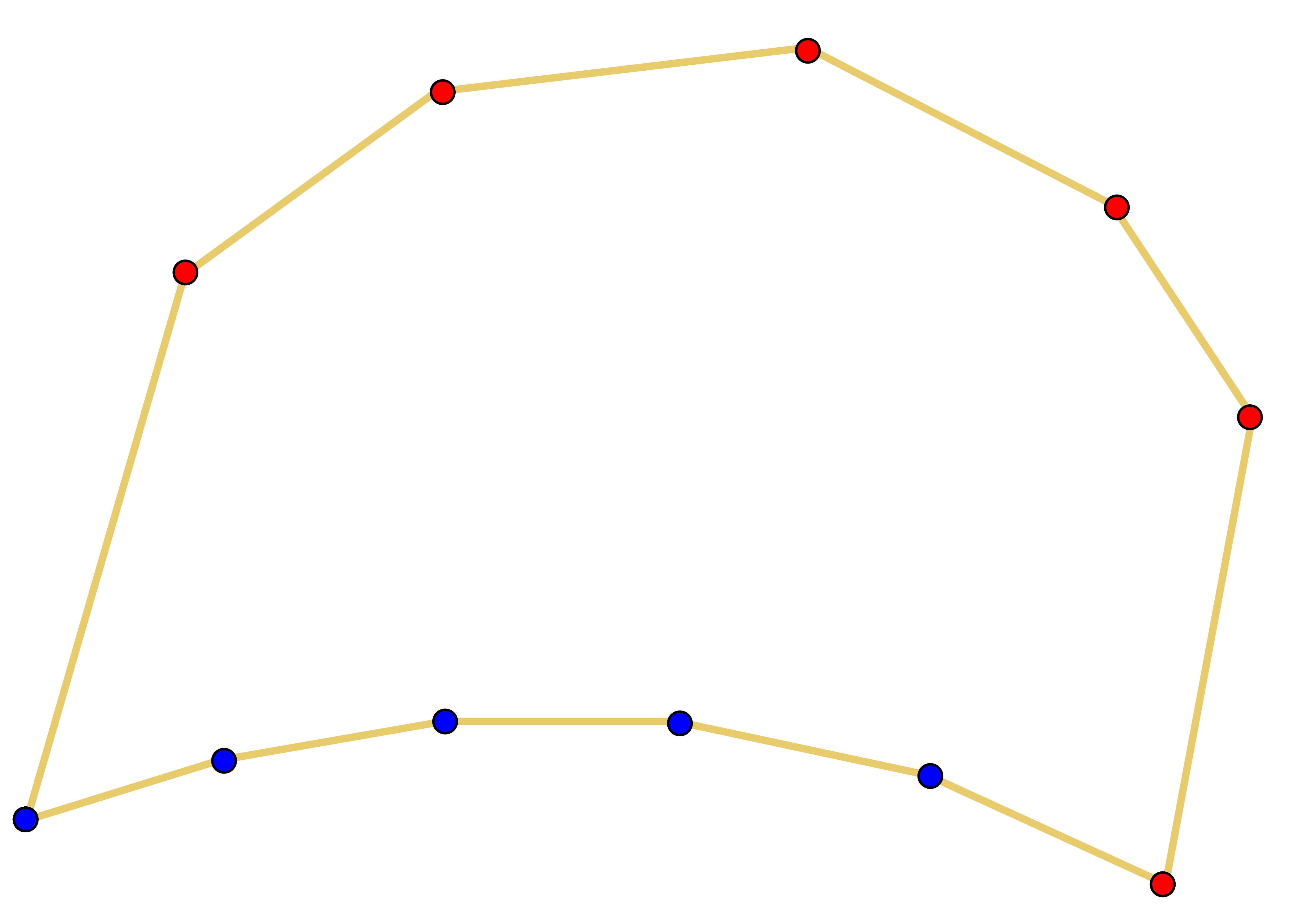}
		\cput{5}{3}{$\vec{v}_-$}
	\end{overpic}	
}
		\qquad
		\subfloat{
			\begin{overpic}[height=0.3\textwidth]{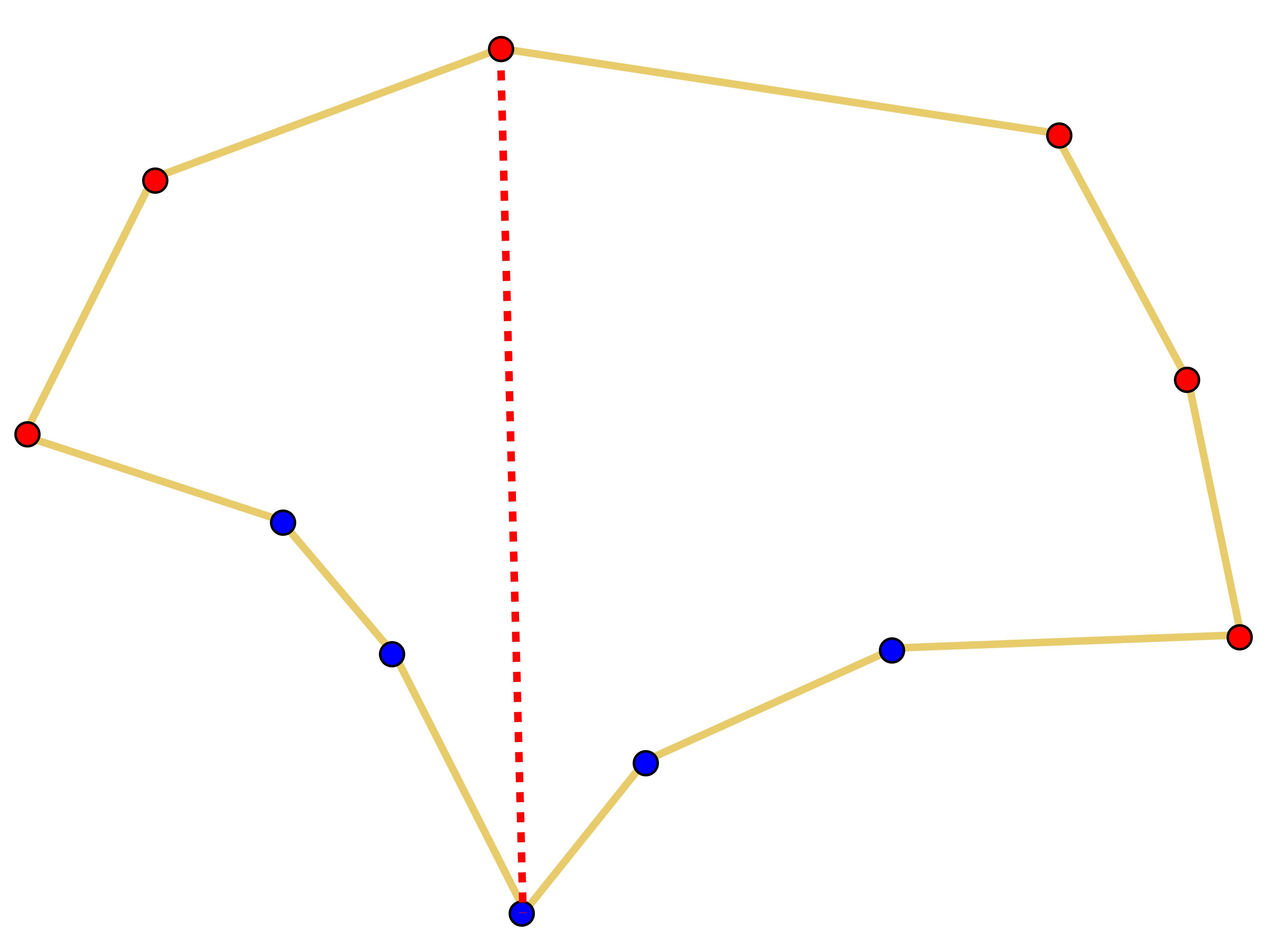}
			\cput{46}{0}{$\vec{v}_-$}	
			\end{overpic}}
   \caption{Two cases for the shape of a face with mixed discrete Gaussian curvature in the case of one vertex $\vec{v}_-$ of negative discrete Gaussian curvature that is a corner and where the face is not an inflection face. Vertices of positive discrete Gaussian curvature are colored red, the others blue. The dotted line indicates how the right figure can be decomposed into two parts of the left type}
   \label{fig:mixed1}
\end{figure}

If we are in the case that $\vec{v}_-$ is not adjacent to a vertex of $f_1$ of positive curvature, then we can divide $f_1$ along $\vec{v}_-$ into two faces, as depicted by the dotted line in the right picture of Fig.~\ref{fig:mixed1}. For this, it may be necessary to add an additional vertex of zero discrete Gaussian curvature on an edge. The resulting two parts now both correspond to the case that $\vec{v}_-$ is adjacent to a vertex of positive curvature.

\subsubsection{No vertices of the first or fourth type}\label{sec:c10a}

We now come to the case $c_1=0$. Then, $c_3=n_--2$ and $c_2+c_4=2$. Due to our assumption that the sum of the interior angles of the Gauss images of negatively curved vertex stars at ${\vec{n}}_{1}$ is less than $2\pi$ we have $c_4 \leq 1$. We start with the case $c_4=0$ and $c_2=2$.

It follows that $f_1$ is an inflection face at two of its vertices $\vec{v}_-$, $\vec{v}'_-$ where the interior angle is less than $\pi$. All other vertices of negative discrete Gaussian curvature have a reflex angle. We now have three options for the positions of $\vec{v}_-$, $\vec{v}'_-$: Both, one, or none of them are adjacent to vertices of positive discrete Gaussian curvature, see Fig.~\ref{fig:mixed2}. As in Section~\ref{sec:c11}, the turning angle of the polyline defined by all edges incident to the corresponding vertices is less than $2\pi$. As a result, $f_1$ is either a pseudo-$(n_++1)$-gon with at most one non-trivial pseudo-edge as before (but it may be more sickle-shaped), a pseudo-$(n_++2)$-gon with two non-trivial pseudo-edges, or a pseudo-$(n_++2)$-gon with three non-trivial pseudo-edges or two non-trivial pseudo-edges connected by a single edge, see Fig.~\ref{fig:mixed2}. Its convex hull contains at least one of $\vec{v}_-$ and $\vec{v}'_-$.

\begin{figure}[htbp]
   \centering
    \subfloat{
    	\begin{overpic}[height=0.18\textwidth]{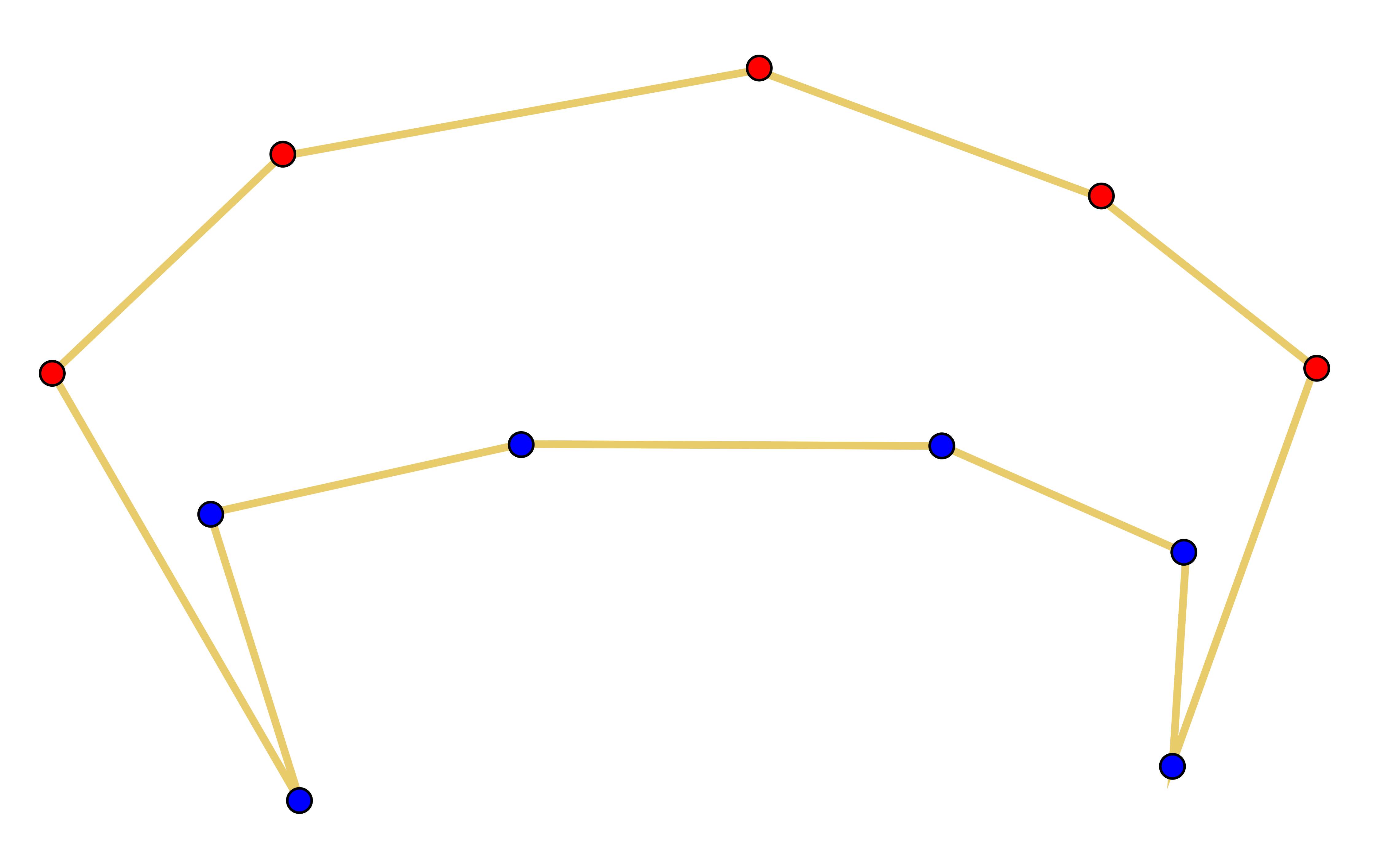}
    		\cput{20}{0}{$\vec{v}_-$} 
    	   	\cput{86}{0}{$\vec{v}'_-$}		
    	\end{overpic}
    	}
		\qquad
		\subfloat{
    	\begin{overpic}[height=0.18\textwidth]{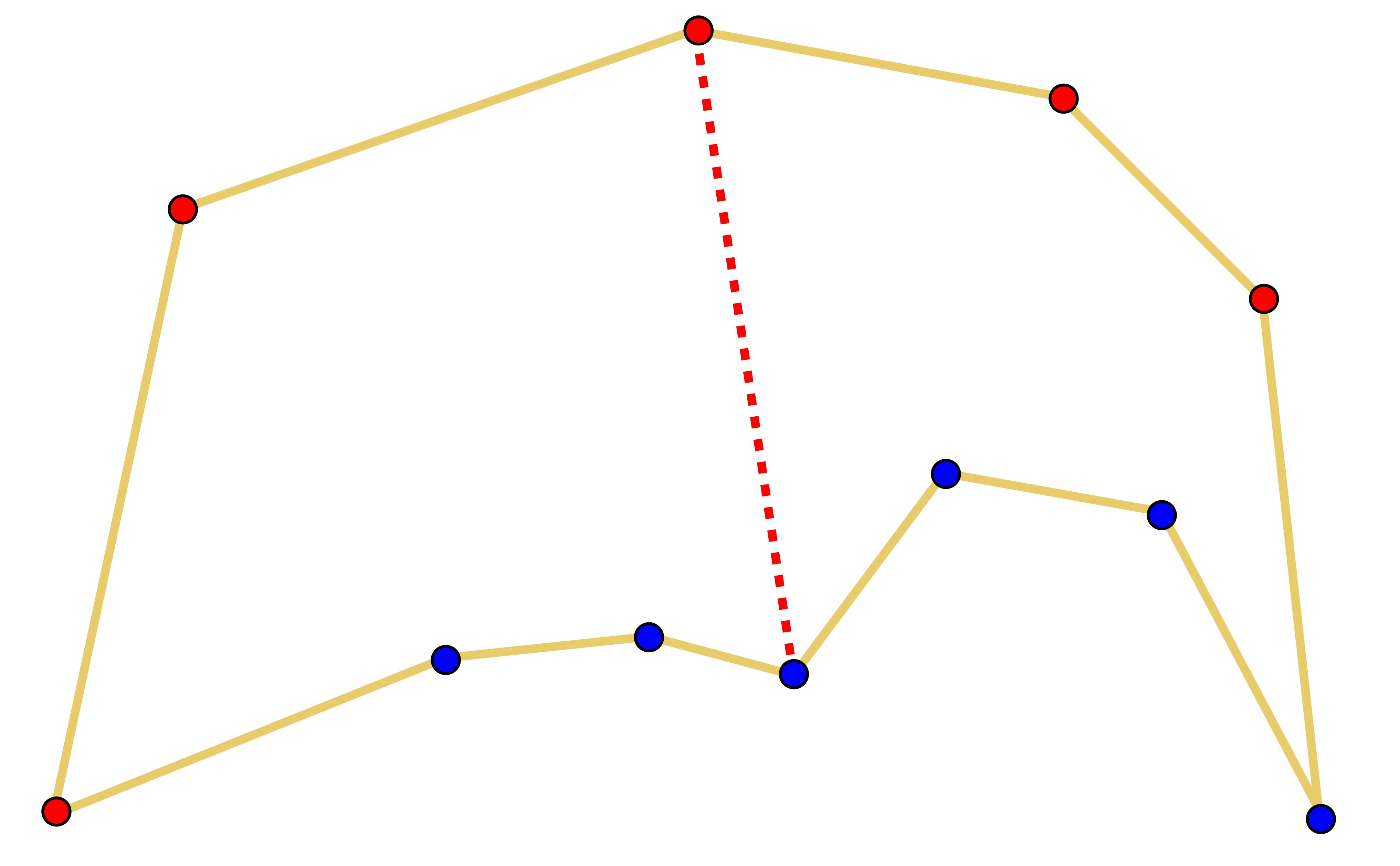}
    		\cput{60}{7}{$\vec{v}_-$}
    			\cput{90}{0}{$\vec{v}'_-$}	
    	\end{overpic}
		}
		\qquad
		\subfloat{
    	\begin{overpic}[height=0.18\textwidth]{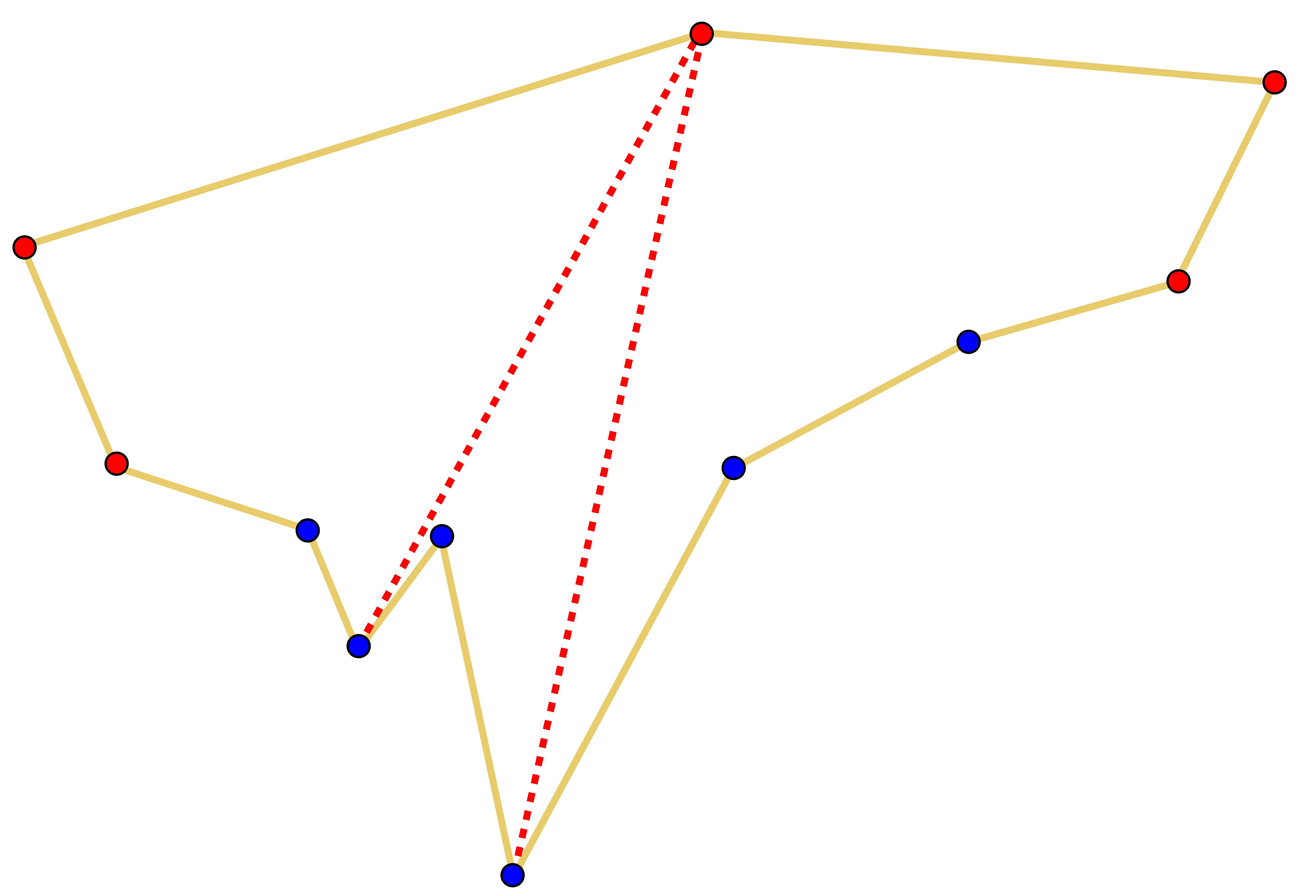}
    		\cput{30}{13}{$\vec{v}_-$}
    		\cput{48}{0}{$\vec{v}'_-$}
    	\end{overpic}
		}
   \caption{Three cases for the shape of a face with mixed discrete Gaussian curvature in the case of two vertices $\vec{v}_-$, $\vec{v}'_-$ of negative discrete Gaussian curvature that are corners of the face and where the face is an inflection face. Vertices of positive discrete Gaussian curvature are colored red, the others blue. The dotted lines indicate how the middle and the right figure can be decomposed into parts of the left type and of the type described in Fig.~\ref{fig:mixed1}}
   \label{fig:mixed2}
\end{figure}

In the case that $\vec{v}_-$ or $\vec{v}'_-$ are not adjacent to a vertex of $f_1$ of positive curvature, we divide $f_1$ along one or possibly two of the vertices $\vec{v}_-, {\vec{v}}'_-$ into two or three faces, as depicted by the dotted lines in the middle and the right picture of Fig.~\ref{fig:mixed2}. Again, it may be necessary to add an additional vertex of zero discrete Gaussian curvature on an edge. In the resulting parts, both vertices are adjacent to vertices of positive curvature, so each part corresponds to either the left picture in Fig.~\ref{fig:mixed1} or in Fig.~\ref{fig:mixed2}. Note that only one of the two parts at $\vec{v}_-$ or $\vec{v}'_-$ inherits the inflection-property in such a way that the number of inflections around each part is still even.

\subsubsection{No vertex of the first type, one of the second and one of the fourth type}\label{sec:c10b}

Within the case $c_1=0$, we are left with the case $c_2=c_4=1$. Then, $f_1$ is an inflection face at two of its vertices. At one vertex $\vec{v}_-$, the interior angle is less than $\pi$, at the other $\vec{v}'_-$ it is greater than $\pi$. All other vertices of negative discrete Gaussian curvature have a reflex angle. Thus, we obtain the shapes of Section~\ref{sec:c11}. In contrast to our discussion in that section, $f_1$ is an inflection face $\vec{v}_-$ and at one further vertex, see Fig.~\ref{fig:mixed3}.

\begin{figure}[htbp]
   \centering
    \subfloat{
    	    	\begin{overpic}[height=0.25\textwidth]{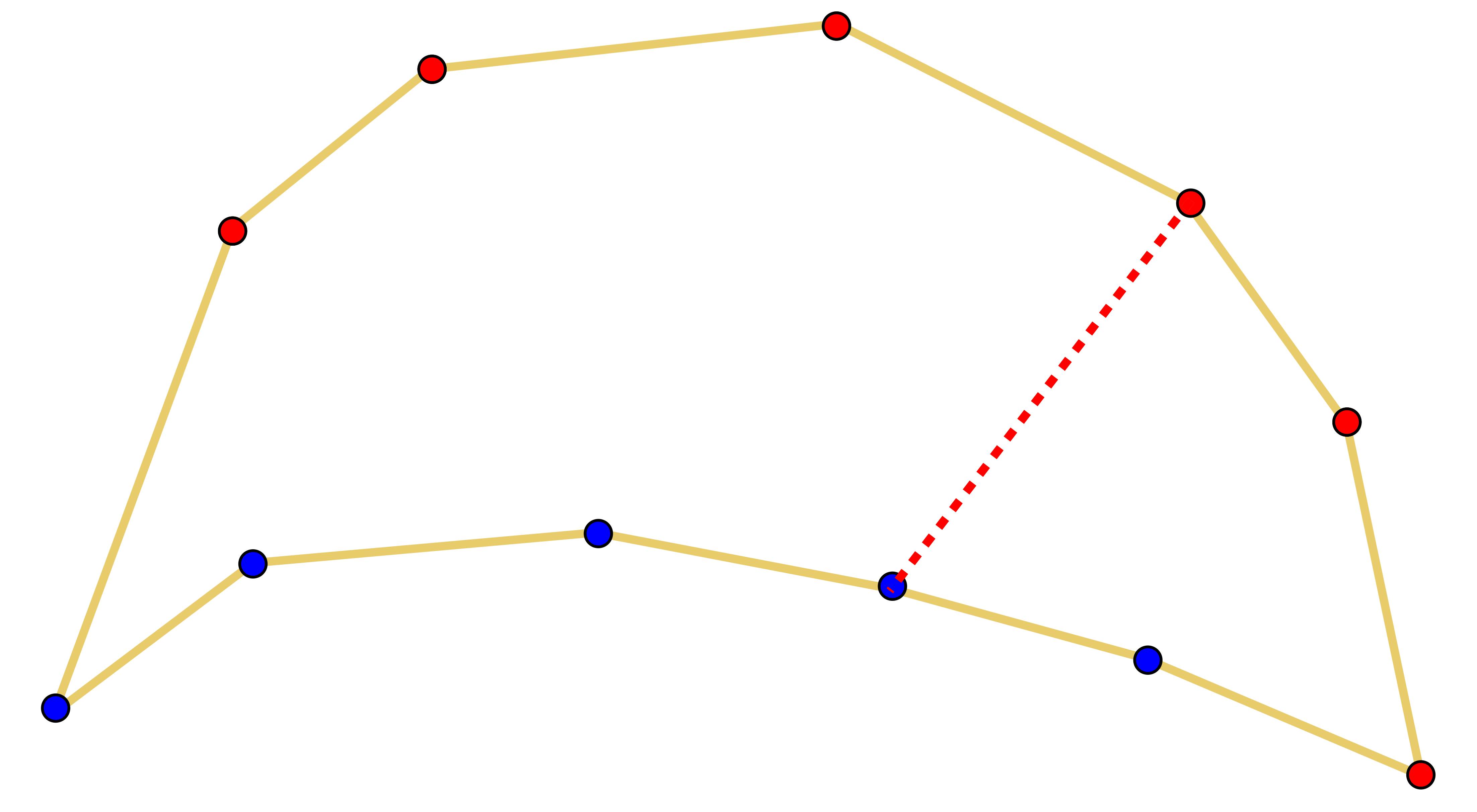}
    		\cput{6}{2}{$\vec{v}_-$}
    		\cput{61}{10}{$\vec{v}'_-$}	
    	    	\end{overpic}
    }
		\qquad
		\subfloat{
			    	\begin{overpic}[height=0.25\textwidth]{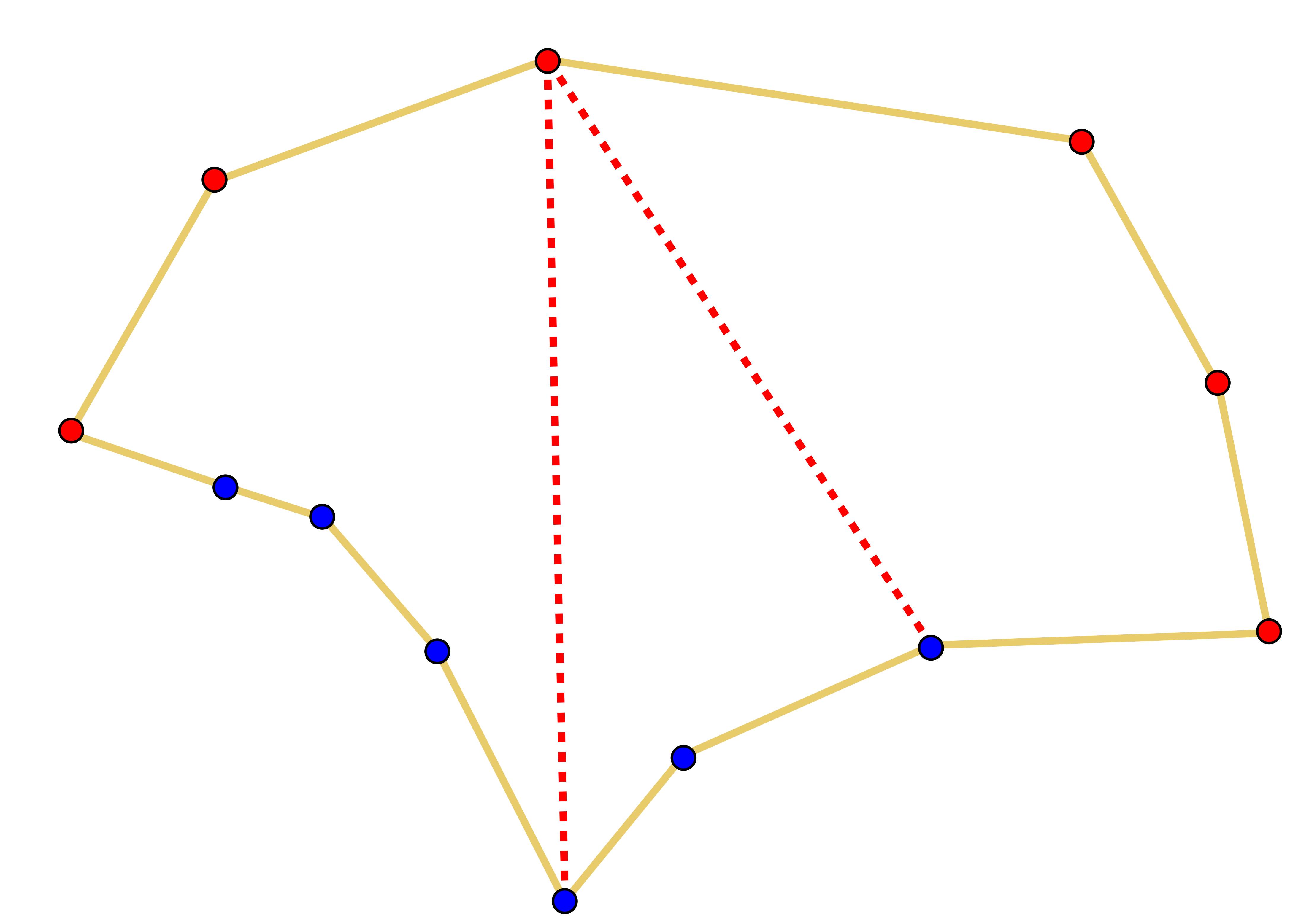}
			    		\cput{50}{0}{$\vec{v}_-$}
			    		\cput{74}{15}{$\vec{v}'_-$}	
			    	\end{overpic}
		}
   \caption{Two cases for the shape of a face with mixed discrete Gaussian curvature in the case of vertices $\vec{v}_-$ and $\vec{v}'_-$ of negative discrete Gaussian curvature where the face is an inflection face such that $\vec{v}_-$ is a corner and $\vec{v}'_-$ is not. Vertices of positive discrete Gaussian curvature are marked red, the others blue. The dotted line indicates how the figures can be decomposed into parts described in Figs.~\ref{fig:mixed1} and~\ref{fig:mixed2}}
   \label{fig:mixed3}
\end{figure}

As before, we can divide the faces along $\vec{v}'_-$ and possibly $\vec{v}_-$ into two or three parts, as depicted by the dotted lines in Fig.~\ref{fig:mixed3}. We do it in such a way that the reflex angle splits into two convex angles. Again, it may be necessary to add an additional vertex of zero discrete Gaussian curvature on an edge. Assigning the inflection-property of each of the vertices to just one of the two adjacent faces in such a way that the number of inflections around faces is even leads to parts corresponding to the left pictures in Fig.~\ref{fig:mixed1} and Fig.~\ref{fig:mixed2}. 

\subsubsection{Summary and discrete parabolic curves}\label{sec:mixedsummary}

The following proposition summarizes our discussion in the previous sections.

\begin{proposition}\label{prop:face_mixed}
Let $f$ be a face of $P$. Assume that $K({\vec{v}})\neq0$ and that the Gauss image of the star of $\vec{v}$ has no self-intersections for all $\vec{v} \sim f$. In addition, we assume that the sign of discrete Gaussian curvature changes along exactly two edges, that the interior angles of the Gauss images of the corresponding vertex stars add up to $0$ at ${\vec{n}}_f$ if they are oriented according to the sign of discrete Gaussian curvature, and that the sum of interior angles that correspond to vertices of positive curvature (or of negative curvature) is less than $2\pi$.

\begin{figure}[htbp]
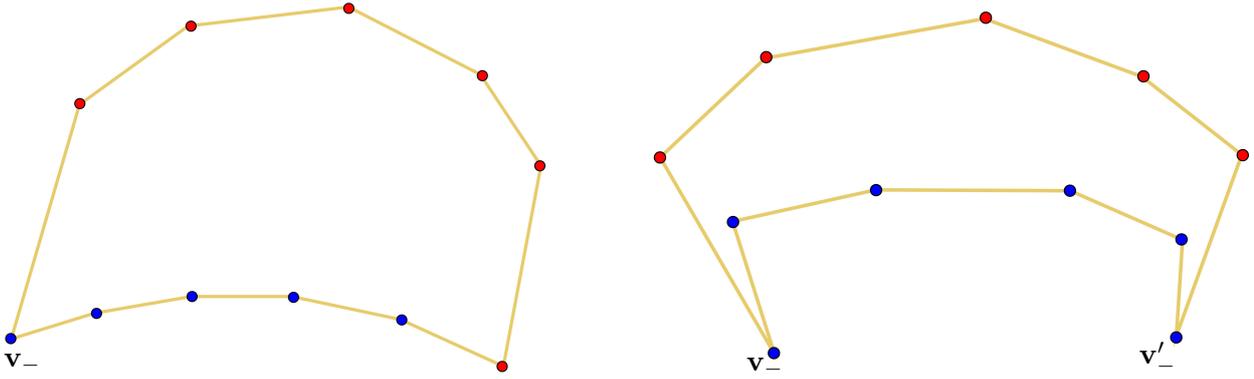

   \centering
    \subfloat{
				    	\begin{overpic}[height=0.3\textwidth]{mixed_1a_1}
				    	\cput{4}{3}{$\vec{v}_-$}	
				    	\end{overpic}
    }
		\qquad
		\subfloat{
				    	\begin{overpic}[height=0.3\textwidth]{mixed_2a_1}
    		\cput{20}{2}{$\vec{v}_-$}
    		\cput{81}{3}{$\vec{v}'_-$}
				    	\end{overpic}
		}
   \caption{Shapes of (parts of) faces of mixed discrete Gaussian curvature. The face on the left has no vertices where it is an inflection face, the face on the right is an inflection face exactly at $\vec{v}_-$ and $\vec{v}'_-$}
   \label{fig:mixed4}
\end{figure}

Then, by adding additional vertices of zero discrete Gaussian curvature on edges if necessary, $f$ can be split into parts that are pseudo-$(n+1)$-gons with at most one non-trivial pseudo-edge, see Fig.~\ref{fig:mixed4}. Each of the new parts is an inflection face either at no vertex or at the two corners of the pseudo-edge. In the first case, one of the corners is of negative discrete Gaussian curvature, in the second case both corners are of negative curvature. The outer part of the pseudo-$(n+1)$-gon (with convex angles) consists of vertices of positive discrete Gaussian curvature, the inner part (with reflex angles) of vertices of negative discrete Gaussian curvature.
\end{proposition}

\begin{remark}
By dividing faces into individual parts and adding additional vertices of zero discrete Gaussian curvature on edges we are violating our conditions that faces that share an edge should not be coplanar, interior angles are not equal to $\pi$, and that the discrete Gaussian curvature should never be zero. In such a case, for example the concept of an inflection face is not well-defined. However, we could slightly move vertices and bend edges in such a way, that we are again fulfilling all our conditions, that the faces are inflection faces at the right place, and that the discrete Gaussian curvature of additional vertices will be positive. For such a small variation, our assessments of smoothness will not be affected.
\end{remark}

So far, we just demanded that the turning angle of the polyline defined by all edges incident to positively curved vertices is less than $2\pi$. It could happen that the polyline winds inside its own convex hull. We would like to have that knowing just the shape of $f$ around positively curved vertices does not tell us whether $f$ has negatively curved vertices or not. One option would be to demand that the positively curved vertices define a convex polygon that the polyline does not enter. This definition also has the advantage of being projectively invariant. But the faces could still be shaped as very thin sickles. Then, the parabolic curve separating positively from negatively curved vertices would lie inside this face and would therefore be strongly bended. Thinking about a fine discretization of a smooth surface, it is natural to ask for a straight discrete parabolic curve inside the face. This is possible if the convex hull of positively curved vertices does not contain any negatively curved vertex. We add this condition to our assessments of smoothness and will show in Section~\ref{sec:projective} that this condition is also projectively invariant.

\begin{remark}
All faces shown in Figs.~\ref{fig:mixed1}, \ref{fig:mixed2}, and~\ref{fig:mixed3} possess the property that the convex hull of positively curved vertices does not contain any vertex of negative discrete Gaussian curvature.
\end{remark}

\begin{figure}[htbp]
   \centering
   \subfloat{
   	\begin{overpic}[height=0.3\textwidth]{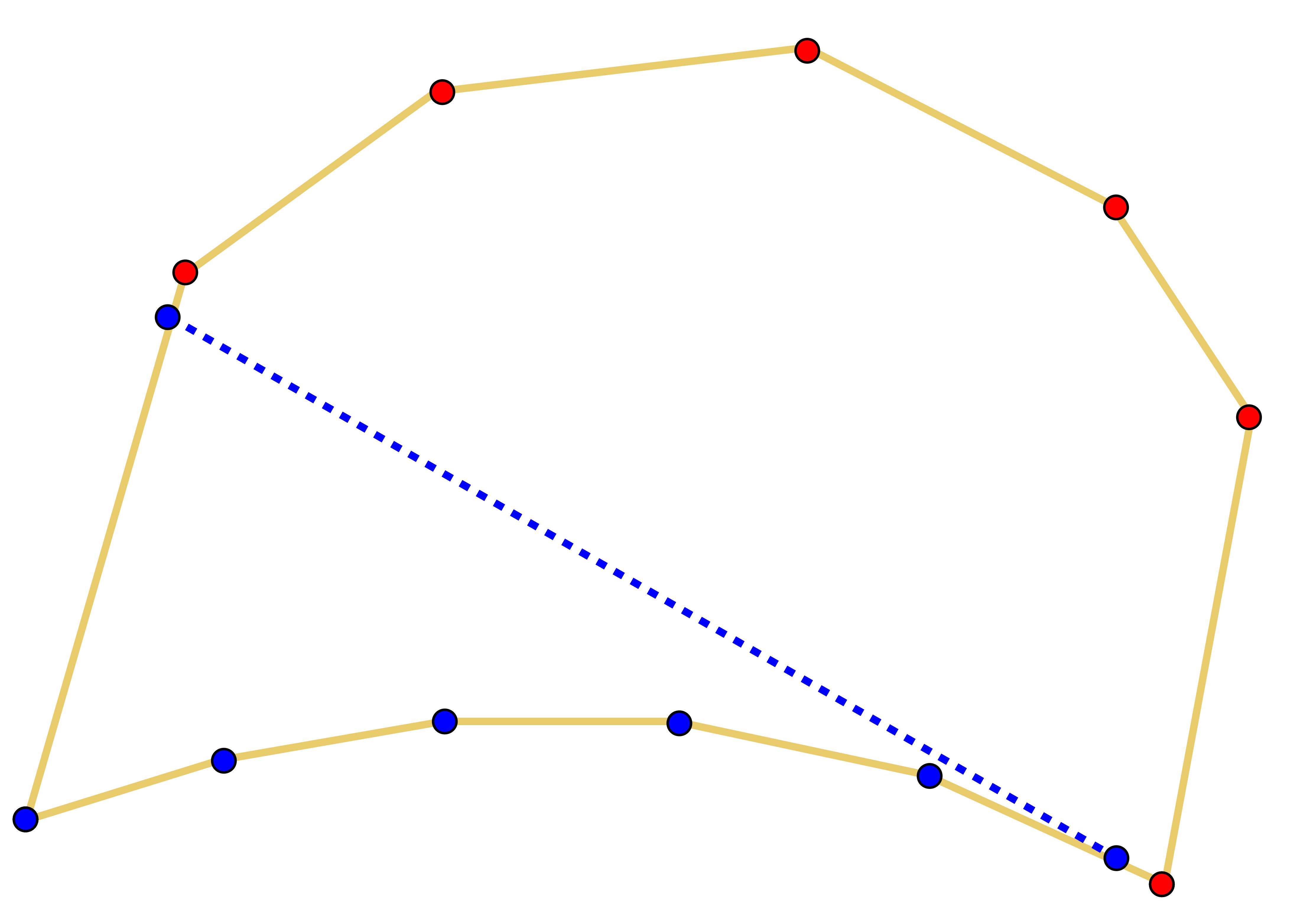}
   		\cput{4}{3}{$\vec{v}_-$}	
   	\end{overpic}
   }
   \qquad
   \subfloat{
   	\begin{overpic}[height=0.3\textwidth]{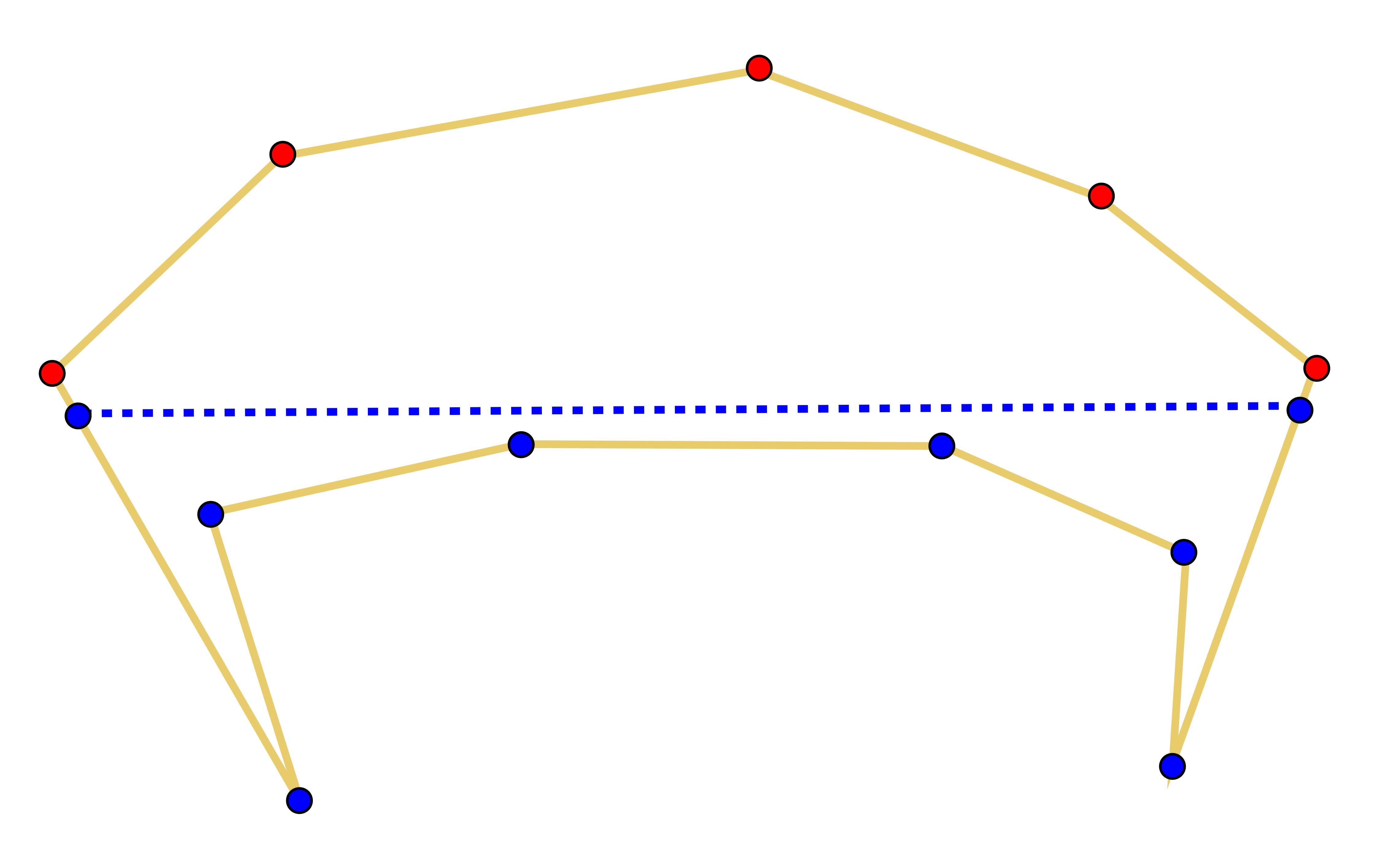}
   		\cput{20}{2}{$\vec{v}_-$}
   		\cput{81}{3}{$\vec{v}'_-$}
   	\end{overpic}
   }
   \caption{Discrete parabolic curve (dashed) in the two ``building blocks'' of faces of mixed curvature}
   \label{fig:mixed5}
\end{figure}

\begin{definition}
Let $f$ be a face of $P$ that has vertices of both positive and negative discrete Gaussian curvature. Furthermore, the sign of curvature changes at exactly two edges and the Gauss images of all vertex stars have no self-intersections. Also, the convex hull of all positively curved vertices shall not contain any negatively curved vertex.

Then, a line segment inside $f$ connecting two interior points of the two edges where the sign of curvature changes is said to be a \textit{discrete parabolic curve}.
\end{definition}

Let us shortly discuss the relation between parabolic curves and asymptotic directions. If the points of a parabolic curve on a smooth surface have a common tangent plane (e.g. on a torus), then the parabolic curve envelopes the asymptotic lines \cite{HCV52}. We know from our discussion in Section~\ref{sec:asymptotes} that the discrete asymptotic directions in all vertices of $f$ of negative discrete Gaussian curvature that are not corners of $f$ and where $f$ is not an inflection face lie in the double cone inside $f$ spanned by the two incident edges. Furthermore, discrete asymptotic directions in the corners of $f$, supposing that $f$ is an inflection face there, lie in the cone spanned by the adjacent edges. Hence, the discrete parabolic curve envelopes the discrete asymptotic directions in the two building blocks depicted in Fig.~\ref{fig:mixed5}.

If the parabolic curve on a smooth surface has a variable tangent plane, then the asymptotic lines have cusps along the parabolic curve \cite{HCV52}. Looking at the right picture in Fig.~\ref{fig:mixed1}, the discrete asymptotic directions along the two pseudo-edges come together at $\vec{v}_-$, so we actually have a cusp-like behavior. Being in analogy to the smooth case, the tangent plane should change there, which is modeled by the division of $f$ into two parts along the dotted line. The partitions shown in Figs.~\ref{fig:mixed2} and~\ref{fig:mixed3} can be motivated in a similar way.


\subsection{Smooth polyhedral surfaces}\label{sec:smooth}

Based on our previous discussion of smoothness around a face, we are now ready to give a definition of a smooth polyhedral surface.

\begin{definition}\label{def:smooth2}
An orientable polyhedral surface $P$ immersed into three-dimen\-sional Euclidean space is said to be \textit{smooth}, if the following conditions are satisfied:
\begin{enumerate}
\item $K({\vec{v}})\neq 0$ for all interior vertices $\vec{v}$.
\item For each interior vertex $\vec{v}$, the spherical polygon defined by the normals ${\vec{n}}_f$ of incident faces is star-shaped with respect to an interior point and contained in an open hemisphere.
\item For each face $f$ not at the boundary, the sign of discrete Gaussian curvature changes at either zero or two edges.
\item For each face $f$ not at the boundary, the interior angles of Gauss images of stars of vertices $\vec{v} \sim f$ oriented according to the sign of $K({\vec{v}})$ add up to $2\pi$ or $0$ depending on whether $K({\vec{v}})$ has the same sign for all $\vec{v} \sim f$ or not. In the first case, we demand in addition that $f$ is star-shaped; in the latter case, we require that the convex hull of vertices of positive discrete Gaussian curvature does not contain any vertex of negative curvature.
\end{enumerate}
\end{definition}

Note that the second condition is always fulfilled for vertices of positive discrete Gaussian curvature. Theorem~\ref{th:shape}~(i) to~(iii) describes which shapes Gauss images of vertex stars can have if the second condition is satisfied, Propositions~\ref{prop:face_positive}, \ref{prop:face_negative}, and \ref{prop:face_mixed} describe the shapes of faces $f$ of $P$ in the cases of positive, negative, and mixed discrete Gaussian curvature, respectively.

Discrete tangent planes at vertices are given by the planes orthogonal to vectors with respect to which the Gauss image of the vertex star is star-shaped, discrete normals by the poles of open hemispheres containing the Gauss image, and the points of contact of the face tangent planes are points with respect to which the face is star-shaped (assuming that the sign of discrete Gaussian curvature does not change). Discrete asymptotic directions were discussed in Sections~\ref{sec:asymptotes} and~\ref{sec:negative}, a discrete parabolic curve was defined in Section~\ref{sec:mixedsummary}.

\begin{remark}
It is tempting to demand in addition that each face is a transverse plane for its neighborhood, even though this would exclude surfaces such as a tetrahedron or a cube from being smooth. This requirement would have the advantage that transverse planes for vertex stars automatically exist and that we could choose a discrete normal as a normal vector of a discrete tangent plane. Furthermore, Theorem~\ref{th:dual} could be applied for any neighborhood of a face with vertices with the same sign of discrete Gaussian curvature, such that the dual pictures in Fig.~\ref{fig:negative_face} are indeed projective duals of neighborhoods of smooth polyhedral vertices. However, already the condition that the neighborhood of a face should have a transverse plane is not invariant under projective transformations, in contrast for the condition of a vertex star having a transverse plane. The reason for that difference is that a projective transformation around a vertex can be approximated by an affine transformation, but this is not any longer true around a face.
\end{remark}


\subsection{Comparison between a smooth surface and its discretizations}\label{sec:comparison}

Our discussion of the shapes of faces of a smooth polyhedral surface nicely explains the transformations of shapes that occurred in the examples generated in \cite{JTVWP15}. One example we want to take a closer look at is a Dupin cyclide and two smooth discretizations of it. One of the polyhedral surfaces in Fig.~\ref{fig:overlap} consists of triangles only, the other realizes the dual situation of having only vertices of valence three.

\begin{figure}[htbp]
   \centering
    \includegraphics[height=0.4\textwidth]{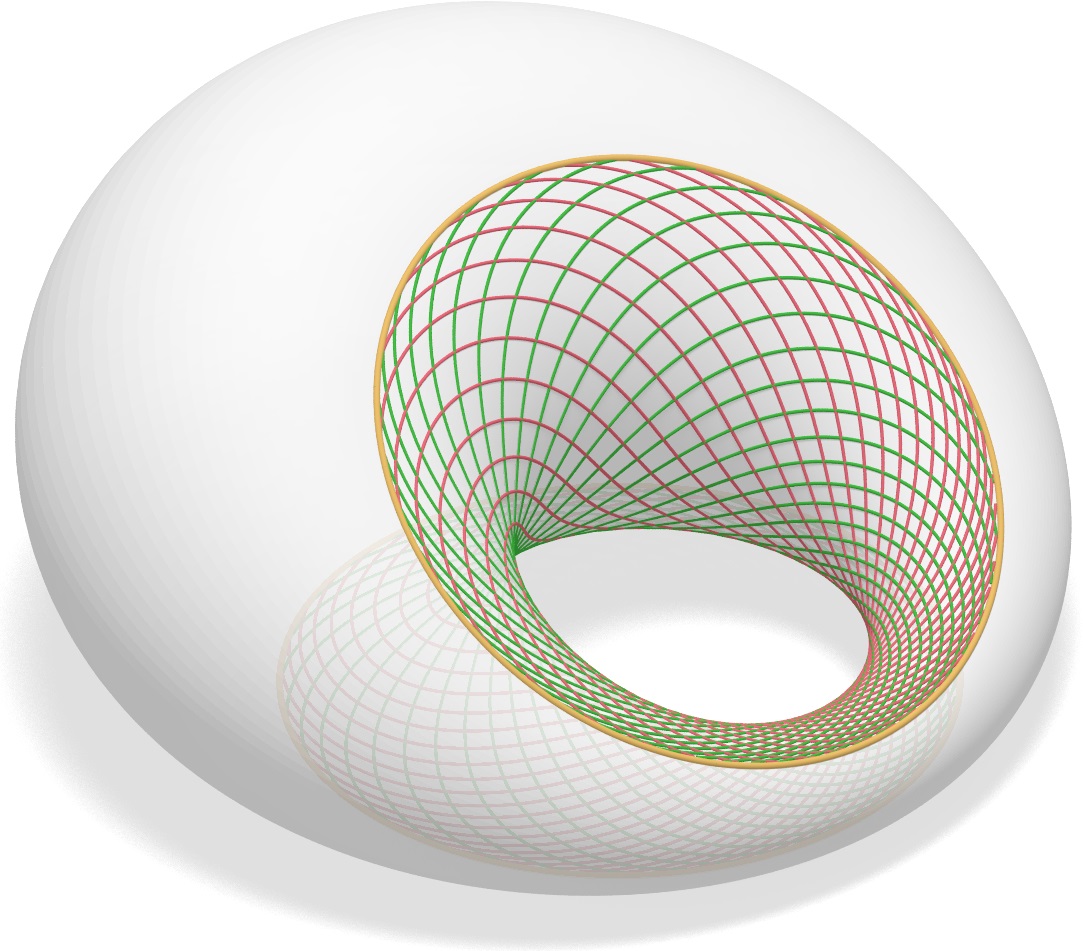}
   \caption{Dupin cyclide with parabolic curves (yellow) and asymptotic curves (red and green)}
   \label{fig:cyclide}
\end{figure}

The Dupin cyclide shown in Fig.~\ref{fig:cyclide} is generated by a spherical inversion of a torus. The cyclide is tangent to two planes along two circles that equal its parabolic curves. They separate the region of positive curvature from the region of negative curvature. In the latter domain, we illustrated two families of asymptotic curves.

\begin{figure}[htbp]
   \subfloat{
   	\includegraphics[height=0.4\textwidth]{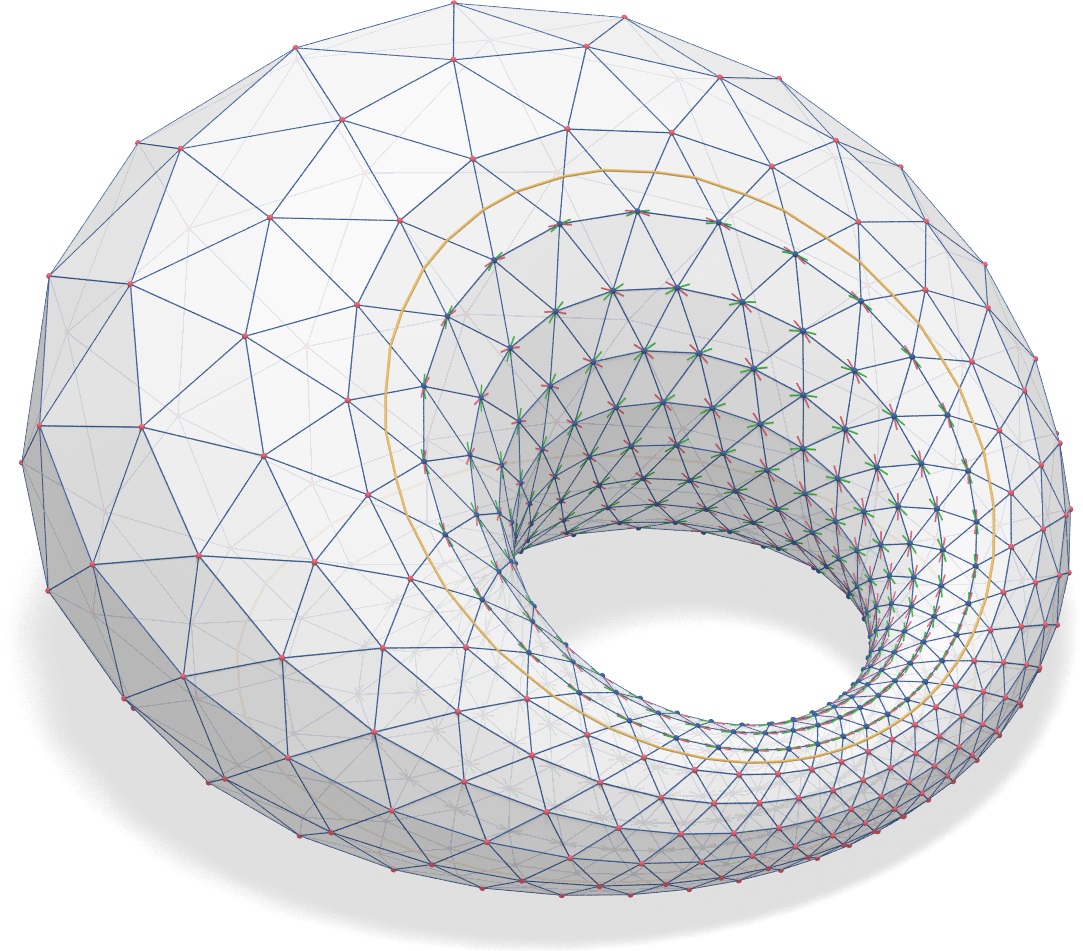}
   }
	\quad
   \subfloat{
   	\includegraphics[height=0.4\textwidth]{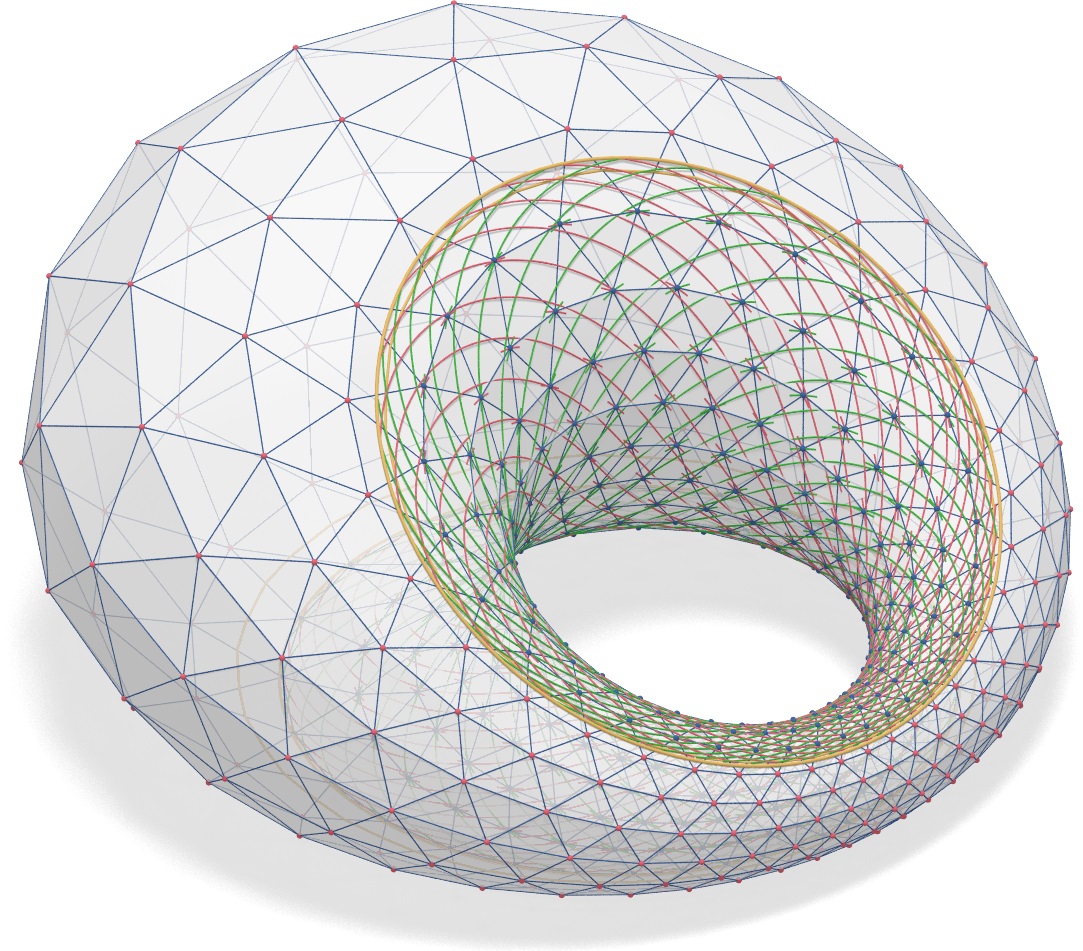}
   }\\
   \subfloat{
   	\includegraphics[height=0.4\textwidth]{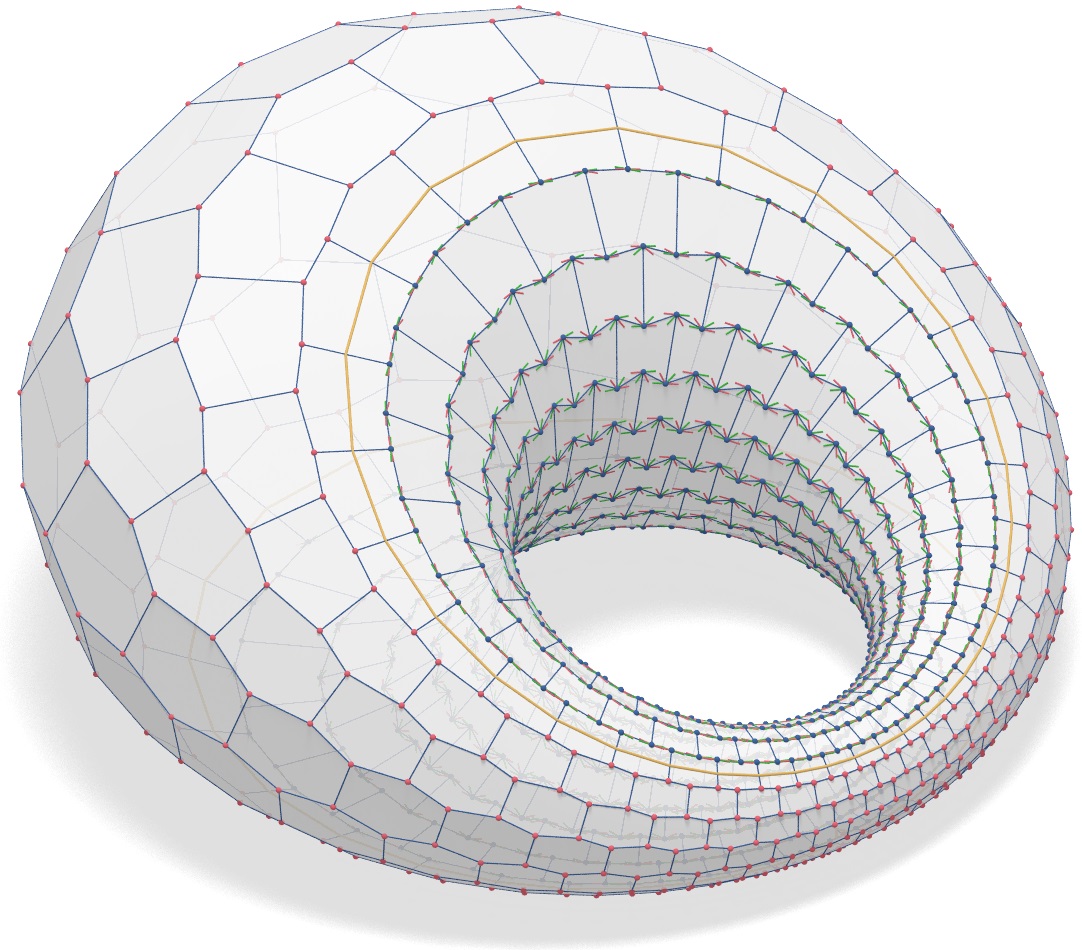}
   }
   \quad
   \subfloat{
   	\includegraphics[height=0.4\textwidth]{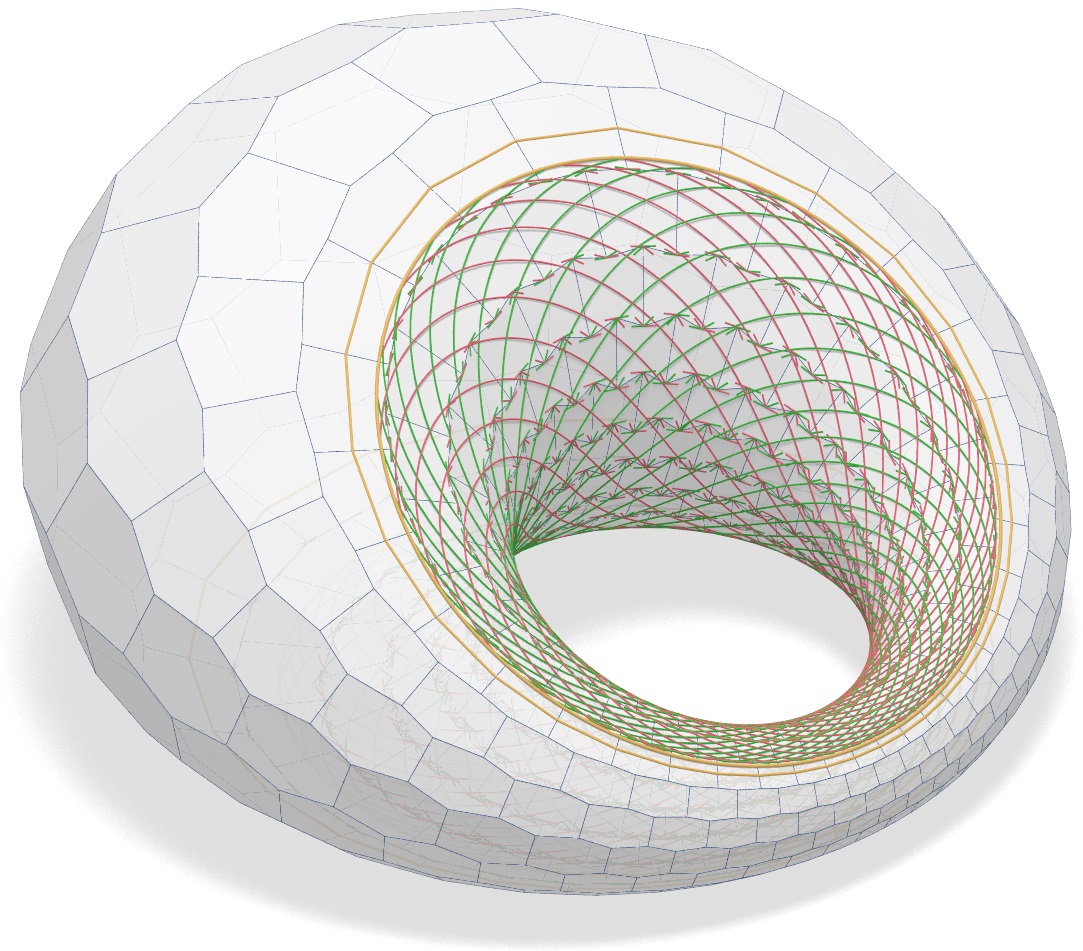}
   }
   \caption{Two smooth polyhedral surfaces that discretize a Dupin cyclide. One of them consists of triangles only, the other contains only vertices of valence three. The yellow polylines are discrete parabolic curves; at each vertex, four discrete asymptotic directions are drawn in with colors corresponding to the colors of the smooth asymptotic directions. The overlaps of the discrete and the smooth image show the alignment of the discrete objects with their smooth counterparts}
   \label{fig:overlap}
\end{figure}

In the case of the polyhedral surface consisting of hexagons, we can see how the shapes of the faces adapt to the sign of curvature exactly as we described in the previous sections. The hexagons are convex if the curvature is positive, they are pseudo-quadrilaterals if the curvature is negative, and along the discrete parabolic curve we recognize the shapes of Fig.~\ref{fig:mixed5}.

For both discrete cyclides, the discrete parabolic curves are close to the smooth parabolic curve. Having the freedom in locating the vertices of the discrete parabolic curve in mind, we could have aligned discrete and smooth parabolic curves even better.

When we come to the comparison of asymptotic directions, we see that the discrete asymptotic directions align very well with the asymptotic curves in the case of the triangulated surface. For the polyhedral surface with vertices of valence three, we are always in the situation of Fig.~\ref{fig:Dupin_triangle_degenerate}. That means, that two asymptotic directions are collinear. If we ignore them and consider for each vertex the other two discrete asymptotic directions only, we obtain a good accordance with the asymptotic directions of the smooth cyclide at these points. Dual to this case is the situation described in the lower right picture of Fig.~\ref{fig:negative_face} and that is exactly what we observe on the triangulated surface. If we had depicted the discrete asymptotic directions inside faces, the two corresponding to the line segments connecting to the vertices where the triangle is an inflection face would align with the smooth asymptotic curves, but the line segment counting as two discrete asymptotic directions would not. Again, it is reasonable to ignore these two directions.


\section{Projective transformations}\label{sec:projective}

In this last section, we want to study projective transformations of polyhedral surfaces and how they affect smoothness. When we compare our discussions for smoothness around a vertex and around a face, we observe several analogies. For example, the Gauss image of a smooth vertex star is either a convex spherical polygon or a spherical pseudo-$n$-gon ($n=3,4$) depending on the sign of discrete Gaussian curvature, and a face of a smooth polyhedral surface all whose vertices have the same sign of discrete Gaussian curvature is either a convex polygon or a pseudo-$n$-gon ($n=3,4$). Such properties hint at an invariance of our notion with respect to correlations in projective geometry. In particular, our notion of smoothness should be already invariant under projective transformations that do not map any point of the surface to infinity.

\begin{theorem}\label{th:projective}
Let $P$ be a smooth polyhedral surface immersed into three-di\-mensional Euclidean space that we consider as the space of finite points in three-dimensional projective space. Furthermore, let $\pi$ be a projective transformation (collineation) that does not map any point of $P$ to infinity. Then, $\pi(P)$ is a smooth polyhedral surface. Moreover, discrete tangent planes, discrete asymptotic directions, discrete parabolic curves, and points of contact are mapped to the corresponding objects of $\pi(P)$.
\end{theorem}
\begin{proof}
Since the faces incident to a vertex $\vec{v}$ of $P$ did not intersect in a neighborhood of $\vec{v}$, their images will not intersect in a neighborhood of $\pi({\vec{v}})$. In particular, $\pi(P)$ is immersed in three-dimensional Euclidean space.

Let us investigate how the Gauss image of the star of $\vec{v}$ changes. The normal vectors of incident faces are already defined by the edges emanating from $\vec{v}$. Thus, the change of the Gaussian image is encoded in the change of an infinitesimal neighborhood of the vertex. If we differentiate $\pi$ at a vertex $\vec{v}$ of $P$, we obtain an affine transformation $\rho$. For this, note that $\pi$ does not map points of $P$ to infinity.

$\rho$ can be seen as a concatenation of a bijective linear transformation and a translation. Since the translation does not affect the Gauss images, we can assume without loss of generality that $\rho$ is a linear transformation of non-vanishing determinant. Since the negative of the identity matrix defines a point reflection in the origin, which does not significantly change the surface besides reversing the orientation, we may assume that $\rho$ is orientation-preserving. Since the general linear group has exactly two connected components that can be distinguished by the determinant, there is a continuous curve $\gamma$ in this space connecting the identity with $\rho$, each point of $\gamma$ defining an affine transformation.

If we move along $\gamma$, also the image of the polyhedral surface as well as the new Gauss images will change continuously. In particular, if $g(\pi({\vec{v}}))$ has a self-intersection but $g({\vec{v}})$ has not, then somewhere along $\gamma$ there is an affine transformation that generates a Gauss image where three normals lie on a common great circle. The planes passing through the arcs in the Gauss image and the origin are orthogonal to the corresponding edge in the star of the vertex. If three normals lie on a common great circle, two edges possess a common orthogonal plane, i.e., they are collinear. But this cannot happen under a linear transformation of non-vanishing determinant. It follows that $g(\pi({\vec{v}}))$ is free of self-intersections. Furthermore, the sign of discrete Gaussian curvature does not change.

Since $P$ possesses a transverse plane at $\vec{v}$, there exist ${\vec{n}}' \in S^2$ such that a disk neighborhood of the star of $\vec{v}$ orthogonally projects to a plane orthogonal to ${\vec{n}}$ in a one-to-one way. We want to prove that $\rho({\vec{n}}')/|\rho({\vec{n}}')|$ does the same for $g(\pi({\vec{v}}))$.

The condition that a disk neighborhood of the star of $\vec{v}$ orthogonally projects to a plane orthogonal to ${\vec{n}}'$ in a one-to-one way is equivalent to the condition that such a neighborhood bijectively projects to any plane that is not parallel to ${\vec{n}}'$, where the direction of projection is determined by that vector. Equivalently, any line parallel to ${\vec{n}}'$ and close to $\vec{v}$ intersects the disk neighborhood in either exactly one interior point of a face or exactly one interior point of an edge or only in $\vec{v}$. The linear transformation $\rho$ preserves parallelity and maps a disk neighborhood of $\vec{v}$ to a disk neighborhood of $\rho({\vec{v}})$. It follows that the star of $\pi({\vec{v}})$ bijectively projects to any plane that is not parallel to $\rho({\vec{n}}')$. Due to the continuity argument of above, $\rho({\vec{n}}')/|\rho({\vec{n}}')|$ lies inside $g(\pi({\vec{v}}))$. In particular, the rescaled image $\rho({\vec{n}}')/|\rho({\vec{n}}')|$ is a discrete normal.

If the discrete Gaussian curvature is positive, the Gauss image is star-shaped by Proposition~\ref{prop:Dupin_positive}. If the discrete Gaussian curvature is negative, we can use Theorem~\ref{th:Dupin_negative} to equivalently describe star-shapedness by the condition that the intersection of a band of planes parallel to a given plane through the vertex intersect the star of the vertex in a discrete hyperbola. For the star of $\vec{v}$, such a plane is given by the tangent plane $E_0$. We want to argue that $\rho(E_0)$ defines a tangent plane through $\pi({\vec{v}})$.

The band of planes parallel to $E_0$ is mapped under $\rho$ to a band of planes parallel to $\rho(E_0)$. If we consider the curve $\gamma$ of linear transformations, then the intersections of the image planes with the star of the vertex will change continuously. A significant change in the shape of the intersection only happens if an edge in a polyline shrinks to a point or two adjacent edges become parallel. In the first case, the original face would not intersect the band of planes anymore, and in the second case, two faces would become coplanar. Both situations cannot occur for an affine transformation. Hence, the images of the discrete hyperbolas will still be discrete hyperbolas, meaning that $g(\pi({\vec{v}}))$ is star-shaped and that discrete tangent planes are mapped to discrete tangent planes. Also, discrete asymptotic directions at vertices are mapped to discrete asymptotic directions.

We have proven that the star of $\pi({\vec{v}})$ is smooth, so we are left with the situation around faces. We have already shown that the sign of discrete Gaussian curvature does not change, in particular, the number of edges of a face where the sign changes remains invariant. Clearly, star-shaped faces remain star-shaped: $f$ is star-shaped with respect to $A$ if and only if all line segments connecting $A$ with one of the vertices of $f$ lie completely in $f$. Since $\pi$ does not map any point of $f$ to infinity, $\pi(f)$ will be star-shaped with respect to $\pi(A)$. It follows that points of contact and discrete asymptotic directions inside faces are mapped to their counterparts in the projective image.

Next, we want to show that neither the property of an angle being greater or less than $\pi$ nor the property of a face being an inflection face for a vertex does change under the projective transformation $\pi$. For this, we will give different characterizations in terms of intersections.

In fact, an angle $\alpha_f$ of a face $f$ is less than $\pi$ if and only if there exists a line segment connecting the two edges inside $f$. The image of that line segment under $\pi$ will be again a line segment connecting the two image edges. Since $\pi$ does not map any point of $P$ to infinity, that line segment will be contained in $\pi(f)$.

A face $f$ is an inflection face for a vertex $\vec{v}\sim f$ if and only if there exist lines $l$ arbitrarily close to $\vec{v}$ that intersect $f$ and both its neighboring faces. The lines $\pi(l)$ will then intersect $\pi(f)$ and both its neighboring faces and are arbitrarily close to $\pi({\vec{v}})$.

In addition to the sign of discrete Gaussian curvature only the properties of a face $f$ having a angle greater or less than $\pi$ at $\vec{v}\sim f$ and being an inflection face or not were necessary to determine the interior angle of the Gauss image of the star of $\vec{v}$ at ${\vec{n}}_f$ by Lemma~\ref{lem:angle}. Since these properties do not change under the projective transformation $\pi$, it follows from our discussion in Section~\ref{sec:positive} that the sum of suitably oriented interior angles of Gauss images (which can be only a multiple of $2\pi$) does not change. In particular, if the sign of discrete Gaussian curvature is the same for all vertices of a face, then smoothness is preserved under $\pi$.

The only statement we are left to show is that in the case of different signs of discrete Gaussian curvature around $f$, the convex hull of all positively curved vertices again does not contain any vertex of negative discrete Gaussian curvature. If this was not the case, all line segments connecting a certain negatively curved vertex with all positively curved vertices in the image would lie inside the image face. Inverting $\pi$ would then yield to the contradiction that the original negatively curved vertex would have been contained in the convex hull. It is also easy to see that the discrete parabolic curve is mapped to a discrete parabolic curve of the image. So $\pi(P)$ is indeed a smooth polyhedral surface if $P$ was. 
\end{proof}

We conclude with an investigation of polyhedral surfaces $P^*$ which are projective duals to $P$. Numerical experiments show an interesting phenomenon, see Fig.~\ref{fig:dual}: A smooth polyhedral surface $P$ whose faces are star-shaped in an area of negative discrete Gaussian curvature has a Gauss image which looks like the dual polyhedral surface $P^*$ which again is a smooth negatively curved polyhedral surface with star-shaped faces. This observation can be made more precise:

\begin{figure}[htbp]
   \begin{center}
    \includegraphics[height=0.6\textwidth]{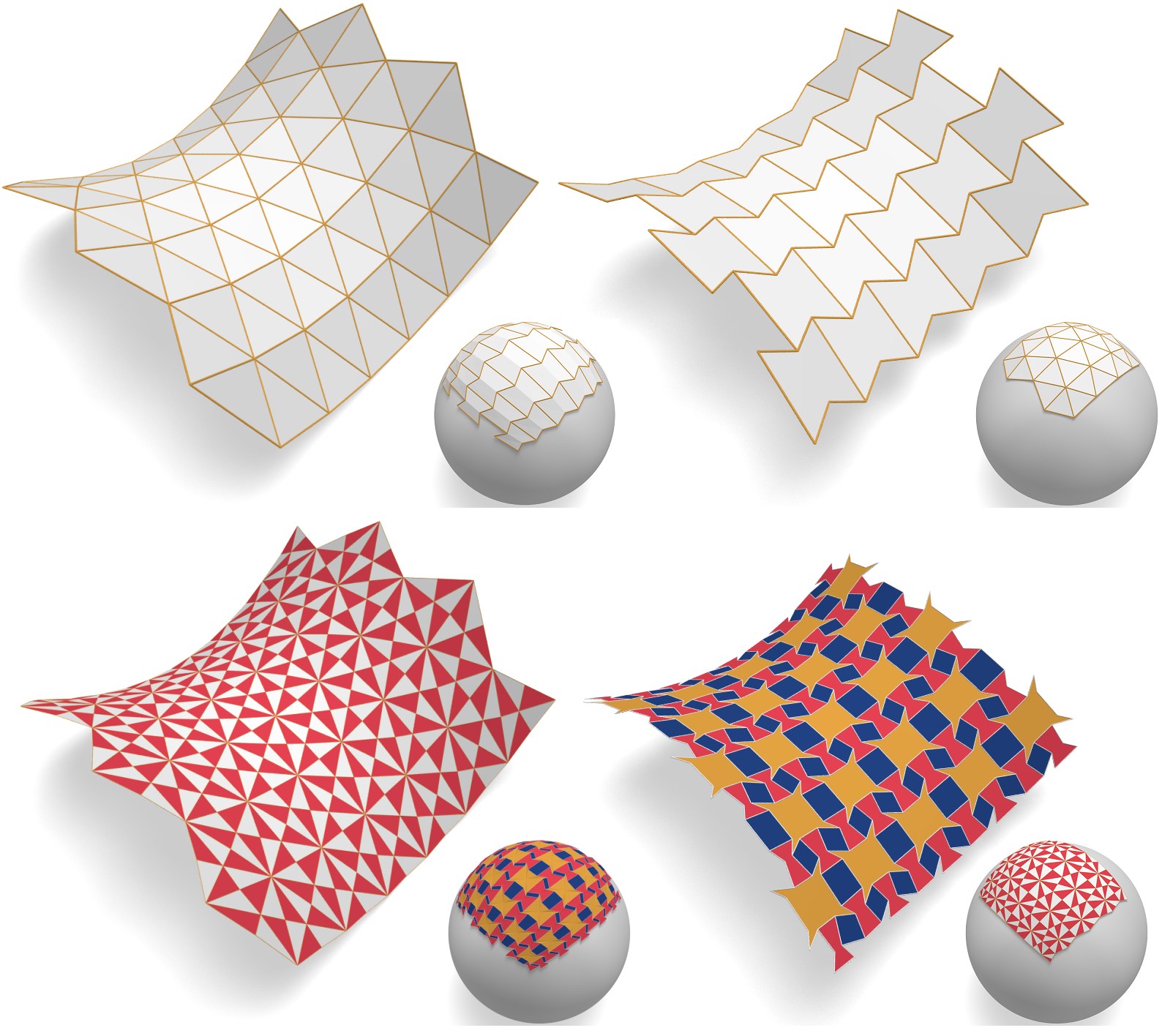}
   \caption{Dual pairs of negatively curved polyhedral surfaces and their Gauss images}
   \label{fig:dual}
	\end{center}
\end{figure}

\begin{theorem}\label{th:dual}
Let $P$ be a smooth polyhedral surface $P$ of either positive or negative discrete Gaussian curvature. If its Gauss image is contained in an open hemisphere, then it can be obtained by projecting a dual polyhedral surface $P^*$ to $S^2$. This dual surface is smooth, has the same sign of curvature as $P$, and the projection center agrees with the center $O$ of the Gaussian sphere $S^2$.
\end{theorem}
\begin{proof}
First of all, we choose a point $O$ that should not lie in any plane that is orthogonal to a point contained in the Gauss image of $P$ and passing through the corresponding vertex -- in some sense, $O$ should not be contained in any plane that is tangent to $P$. This is possible since the Gauss image is contained in a hemisphere: Let ${\vec{n}}$ be the pole of such a hemisphere and consider a ray $r$ with direction $-{\vec{n}}$ that starts in an interior point of a face of $P$. Due to our assumption, no plane tangent to $P$ can be parallel to $r$. Hence, a point $O$ far away on $r$ will satisfy the condition we are looking for. We choose $O$ to be the origin of the coordinate system and place the center of $S^2$ there.

By definition, the tangent planes $G_f$ of $S^2$ at Gauss image points ${\vec{n}}_f$ are parallel to the planes $P_f$ of the faces $f$ of $P$. Let us now apply the polarity $\Pi$ with respect to $S^2$. The polyhedral surface $P$ is mapped to a dual polyhedral surface $P^*$. The vertices $\vec{v}^*_f$ of $P^*$ are the polar images of the face planes $P_f$ (as $P_f$ does not pass through $O$, $\vec{v}^*_f$ is not at infinity). Tangent planes $G_f$ of $S^2$ are mapped via $\Pi$ to their contact points ${\vec{n}}_f$. Parallel planes are mapped to points on a line through the center of $O$ of $S^2$. Thus, for each face $f$, points $\vec{v}^*_f$, ${\vec{n}}_f$, and $O$ are collinear. On the line passing through these points, $O$ lies either in between the two points or not. Due to our assumption that $O$ is not contained in any plane that is orthogonal to a normal vector in the Gauss image of $P$ and passing through the corresponding vertex it follows that the position of $O$ with respect to the points $\vec{v}^*_f$ and ${\vec{n}}_f$ is the same for all faces $f$ of $P$. Hence, the Gauss image of $P$ is generated by projecting the polar mesh $P^*$ from $O$ onto $S^2$.

Since the interior angles of the Gauss image of $P$ add up to $2\pi$ at ${\vec{n}}_f$, the plane orthogonal to ${\vec{n}}_f$ and passing through $\vec{v}^*_f$ serves as a transverse plane for the star of $\vec{v}^*_f$. The faces of $P^*$ are star-shaped because the Gauss image of $P$ has star-shaped faces.

When we apply $\Pi$ again, $\Pi:P^* \to P^{**}=P$, we see that the Gauss image of $P^*$ is obtained by projecting $P$ from $O$ onto $S^2$. Using that the planes orthogonal to the line from $O$ to $\vec{v}$ and passing through that vertex serve as transverse planes for the star of $\vec{v}$, the interior angles of the Gauss images of $P^*$ add up to $2\pi$ at ${\vec{n}}_{{\vec{v}}^*}$. Furthermore, the faces of the Gauss image of $P^*$ are star-shaped since the faces of $P$ are.

Since the Gauss images of the vertices of $P^*$ have non-vanishing area, the discrete Gaussian curvature of $P^*$ is nowhere vanishing. When we consider the dual surface, the orientation changes if the curvature is negative and stays the same if it is positive. The dual of the dual surface has again the original orientation, so the sign of curvature has to be the same for $P$ and $P^*$.
\end{proof}

\begin{proposition}\label{prop:inflection_dual}
Let $P$ be a smooth polyhedral surface such that there exists a finite point that does not lie in any plane that is orthogonal to a normal vector contained in the Gauss image of $P$ and passing through the corresponding vertex of $P$. Let $P^*$ be its dual as in Theorem~\ref{th:dual}.

Then, if $f$ is an inflection face at the vertex $\vec{v}\sim f$, then the face of $P^*$ dual to $\vec{v}$ is an inflection face at the vertex $\vec{v}^*_f$ dual to $f$ (and vice versa).
\end{proposition}

\begin{proof}
The key of proving this statement lies in Lemma~\ref{lem:angle}. If the angle $\alpha_f$ of $f$ at $\vec{v}$ is a reflex angle, then we are in the setting of negative discrete Gaussian curvature by Theorem~\ref{th:shape}. Due to Theorem~\ref{th:dual}, $P^*$ is also negatively curved. Lemma~\ref{lem:angle} then shows that the angle $\alpha_{\vec{v}}^*$ of the face dual to $\vec{v}^*$ at $\vec{v}^*_f$ is less than or greater than $\pi$ if $f$ is an inflection face or not, respectively. Note that the angle $\alpha_{\vec{v}}^*$ does not coincide in general with the corresponding spherical angle at ${\vec{n}}_f$.

In the first case, we have a reflex angle $\alpha_{\vec{v}}^*$ whose dual angle is also a reflex angle. By Lemma~\ref{lem:angle}, this means that the face dual to $\vec{v}$ is an inflection face at $\vec{v}^*_f$. In the second case, when $\alpha_{\vec{v}}^*$ is convex and its dual angle is not, it follows from Lemma~\ref{lem:angle} that the face dual to $\vec{v}$ is not an inflection face at $\vec{v}^*_f$.

Let us now assume that $\alpha_f<\pi$. If $f$ is an inflection face, then we are again in the setting of negative curvature. Using that the dual angle is also convex, we conclude that the face dual to $\vec{v}$ is an inflection face at $\vec{v}^*_f$.

If $f$ is not an inflection face and the discrete Gaussian curvature is negative or positive, then the dual angle is a reflex or a convex angle, respectively. Since the dual of the dual angle is convex, we use again Lemma~\ref{lem:angle} to deduce that the face dual to $\vec{v}$ is not an inflection face at $\vec{v}^*_f$.
\end{proof}

Note that we can use Theorem~\ref{th:dual} in combination with Theorem~\ref{th:projective} to show that the projective dual surface of a smooth polyhedral surface of either positive or negative discrete Gaussian curvature whose Gauss image is contained in an open hemisphere is smooth for more general correlations than the one we used in Theorem~\ref{th:dual}.

\begin{theorem}\label{th:correlation}
Let $P$ be a smooth polyhedral surface of either positive or negative discrete Gaussian curvature. We assume that there exists a finite point that does not lie in any plane that is orthogonal to a normal vector contained in the Gauss image of $P$ and passing through the corresponding vertex of $P$. Let $\Gamma$ be a projective correlation that maps no such plane to a point at infinity. 

Then, the dual surface $P^*=\Gamma(P)$ is a smooth polyhedral surface that has the same sign of curvature as $P$. If $E$ is a discrete tangent plane in $\vec{v}$, then $\Gamma(E)$ is a point of contact of the face $\vec{v}^*$. Conversely, $\Gamma(A)$ is a discrete tangent plane in the vertex $\vec{v}^*_f$ dual to the face $f$ of $P$ if $A$ is a point of contact in $f$. Furthermore, $\Gamma$ maps discrete asymptotic directions to discrete asymptotic directions.
\end{theorem}
\begin{proof}
The assumption that $\Gamma$ does not map any plane tangent to $P$ to infinity guarantees that the dual surface $P^*$ is indeed a polyhedral surface that does not contain any point at infinity. Let $\Pi$ be the polarity that we constructed in Theorem~\ref{th:dual} and let $P_{\Pi}^*$ be the corresponding dual surface. $\Pi$ and $\Gamma$ then define a projective transformation $\pi$ that maps $P_{\Pi}^*$ to $P^*$.

Since $P_{\Pi}^*$ is a smooth polyhedral surface with the same sign of discrete Gaussian curvature as $P$ by Theorem~\ref{th:dual}, we can apply Theorem~\ref{th:projective} to obtain smoothness of $P^*$.

By Theorem~\ref{th:projective}, discrete tangent planes, points of contact, and discrete asymptotic directions are mapped to their counterparts. In particular, it is sufficient to show the duality statements for the surfaces $P$ and $P_{\Pi}^*$. In what follows, let $P^*=P_{\Pi}^*$

If $E$ is a discrete tangent plane in the vertex $\vec{v}$ of $P$, then the Gauss image $g({\vec{v}})$ is star-shaped with respect to the normal ${\vec{n}}\in g({\vec{v}})$ of $E$. It follows from our considerations in the proof of Theorem~\ref{th:projective} that the face $\vec{v}^*$ is star-shaped with respect to the intersection of the face with the line passing through ${\vec{n}}$ and the center of projection, so $\Pi(E)$ is indeed a point of contact.

Conversely, if $A$ is a point of contact in the face $f$ of $P$, $f$ is star-shaped with respect to $A$. If we apply the polarity $\Pi$, this means that the Gauss image of the star of the dual vertex $\vec{v}^*_f$ is star-shaped with respect to the normal of $\Pi(A)$ that lies inside the Gauss image. In particular, $\Pi(A)$ is a discrete tangent plane.

It remains to consider the discrete asymptotic directions, so we are in the case that both $P$ and $P_{\Pi}^*$ are negatively curved. Let $\vec{v}$ be an interior vertex of $P$. Then, we defined the discrete asymptotic directions in Section~\ref{sec:asymptotes} as the four directions defined by the four line segments in the discrete tangent plane with a disk neighborhood of $\vec{v}$.

In the case that the vertex star contains four inflection faces, these four line segments are exactly defined by the intersections of the discrete tangent plane with these four inflection faces. Under the polarity $\Pi$, the lines are mapped to the four lines in $\vec{v}^*$ each connecting a point of contact with one of the four vertices of $f$ that are dual to the four inflection faces in the star of $\vec{v}$. By Proposition~\ref{prop:inflection_dual}, $\vec{v}^*$ is an inflection face at these four vertices. It then follows from our discussion in Section~\ref{sec:negative} that the projective duals of the discrete asymptotic directions are again discrete asymptotic directions.

In the case that the vertex star contains only two inflection faces, two discrete asymptotic directions are the intersections of the discrete tangent plane with the two inflection faces and the other two asymptotic directions are collinear and are given by the intersection with the face $f'$ that has a reflex angle. Similar to the previous paragraph, the first two discrete asymptotic directions correspond to the lines connecting a point of contact with the two corners of $\vec{v}^*$ where the face is an inflection face. The latter two discrete asymptotic directions are both mapped to a line connecting the vertex $\vec{v}^*_{f'}$ dual to $f'$ with the point of contact. Since $f'$ is not an inflection face, the interior angle of the Gauss image of the star of $\vec{v}$ at ${\vec{n}}_{f'}$ is less than $\pi$ by Theorem~\ref{th:shape}~(iii). Thus, $\vec{v}^*_{f'}$ is a corner of  $\vec{v}^*$. Since we defined the line segment connecting the point of contact with that corner as two discrete asymptotic directions, we have shown that the discrete asymptotic directions in $\vec{v}$ correspond to the discrete asymptotic directions in $\vec{v}^*$.

Using the same ideas as above, it also follows that the discrete asymptotic directions in a face $f$ of $P$ correspond to the discrete asymptotic directions in the dual vertex $\vec{v}_f^*$.
\end{proof}


\section{Future research}\label{sec:future}

In our paper we did not address questions of convergence and approximation:
\begin{itemize}
\item Given a sequence of smooth polyhedral surfaces converging to a smooth surface, do the discrete normal vectors, discrete tangent planes, discrete asymptotic directions, and the discrete parabolic curves converge to their smooth counterparts?
\item Given a smooth surface $S$ and $\varepsilon>0$, is it possible to find a smooth polyhedral surface $P$ within a $\varepsilon$-neighborhood of $S$? Can $P$ be chosen to be a simplicial surface? As $\varepsilon \to 0$, do some shapes of vertex stars and faces shown in Fig.~\ref{fig:negative_face} not occur anymore?
\item In the previous papers \cite{JGWP16,JTVWP15}, only the regularization of a given mesh was investigated. What is missing is an algorithm which finds a good initial mesh to start with. As was shown in these papers, the combinatorics of the polyhedral surface and the initial realization of it may prevent a successful optimization. We expect that a careful comparison of smooth and discrete asymptotic directions will be essential for finding smooth polyhedral approximations of negatively curved surfaces.
\end{itemize}

\vspace{-0.5cm}
\section*{Acknowledgments}
The authors are grateful to Thomas Banchoff for fruitful discussions concerning the Gauss image of a vertex star and to G\"unter Rote for pointing out the connection to \cite{ORSSSW04}. This research was initiated during the first author's stay at the Erwin Schr\"odinger International Institute for Mathematical Physics in Vienna and continued during his stays at the Institut des Hautes \'Etudes Scientifiques in Bures-sur-Yvette and the Max Planck Institute for Mathematics in Bonn. The first author thanks the institutes for their hospitality and the European Post-Doctoral Institute for Mathematical Sciences for the opportunity to visit the afore mentioned research institutes. The first and last author are grateful for support by the DFG Collaborative Research Center TRR 109 ``Discretization in Geometry and Dynamics'' and corresponding FWF grants I 706-N26 and I 2978-N35. 

\bibliographystyle{plain}
\bibliography{Polyhedral_smooth}

\end{document}